\documentclass[11pt]{article}
\usepackage{lmodern} 
\usepackage{microtype}
\usepackage[top=25mm,bottom=25mm,inner=29mm,outer=29mm,headheight=14pt]{geometry}
\usepackage{amsmath,amssymb,amsthm,mathtools}
\usepackage{bm}%$\bm{\alpha}$  % 有效，且更稳健
\usepackage{enumitem}
\usepackage{xcolor} 
\usepackage{hyperref}
\usepackage[nameinlink,capitalise,noabbrev]{cleveref}
\crefname{equation}{}{}
\Crefname{equation}{}{}
\usepackage{authblk}
\usepackage{orcidlink}
\allowdisplaybreaks[3] 
\newtheorem{theorem}{Theorem}[section]
\newtheorem{lemma}[theorem]{Lemma}
\newtheorem{proposition}[theorem]{Proposition}
\newtheorem{corollary}[theorem]{Corollary}
\newtheorem{definition}[theorem]{Definition}

\newtheorem{remark}[theorem]{Remark}
\theoremstyle{definition}   
\theoremstyle{remark} 
\numberwithin{equation}{section} 
\title{\bfseries Boundary Layers and Initial–Boundary Corner Asymptotics for the 2D Navier–Stokes Equations with Vanishing Vertical Viscosity}
\author[1]{\rm Siwei Chen\orcidlink{0009-0008-2999-5255}}
\author[2]{\rm Yinghui Wang\orcidlink{0000-0002-7565-5525}\footnote{Corresponding author.\\ $\quad\quad\quad$ E-mail addresses: yhwangmath@hunnu.edu.cn, yhwangmath@163.com (Y. Wang), swchenmath@outlook.com (S. Chen), weihao-zhang@qq.com (W. Zhang).}}
\author[3]{\rm Weihao Zhang\orcidlink{0000-0002-1361-2108}}
\affil[1,2]{\footnotesize MOE-LCSM, School of Mathematics and Statistics, Hunan Normal University, Changsha, Hunan 410081, P. R. China}
\affil[3]{\footnotesize Zhuhai NO.1 High School, Zhuhai 519000, China.}
 
\date{}

\begin{document}
\maketitle
\begin{abstract}
Motivated by anisotropic viscosity in geophysical fluid dynamics, where vertical momentum diffusion is often much weaker than horizontal diffusion and may be negligible in the interior yet remains essential near a solid wall, we study the vanishing vertical viscosity limit for the two-dimensional incompressible Navier--Stokes equations in the upper half-plane, with horizontal viscosity fixed at one and vertical viscosity \(\varepsilon^2\). For arbitrary divergence-free \(H^4\) no-slip initial data, with no time-differentiated compatibility conditions required, we construct the limiting horizontally viscous flow and the boundary-layer profiles on every prescribed finite interval \([0,T]\). The leading-order corrected approximation has \(O(\varepsilon)\) error in \(L^\infty\), and the full finite-order expansion reduces this error to \(O(\varepsilon^{3/2})\), uniformly down to \(t=0\). We further describe the development of the layer from data that vanish initially at the wall. At the initial--boundary corner, we determine the first two self-similar coefficients and characterize matching through exact profile tails, thereby separating the fixed-time scale \(\varepsilon\) from the short-time scale \(\varepsilon\sqrt t\). When the initial wall acceleration is nonzero, both the leading correction rate and the first normalized corner remainder are sharp. Finally, continuity estimates for fixed data yield joint small-time and small-viscosity limits for the exact solution, as well as finite-\(L^p\) asymptotics.
\end{abstract}

\section{Introduction}

\subsection{Background and the limiting problem}
Anisotropic momentum diffusion is a common feature of many fluid models. In geophysical and multi-scale flows, unresolved mixing is often represented by effective, or eddy, viscosities whose horizontal and vertical strengths differ substantially \cite{CheminEtAlBook,Ericksen1960,Pedlosky_1978}. A weak vertical diffusion may be negligible in the the interior of fluids, yet it can remain essential near a solid boundary wall, where the no-slip condition must be enforced. This creates a singular perturbation problem in which a thin boundary layer adjusts the interior motion to the wall. The present paper studies this mechanism in the simplest anisotropic setting: the two-dimensional incompressible Navier--Stokes equations in the upper half-plane, with horizontal viscosity fixed and vertical viscosity tending to zero \cite{CheminEtAl2000}. More precisely, let
\( 
 \Omega:=\mathbb{R}_+^2=\{(x,y)\in\mathbb{R}^2:y>0\}\). For \(0<\varepsilon\leq1\), the velocity \(u^\varepsilon:=(u_1^\varepsilon,u_2^\varepsilon)(x,y,t)\) and the pressure \(p^\varepsilon:=p^\varepsilon(x,y,t)\) satisfy
\begin{equation}
\label{eq:1.1}
\left\{
\begin{aligned}
 &\partial_tu^\varepsilon+(u^\varepsilon \cdot \nabla) u^\varepsilon +\nabla p^\varepsilon-\partial_x^2u^\varepsilon-\varepsilon^2\partial_y^2u^\varepsilon&=0,\\[0.2cm]
 &\mathrm{div}u^\varepsilon=0,
\end{aligned}
\right.
\end{equation}
in \(\Omega\times(0,\infty)\). Here the horizontal viscosity is normalized to one, and \(\varepsilon^2\) measures the relative strength of vertical diffusion. The divergence-free condition $\mathrm{div}u^\varepsilon = \partial_xu_1^\varepsilon+\partial_yu_2^\varepsilon$ expresses incompressibility. We equip \eqref{eq:1.1} with the no-slip initial-boundary value conditions
\begin{equation}
\label{eq:1.2}
 u^\varepsilon(x,y,0)=\tilde u(x,y),
 \qquad
 u^\varepsilon(x,0,t)=0.
\end{equation} 
Formally, letting \(\varepsilon\to0\) in \eqref{eq:1.1}, one can obtain the following zero vertical viscosity model,
\begin{equation}
\label{eq:1.3}
\left\{
\begin{aligned}
 &\partial_tu^{I,0}+(u^{I,0}\!\cdot\nabla)u^{I,0}
 +\nabla p^{I,0}-\partial_x^2u^{I,0}=0,\\[0.2cm]
 &\operatorname{div}u^{I,0}=0,
\end{aligned}
\right.
\end{equation}
As usual, we supplement \eqref{eq:1.3} with the following initial-boundary conditions:
\begin{equation}
\label{eq:1.4}
    u^{I,0}(x,y,0)=\tilde u(x,y),\quad u^{I,0}_2(x,0,t)=0.
\end{equation}
The boundary condition in \eqref{eq:1.4} is the impermeability condition, which does not prescribe the tangential velocity. The latter is the boundary trace
\[
 a(x,t)=u^{I,0}_1(x,0,t).
\]
It need not remain zero for \(t>0\), even though the initial datum satisfies no slip and \(a(x,0)=0\). Indeed, evaluating the tangential component of \eqref{eq:1.3} at \(y=0\) gives the wall equation
\[
 a_t-a_{xx}+aa_x=-\partial_xp^{I,0}|_{y=0},
\]
so the tangential pressure gradient can accelerate the limiting flow at the wall. The wall acceleration
\[
 \alpha_0(x):=\partial_ta(x,0)=-\partial_xp^{I,0}(x,0,0)
\]
is therefore generally nonzero. The viscous velocity \(u^\varepsilon\), by contrast, must vanish on \(\{y=0\}\) in view of \eqref{eq:1.2}.
 Hence
\[
 u_1^\varepsilon(x,0,t)-u_1^{I,0}(x,0,t)=-a(x,t),
\]
and the tangential velocity defect at the wall is of order one in \(\varepsilon\) whenever \(a\not\equiv0\). This mismatch cannot be removed by the interior flow alone and is the source of the boundary layer.

The normal scale of the layer is dictated by the balance between vertical diffusion and time evolution. Across a strip of width \(\delta\), the vertical diffusion term \(\varepsilon^2\partial_y^2\) acts at the rate \(\varepsilon^2/\delta^2\). Balancing this rate with \(\partial_t\) on a fixed time interval gives \(\delta=O(\varepsilon)\), so we introduce the stretched normal variable
\[
 z=\frac{y}{\varepsilon},\qquad \Omega_{\mathrm b}:=\mathbb{R}_x\times(0,\infty)_z.
\]
On this scale, the leading boundary-layer profile satisfies a two-dimensional parabolic equation on \(\Omega_{\mathrm b}\) with tangential diffusion retained; see \eqref{eq:3.3} below. This is the same normal scale as in the classical Prandtl expansion for isotropic vanishing viscosity, but the presence of horizontal dissipation in the limiting system changes the structure of the layer and supplies the regularity needed to avoid analyticity or monotonicity assumptions on the data.

\subsection{Related results on vanishing viscosity limits and boundary layers}

When viscosity tends to zero in all directions, the interior equation is the Euler system and the leading correction is described by Prandtl's boundary-layer equations \cite{Prandtl1905}. Kato's criterion \cite{Kato1984} characterizes energy convergence to a smooth Euler flow through viscous dissipation in a strip whose width is proportional to the viscosity. Related criteria involve vorticity or selected velocity derivatives \cite{ConstantinEtAl2015,Kelliher2007,TemamWang1997,Wang2001}. Such criteria decide whether energy convergence occurs, but do not by themselves construct the pointwise correction or determine its shape; see \cite{MaekawaMazzucato2018} for a survey.

The construction and stability of the Euler--Prandtl expansion require additional control of the boundary dynamics. Sammartino and Caflisch \cite{SammartinoCaflischI,SammartinoCaflischII} justified the expansion for analytic data on a short time interval. Maekawa \cite{Maekawa2014} treated the case where the initial vorticity is supported away from the boundary. Wang, Wang, and Zhang \cite{WangWangZhang2017} developed an energy method based on the vorticity formulation and conormal derivatives to justify the zero-viscosity limit in the analytic setting. Subsequent results treat data analytic only near the wall and of Sobolev regularity in the interior \cite{KukavicaVicolWang2020,KukavicaEtAl2022}; a direct analytic-data approach to the inviscid limit is given in \cite{NguyenNguyen2018}. The classical monotone theory \cite{OleinikSamokhin1999,XinZhang2004}, Gevrey well-posedness results \cite{GerardVaretMasmoudi2015,DietertGerardVaret2019}, and Sobolev ill-posedness and instability mechanisms \cite{GerardVaretDormy2010,Grenier2000} illustrate the sensitivity of the Prandtl equation to its structural and regularity assumptions. The well-posedness of that equation and the stability of a Navier--Stokes expansion are separate issues.

Compared with the case in which viscosity vanishes in all directions, keeping the horizontal viscosity fixed while letting the vertical viscosity tend to zero yields completely different limiting equations and boundary-layer equations. For the Cauchy problem, the anisotropic framework was developed in \cite{CheminEtAl2000,Iftimie2002}; global two-dimensional results for the horizontally dissipative equation include \cite{LiangZhangZhu2021,CaoGuoHorizontal2025}. 

When considering the initial-boundary problem,
in contrast to the classical Prandtl equation, the leading layer considered here retains tangential diffusion. This supplies the dissipation needed for a Sobolev construction without analyticity or monotonicity. Nevertheless, the tangential trace of the interior flow is not constrained to vanish, so a no-slip boundary layer is still present. 
The closest predecessors already establish the vanishing vertical viscosity limit with a boundary-layer correction. Liu and Wang \cite{LiuWang2013} considered the three-dimensional half-space, derived a nonlinear parabolic--elliptic layer system, and justified the expansion in energy and supremum norms. In the two-dimensional half-plane, Tao \cite{Tao2018} proved corrected convergence in \(L^\infty(0,T;L^2\cap L^\infty)\), including an optimal convergence rate, for sufficiently smooth data. These works establish the boundary-layer mechanism and corrected convergence that form the starting point of the present study. 

A different regularity regime is addressed by Cao and Guo \cite{CaoGuo2025}. In a horizontally periodic finite channel, with no slip at the lower wall and impermeability and tangential free slip at the upper wall, they prove convergence for \(H^2\) initial data in every finite \(L^p\) space, \(2\le p<\infty\), on each fixed interval \([0,T]\). For identical initial data, their \(L^2\) rate is \(O(\nu_2^{1/8})\), where \(\nu_2\) is the vertical viscosity, hence \(O(\varepsilon^{1/4})\) when \(\nu_2=\varepsilon^2\). Their result has a lower initial regularity requirement, but does not give the boundary-layer-corrected \(L^\infty\) endpoint studied here. Anisotropic limits for plane-parallel, pipe-parallel, and circularly symmetric flows are considered in \cite{GalbiatiEtAl2026}; those reductions provide further information on viscosity-dependent rates and boundary layers under symmetry assumptions.

Our contribution is to resolve the regularity and initial-time structure of this pointwise expansion together. We work with arbitrary \(H^4\) no-slip data, whose spatial derivatives through order four are square integrable, and justify the approximation on any prescribed finite interval, including its initial endpoint. No smallness, symmetry, analyticity, monotonicity, or time-differentiated compatibility condition is imposed. We estimate the leading corrected approximation and the full finite expansion separately, identify the first two corner coefficients, and characterize matching through the exact tails of the leading profile. Thus the result combines finite-regularity pointwise accuracy with a quantitative description of how the layer forms and where the interior approximation becomes valid. This is the distinction from the smooth-data pointwise theory and the lower-regularity finite-\(L^p\) theory discussed above.
\subsection{The boundary layer and the initial corner}\label{subsec:corner-scales}

The correction must cancel the tangential trace \(a\) at the wall and decay toward the interior.  With \(a=u_1^{I, 0}(x, 0, t)\), the leading tangential profile satisfies  
\begin{equation}
\label{eq:3.3}
\left\{
\begin{aligned}
 &\partial_tu^{b,0}_1-\partial_x^2u^{b,0}_1-\partial_z^2u^{b,0}_1
 +(a+u^{b,0}_1)\partial_xu^{b,0}_1+u^{b,0}_1\partial_xa\\
 &\qquad-\left(z\partial_xa+\int_0^z\partial_xu^{b,0}_1(x,s,t)\,ds\right)
            \partial_zu^{b,0}_1=0,\\
 &u^{b,0}_1(x,0,t)=-a(x,t),\qquad
   \lim_{z\to\infty}u^{b,0}_1(x,z,t)=0,\\
 &u^{b,0}_1(x,z,0)=0.
\end{aligned}
\right.
\end{equation}
%The corresponding decaying normal profile is $u^{b,1}_2(x,z,t)=\int_z^\infty\partial_xu^{b,0}_1(x,s,t)\,ds$; thus $\partial_xu^{b,0}_1+\partial_zu^{b,1}_2=0$, and its physical amplitude is $\varepsilon$.
The leading approximation to the velocity is therefore
\[
 u^{I,0}(x,y,t)+\bigl(u^{b,0}_1(x,y/\varepsilon,t),0\bigr)^{\top}.
\]
The higher profiles restore incompressibility and the boundary condition at successive orders; their construction is given in Section~\ref{sec:Construction}.

At the regularity used below, the wall equation and $a(0)=0$ give $a(t)=t\alpha_0+O_{H_x^1}(t^2)$. Consequently, the boundary forcing starts at order \(t\), while normal diffusion spreads over a distance \(z=O(\sqrt t)\). This produces the corner variable
\[
 \eta=\frac{z}{\sqrt t}=\frac{y}{\varepsilon\sqrt t}.
\]
The fixed-positive-time layer has normal scale \(\varepsilon\), whereas its initial diffusion front has scale \(\varepsilon\sqrt t\). Here ``corner'' refers to the intersection of the initial surface \(t=0\) with the boundary \(y=0\), not to a spatial corner of the domain.

This onset differs from that caused by incompatible initial and boundary values, for which the leading short-time layer can have order-one amplitude; see \cite{CannoneEtAl2013,ArgenzianoEtAl2024}. Here the velocities agree initially, but their time derivatives need not satisfy the corresponding wall conditions. Higher-order compatibility conditions are obtained by differentiating the no-slip condition in time and using the equations to impose additional identities on the initial velocity and pressure; see \cite{Temam1982}. We do not impose these identities. Instead, we quantify the initial adjustment and control the velocity approximation on the closed interval \([0,T]\), even though higher derivatives may be singular at its initial endpoint.

\subsection*{Notation} 
We introduce the following notation conventions. We denote by $C$ a generic constant which may change from line to line, independent of $\varepsilon$ but dependent on $T$. Some other notations are defined as follows:
\begin{itemize}
	\item $A \lesssim B \Longleftrightarrow A \le C B.$
\end{itemize}
\begin{itemize}
	\item For a scalar-value function $p$, two vector-valued functions $u$ and $v$, we use the following notations:
	
	$(\nabla p)_i:=\partial_i p,(\nabla u)_{i j}:=\partial_j u_i ,u \cdot \nabla:=u_{1}\partial_{x}+u_{2}\partial_{y},(u \otimes v)_{i j}:=u_i v_j$. We set $\nabla^\perp=(\partial_y,-\partial_x)$, so that $\operatorname{curl}(\nabla^\perp\psi)=-\Delta\psi$.
\end{itemize}
\begin{itemize}
	\item $\langle\cdot\rangle:=\sqrt{1+|\cdot|^2}$.
\end{itemize}
\begin{itemize}
	\item $L_{xy}^{p}$ and $H_{xy}^{s}$ denote the usual Lebesuge and Sobolev space over $\mathbb{R}_{+}^2$ with corresponding norms $\|\cdot\|_{L_{x y}^p}$ and $\|\cdot\|_{H_{x y}^s}$, respectively.
\end{itemize}
\begin{itemize}
	\item Unsubscripted norms and the inner product $\langle\cdot,\cdot\rangle$ are taken in $L^2$ over the spatial variables of the functions involved.
\end{itemize}
\begin{itemize}
    \item The anisotropic Sobolev space is denoted as
    $$
    H_x^m H_y^{\ell}:=\left\{f \in L^2\left(\mathbb{R}_{+}^2\right) \mid \sum_{0 \leq i \leq m, 0 \leq j \leq \ell}\left\|\partial_x^i \partial_y^j f(x, y)\right\|_{L_{x y}^2}<\infty\right\},
    $$
    with norm $\|\cdot\|_{H_x^m H_y^{\ell}}$. Mixed norms involving $L^\infty$ retain the indicated order; for example,
    \[
    \|h\|_{H_x^mL_z^\infty}^2
      :=\sum_{j=0}^m\int_{\mathbb R}
          \operatorname*{ess\,sup}_{z>0}|\partial_x^jh(x,z)|^2\,dx.
    \]
\end{itemize}
\begin{itemize}
	\item $\bar{f}:=f(x, 0, t)$.
\end{itemize}
\begin{itemize}
	\item $z:=\frac{y}{\varepsilon}$ for $\varepsilon >0$. The notations $L_{xz}^{p}$ and $H_{xz}^{s}$ denote that their components are functions of $(x,z)$. For an inner profile $h$, we write $h^\varepsilon(x,y,t):=h(x,y/\varepsilon,t)$; in particular, $\partial_yh^\varepsilon=\varepsilon^{-1}(\partial_zh)^\varepsilon$.
\end{itemize}
\begin{itemize}
	\item For vector-valued functions, all $L^p$ norms use the Euclidean pointwise norm. For an $L^2$-based Hilbert space $X$, we write $\|(u,v)\|_X^2:=\|u\|_X^2+\|v\|_X^2$. The norm of $L^q(0,T;X)$, $1\leq q\leq\infty$, is denoted by $\|\cdot\|_{L_T^qX}$. When the time interval is fixed, $L_t^qX$ has the same meaning; shorter time intervals are written explicitly.
\end{itemize}
\begin{itemize}
    \item The complementary error function and Gamma function are normalized by
    $$
    \operatorname{erfc}(r):=\frac{2}{\sqrt{\pi}} \int_r^{\infty} e^{-s^2} \mathrm{~d} s, \quad \Gamma(q):=\int_0^{\infty} s^{q-1} e^{-s} \mathrm{~d} s \quad(q>0).
    $$
\end{itemize}
\begin{itemize}
    \item For $s\geq 0$ and $\ell \geq 0$, we use the mixed weighted norm
    $$
    \|f\|_{H_x^s L_{z, \ell}^2}:=\left\|\langle z\rangle^{\ell} f\right\|_{H_x^s L_z^2}.
    $$
\end{itemize}
\begin{itemize}
    \item For integers $N,\ell \geq 0$, the weighted energy used for a boundary-layer scalar $f$ is
    $$
    E_{N,\ell}[f](t) = \sum_{j=0}^{N}\left\|\langle z\rangle^{\ell+N-j} \partial_x^j f(t)\right\|_{L_{x, z}^2}^2, \quad D_{N,\ell}[f](t) = \sum_{j=0}^{N}\left\|\langle z\rangle^{\ell+N-j} \nabla_{x,z}\partial_x^j f(t)\right\|_{L_{x, z}^2}^2.
    $$
\end{itemize}
\begin{itemize}
	\item Let $\varphi$ be a smooth non-negative function defined on $[0,+\infty)$ satisfying
	\begin{equation}
		\label{eq:1.16}
		\varphi(0)=1, \varphi^{\prime}(0)=0, \varphi(s)=0 \text { for } s>1 .
	\end{equation}
\end{itemize}

\subsection{Main results}
Write \(L^2_\sigma(\Omega)\) for the \(L^2\) closure of smooth compactly supported divergence-free fields in \(\Omega\). 
\begin{proposition}[Solutions at fixed vertical viscosity $\varepsilon^2>0$]\label{prop:exact-solution}
Let \(\varepsilon>0\) and let \(\tilde{u}\in H^4(\Omega)^2\cap H^1_0(\Omega)^2\) be divergence free. The problem \eqref{eq:1.1}--\eqref{eq:1.2} has a unique global  solution satisfying, for every \(T>0\),
\[
 u^\varepsilon\in C\bigl([0,T];H^2(\Omega)^2\cap H_0^1(\Omega)^2\cap L^2_\sigma(\Omega)\bigr).
\]
For every \(0<\delta<T\), it is a strong solution with
\begin{align}
 u^\varepsilon&\in L^\infty(\delta,T;H^3)\cap L^2(\delta,T;H^4),\notag\\
 \partial_tu^\varepsilon&\in L^\infty(\delta,T;H^1)\cap L^2(\delta,T;H^2).
 \label{r6:finite-viscous-regularity}
\end{align}
\end{proposition}
\begin{remark}
    No time-differentiated compatibility condition is required in Proposition \ref{prop:exact-solution}. The strong-regularity bounds may depend on \(\varepsilon\), \(\delta\), \(T\), and \(\tilde{u}\). The proof of Proposition \ref{prop:exact-solution} is standard, one can refer to \cite{Sohr2001} for some related discuss. 
\end{remark} 
\begin{theorem}[Pointwise vanishing vertical viscosity limit]\label{thm:main}  
  Let the initial datum satisfy the assumptions of Proposition
\ref{prop:exact-solution}, and fix $T>0$. For every $0<\varepsilon\leq1$,
problems \eqref{eq:1.1}--\eqref{eq:1.2},
\eqref{eq:1.3}--\eqref{eq:1.4}, and \eqref{eq:3.3}
admit velocities $u^\varepsilon$, $u^{I,0}$, and $u_1^{b,0}$ on $[0,T]$,
respectively, unique in the classes specified in Propositions
\ref{prop:exact-solution}, \ref{prop:outer-flow}, and
\ref{pro:boundary-layer equation}.  The viscous velocity has the regularity
stated in Proposition \ref{prop:exact-solution}, and
\[
 \begin{aligned}
 &u^{I,0}\in C([0,T];H^4(\Omega)^2),\quad
       \partial_xu^{I,0}\in L^2(0,T;H^4(\Omega)^2),\\
 &u_1^{b,0}\in C([0,T];L^2(\Omega_{\mathrm b}))
       \cap L^2(0,T;H^1(\Omega_{\mathrm b})).
 \end{aligned}
\]
The associated pressures are unique up to functions of time. The higher profiles constructed from \(\tilde{u}\) give the divergence-free, no-slip approximation \(u^a\) in \eqref{eq:3.70}, with
\begin{gather}
 \sup_{0\le t\le T}
 \left\|u^\varepsilon(t)-u^{I,0}(t)
       -\bigl(u_1^{b,0}(x,y/\varepsilon,t),0\bigr)\right\|_{L^\infty(\Omega)}
  \le C_T\varepsilon,
 \label{eq:leading-uniform-main}\\*
 \sup_{0\le t\le T}\|u^\varepsilon(t)-u^a(t)\|_{L^\infty(\Omega)}
  \le C_T\varepsilon^{3/2}.
 \label{eq:full-error-uniform}
\end{gather}
The constant \(C_T\) depends only on \(T\) and \(\|\tilde{u}\|_{H^4}\), with the cutoff fixed. Both bounds hold on the closed interval \([0,T]\), and both errors vanish at \(t=0\).
\end{theorem} 

\begin{remark}
No smallness or additional time-differentiated
compatibility condition is imposed in Theorem \ref{thm:main}.
 If \(\alpha_0\not\equiv0\), the order-\(\varepsilon\) error in \eqref{eq:leading-uniform-main} is sharp at sufficiently small fixed positive times; see Corollary \ref{cor:corrected-rate-sharp}.  
\end{remark}

\begin{remark}[The role of the correction]\label{rem:corrector-necessary}
At the boundary, \(u_1^\varepsilon-u_1^{I,0}=-u_1^{I, 0}(x, 0, t)\). Hence uniform convergence to the interior flow alone fails whenever \(a\) is nonzero. The tangential layer cancels this order-one mismatch in \(\varepsilon\); the normal layer has amplitude \(\varepsilon\) and is included in the error in \eqref{eq:leading-uniform-main}.
\end{remark}

\begin{remark}[The initial endpoint]\label{rem:main-time-weight}
The uniform error estimates up to \(t=0\) do not require unweighted high-order parabolic regularity there. The sharper initial continuity modulus \eqref{r6:first-correction-onset} for the first correction also permits simultaneous limits \(t\downarrow0\) and \(\varepsilon\downarrow0\). This modulus is associated with a fixed initial datum, not asserted uniformly over bounded \(H^4\) sets.
\end{remark}

\begin{definition}[Thickness on a fixed time interval]
\label{def:thickness}
 Under the assumptions of Theorem \ref{thm:main},   write \(\Omega_\delta=\{(x,y)\in\Omega:y\geq\delta\}\). A nonnegative scale \(\delta_\varepsilon\to0\) as $\varepsilon\downarrow0$ is called an admissible thickness on \([0,T]\) if
\[
 \liminf_{\varepsilon\downarrow0}
   \left\lVert u^\varepsilon-u^{I,0}\right\rVert_{{L^\infty([0,T]\times\Omega)}}>0,
 \qquad
 \left\lVert 
 u^\varepsilon-u^{I,0}\right\rVert_{{L^\infty([0,T]\times\Omega_{\delta_\varepsilon})}}\longrightarrow0.
\]
\end{definition}

\begin{theorem}[Exact tails and the matching scale on a fixed time interval]
\label{thm:thickness}
Under the assumptions of Theorem \ref{thm:main},  define
\[
 F_T(\rho)
 =\sup_{0\le t\le T}
     \|u^{b,0}_1(t)\|_{L^\infty(\mathbb R\times[\rho,\infty))}.
\]
For every $\delta_\varepsilon\ge0$,
\begin{equation}
 \left|\|u_1^\varepsilon-u^{I,0}_1\|_{L^\infty([0,T]\times\Omega_{\delta_\varepsilon})}
       -F_T(\delta_\varepsilon/\varepsilon)\right|\le C_T\varepsilon .
 \label{eq:interval-tail}
\end{equation}
The function $F_T$ is continuous, nonincreasing, and tends to zero at infinity. Hence $\delta_\varepsilon/\varepsilon\to\infty$ implies uniform convergence to $u^{I,0}$ outside the thin strip. If $\delta_\varepsilon/\varepsilon\to\ell<\infty$, the norm in \cref{eq:interval-tail} converges to $F_T(\ell)$. Under the same condition $\delta_\varepsilon/\varepsilon\to\ell\in[0,\infty)$, for each fixed $t_0\in[0,T]$,
\begin{equation}
 \lim_{\varepsilon\downarrow0}
 \|u_1^\varepsilon(t_0)-u^{I,0}_1(t_0)\|_{L^\infty(\mathbb R\times[\delta_\varepsilon,\infty))}
 =\|u^{b,0}_1(t_0)\|_{L^\infty(\mathbb R\times[\ell,\infty))}.
 \label{eq:tail-limit}
\end{equation}
If the leading boundary layer is nonzero on $[0,T]$, a strip of width $o(\varepsilon)$ cannot remove its uniform defect. For a finite limiting ratio $\ell$, convergence holds if and only if $F_T(\ell)=0$.
\end{theorem}

\begin{remark}[Meaning of the matching scale]
For a nonzero layer, a strip of width $o(\varepsilon)$ cannot remove the uniform defect, while a vanishing width $\delta_\varepsilon$ with $\delta_\varepsilon/\varepsilon\to\infty$ does remove it. At a finite limiting ratio \(\ell\), the exact condition is \(F_T(\ell)=0\); continuity and decay alone do not imply a nonzero tail beyond every finite \(\ell\). The matching length is \(\varepsilon\), not the viscosity coefficient \(\varepsilon^2\).
\end{remark}

For the corner estimates, $x\in\mathbb R$ and $\eta\ge0$. We use the mixed norms
\begin{equation}
\begin{aligned}
 \|f\|_{H_\eta^1(H_x^1)}^2
 &=\sum_{j,k=0}^1
    \|\partial_x^j\partial_\eta^k f\|_{L^2(\mathbb R\times\mathbb R_+)}^2,\\
 \|f\|_{L_\eta^\infty(H_x^1)}
 &=\operatorname*{ess\,sup}_{\eta\ge0}\|f(\cdot,\eta)\|_{H^1(\mathbb R)}.
\end{aligned}
\label{eq:corner-mixed-norms}
\end{equation}

\begin{theorem}[The first corner profile]
\label{thm:corner}
Under the assumptions of Theorem \ref{thm:main}, the initial wall acceleration
\[
 \alpha_0(x)=\partial_tu_1^{I,0}(x,0,0)\in H^3(\mathbb R)
\]
is well defined. Set
\begin{equation}
 \Phi_1(\eta)=
 \left(1+\frac{\eta^2}{2}\right)\operatorname{erfc}\!\left(\frac\eta2\right)
 -\frac{\eta}{\sqrt\pi}\mathrm{e}^{-\eta^2/4}.
\label{eq:Phi1-explicit}
\end{equation}
Then there exist \(0<t_c\le\min\{1,T\}\) and \(C>0\) such that
\begin{equation}
 \left\lVert 
 t^{-1}u^{b,0}_1(x,\sqrt t\,\eta,t)+\alpha_0(x)\Phi_1(\eta)\right\rVert_{H_\eta^1(H_x^1)}
 \leq Ct,
 \qquad 0<t\leq t_c.
\label{eq:corner-limit}
\end{equation}
Consequently,
\begin{equation}
 \left\lVert 
 t^{-1}u^{b,0}_1(x,\sqrt t\,\eta,t)+\alpha_0(x)\Phi_1(\eta)\right\rVert_{L_\eta^\infty(H_x^1)}
 +\left\lVert 
 t^{-1}u^{b,0}_1(x,\sqrt t\,\eta,t)+\alpha_0(x)\Phi_1(\eta)\right\rVert_{L^\infty_{x,\eta}}
 \leq Ct.
\label{eq:corner-limit-Linf}
\end{equation}
Thus the first corner term is $-t\alpha_0(x)\Phi_1(z/\sqrt t)$, with physical scale $y=O(\varepsilon\sqrt t)$. The normalized profile is defined for $t>0$ and converges to \(-\alpha_0\Phi_1\) as \(t\downarrow0\); the unnormalized profile has zero initial value.
\end{theorem}

\begin{remark}[Profile asymptotics and the exact solution]\label{rem:profile-versus-exact}
The order-\(t\) amplitude reflects \(a(0)=0\) and a possibly nonzero \(a_t(0)\). The estimate \eqref{eq:corner-limit} describes the boundary-layer profile, not yet the exact viscous solution in a joint limit. That passage additionally uses the initial modulus of the first correction. For a fixed datum, sufficient ranges are \(t_\varepsilon\ge c\varepsilon^{4/3}\) for the first coefficient and \(t_\varepsilon\ge c\varepsilon^{4/7}\) for the second, with \(t_\varepsilon\to0\) and fixed \(c>0\); see Corollaries \ref{cor:exact-corner-coupled} and \ref{cor:r7-exact-second-corner}. These ranges are not asserted to be necessary.
\end{remark}

\begin{corollary}[Second corner coefficient and optimality of the normalized remainder]
\label{cor:second-corner}
Under the assumptions of Theorem \ref{thm:main}, the second wall acceleration $\beta_0=a_{tt}(0)$ belongs to $H_x^1$. Let $\Phi_1$ be as in \eqref{eq:Phi1-explicit}, and let $\Phi_2$ be the positive decaying solution of
\[
 \Phi_2''+\tfrac\eta2\Phi_2'-2\Phi_2=0,\qquad
 \Phi_2(0)=1,\qquad \lim_{\eta\to\infty}\Phi_2(\eta)=0.
\]
Set
\begin{equation}
 \Psi_2(x,\eta)=\tfrac12\beta_0(x)\Phi_2(\eta)
      +\alpha_{0,xx}(x)(\Phi_1(\eta)-\Phi_2(\eta)).
 \label{r6:second-corner-coefficient}
\end{equation}
Then
\begin{equation}
 u_1^{b,0}(x,\sqrt t\eta,t)=-t\alpha_0(x)\Phi_1(\eta)
                    -t^2\Psi_2(x,\eta)+o_{H_\eta^1(H_x^1)}(t^2).
 \label{r6:second-corner-expansion}
\end{equation}
If $\alpha_0\not\equiv0$, then $\Psi_2\not\equiv0$, and for both $Y=H_\eta^1(H_x^1)$ and $Y=L^\infty_{x,\eta}$,
\begin{equation}
 \lim_{t\downarrow0}t^{-1}
 \|t^{-1}u_1^{b,0}(x,\sqrt t\eta,t)+\alpha_0\Phi_1\|_Y=\|\Psi_2\|_Y>0.
 \label{r6:sharp-corner-remainder}
\end{equation}
In particular, $t^{-2}[u_1^{b,0}(x,\sqrt t\eta,t)+t\alpha_0\Phi_1]\to-\Psi_2$.
\end{corollary}

\begin{remark}[The corner coefficient]
\label{rem:corner-coefficient}
Evaluating the tangential equation in \cref{eq:1.3} at \(y=t=0\) gives
\[
 \alpha_0(x)=-\partial_xp^{I,0}(x,0,0),
\]
because the initial no-slip trace and all its tangential derivatives vanish. Thus \(\alpha_0\) is determined nonlocally by the initial pressure. In Proposition \ref{prop:Gaussian-data}, we construct smooth divergence-free initial data for which \(\alpha_0\not\equiv0\). If $\alpha_0=0$, the order-$t$ corner profile identified above vanishes. The theorem then gives
$\|u^{b,0}_1(x,\sqrt t\,\eta,t)\|_{H_\eta^1(H_x^1)}\le Ct^2$.
The estimate alone does not imply a nonzero higher-order coefficient.

In \cref{r6:second-corner-coefficient}, the term $\alpha_{0,xx}(\Phi_1-\Phi_2)$ records horizontal diffusion during the wall heat history. The second coefficient therefore depends on both that diffusion and the second wall acceleration $\beta_0$. When $\alpha_0\not\equiv0$, the corollary proves $\Psi_2\not\equiv0$ without a separate nondegeneracy assumption on $\beta_0$.
\end{remark}

\begin{remark}[The corner estimate in physical variables]
\label{rem:corner-physical}
After the common rescaling \(z=\sqrt t\,\eta\), the error in \cref{eq:corner-limit-Linf} is pointwise \(O(t^2)\). The leading term has pointwise amplitude \(t\), normal \(L^2_z\) size \(t^{5/4}\), and normal-derivative size \(t^{3/4}\). Under \(y=\varepsilon z\), these quantities yield the exact powers of \(\varepsilon\) and \(t\) listed in \cref{eq:corner-physical-Lp}, \cref{eq:corner-physical-gradient}. Together with \(\alpha_0\neq0\), these identities give upper and lower bounds for the corner width of the leading profile. The transfer of these conclusions to the exact tangential defect in a simultaneous limit is stated separately in Corollary \ref{cor:exact-corner-coupled}.
\end{remark}

\subsection{Main ideas}\label{sec:main-ideas}

The main difficulty is to control the approximation as both the vertical viscosity and the time from initialization become small. Normal derivatives grow as the layer becomes thinner, while repeated time differentiation at the wall would impose the compatibility conditions excluded here. Three features of the equations allow the expansion to be controlled at the \(H^4\) level. First, the pressure-driven wall equation supplies more tangential regularity than a direct bulk trace estimate. Starting from \(a(0)=0\), horizontal heat estimates applied to
\[
 a_t-a_{xx}+aa_x=-\partial_xp^{I,0}|_{y=0}
\]
give the wall derivatives needed by the profile hierarchy. At the highest order, one horizontal derivative is transferred to the heat dissipation rather than demanded from the initial datum. The interior corrections obey analogous wall equations.

Second, the leading layer has a cancellation specific to two dimensions. Set
\[
 v=a+u_1^{b,0},\qquad
 w=-\int_0^z\partial_xv(x,s,t)\,ds,\qquad q=\partial_zv.
\]
Then $\partial_xv+\partial_zw=0$, and the shear satisfies
\[
 q_t-\Delta_{x,z}q+v\partial_xq+w\partial_zq=0,
 \qquad \partial_zq|_{z=0}=\partial_xp^{I,0}|_{y=0}.
\]
There is no stretching term. We first control the velocity dissipation, then the weighted shear and its tangential derivatives, and finally recover the velocity by integration in $z$. The full normal moments
$M_j=\int_0^\infty u_1^{b,j}\,dz$ determine the boundary data of the interior corrections through $u_2^{b,j+1}|_{z=0}=\partial_xM_j$. A final divergence-free lifting makes the approximation exactly no-slip.

Third, the remainder $E^\varepsilon=u^\varepsilon-u^a$ is estimated by spatial energies for
$E^\varepsilon$, $\partial_xE^\varepsilon$,
$\varepsilon\partial_yE_1^\varepsilon$, and
$\varepsilon\partial_x\partial_yE_1^\varepsilon$.
The pressure boundary terms are kept in these estimates; the velocity dissipation controls the pressure quantities needed for the shear energies.  For each fixed positive viscosity, the strong solution is continuous at the initial time in $H^2$. The zero-data heat estimates give the corresponding initial continuity of the mixed profile derivatives. These facts allow the spatial error estimates to pass to $t=0$ without time-differentiated boundary compatibility.

For the corner expansion, we separate the leading profile into the heat lifting of the complete wall history and a higher-order nonlinear remainder. The first two coefficients arise from the wall acceleration and its time derivative, together with horizontal diffusion. Initial continuity of the higher profiles then transfers these asymptotics to the exact solution in the joint limits stated above.

The rest of the paper is organized as follows.
Section \ref{sec:Preliminaries} collects the analytic tools.
Section \ref{sec:Construction} constructs the profiles and the composite
approximation. Section \ref{sr:finite-criterion} justifies the vanishing
vertical viscosity limit. Section \ref{sec:thickness-sharpness} establishes
the tail identities, corner expansions, sharpness, and finite-$L^p$
asymptotics. The appendices

\section{Preliminaries}\label{sec:Preliminaries}
In this section, we recall some useful results which will be used later. We begin with the following inequalities.
\begin{lemma}[Hardy's inequality {\cite{HardyLittlewoodPolya}}]
\label{lemma:2.1}
Let $1<p<\infty$ and $f\in L^p(0,\infty)$. Then
\begin{equation}
\label{eq:2.1}
 \int_0^\infty
 \left|\frac1y\int_0^y f(t)\,dt\right|^p\,dy
 \leq C_p\int_0^\infty |f(y)|^p\,dy.
\end{equation}
\end{lemma}
\begin{lemma}[\cite{CaoWu2011,WangWen2024}]
	\label{lemma:2.2}
	Assume that $f,g,\partial_{y}g,h,\partial_{x}h \in L^{2}(\mathbb{R}_{+}^{2})$. Then the inequality
	\begin{equation}
		\label{eq:2.2}
		\iint_{\mathbb{R}_{+}^{2}}|f g h| d x d y \leqslant C\|f\|_{L^{2}}\|g\|_{L^{2}}^{\frac{1}{2}}\left\|\partial_{y}g\right\|_{L^{2}}^{\frac{1}{2}}\|h\|_{L^{2}}^{\frac{1}{2}}\left\|\partial_{x}h\right\|_{L^{2}}^{\frac{1}{2}} ,
	\end{equation}
	hold.
\end{lemma} 
\begin{lemma}
\label{lem:one-dimensional-Sobolev}
Let \(I\) be either \(\mathbb{R}\) or \(\mathbb{R}_+\), and let \(X\) be a real Hilbert space. If \(f\in H^1(I;X)\), then its continuous representative satisfies
\begin{equation}
 \left\lVert f\right\rVert_{L^\infty(I;X)}^{2}
 \leq C\left\lVert f\right\rVert_{L^2(I;X)}\left\lVert \partial_sf\right\rVert_{L^2(I;X)}.
\label{eq:one-dimensional-Sobolev}
\end{equation}
In particular, for a scalar-valued function \(f\),
\begin{equation}
 \left\lVert f\right\rVert_{L^4(I)}
 \leq C\left\lVert f\right\rVert_{L^2(I)}^{3/4}
                    \left\lVert \partial_sf\right\rVert_{L^2(I)}^{1/4}.
\label{eq:one-dimensional-L4}
\end{equation}
The numerical constants $C$ above are independent of the choice of \(I\) and of \(f\).
\end{lemma} 
\begin{lemma}
\label{lem:one-dimensional-H2-interpolation}
If \(f\in H^2(\mathbb{R}_+)\), then
\begin{equation}
 \left\lVert \partial_zf\right\rVert_{L^2(\mathbb{R}_+)}^{2}
 \leq C\left\lVert f\right\rVert_{L^2(\mathbb{R}_+)}
 \left(\left\lVert \partial_z^2f\right\rVert_{L^2(\mathbb{R}_+)}+\left\lVert f\right\rVert_{L^2(\mathbb{R}_+)}\right).
\label{eq:one-dimensional-H2-interpolation}
\end{equation}
The same estimate holds on \(\mathbb{R}\), in which case the last term
\(\left\lVert f\right\rVert_{L^2(\mathbb{R})}\) in the parentheses can be omitted. No trace condition is imposed a priori on either \(f\) or \(\partial_z f\).
\end{lemma}

\begin{lemma}
\label{lem:two-dimensional-Sobolev}
Let \(Q\) be \(\mathbb{R}^2\), \(\Omega\), or \(\Omega_{\mathrm b}\). If \(f\in H^1(Q)\), then, for \(2\leq p<\infty\),
\begin{equation}
 \left\lVert f\right\rVert_{L^p(Q)}
 \leq C_p\left\lVert f\right\rVert_{L^2(Q)}^{2/p}
              \left\lVert \nabla f\right\rVert_{L^2(Q)}^{1-2/p}.
\label{eq:two-dimensional-GN}
\end{equation}
In particular, the Ladyzhenskaya inequality takes the form
\begin{equation}
 \left\lVert f\right\rVert_{L^4(Q)}^{2}
 \leq C\left\lVert f\right\rVert_{L^2(Q)}\left\lVert \nabla f\right\rVert_{L^2(Q)}.
\label{eq:Ladyzhenskaya-prelim}
\end{equation}
If \(f\in H^2(Q)\), then
\begin{equation}
 \left\lVert f\right\rVert_{L^\infty(Q)}\leq C\left\lVert f\right\rVert_{H^2(Q)}.
\label{eq:H2-Linfty-prelim}
\end{equation}
\end{lemma}

\begin{lemma}
\label{lem:wall-trace-prelim}
Let \(\Omega=\mathbb{R}\times\mathbb{R}_+\), with \(r\) denoting the normal variable. If \(f,\partial_r f\in L^2(\Omega)\), then the trace of \(f\) on the boundary \(r=0\) is well defined in \(L_x^2(\mathbb{R})\) and satisfies
\begin{equation}
 \left\lVert f(\cdot,0)\right\rVert_{L_x^2}^{2}
 \leq 2\left\lVert f\right\rVert_{L^2(\Omega)}\left\lVert \partial_r f\right\rVert_{L^2(\Omega)}.
\label{eq:wall-trace-L2}
\end{equation}
More generally,
\begin{equation}
 \|f(\cdot,0)\|_{H_x^m}^2
 \leq 2\|f\|_{L_r^2H_x^m}\|\partial_r f\|_{L_r^2H_x^m},
 \qquad m\in\mathbb N_0.
\label{eq:wall-trace-general}
\end{equation}
In particular,
\[
 \|f(\cdot,0)\|_{H_x^m}
 \leq C_m\|f\|_{H^{m+1}(\Omega)}.
\]
\end{lemma}

\begin{lemma}
\label{lem:Sobolev-products-prelim}
Let \(Q\) be \(\mathbb{R}^2\), \(\Omega\), or \(\Omega_{\mathrm b}\), and let \(m\geq1\) be an integer. For smooth functions \(f,g\) and every spatial multi-index \(\alpha\) with \(|\alpha|\leq m\), we have
\begin{align}
 \left\lVert \partial^\alpha(fg)\right\rVert_{L^2(Q)}
 &\leq C_m\left(
   \left\lVert f\right\rVert_{L^\infty(Q)}\left\lVert g\right\rVert_{H^m(Q)}
  +\left\lVert g\right\rVert_{L^\infty(Q)}\left\lVert f\right\rVert_{H^m(Q)}\right),
 \label{eq:tame-product-prelim}\\
 \left\lVert \partial^\alpha(fg)-f\partial^\alpha g\right\rVert_{L^2(Q)}
 &\leq C_m\left(
   \left\lVert \nabla f\right\rVert_{L^\infty(Q)}\left\lVert g\right\rVert_{H^{m-1}(Q)}
  +\left\lVert f\right\rVert_{H^m(Q)}\left\lVert g\right\rVert_{L^\infty(Q)}\right).
 \label{eq:tame-commutator-prelim}
\end{align}
In particular, since \(H^2(Q)\hookrightarrow L^\infty(Q)\), it follows that
\begin{equation}
 \left\lVert fg\right\rVert_{H^m(Q)}
 \leq C_m\left\lVert f\right\rVert_{H^m(Q)}\left\lVert g\right\rVert_{H^m(Q)},
 \qquad m\geq2.
\label{eq:Hm-algebra-prelim}
\end{equation}
The constants depend only on \(m\) and the fixed domain \(Q\), and are independent of the functions \(f,g\) and the parameter \(\varepsilon\).
\end{lemma}

\begin{lemma}
\label{lem:time-interpolation-prelim}
Let \(X\) be a Hilbert space, let \(t_0<t_1\), and suppose that
\(h,\partial_t h\in L^2(t_0,t_1;X)\). Then
\begin{equation*}
 \sup_{t_0\leq t\leq t_1}\left\lVert h(t)\right\rVert_{X}^{2}
 \leq \frac{2}{t_1-t_0}\left\lVert h\right\rVert_{L^2(t_0,t_1;X)}^{2}
 +2\left\lVert h\right\rVert_{L^2(t_0,t_1;X)}
    \left\lVert \partial_t h\right\rVert_{L^2(t_0,t_1;X)}.
\end{equation*}
The constant depends only on the length of the time interval.
\end{lemma}

\begin{lemma}[\cite{Sohr2001}]
\label{lem:divcurl}
Let \(s\geq0\) be an integer. Suppose that
\(v\in H^{s+1}(\Omega)^2\), \(\operatorname{div} v=0\), and
\(v_2|_{y=0}=0\), and assume that \(v\to0\) at spatial infinity.
If
\[
 \omega=\partial_xv_2-\partial_yv_1,
\]
then
\begin{equation}
 \left\lVert v\right\rVert_{H^{s+1}(\Omega)}
 \leq C_s\bigl(\left\lVert v\right\rVert_{L^2(\Omega)}
 +\left\lVert \omega\right\rVert_{H^s(\Omega)}\bigr).
\label{eq:divcurl}
\end{equation}
\end{lemma}

\begin{lemma}[\cite{Sohr2001}]
\label{lem:r7-stokes}
Fix $\varepsilon>0$ and set $V_\sigma=H_0^1(\Omega)^2\cap L^2_\sigma(\Omega)$. Let $\mathbb P$ be the orthogonal projection onto $L^2_\sigma(\Omega)$. The positive self-adjoint operator on $L^2_\sigma(\Omega)$ associated with the quadratic form of domain $V_\sigma$
\[
 \|v_x\|_2^2+\varepsilon^2\|v_y\|_2^2
\]
has domain and action given by
\[
 D(A_\varepsilon)=H^2(\Omega)^2\cap V_\sigma,\qquad
 A_\varepsilon v=-\mathbb P(\partial_x^2+\varepsilon^2\partial_y^2)v.
\]
If \(v\in D(A_\varepsilon)\) and \(A_\varepsilon v\in H^j\), then
\begin{equation}
 \|v\|_{H^{j+2}}\le C_{\varepsilon,j}
       \bigl(\|A_\varepsilon v\|_{H^j}+\|v\|_2\bigr),
 \qquad j=0,1,2.
 \label{r7:stokes-elliptic}
\end{equation}
Moreover, the following homogeneous Hessian estimate holds:
\begin{equation}
 \|\nabla^2v\|_2\le C_\varepsilon\|A_\varepsilon v\|_2,
 \qquad v\in D(A_\varepsilon).
 \label{r8:homogeneous-stokes}
\end{equation}
The constants can be chosen uniformly when \(\varepsilon\) ranges over a compact subset of \((0,\infty)\).
\end{lemma}

\begin{lemma}
\label{lem:zero-data-weighted-heat}
Fix a nonnegative integer \(m\) and an integer \(r\geq1\). Let \(f\) be an energy solution on the half-plane
\(\Omega_{\mathrm b}\) of
\[
 f_t-\Delta f=H,\qquad
 f|_{z=0}=0,\qquad
 f(0)=0,
\]
where the zero initial condition is understood in the sense of the strong
\(L^2\)-trace. If the full right-hand side satisfies
\[
 \langle z\rangle^mH\in L^2(0,T;H_x^{r-1}L_z^2),
\]
then
\begin{equation}
 \nabla f\in L_T^\infty H_x^{r-1}L_{z,m}^2,\qquad
 f_t,\nabla^2f\in L_T^2H_x^{r-1}L_{z,m}^2.
 \label{r4:gradient-energy}
\end{equation}
More precisely, there exists a constant \(C_{T,m,r}>0\) such that
\begin{align*}
 &\sup_{0\le s\le t}\|\langle z\rangle^m\nabla f(s)\|_{H_x^{r-1}L_z^2}^2\\
 &\quad+\int_0^t\bigl(
       \|\langle z\rangle^mf_t\|_{H_x^{r-1}L_z^2}^2
       +\|\langle z\rangle^m\nabla^2f\|_{H_x^{r-1}L_z^2}^2
       \bigr)\,ds
 \leq C_{T,m,r}
       \int_0^t\|\langle z\rangle^mH\|_{H_x^{r-1}L_z^2}^2\,ds .
\end{align*}
The weighted gradient admits a continuous \(L^2\)-valued representative,
from which \(\nabla f(0)=0\) follows. 
\end{lemma}

\begin{lemma}
Let $\Omega=\mathbb{R}\times (0,\infty)$, If $f,\partial_{x}f,\partial_{y}^{2}f\in L^{2}(\Omega)$. Then the following inequality holds,
$$
\|f\|_{L^{\infty}(\Omega)} \leq C\| f\|_{L^2(\Omega)}^{1 / 4}\| \partial_x f\|_{L^2(\Omega)}^{1 / 2}\| \partial_y^2 f \|_{L^2(\Omega)}^{1 / 4}. 
$$ 
\end{lemma}
\begin{proof}
The assumptions and one-dimensional interpolation imply $f_y\in L^2$. Define 
$$
\widetilde f(x,y)=
\begin{cases}
f(x,y),& y\geq0,\\
3f(x,-y)-2f(x,-2y),& y<0.
\end{cases}
$$
The value and the first normal derivative agree at $y=0$, since $3-2=1$ and $-3+4=1$. Thus
$$
\|\widetilde f\|_2\leq C\|f\|_2,\qquad
\|\partial_x\widetilde f\|_2\leq C\|\partial_xf\|_2,\qquad
\|\partial_y^2\widetilde f\|_2\leq C\|\partial_y^2f\|_2.
$$
Since
\[
\int_{\mathbb R^2}\frac{1}{1+\xi^2+\zeta^4}\,d\xi\,d\zeta<\infty,
\]
the Cauchy--Schwarz inequality in Fourier space gives
\[
\|f\|_{L^\infty}
\leq C\bigl(\|f\|_{L^2}+\|\partial_{x}f\|_{L^2}+\|\partial_{y}^{2}f\|_{L^2}\bigr).
\]
Applying this estimate to \(g(x,y)=f(ax,by)\) yields
\[
\|f\|_{L^\infty}
\leq C\left(
a^{-1/2}b^{-1/2}\|f\|_{L^2}
+a^{1/2}b^{-1/2}\|\partial_{x}f\|_{L^2}
+a^{-1/2}b^{3/2}\|\partial_{y}^{2}f\|_{L^2}
\right).
\]
Choosing
\[
a=\frac{\|f\|_{L^2}}{\|\partial_{x}f\|_{L^2}},
\qquad
b^2=\frac{\|f\|_{L^2}}{\|\partial_{y}^{2}f\|_{L^2}},
\]
gives the desired inequality. The degenerate cases follow by a standard regularization argument.
\end{proof}
Next, we introduce a Dirichlet problem for a linear parabolic equation involving an integral term, which will be employed to analyze the well-posedness of the boundary layer profiles:
\begin{equation}
\label{eq:2.18}
\left\{\begin{aligned}
&\partial_t\theta-\partial_x^2\theta-\partial_z^2\theta
 +A\partial_x\theta+B\partial_z\theta+C\theta
 +D\int_0^z\partial_x\theta(x,r,t)\,\mathrm{d}r=F,\\
&\theta(x,z,0)=0,\\
&\theta(x,0,t)=0,\qquad \lim_{z\to\infty}\theta(x,z,t)=0.
\end{aligned}\right.
\end{equation}
The coefficients are given by
\[
A=\mathfrak{a}+\mathfrak{b},\qquad
B=-z\partial_x\mathfrak{a}
 -\int_0^z\partial_x\mathfrak{b}(x,r,t)\,\mathrm{d}r,
\qquad
C=\partial_x\mathfrak{a}+\partial_x\mathfrak{b},\qquad
D=-\partial_z\mathfrak{b},
\]
where $\mathfrak{a}$ is independent of $z$,
$\mathfrak{b}|_{z=0}=-\mathfrak{a}$, and $\mathfrak{b}$ decays as
$z\to\infty$. In particular,
$\partial_xA+\partial_zB=0$ and $B|_{z=0}=0$.

\begin{lemma}[Existence and uniqueness for a linear boundary layer equation]
\label{lemma:2.13}
Let $s\in\{2,3\}$ and $\ell\geq0$. Assume that
\[
\mathfrak{a}\in C([0,T];H_x^5),\qquad
\partial_x\mathfrak{a}\in L^2(0,T;H_x^5),
\]
and
\begin{gather*}
\mathfrak{b}\in C([0,T];H_x^5L_{z,\ell+s+3}^2),\qquad
\partial_x\mathfrak{b}\in L^2(0,T;H_x^5L_{z,\ell+s+3}^2),\\
\partial_z\mathfrak{b}
\in L^\infty(0,T;H_x^4L_{z,\ell+s+3}^2)
\cap L^2(0,T;H_x^5L_{z,\ell+s+3}^2).
\end{gather*}
If $F\in L^2(0,T;H_x^{s-1}L_{z,\ell+s+2}^2)$, then
problem \eqref{eq:2.18} has a unique solution satisfying
\begin{gather*}
\theta\in C([0,T];H_x^sL_{z,\ell}^2),\qquad
\nabla\theta\in L^2(0,T;H_x^sL_{z,\ell}^2),\\
\partial_z\theta\in C([0,T];H_x^{s-1}L_{z,\ell}^2),\qquad
\partial_t\theta,\ \partial_z^2\theta
\in L^2(0,T;H_x^{s-1}L_{z,\ell}^2).
\end{gather*}
Moreover, $\partial_z\theta(0)=0$ in $H_x^{s-1}L_{z,\ell}^2$.
\end{lemma}

\begin{proof}
The existence and uniqueness part of Lemma \ref{lemma:2.13} can be proved by the standard theory of parabolic problem, one
can refer to \cite{Evans2010} for instance. Here, we only give the uniform regularity estimates.
For $0\leq j\leq s$, apply $\partial_x^j$ to \eqref{eq:2.18} and
test by $\langle z\rangle^{2(\ell+s-j)}\partial_x^j\theta$.
The identity $\partial_xA+\partial_zB=0$ gives
\begin{align*}
&\int_{\Omega_{\mathrm{b}}}
\langle z\rangle^{2(\ell+s-j)}
\bigl(A\partial_x^{j+1}\theta
+B\partial_z\partial_x^j\theta\bigr)
\partial_x^j\theta\,\mathrm{d}x\,\mathrm{d}z\\
&\qquad=-\frac12\int_{\Omega_{\mathrm{b}}}
B\,\partial_z\!\left(\langle z\rangle^{2(\ell+s-j)}\right)
|\partial_x^j\theta|^2\,\mathrm{d}x\,\mathrm{d}z.
\end{align*}
This weight contribution is controlled by
$\|B/\langle z\rangle\|_\infty
\|\langle z\rangle^{\ell+s-j}\partial_x^j\theta\|^2$, since
\[
\left\|\frac{B}{\langle z\rangle}\right\|_\infty
\leq C\left(\|\mathfrak{a}\|_{H_x^2}
+\|\mathfrak{b}\|_{H_x^2L_z^2}\right).
\]
Integration by parts in the diffusion terms similarly yields
\begin{align*}
&\int_{\Omega_{\mathrm{b}}}
\bigl(-\partial_x^{j+2}\theta-\partial_z^2\partial_x^j\theta\bigr)
\langle z\rangle^{2(\ell+s-j)}\partial_x^j\theta
\,\mathrm{d}x\,\mathrm{d}z\\
&\qquad\geq
\|\langle z\rangle^{\ell+s-j}\nabla\partial_x^j\theta\|^2
-C\|\langle z\rangle^{\ell+s-j}\partial_x^j\theta\|^2.
\end{align*}

For $1\leq k\leq j$, differentiation of the affine part of $B$
produces $z\partial_x^{k+1}\mathfrak{a}\,
\partial_z\partial_x^{j-k}\theta$. The factor $z$ is absorbed by
the weight at the lower tangential index because
\[
z\langle z\rangle^{\ell+s-j}
\leq\langle z\rangle^{\ell+s-(j-k)}.
\]
Thus these commutator terms are estimated by Sobolev embedding in $x$
and Young's inequality. For the primitive part of $B$, we use
\[
\left\|\int_0^z h(x,r)\,\mathrm{d}r\right\|_{L_x^\infty L_z^\infty}
\leq C\|\langle z\rangle h\|_{H_x^1L_z^2},
\]
which follows from Cauchy--Schwarz in $z$ and Sobolev embedding in $x$.
Since $\langle z\rangle^{\ell+s-j}
\leq\langle z\rangle^{\ell+s-(j-k)}$, its commutators are controlled
in the same way. The commutators with $A\partial_x$ and the derivatives
of $C\theta$ contain at most $s$ tangential derivatives of $\theta$;
their coefficient factors are bounded using
$H_x^1H_z^1\hookrightarrow L^\infty$ and the stated assumptions.

For the nonlocal term, the pointwise estimate
\[
\left|\int_0^z h(x,r)\,\mathrm{d}r\right|
\leq\sqrt z\,\|h(x,\cdot)\|_{L_z^2}
\]
and Sobolev embedding in $x$ give, for $0\leq k\leq j$,
\begin{align*}
&\left|\int_{\Omega_{\mathrm{b}}}
\langle z\rangle^{2(\ell+s-j)}
\partial_x^k\partial_z\mathfrak{b}
\left(\int_0^z\partial_x^{j-k+1}\theta(x,r,t)\,\mathrm{d}r\right)
\partial_x^j\theta\,\mathrm{d}x\,\mathrm{d}z\right|\\
&\qquad\leq
C\|\partial_x^k\partial_z\mathfrak{b}\|_{H_x^1L_{z,\ell+s-j+1}^2}
\|\partial_x^{j-k+1}\theta\|
\|\langle z\rangle^{\ell+s-j}\partial_x^j\theta\|.
\end{align*}
Only $j=s$, $k=0$ involves $\partial_x^{s+1}\theta$.
Since $\ell\geq0$, its unweighted norm is controlled by
$\|\langle z\rangle^\ell\partial_x^{s+1}\theta\|$, and Young's
inequality absorbs this term into the horizontal dissipation.
All other cases involve at most $s$ tangential derivatives of $\theta$.
The source is estimated directly when $j<s$; at $j=s$, we use
\begin{align*}
&\int_{\Omega_{\mathrm{b}}}
\langle z\rangle^{2\ell}\partial_x^sF\,\partial_x^s\theta
\,\mathrm{d}x\,\mathrm{d}z\\
&\qquad=-\int_{\Omega_{\mathrm{b}}}
\langle z\rangle^{2\ell}\partial_x^{s-1}F\,
\partial_x^{s+1}\theta\,\mathrm{d}x\,\mathrm{d}z.
\end{align*}

Summing over $0\leq j\leq s$ and applying Young's inequality yields
\begin{align*}
&\frac{\mathrm{d}}{\mathrm{d}t}
\sum_{j=0}^s
\|\langle z\rangle^{\ell+s-j}\partial_x^j\theta\|^2
+\sum_{j=0}^s
\|\langle z\rangle^{\ell+s-j}\nabla\partial_x^j\theta\|^2\\
&\quad\leq
C\Bigl(1+\|\mathfrak{a}\|_{H_x^5}^2
+\|\mathfrak{b}\|_{H_x^5L_{z,\ell+s+3}^2}^2
+\|\partial_z\mathfrak{b}\|_{H_x^5L_{z,\ell+s+3}^2}^2\Bigr)\\
&\qquad\qquad\times\sum_{j=0}^s
\|\langle z\rangle^{\ell+s-j}\partial_x^j\theta\|^2
+C\|F\|_{H_x^{s-1}L_{z,\ell+s+2}^2}^2.
\end{align*}
The coefficient multiplying the sum on the right belongs to
$L^1(0,T)$ by the assumptions. Gronwall's inequality and the zero
initial data therefore give
\begin{align*}
&\sup_{0\leq t\leq T}\sum_{j=0}^s
\|\langle z\rangle^{\ell+s-j}\partial_x^j\theta(t)\|^2\\
&\quad+\int_0^T\sum_{j=0}^s
\|\langle z\rangle^{\ell+s-j}\nabla\partial_x^j\theta(t)\|^2
\,\mathrm{d}t\\
&\qquad\leq C\int_0^T
\|F(t)\|_{H_x^{s-1}L_{z,\ell+s+2}^2}^2\,\mathrm{d}t,
\end{align*}
where $C$ depends only on $T,s,\ell$ and the coefficient norms in the
assumptions.

To estimate the normal and time derivatives, write
\[
\partial_t\theta-\partial_x^2\theta-\partial_z^2\theta
=F-A\partial_x\theta-B\partial_z\theta-C\theta
-D\int_0^z\partial_x\theta(x,r,t)\,\mathrm{d}r.
\]
The preceding product estimates imply
\begin{align*}
&\left\|A\partial_x\theta+B\partial_z\theta+C\theta
+D\int_0^z\partial_x\theta(x,r,t)\,\mathrm{d}r
\right\|_{H_x^{s-1}L_{z,\ell}^2}\\
&\quad\leq C\Bigl(
\|\mathfrak{a}\|_{H_x^5}
+\|\mathfrak{b}\|_{H_x^5L_{z,\ell+s+3}^2}
+\|\partial_z\mathfrak{b}\|_{H_x^4L_{z,\ell+s+3}^2}\Bigr)\\
&\qquad\qquad\times
\Bigl(\|\theta\|_{H_x^sL_{z,\ell}^2}
+\|\partial_z\theta\|_{H_x^{s-1}L_{z,\ell+1}^2}\Bigr).
\end{align*}
Here the extra weight needed for $B\partial_z\theta$ is available
because $\ell+s-j\geq\ell+1$ for $j\leq s-1$.
The coefficient norms in this estimate are bounded in time, whereas
the solution norms on its right belong to $L^2(0,T)$ by the tangential
estimate. Hence the right-hand side of the heat equation belongs to
$L^2(0,T;H_x^{s-1}L_{z,\ell}^2)$.

For $0\leq j\leq s-1$, the function
$\langle z\rangle^\ell\partial_x^j\theta$ consequently satisfies
a zero-data Dirichlet heat equation with an $L^2$ right-hand side.
Indeed, commuting the weight with $\partial_x^2+\partial_z^2$ adds only
\[
-2\partial_z\!\left(\langle z\rangle^\ell\right)
\partial_z\partial_x^j\theta
-\partial_z^2\!\left(\langle z\rangle^\ell\right)
\partial_x^j\theta,
\]
which belongs to $L^2((0,T)\times\Omega_{\mathrm{b}})$ by the estimate
above. The zero-data Dirichlet heat estimate, obtained by testing with
$-(\partial_x^2+\partial_z^2)(\langle z\rangle^\ell\partial_x^j\theta)$ and then using
the equation, gives
\[
\partial_t\theta,\ \nabla^2\theta
\in L^2(0,T;H_x^{s-1}L_{z,\ell}^2),\qquad
\langle z\rangle^\ell\partial_x^j\theta
\in C([0,T];H_0^1(\Omega_{\mathrm{b}})),
\]
with zero initial value in $H_0^1(\Omega_{\mathrm{b}})$.
Taking tangential and normal derivatives in this last statement yields
\[
\theta\in C([0,T];H_x^sL_{z,\ell}^2),\qquad
\partial_z\theta\in C([0,T];H_x^{s-1}L_{z,\ell}^2),\qquad
\partial_z\theta(0)=0.
\] 
The proof is complete.
\end{proof}
We also need the following linearized problem to analyze the
well-posedness of the higher-order outer layer profiles:
\begin{equation}
\label{eq:2.19}
    \left\{\begin{array}{l}
    \partial_t v+(v\cdot\nabla)\mathfrak{d}
    +(\mathfrak{d}\cdot\nabla)v+\nabla p-\partial_x^2v=F,
    \\[0.2cm]
    \operatorname{div}v=0,
    \\[0.2cm]
    v_2(x,0,t)=0,\quad v(x,y,0)=0.
    \end{array}\right.
\end{equation}
The well-posedness of \eqref{eq:2.19} is stated as follows.

\begin{lemma}[Existence and uniqueness for a linear outer equation]
\label{lemma:2.14}
Let $s\geq2$ be an integer. Assume that
$$
\mathfrak{d}\in C([0,T];H^{s+1}),
\qquad
\operatorname{div}\mathfrak{d}=0,
\qquad
\mathfrak{d}_2(x,0,t)=0.
$$
If $F\in L^2(0,T;H^s)$, then problem \eqref{eq:2.19}
has a unique solution $v$ satisfying
$$
v\in C([0,T];H^s),
\qquad
\partial_xv\in L^2(0,T;H^s),
$$
and
$$
\partial_tv\in L^2(0,T;H^{s-1}),
\qquad
\nabla p\in L^2(0,T;H^s).
$$
Moreover, the trace of the first component $v_1$ on $y=0$
satisfies
$$
v_1|_{y=0}\in C([0,T];H_x^s),
\qquad
\partial_xv_1|_{y=0}\in L^2(0,T;H_x^s),
\qquad
\partial_tv_1|_{y=0}\in L^2(0,T;H_x^{s-1}).
$$
\end{lemma}

\begin{proof}
Owing to the linearity of problem \eqref{eq:2.19}, the existence and uniqueness part Lemma \ref{lemma:2.14} can be proved by standard argument similar to the one in \cite{CaoWu2011} and \cite{CheminEtAlBook}. Here, we only sketch the uniform regularity estimates. 
We combine a bulk div--curl estimate with a parabolic estimate
for the tangential velocity on the boundary.  
Put $\zeta=\partial_xv_2-\partial_yv_1$.
Since both $v$ and $\mathfrak{d}$ are divergence-free, the
two-dimensional identity
$$
\operatorname{curl}\bigl(
(\mathfrak{d}\cdot\nabla)v+(v\cdot\nabla)\mathfrak{d}
\bigr)
=
\mathfrak{d}\cdot\nabla\zeta
+v\cdot\nabla\operatorname{curl}\mathfrak{d}
$$
gives
$$
\partial_t\zeta-\partial_x^2\zeta
+\mathfrak{d}\cdot\nabla\zeta
=
\operatorname{curl}F
-v\cdot\nabla\operatorname{curl}\mathfrak{d}.
$$
The div--curl estimate, applied also to $\partial_xv$, yields
$$
\begin{aligned}
\|v\|_{H^s}
&\leq C\bigl(\|v\|_2+\|\zeta\|_{H^{s-1}}\bigr),\\
\|\partial_xv\|_{H^s}
&\leq C\bigl(
\|\partial_xv\|_2+\|\partial_x\zeta\|_{H^{s-1}}
\bigr).
\end{aligned}
$$
For $s\geq2$, the integer Sobolev product estimates imply
$$
\|v\cdot\nabla\operatorname{curl}\mathfrak{d}\|_{H^{s-1}}
\leq
C\|\mathfrak{d}\|_{H^{s+1}}\|v\|_{H^s},
$$
while the transport commutators contribute at most
$C\|\mathfrak{d}\|_{H^{s+1}}\|v\|_{H^s}^2$.
Combining the $H^{s-1}$ vorticity energy with the velocity
$L^2$ energy, we obtain
$$
\begin{aligned}
&\frac{d}{dt}
\bigl(\|v\|_2^2+\|\zeta\|_{H^{s-1}}^2\bigr)
+c\|\partial_xv\|_{H^s}^2\\
&\qquad\leq
C\bigl(1+\|\mathfrak{d}\|_{H^{s+1}}^2\bigr)
\|v\|_{H^s}^2
+C\|F\|_{H^s}^2.
\end{aligned}
$$
No vorticity boundary condition is needed: diffusion is
tangential, and the boundary contribution from transport
vanishes because $\mathfrak{d}_2|_{y=0}=0$.
Gronwall's inequality proves the bulk spatial bounds.

Taking the divergence of the velocity equation and evaluating
its normal component on the boundary gives
$$
-\Delta p
=
2\partial_i\mathfrak{d}_j\,\partial_jv_i
-\operatorname{div}F,
\qquad
\partial_yp|_{y=0}=F_2|_{y=0}.
$$
The finite-energy formulation of this Neumann problem yields
$$
\|\nabla p\|_2
\leq
\|F-(v\cdot\nabla)\mathfrak{d}
-(\mathfrak{d}\cdot\nabla)v\|_2.
$$
Combining this bound with the integer elliptic estimates and
the trace theorem, we obtain
$$
\begin{aligned}
\|\nabla p\|_{H^s}
&\leq C\bigl(
\|\nabla p\|_2
+\|\Delta p\|_{H^{s-1}}
+\|F_2|_{y=0}\|_{H_x^{s-\frac12}}
\bigr)\\
&\leq C\bigl(
\|F\|_{H^s}
+\|\mathfrak{d}\|_{H^{s+1}}\|v\|_{H^s}
\bigr).
\end{aligned}
$$
Hence $\nabla p\in L^2(0,T;H^s)$.
Since $\partial_x^2v\in L^2(0,T;H^{s-1})$,
the velocity equation also gives
$\partial_tv\in L^2(0,T;H^{s-1})$.

The improved tangential trace is obtained from its own wall
equation rather than from the bulk trace theorem.
For $h=v_1|_{y=0}$, using an overbar to denote boundary traces,
we have
$$
h_t-h_{xx}
+\partial_x\bigl(\overline{\mathfrak{d}_1}h\bigr)
=
\bar F_1-\overline{\partial_xp},
\qquad h(0)=0.
$$
The trace theorem and the pressure estimate imply
$$
\overline{\mathfrak{d}_1}
\in L^\infty(0,T;H_x^s),
\qquad
\bar F_1-\overline{\partial_xp}
\in L^2(0,T;H_x^{s-1}).
$$
Moreover, the one-dimensional Sobolev product estimate gives
$$
\|\overline{\mathfrak{d}_1}h\|_{H_x^s}
\leq
C\|\overline{\mathfrak{d}_1}\|_{H_x^s}
\|h\|_{H_x^s}.
$$
Testing the wall equation with its tangential derivatives
through order $s$, and integrating both the flux term and
the highest-order source term once in $x$, yields
$$
\begin{aligned}
\frac{d}{dt}\|h\|_{H_x^s}^2
+c\|\partial_xh\|_{H_x^s}^2
&\leq
C\bigl(1+\|\overline{\mathfrak{d}_1}\|_{H_x^s}^2\bigr)
\|h\|_{H_x^s}^2\\
&\quad+
C\|\bar F_1-\overline{\partial_xp}\|_{H_x^{s-1}}^2.
\end{aligned}
$$
Gronwall's inequality therefore gives
$h\in L^\infty(0,T;H_x^s)$ and
$\partial_xh\in L^2(0,T;H_x^s)$.
The wall equation then implies
$h_t\in L^2(0,T;H_x^{s-1})$.
Together with $h\in L^2(0,T;H_x^{s+1})$, this yields
$h\in C([0,T];H_x^s)$. 
\end{proof}

\section{Construction of an approximate solution}\label{sec:Construction}

In this section, we present the equations of outer and inner (i.e. boundary) layer profiles via asymptotic analysis. The derivations is given in Appendix \ref{sec:appendixA}. Based on the outer and inner layer profiles, we can construct an approximate solution to the problem \eqref{eq:1.1}-\eqref{eq:1.2}.

\subsection{Asymptotic analysis}
\label{subsec:Asymptotic analysis}

In this subsection, we derive the equations of the outer and inner layer profiles by the asymptotic analysis(see,e.g.\cite{Holmes2013,WangWen2024}). To begin with, we introduce the following formal Prandtl type boundary layer expansions:
\begin{equation}
	\label{eq:3.1}
	\left\{\begin{array}{l}
		u^\varepsilon(x, y, t)\sim\displaystyle \sum_{j=0}^{+\infty} \varepsilon^{j}\left(u^{I, j}(x, y, t)+u^{b, j}(x, z, t)\right), \\[0.2cm]
		p^\varepsilon(x, y, t)\sim\displaystyle \sum_{j=0}^{+\infty} \varepsilon^{j}\left(p^{I, j}(x, y, t)+p^{b, j}(x, z, t)\right),
	\end{array}\right.
\end{equation}
where $z=\frac{y}{\varepsilon}$. We assume that
$$
u^{b, j}(x, z, t) \rightarrow 0, \quad p^{b, j}(x, z, t) \rightarrow 0,
$$
fast enough as $z \rightarrow +\infty$. Substituting \eqref{eq:3.1} in to \eqref{eq:1.1}-\eqref{eq:1.2}, and applying the matched asymptotic method, we deduce the equations of outer and inner layer profiles in sequence, see Appendix \ref{sec:appendixA} for the details.

\subsubsection{The leading order inner and outer profiles}
Due to the analysis in Lemma \ref{lemma:A1} (see Appendix \ref{sec:appendixA}), we find that the leading order outer layer profile $(u^{I,0},p^{I,0})$ satisfies problem \eqref{eq:1.3}-\eqref{eq:1.4}. Moreover,the leading boundary layer profiles of vertical velocity and pressure satisfy
$$
u_{2}^{b,0}=0,p^{b,0}=0.
$$ 
For the horizontal velocity, a nontrivial boundary-layer profile \(u_{1}^{b,0}\) arises, which satisfies \eqref{eq:3.3}.

\subsubsection{The first order inner and outer profiles}
From Corollary \ref{corollary:A2} and Lemma \ref{lemma:A3}, we find that
\begin{equation}
	\label{eq:3.4}
	u_{2}^{b,1}=\int_{z}^{+\infty} \partial_{x}u_{1}^{b,0}(x,s,t)ds,\quad p^{b,1}=0.
\end{equation}
Then, the outer profiles $(u^{I,1},p^{I,1})$ satisfy the following problem:
\begin{equation}
	\label{eq:3.5}
	\left\{\begin{array}{l}
		\partial_t u^{I, 1}+u^{I, 1} \cdot \nabla u^{I,0}+u^{I,0} \cdot \nabla u^{I,1}+\nabla p^{I, 1}-\partial_{x}^{2} u^{I, 1}=0, \\[0.2cm]
		\partial_x u_1^{I, 1}+\partial_y u_2^{I, 1}=0, \\[0.2cm]
		u_{2}^{I,1}(x,0,t)=-\int_{0}^{+\infty} \partial_{x}u_{1}^{b,0}(x,s,t)ds, \\[0.2cm]
		u^{I, 1}(x, y, 0)=0 .
	\end{array}\right.
\end{equation}
Furthermore, $u_{1}^{b,1}$ satisfies
\begin{equation}
	\label{eq:3.6}
	\left\{\begin{array}{l}
		\partial_t u_{1}^{b,1} - \partial_{x}^{2} u_{1}^{b,1} - \partial_{z}^{2} u_{1}^{b,1} \\[0.2cm]
		\quad+\left(a+u_{1}^{b,0}\right) \partial_x u_{1}^{b,1}+ \left(-z\partial_{x}a-\int_{0}^{z}\partial_{x}u_{1}^{b,0}(x,s,t)ds  \right) \partial_z u_{1}^{b,1}\\[0.2cm]
		\qquad+\left( \partial_x u_{1}^{b,0}+\partial_{x}a\right) u_{1}^{b,1}-\left(\int_{0}^{z}\partial_{x}u_{1}^{b,1}(x,s,t)ds\right) \partial_z u_{1}^{b,0}=g^{b,1} \\[0.2cm]
		u_{1}^{b,1}(x, z,0)=0, \\[0.2cm]
		\displaystyle \lim _{z \rightarrow \infty} u_{1}^{b,1}(x,z,t)=0, \quad u_{1}^{b,1}(x,0,t)=-u_{1}^{I,1}(x,0,t).
	\end{array}\right.
\end{equation}
where
$$
\begin{aligned}
	-g^{b, 1}&=\left(z \overline{\partial_{y} \partial_{x} u_{1}^{I,0}} u_{1}^{b,0} +u_{1}^{b,0} \overline{\partial_x u_{1}^{I,1}} +u_{2}^{b,1} \overline{\partial_y u_{1}^{I,0}} \right)+\left( \overline{u_{1}^{I,1}} + z \overline{\partial_y u_{1}^{I,0}}\right) \partial_x u_{1}^{b,0} \\[0.2cm]
	&\quad+\left(z \overline{\partial_y u_{2}^{I,1}} + \frac{z^2}{2} \overline{\partial_{y}^{2} u_{2}^{I,0}}\right) \partial_z u_{1}^{b,0}.
\end{aligned}
$$
See Step 3 in Appendix \ref{sec:appendixA} for details.

\subsubsection{The second order inner and outer profiles}
From Corollary \ref{corollary:A4} and Lemma \ref{lemma:A5}, we find the second order boundary layer profile $u_{2}^{b,2}$ and $p^{b,2}$ satisfy
\begin{equation}
	\label{eq:3.7}
	u_{2}^{b,2}=\int_{z}^{+\infty}\partial_{x}u_{1}^{b,1}(x,s,t)ds,\quad p^{b,2}=-\int_{z}^{+\infty}\mathcal{P}_{2}(x,s,t)ds.
\end{equation}
where
$$
\begin{aligned}
	\mathcal{P}_{2}(x,z,t)&=-\partial_{t}u_{2}^{b,1}+\partial_{x}^{2}u_{2}^{b,1}+\partial_{z}^{2}u_{2}^{b,1}-\overline{u_{1}^{I,0}}\partial_{x}u_{2}^{b,1} \\[0.2cm]
	&-(\partial_{x}u_{2}^{b,1}+\overline{\partial_{x}u_{2}^{I,1}})u_{1}^{b,0}-\overline{\partial_{y}u_{2}^{I,0}}u_{2}^{b,1}-(u_{2}^{b,1}+\overline{u_{2}^{I,1}})\partial_{z}u_{2}^{b,1} \\[0.2cm]
	&-z\overline{\partial_{y}\partial_{x}u_{2}^{I,0}}u_{1}^{b,0}-z\overline{\partial_{y}u_{2}^{I,0}}\partial_{z}u_{2}^{b,1}.
\end{aligned}
$$
The second order outer profiles $(u^{I,2},p^{I,2})$ satisfy the following problem:
\begin{equation}
	\label{eq:3.8}
	\left\{\begin{array}{l}
		\partial_t u^{I,2}+u^{I,2} \cdot \nabla u^{I,0}+u^{I,0} \cdot \nabla u^{I,2}+\nabla p^{I,2}-\partial_{x}^{2} u^{I,2}=g^{I,2}, \\[0.2cm]
		\partial_x u_{1}^{I,2}+\partial_y u_{2}^{I,2}=0, \\[0.2cm]
		u_{2}^{I,2}(x,0,t) = - \int_{0}^{+\infty}\partial_{x}u_{1}^{b,1}(x,s,t)ds, \\[0.2cm]
		u^{I,2}(x,y,0)= 0 .
	\end{array}\right.
\end{equation}
where
$$
g^{I,2}=\partial_{y}^{2}u^{I,0}-u^{I,1} \cdot \nabla u^{I,1}.
$$
Furthermore, $u_{1}^{b,2}$ satisfies
\begin{equation}
	\label{eq:3.9}
	\left\{\begin{array}{l}
		\partial_t u_{1}^{b,2} - \partial_{x}^{2} u_{1}^{b,2} - \partial_{z}^{2} u_{1}^{b,2} \\[0.2cm]
		\quad+\left(a + u_{1}^{b,0}\right) \partial_x u_{1}^{b,2} + \left(-z\partial_{x}a-\int_{0}^{z}\partial_{x}u_{1}^{b,0}(x,s,t)ds \right) \partial_z u_{1}^{b,2} \\[0.2cm]
		\quad+\left( \partial_x u_{1}^{b,0}+\partial_{x}a\right) u_{1}^{b,2}-\left(\int_{0}^{z}\partial_{x}u_{1}^{b,2}(x,s,t)ds \right) \partial_{z} u_{1}^{b,0}=g^{b,2} \\[0.2cm]
		u_{1}^{b,2}(x, z,0)=0, \\[0.2cm]
		\displaystyle \lim _{z \rightarrow \infty} u_{1}^{b,2}(x,z,t)=0, \quad u_{1}^{b,2}(x,0,t)=-u_{1}^{I,2}(x,0,t).
	\end{array}\right.
\end{equation}
where
$$
\begin{aligned}
	-g^{b,2}=&\left(\overline{u_{1}^{I,1}}+u_{1}^{b,1}+z\overline{\partial_{y}u_{1}^{I,0}}\right) \partial_{x} u_{1}^{b,1} +\left(z \overline{\partial_{y}u_{2}^{I,1}} +\frac{z^2}{2}\overline{\partial_{y}^{2}u_{2}^{I,0}}\right)\partial_{z}u_{1}^{b,1} \\[0.2cm] 
	&+\left(-\int_{0}^{z}\partial_{x}u_{1}^{b,1}(x,s,t)ds \right) \partial_{z}u_{1}^{b,1} + \left(\overline{u_{1}^{I,2}}+z \overline{\partial_{y}u_{1}^{I,1}} + \frac{z^2}{2} \overline{\partial_{y}^{2}u_{1}^{I,0}}\right) \partial_{x} u_{1}^{b,0} \\[0.2cm]
	&+ \left( z \overline{\partial_{y} u_{2}^{I,2}} + \frac{z^2}{2}\overline{\partial_{y}^{2}u_{2}^{I,1}} + \frac{z^3}{6} \overline{\partial_{y}^{3}u_{2}^{I,0}}\right) \partial_{z} u_{1}^{b,0} \\[0.2cm]
	&+\frac{z^2}{2}\overline{\partial_{y}^{2}\partial_{x}u_{1}^{I,0}} u_{1}^{b,0} + z \overline{\partial_{y}\partial_{x}u_{1}^{I,1}} u_{1}^{b,0} + \overline{\partial_{x} u_{1}^{I,2}} u_{1}^{b,0}+ z \overline{\partial_{y}^{2}u_{1}^{I,0}} u_{2}^{b,1} \\[0.2cm]
	&+ \overline{\partial_{y}u_{1}^{I,1}} u_{2}^{b,1} +z \overline{\partial_{y} \partial_{x} u_{1}^{I,0}} u_{1}^{b,1}+ \overline{\partial_{x}u_{1}^{I,1}} u_{1}^{b,1}+\overline{\partial_{y} u_{1}^{I,0}} u_{2}^{b,2} + \partial_{x} p^{b,2}.
\end{aligned}
$$
See Step 4 in Appendix \ref{sec:appendixA} for details.
\subsubsection{Some higher order profiles}
We also need the following higher order profiles of velocity:
\begin{equation}
	\label{eq:3.10}
	u_{2}^{b,3}=\int_{z}^{+\infty}\partial_{x}u_{1}^{b,2}(x,s,t)ds.
\end{equation}
See Lemma \ref{lemma:A6} for details.

\subsection{Regularity of the outer and boundary layer profiles}
\label{sec:profile-regularity}
\subsubsection{Regularity of the limiting flow }

In order to use the outer and inner layer profiles deduced in Section \ref{subsec:Asymptotic analysis}. we prove the well-posedness of those profiles. For the initial data of system \eqref{eq:1.3}, we do not impose any higher-order boundary compatibility conditions.  
In particular, the initial wall acceleration $\partial_tu_1^{I,0}(x,0,0)=\alpha_0(x)$ need not vanish. Then, for $u^{I,0}$, we have the following results.
\begin{proposition}
\label{prop:outer-flow}
    Under the assumptions of Theorem \ref{thm:main}. For every $T> 0$, the limit problem \eqref{eq:1.3} has a unique solution satisfying 
    \begin{equation}
    \label{eq:4.2}
        u^{I,0}\in C([0,T];H^{4}),\quad \partial_x u^{I,0}\in L^2(0,T;H^{4}).
    \end{equation}
    The pressure is defined modulo functions of time, and its gradient is uniquely determined by the finite-energy Neumann problem. Moreover,
    \begin{equation}
    \label{eq:4.3}
        \begin{gathered}
            \nabla p^{I,0}\in C([0,T];H^4), \quad \partial_x\nabla p^{I,0}\in L^2(0,T;H^4), \\[0.2cm]
            \partial_t u^{I,0},\ \nabla\partial_t p^{I,0} \in C([0,T];H^2),\quad \partial_x\partial_t u^{I,0},\ \partial_x\nabla\partial_t p^{I,0} \in L^2(0,T;H^2), \\[0.2cm]
            \partial_t^2u^{I,0}\in C([0,T];L^2), \quad \partial_x\partial_t^2u^{I,0}
            \in L^2(0,T;L^2).
        \end{gathered}
    \end{equation}
    The boundary quantities required below are already well defined at the $H^4$ level. More precisely, set
    \begin{equation}
    a(x,t):=u_1^{I,0}(x,0,t),\qquad \alpha_0(x):=\partial_t a(x,0)=-\partial_xp^{I,0}(x,0,0),
    \end{equation}
    where the initial pressure is determined by the Neumann problem. On every finite interval $[0,T]$,
    \begin{equation}
    \label{eq:4.5}
        \begin{gathered}
            -\partial_x p^{I,0}(x,0,t) \in C([0,T];H_x^3) \cap L^2(0,T;H_x^4), \\[0.2cm]
            -\partial_x\partial_t p^{I,0}(x,0,t) \in C([0,T];H_x^1) \cap L^2(0,T;H_x^2).
        \end{gathered}
    \end{equation}
    Furthermore,
    \begin{equation}
    \label{eq:4.6}
        \begin{gathered}
            a\in C([0,T];H_x^5), \partial_x a\in L^2(0,T;H_x^5), \\[0.2cm]
            \partial_t a\in C([0,T];H_x^3), \partial_x\partial_t a \in L^2(0,T;H_x^3), \partial_t^2a\in C([0,T];H_x^1).
        \end{gathered}
    \end{equation}
    In particular,
    \begin{equation}
    \label{eq:4.7}
        \begin{gathered}
            \alpha_0\in H_x^3,\quad \|a(t)\|_{H_x^3} \le C_T t, \quad \|a(t)-t\alpha_0\|_{H_x^1} \le C_T t^2, \quad \|\partial_t a(t)-\alpha_0\|_{H_x^1} \le C_T t.
        \end{gathered}
    \end{equation}
    The ordinary wall jets needed below satisfy
    \[
    \begin{aligned}
    &\overline{\partial_yu_1^{I,0}}\in C([0,T];H_x^2),\qquad
      \overline{\partial_x\partial_yu_1^{I,0}}\in L^2(0,T;H_x^2),\\[0.2cm]
    &\overline{\partial_y^2u_1^{I,0}}\in C([0,T];H_x^1),\qquad
      \overline{\partial_x\partial_y^2u_1^{I,0}}\in L^2(0,T;H_x^1).
    \end{aligned}
    \]
    All constants depend only on $T$ and $\|\tilde u\|_{H^4}$. No higher-order boundary compatibility condition is imposed.
\end{proposition}
\begin{proof}
    For this proof write $u=u^{I,0}$, $\omega=\partial_x u_2-\partial_y u_1$ and $p=p^{I,0}$. We first derive the velocity and vorticity estimates and their spatial induction. We then construct the solution and prove uniqueness, before obtaining the pressure, time-derivative, and wall estimates.
    
    \textit{Step 1: the two basic energies.} Taking the $L^2$ product of \eqref{eq:1.3} with $u$, and using $\operatorname{div} u=0$ and $u_{2}|_{y=0}=0$, gives
    \begin{equation}
    \label{eq:4.8}
        \|u(t)\|_2^2+2 \int_0^t\left\|\partial_x u(s)\right\|_2^2 \mathrm{~d} s=\|\tilde{u}\|_2^2.
    \end{equation}
    Taking the curl of \eqref{eq:1.3}, we obtain 
    \begin{equation}
    \label{eq:4.9}
        \partial_t \omega+u \cdot \nabla \omega-\partial_x^2 \omega=0 .
    \end{equation}
    Multiplying \eqref{eq:4.9} $\omega$ and integrating by parts, we obtain
    \begin{equation}
    \label{eq:4.10}
        \|\omega(t)\|_2^2+2 \int_0^t\left\|\partial_x \omega(s)\right\|_2^2 \mathrm{~d} s=\|\operatorname{curl} \tilde{u}\|_2^2 .
    \end{equation}
    Equations\eqref{eq:4.8}-\eqref{eq:4.10} and Lemma \ref{lem:divcurl} control $u$ in $L^{\infty}H^{1}$ and $\partial_{x}u$ in $L^{2}H^{1}$.
    
    \textit{Step 2: the $H^{1}$ vorticity estimate.} Acting $\partial_{x}$ and $\partial_{y}$ to \eqref{eq:4.9}, testing the resulting equations with $\partial_{x}\omega$ and $\partial_{y}\omega$, respectively, summing up and integrating by parts, we obtain 
    \begin{equation}
    \label{eq:4.11}
        \begin{aligned}
			\frac{1}{2}\frac{d}{dt}&\left(\left\|\partial_{x}\omega\right\|^{2}+\left\|\partial_{y}\omega\right\|^{2} \right)+\left\|\partial_{x}^{2}\omega\right\|^{2}+\left\|\partial_{x}\partial_{y}\omega\right\|^{2} \\[0.2cm]
			&=-\iint \partial_{x}u\cdot\nabla \omega\partial_{x}\omega dxdy -\iint \partial_{y}u\cdot\nabla \omega\partial_{y}\omega dxdy.
		\end{aligned}
    \end{equation}
    Incompressibility and $u_2|_{y=0}=0$ give
    \[
    \|\nabla u\|_2=\|\omega\|_2,\qquad
    \|\nabla\partial_xu\|_2=\|\partial_x\omega\|_2,\qquad
    \|\partial_y^2u\|_2\leq C\|\nabla\omega\|_2,
    \]
    where the last inequality also uses
    $\partial_y^2u_2=-\partial_x\partial_yu_1$ and
    $\partial_y^2u_1=\partial_x\partial_yu_2-\partial_y\omega$.
    Using Lemma \ref{lemma:2.2}, we can bound the terms on the right as follows:
    $$
	\begin{aligned}
		\iint& |\partial_{x}u\cdot\nabla \omega\partial_{x}\omega|dxdy \\[0.2cm]
		\leq& C\left\|\partial_{x}u_{1} \right\|^{\frac{1}{2}}\left\|\partial_{y}\partial_{x}u_{1} \right\|^{\frac{1}{2}}\left\|\partial_{x}\omega \right\|^{\frac{1}{2}}\left\|\partial_{x}^{2}\omega \right\|^{\frac{1}{2}}\left\|\partial_{x}\omega \right\| \\[0.2cm]
		&+C\left\|\partial_{x}u_{2} \right\|^{\frac{1}{2}}\left\|\partial_{y}\partial_{x}u_{2} \right\|^{\frac{1}{2}}\left\|\partial_{y}\omega \right\|^{\frac{1}{2}}\left\|\partial_{x}\partial_{y}\omega \right\|^{\frac{1}{2}}\left\|\partial_{x}\omega \right\| \\[0.2cm]
		\leq& \frac{1}{8}\left(\left\|\partial_{x}^{2}\omega\right\|^{2}+\left\|\partial_{x}\partial_{y}\omega \right\|^{2}\right)+C\left\|\omega \right\|^{\frac{2}{3}}\left\|\partial_{x}\omega \right\|^{\frac{2}{3}}\left\|\partial_{x}\omega \right\|^{2} \\[0.2cm]
		&+C\left\|\omega \right\|^{2}\left\|\partial_{y}\omega\right\|^{2}+C\left\|\partial_{x}\omega\right\|\left\|\partial_{x}\omega \right\|^{2},
	\end{aligned}
	$$
	and
	$$
	\begin{aligned}
		\iint & |\partial_{y}u\cdot\nabla \omega\partial_{y}\omega|dxdy \\[0.2cm]
		\leq&C\left\|\partial_{y}u_{1} \right\|^{\frac{1}{2}}\left\|\partial_{y}^{2}u_{1} \right\|^{\frac{1}{2}}\left\|\partial_{x}\omega \right\|^{\frac{1}{2}}\left\|\partial_{x}^{2}\omega \right\|^{\frac{1}{2}}\left\|\partial_{y}\omega \right\| \\[0.2cm]
		&+C\left\|\partial_{y}u_{2} \right\|^{\frac{1}{2}}\left\|\partial_{y}^{2}u_{2} \right\|^{\frac{1}{2}}\left\|\partial_{y}\omega \right\|^{\frac{1}{2}}\left\|\partial_{x}\partial_{y}\omega \right\|^{\frac{1}{2}}\left\|\partial_{y}\omega \right\| \\[0.2cm]
		\leq& \frac{1}{8}\left(\left\|\partial_{x}^{2}\omega \right\|^{2}+\left\|\partial_{x}\partial_{y}\omega \right\|^{2}\right)+C\left\|\omega \right\|^{\frac{2}{3}}\left\|\partial_{x}\omega \right\|^{\frac{2}{3}}\left\|\nabla\omega \right\|^{2}.
	\end{aligned}
	$$
	Combining these estimates, we have
	$$
	\begin{aligned}
		\frac{d}{dt}&\left(\left\|\partial_{x}\omega \right\|^{2}+\left\|\partial_{y}\omega \right\|^{2} \right)+\left\|\partial_{x}^{2}\omega\right\|^{2}+\left\|\partial_{x}\partial_{y}\omega \right\|^{2} \\[0.2cm]
		&\leq C\left(\left\|\omega\right\|^{\frac{2}{3}}\left\|\partial_{x}\omega\right\|^{\frac{2}{3}}+\left\|\omega\right\|^{2}+\left\|\partial_{x}\omega\right\| \right) \left( \left\|\partial_{x}\omega\right\|^{2}+\left\|\partial_{y}\omega\right\|^{2}\right).
	\end{aligned}
	$$
    which, together with \eqref{eq:4.10} and Gronwall inequality, we have
	\begin{equation}
    \label{eq:4.12}
	    \sup _{0 \leq t \leq T}\|\omega(t)\|_{H^1}^2+\int_0^T\left\|\partial_x \omega(t)\right\|_{H^1}^2 \mathrm{~d} t \leq C \|\tilde{u}\|_{H^2}^2.
	\end{equation}

    \textit{Step 3: the general spatial induction.} For $m\geq 2$, put 
    $$
    E_m(t)=\sum_{|\alpha| \leq m}\left\|D^\alpha \omega(t)\right\|_2^2, \quad D_m(t)=\sum_{|\alpha| \leq m}\left\|\partial_x D^\alpha \omega(t)\right\|_2^2 .
    $$  
    For $|\alpha|=m$, the differentiated identity is
    \begin{equation}
    \label{eq:4.13}
        \frac{1}{2} \frac{\mathrm{~d}}{\mathrm{~d} t}\left\|D^\alpha \omega\right\|_2^2+\left\|\partial_x D^\alpha \omega\right\|_2^2=-\sum_{0<\beta \leq \alpha}\binom{\alpha}{\beta} \int_{\Omega} D^\beta u \cdot \nabla D^{\alpha-\beta} \omega D^\alpha \omega \mathrm{d} x \mathrm{~d} y .
    \end{equation}
    By expanding the right-hand side of the above equation and using Lemma \ref{lemma:2.2}, we obtain
    $$
	\begin{aligned}
		&\sum_{0<\beta \leq \alpha}\binom{\alpha}{\beta} \int_{\Omega} D^\beta u \cdot \nabla D^{\alpha-\beta} \omega D^\alpha \omega \mathrm{d} x \mathrm{~d} y \\[0.2cm]
		&\quad\lesssim \sum_{0<\beta \leq \alpha} \left\|D^{\beta}u_{1}\right\|^{\frac{1}{2}} \left\|\partial_{y}D^{\beta}u_{1}\right\|^{\frac{1}{2}} \left\|\partial_{x}D^{\alpha-\beta}\omega\right\|^{\frac{1}{2}} \left\|\partial_{x}^{2}D^{\alpha-\beta}\omega\right\|^{\frac{1}{2}} \left\|D^{\alpha}\omega\right\| \\[0.2cm]
		&\quad\quad+\sum_{0<\beta \leq \alpha} 
        \left\|D^{\beta}u_{2}\right\|^{\frac{1}{2}} \left\|\partial_{y}D^{\beta}u_{2}\right\|^{\frac{1}{2}} \left\|\partial_{y}D^{\alpha-\beta}\omega\right\|^{\frac{1}{2}} \left\|\partial_{x}\partial_{y}D^{\alpha-\beta}\omega\right\|^{\frac{1}{2}} \left\|D^{\alpha}\omega\right\|.
	\end{aligned}
	$$
    Notice that
    $$
	\begin{aligned}
		&\sum_{|\beta|=1} 
        \left\|D^{\beta}u_{1}\right\|^{\frac{1}{2}} \left\|\partial_{y}D^{\beta}u_{1}\right\|^{\frac{1}{2}} \left\|\partial_{x}D^{\alpha-\beta}\omega\right\|^{\frac{1}{2}} \left\|\partial_{x}^{2}D^{\alpha-\beta}\omega\right\|^{\frac{1}{2}} \left\|D^{\alpha}\omega\right\| \\[0.2cm]
		&\quad\leq \sum_{|\beta|=1} 
        \left\|\omega\right\|^{\frac{1}{2}} \left\|\omega\right\|_{H^{1}}^{\frac{1}{2}} 
        \left\|\partial_{x}D^{\alpha-\beta}\omega\right\|^{\frac{1}{2}} \left\|\partial_{x}^{2}D^{\alpha-\beta}\omega\right\|^{\frac{1}{2}} \left\|D^{\alpha}\omega\right\| \\[0.2cm]
		&\quad\leq 
        \frac{1}{16}D_{m}(t)
        +C\left(1+\left\|\omega\right\|^{2} \left\|\omega\right\|_{H^{1}}^{2} \right) E_{m}(t).
	\end{aligned}
	$$
    Moreover,
    $$
	\begin{aligned}
		&\sum_{2\leq |\beta| \leq m-1} 
        \left\|D^{\beta}u_{1}\right\|^{\frac{1}{2}} \left\|\partial_{y}D^{\beta}u_{1}\right\|^{\frac{1}{2}} \left\|\partial_{x}D^{\alpha-\beta}\omega\right\|^{\frac{1}{2}} \left\|\partial_{x}^{2}D^{\alpha-\beta}\omega\right\|^{\frac{1}{2}} \left\|D^{\alpha}\omega\right\| \\[0.2cm]
		&\quad\leq 
        C\left(1+E_{m-1}(t)+\|u_{1}\|^{2} \right) E_{m}(t)+E_{m-1}(t)\left(1+E_{m-1}(t)+\|u_{1}\|^{2} \right).
	\end{aligned}
	$$
    Finally,
    $$
    \begin{aligned}
    &\sum_{|\beta|=m} 
    \left\|D^{\beta}u_{1}\right\|^{\frac{1}{2}} \left\|\partial_{y}D^{\beta}u_{1}\right\|^{\frac{1}{2}} \left\|\partial_{x}\omega\right\|^{\frac{1}{2}} \left\|\partial_{x}^{2}\omega\right\|^{\frac{1}{2}} \left\|D^{\alpha}\omega\right\| \\[0.2cm]
	&\quad\leq C\left(1+E_{m-1}(t)+\|u_{1}\|^{2} \right) E_{m}(t)+\|\partial_{x}\omega\|^{2}\|\partial_{x}^{2}\omega\|^{2}+\|u_{1}\|^{2},
    \end{aligned}
    $$
    thus, we get
    $$
	\begin{aligned}
		&\sum_{0<\beta \leq \alpha} \left\|D^{\beta}u_{1}\right\|^{\frac{1}{2}} \left\|\partial_{y}D^{\beta}u_{1}\right\|^{\frac{1}{2}} \left\|\partial_{x}D^{\alpha-\beta}\omega\right\|^{\frac{1}{2}} \left\|\partial_{x}^{2}D^{\alpha-\beta}\omega\right\|^{\frac{1}{2}} \left\|D^{\alpha}\omega\right\| \\[0.2cm]
		&\quad\leq 
        \frac{1}{16}D_{m}(t)+C \left(1+\left\|\omega\right\|^{2}\left\|\omega\right\|_{H^{1}}^{2}+E_{m-1}(t)+\|u_{1}\|^{2} \right)E_{m}(t)  \\[0.2cm]
		&\quad\quad+E_{m-1}(t) \left(1+E_{m-1}(t)+\|u_{1}\|^{2}\right)+\|\partial_{x}\omega\|^{2}\|\partial_{x}^{2}\omega\|^{2}+\|u_{1}\|^{2}.
	\end{aligned}
	$$
	Similarly, we have
	$$
	\begin{aligned}
		&\sum_{0<\beta \leq \alpha} 
        \left\|D^{\beta}u_{2}\right\|^{\frac{1}{2}} \left\|\partial_{y}D^{\beta}u_{2}\right\|^{\frac{1}{2}} \left\|\partial_{y}D^{\alpha-\beta}\omega\right\|^{\frac{1}{2}} \left\|\partial_{x}\partial_{y}D^{\alpha-\beta}\omega\right\|^{\frac{1}{2}} \left\|D^{\alpha}\omega\right\| \\[0.2cm]
		&\quad\leq 
        \frac{1}{16}D_{m}(t)+C \left(1+\left\|\omega\right\|^{2}\left\|\omega\right\|_{H^{1}}^{2}+E_{m-1}(t)+\|u_{2}\|^{2} \right)E_{m}(t)  \\[0.2cm]
		&\quad\quad+E_{m-1}(t) \left(1+E_{m-1}(t)+\|u_{2}\|^{2}\right)+\|\partial_{y}\omega\|^{2}\|\partial_{x}\partial_{y}\omega\|^{2}+\|u_{2}\|^{2}.
	\end{aligned}
	$$
	Combining these estimates, we have
	$$
	\begin{aligned}
		&\sum_{0<\beta \leq \alpha}\binom{\alpha}{\beta} \int_{\Omega} D^\beta u \cdot \nabla D^{\alpha-\beta} \omega D^\alpha \omega \mathrm{d} x \mathrm{~d} y \\[0.2cm]
		&\quad\leq 
        \frac{2}{16}D_{m}(t)+C \left(1+\left\|\omega\right\|^{2}\left\|\omega\right\|_{H^{1}}^{2}+E_{m-1}(t)+\|u\|^{2} \right)E_{m}(t)  \\[0.2cm]
		&\quad\quad+E_{m-1}(t) \left(1+E_{m-1}(t)+\|u\|^{2}\right)+\|\omega\|_{H^{1}}^{2}\|\partial_{x}\omega\|_{H^{1}}^{2}+\|u\|^{2}.
	\end{aligned}
	$$
    Thus, we obtain
    $$
    \begin{aligned}
        \frac{d}{dt}E_{m}(t)+D_{m}(t)&\leq C\left(1+\|\omega\|^{2}\left\|\omega\right\|_{H^{1}}^{2}+E_{m-1}(t)+\|u\|^{2}\right)E_{m}(t) \\[0.2cm]
		&\quad+ E_{m-1}(t)\left(1+E_{m-1}(t)+\|u\|^{2}\right) +\|\omega\|_{H^{1}}^{2}\|\partial_{x}\omega\|_{H^{1}}^{2}+\|u\|^{2} .
	\end{aligned}
	$$
    The base $m=1$ is \eqref{eq:4.12}. Induction through $m=3$ and Lemma \ref{lem:divcurl} yield
    $$
    u^{I,0}\in L^{\infty}(0,T;H^{4}),\quad \partial_x u^{I,0}\in L^2(0,T;H^{4}).
    $$
    
    \textit{Step 4: Existence and Uniqueness.} The preceding estimates were obtained for smooth solutions. We now construct such a solution without regularizing the two velocity factors in different ways. For a scalar $\zeta$ in the homogeneous Dirichlet negative space, put
    $$
    \|\zeta\|_{\dot{H}_D^{-1}}^2:=\left\langle\zeta,\left(-\Delta_D\right)^{-1} \zeta\right\rangle, \quad K_D \zeta=\nabla^{\perp}\left(-\Delta_D\right)^{-1} \zeta .
    $$
    We first record the linear estimate used in the iteration. Let $a$ be divergence free and tangent, with 
    $$
    a \in C\left(\left[0, T_0\right] ; H^{4}\right), \quad \partial_x a \in L^2\left(0, T_0 ; H^{4}\right),
    $$ 
    and consider
    \begin{equation}
    \label{eq:4.14}
        \begin{gathered}
        \partial_t \zeta+a \cdot \nabla \zeta-\partial_x^2 \zeta=0, \quad \zeta(0)=\zeta_0.
        \end{gathered}
    \end{equation}
    A standard Friedrichs construction for this linear transport–diffusion equation gives a unique solution in
    $$
    \zeta \in L^{\infty}\left(0, T_0 ; H^{3} \cap \dot{H}_D^{-1}\right), \quad \partial_x \zeta \in L^2\left(0, T_0 ; H^{3}\right) .
    $$

    Set 
    $$
    \omega_0=\operatorname{curl} \tilde{u}, \quad u^{(0)}=e^{t \partial_x^2} \widetilde{u}, \quad \omega^{(0)}=e^{t \partial_x^2} \omega_0 .
    $$
    By the strong continuity of the heat semigroup, we obtain $u^{(0)} \in C\left(\left[0, T_0\right] ; H^{4}\right)$ and $\partial_{x}u^{(0)} \in L^{2}\left(0, T_0 ; H^{4}\right)$. Given $u^{(n)}$, solve 
    \begin{equation}
    \label{eq:4.15}
        \begin{gathered}
            \partial_t \omega^{(n+1)}+u^{(n)} \cdot \nabla \omega^{(n+1)}-\partial_x^2 \omega^{(n+1)}=0, \quad \omega^{(n+1)}(0)=\omega_0,
        \end{gathered}
    \end{equation}
    and define $u^{(n+1)}=K_D \omega^{(n+1)}$. We next verify in detail that the iterative equations satisfy the conditions required for applying \eqref{eq:4.14}, thereby ensuring that the iteration scheme is well defined.

    Noting that
    $$
    \|u^{(n)}\cdot\nabla \omega^{(n+1)} \|_{H^{2}} \lesssim \|u^{(n)}\|_{H^{4}} \|\omega^{(n+1)}\|_{H^{3}}.
    $$
    It follows directly from \eqref{eq:4.15} that 
    $$
    \partial_{t}\omega^{(n+1)} \in L^{2}(0,T_{0};H^{2}).
    $$ 
    Combining this with 
    $$
    \omega^{(n+1)} \in L^{\infty}\left(0, T_0 ; H^3\right),
    $$
    and using the standard weak-continuity argument, we obtain 
    \begin{equation}
    \label{eq:4.16}
    \begin{gathered}
        \omega^{(n+1)} \in C_w\left(\left[0, T_0\right] ; H^3\right).
    \end{gathered}
    \end{equation}
    Acting $D^{\alpha}$ on both sides of \eqref{eq:4.15}, then taking $L^{2}$ inner product with $D^{\alpha}\omega^{(n+1)}$, integrate by parts and summing about $|\alpha|\leq 3$, we get
    $$
    \frac{1}{2}\frac{d}{dt}\|\omega^{(n+1)}\|_{H^{3}}^{2}+\|\partial_{x}\omega^{(n+1)}\|_{H^{3}}^{2}=-\sum_{|\alpha|\leq3} <[D^{\alpha},u^{(n)}\cdot\nabla]\omega^{(n+1)},D^{\alpha}\omega^{(n+1)}>:=F(t).
    $$
    Using Lemma \ref{lem:Sobolev-products-prelim}, we get
    $$
    \begin{aligned}
    F(t)
    &\lesssim \sum_{|\alpha|\leq 3}
    \|[D^{\alpha},u^{(n)}\cdot\nabla]\omega^{(n+1)}\|_{2}
    \|D^{\alpha}\omega^{(n+1)}\|_{2} \\[0.2cm]
    &\lesssim
    \left(
     \|\nabla u^{(n)}\|_{L^{\infty}}
    \|\omega^{(n+1)}\|_{H^{3}}
    +
     \|u^{(n)}\|_{H^{3}}
    \|\nabla\omega^{(n+1)}\|_{L^{\infty}}
    \right)
    \|\omega^{(n+1)}\|_{H^{3}}.
    \end{aligned}
    $$
    Thus, we get
    $$
    \|\omega^{(n+1)}(t)\|_{H^{3}}^{2}=\|\omega_{0}\|_{H^{3}}^{2}+2\int_{0}^{t}\left( F(s)-\|\partial_{x}\omega^{(n+1)}(s)\|_{H^{3}}^{2}\right)ds \in C([0,T_{0}]).
    $$
    Combining with \eqref{eq:4.16}, we get
    $$
    \omega^{(n+1)} \in C([0,T_{0}];H^{3}).
    $$
    Noting that
    $$
    \|u^{(n)} \cdot \nabla \omega^{(n+1)}\|_{\dot{H}_D^{-1}} \lesssim \|u^{(n)} \omega^{(n+1)}\|_2 \lesssim \|u^{(n)}\|_{\infty}\|\omega^{(n+1)}\|_2 \in L^2\left(0, T_0\right),
    $$
    and
    $$
    \left\|\partial_x^2 \omega^{(n+1)}\right\|_{\dot{H}_D^{-1}} \lesssim \left\|\partial_x \omega^{(n+1)}\right\|_2 \in L^2\left(0, T_0\right) .
    $$
    Thus, we obtain 
    $$
    \partial_{t}\omega^{(n+1)} \in L^2\left(0, T_0 ; \dot{H}_D^{-1}\right).
    $$
    Combining with $\omega^{(n+1)} \in L^{\infty}\left(0, T_0 ; \dot{H}_D^{-1}\right)$, we get 
    $$
    \omega^{(n+1)} \in C\left(\left[0, T_0\right] ; \dot{H}_D^{-1}\right) .
    $$
    Since $\left\|K_D \zeta\right\|_2=\|\zeta\|_{\dot{H}_D^{-1}}$, we obtain
    $$
    u^{(n+1)} \in C\left(\left[0, T_0\right] ; L^2\right) .
    $$
    Using Lemma \ref{lem:divcurl}, we get
    $$
    u^{(n+1)} \in C\left(\left[0, T_0\right] ; H^{4}\right).
    $$
    However, the regularity
    $$
    \partial_x u^{(n+1)} \in L^2\left(0, T_0 ; H^4\right)
    $$
    follows directly from the regularity of $\omega^{(n+1)}$. Hence, the iterative system satisfies the assumptions required for applying \eqref{eq:4.14}.

    Choose $T_{0}>0$, depending only on $\|\tilde{u}\|_{H^{4}}$, and $R<\infty$ such that
    \begin{equation}
    \label{eq:4.17}
        \begin{aligned}
        & \sup _n \sup _{0 \leq t \leq T_0}\left(\left\|u^{(n)}(t)\right\|_{H^{4}}^2+\left\|\omega^{(n)}(t)\right\|_{H^{3}}^2\right) \\[0.2cm]
        & \quad+\sup _n \int_0^{T_0}\left(\left\|\partial_x u^{(n)}\right\|_{H^{4}}^2+\left\|\partial_x \omega^{(n)}\right\|_{H^{3}}^2\right) \mathrm{d} t \leq R .
        \end{aligned}
    \end{equation}
    Put 
    $$
    \delta \omega^{(n+1)}=\omega^{(n+1)}-\omega^{(n)}, \quad \delta u^{(n+1)}=K_D \delta \omega^{(n+1)}.
    $$
    For $n \geq 1$, 
    \begin{equation}
        \begin{aligned} 
        \partial_t \delta \omega^{(n+1)}+u^{(n)} \cdot \nabla \delta \omega^{(n+1)}-\partial_x^2 \delta \omega^{(n+1)} & =-\delta u^{(n)} \cdot \nabla \omega^{(n)} \\[0.2cm] 
        \delta \omega^{(n+1)}(0) & =0 .
        \end{aligned}
    \end{equation}
    Testing first by $\delta \omega^{(n+1)}$, and then by $\delta \psi^{(n+1)}=\left(-\Delta_D\right)^{-1} \delta \omega^{(n+1)}$, gives
    \begin{equation}
        \begin{aligned}
            \frac{1}{2} \frac{\mathrm{~d}}{\mathrm{~d} t}\left\|\delta \omega^{(n+1)}\right\|_2^2+\left\|\partial_x \delta \omega^{(n+1)}\right\|_2^2 &= <-\delta u^{(n)}\cdot\nabla \omega^{(n)},\delta\omega^{(n+1)}> \\[0.2cm]
            &\leq \|\nabla\omega^{(n)}\|_{L^{\infty}} \|\delta u^{(n)}\|_{2}\|\delta \omega^{(n+1)}\|_{2} \\[0.2cm]
            &\leq C_R\left(\left\|\delta u^{(n)}\right\|_2^2+\left\|\delta \omega^{(n+1)}\right\|_2^2\right).
        \end{aligned}
    \end{equation}
    Noting that 
     $$
    \begin{aligned}
    \left\langle u^{(n)}\cdot\nabla\delta\omega^{(n+1)},\delta\psi^{(n+1)}\right\rangle
    &=-\sum_{i,j}\int_\Omega
     \partial_i\delta\psi^{(n+1)}\,\partial_i u_j^{(n)}\,
     \partial_j\delta\psi^{(n+1)}\\[0.2cm]
    &\leq C\|\nabla u^{(n)}\|_{L^\infty}\|\delta u^{(n+1)}\|_2^2.
    \end{aligned}
    $$
    and
    $$
    \begin{gathered}
        <\delta u^{(n)}\cdot\nabla \omega^{(n)},\delta \psi^{(n+1)}> \lesssim \|\omega^{(n)}\|_{L^{\infty}}\|\delta u^{(n)}\|_{2} \|\nabla\delta\psi^{(n+1)}\|_{2}.
    \end{gathered}
    $$
    Thus, we have
    \begin{equation}
        \begin{gathered}
            \frac{1}{2} \frac{\mathrm{~d}}{\mathrm{~d} t}\left\|\delta u^{(n+1)}\right\|_2^2+\left\|\partial_x \delta u^{(n+1)}\right\|_2^2 \leq C_R\left(\left\|\delta u^{(n+1)}\right\|_2^2+\left\|\delta u^{(n)}\right\|_2^2\right) .
        \end{gathered}
    \end{equation}
    Put $X_{n+1}(t)=\left\|\delta u^{(n+1)}(t)\right\|_2^2+\left\|\delta \omega^{(n+1)}(t)\right\|_2^2$, Gronwall inequality gives
    $$
    \sup _{0 \leq t \leq T_0} X_{n+1}(t) \leq C_R T_0 e^{C_R T_0} \sup _{0 \leq t \leq T_0} X_n(t) .
    $$
    After decreasing $T_{0}$, the factor is below one. Hence both $u^{(n)}$ and $\omega^{(n)}$ are Cauchy in $C\left(\left[0, T_0\right] ; L^2\right)$. Let $(u,\omega)$ be the limit. By the weak formulation, we obtain 
    $$
    \partial_t \omega+u \cdot \nabla \omega-\partial_x^2 \omega=0, \quad u=K_D \omega,\left.\quad u_2\right|_{y=0}=0, \quad u(0)=\widetilde{u} .
    $$
    The curl of $\partial_t u+u \cdot \nabla u-\partial_x^2 u$ is zero. Since the half-plane is simply connected, a pressure $p$, unique up to a function of time, recovers \eqref{eq:1.3}-\eqref{eq:1.4}. By the above a priori estimate, the $H^{4}$-norm of the solution remains bounded on every finite time interval. Since the local existence time depends only on the $H^{4}$-norm of the solution, no finite maximal existence time can occur. Therefore, the local solution can be continued successively to any prescribed time interval $[0,T]$. The uniqueness follows from the same joint $L^2-\dot{H}_D^{-1}$ difference estimate as above, applied to two nonlinear solutions.

    By \eqref{eq:4.17}, together with the standard weak and weak-* compactness arguments, we obtain
    $$
    u\in L^{\infty}(0,T;H^{4}), \omega \in L^{\infty}(0,T;H^{3}),\quad \partial_{x}u\in L^{2}(0,T;H^{4}), \partial_{x}\omega \in L^{2}(0,T;H^{3}).
    $$
    Proceeding as in the proof of $u^{(n+1)} \in C\left(\left[0, T_0\right] ; H^4\right)$, we further obtain
    $$
    u \in C\left(\left[0, T_0\right] ; H^4\right),
    $$
    which is \eqref{eq:4.2}.

    \textit{Step 5: Actual Neumann pressure and regularity of time derivatives.} The pressure solves 
    \begin{equation}
    \label{eq:4.21}
        \begin{gathered}
            -\Delta p=\partial_i u_j \partial_j u_i =\operatorname{div} \left(u\cdot\nabla u \right) , \quad \partial_{y}p|_{y=0}=0 .
        \end{gathered}
    \end{equation}
    The energy estimate for the elliptic equation with Neumann boundary conditions yields
    $$
    \|\nabla p \|_{H^{4}} \lesssim \left(\|\partial_i u_j \partial_j u_i\|_{H^{3}}+\|\nabla p\|_{2} \right)
    $$
    Multiplying \eqref{eq:4.21} by $p$ and then integrating by parts, we obtain  
    $$
     \|\nabla p\|_{2}^{2} \leq \|u\cdot\nabla u\|_{2} \|\nabla p\|_{2}.
    $$
    Thus, we get
    $$
    \|\nabla p\|_{H^{4}} \lesssim \|u\|_{H^{4}}^{2} \in C([0,T]).
    $$
    A similar treatment yields that
    $$
    \|\partial_{x}\nabla p\|_{H^{4}} \lesssim \|\partial_{x}u\|_{H^{4}}\|u\|_{H^{4}} \in L^{2}(0,T). 
    $$
    A direct calculation yields
    $$
    \begin{aligned}
        \partial_{t}u &=\partial_{x}^{2}u-u\cdot\nabla u -\nabla p \in C([0,T];H^{2}),\\[0.2cm]
        \partial_{x}\partial_{t}u &= \partial_{x}^{3}u-\partial_{x}u\cdot\nabla u-u\cdot\nabla \partial_{x}u -\nabla \partial_{x}p \in L^{2}(0,T;H^{2}).
    \end{aligned}
    $$
    Acting $\partial_{t}$ on both sides of \eqref{eq:4.21}, a similar argument yields that
    $$
    \begin{aligned}
        \|\nabla \partial_{t}p\|_{H^{2}} &\lesssim \|\partial_{t}u\|_{H^{2}}\|u\|_{H^{4}} \in C([0,T]), \\[0.2cm]
        \|\partial_{x}\nabla\partial_{t}p\|_{H^{2}} &\lesssim \|\partial_{x}\partial_{t}u\|_{H^{2}}\|u\|_{H^{4}}+\|\partial_{t}u\|_{H^{2}}\|\partial_{x}u\|_{H^{4}} \in L^{2}(0,T).
    \end{aligned}
    $$
    A direct calculation yields
    $$
    \begin{aligned}
        \partial_{t}^{2}u=\partial_{x}^{2}\partial_{t}u -\partial_{t}u\cdot\nabla u -u\cdot\nabla \partial_{t}u -\nabla \partial_{t}p \in C([0,T];L^{2}), \\[0.2cm]
        \partial_{x}\partial_{t}^{2}u=\partial_{x}^{3}\partial_{t}u -\partial_{x}\partial_{t}\left(u\cdot\nabla u\right) - \partial_{x}\nabla \partial_{t}p \in L^{2}(0,T;L^{2}),
    \end{aligned}
    $$
    which is \eqref{eq:4.3}.

    \textit{The wall bounds.} Noticing that, for $h(xy)\in H_{x}^{k}H_{y}^{1}$ with fixed $k\in \mathbb{N}$, we have 
    $$
    \begin{aligned}
    \|\bar{h}\|_{H_x^k}^2 & =\sum_{i=0}^k \int_{\mathbb{R}}\left|\partial_x^i h(x, 0)\right|^2 \mathrm{~d} x \leq \sum_{i=0}^k \int_{\mathbb{R}}\left\|\partial_x^i h(x, y)\right\|_{L_y^{\infty}}^2 \mathrm{~d} x \\[0.2cm]
    & \lesssim \sum_{i=0}^k \int_{\mathbb{R}}\left\|\partial_x^i h(x, y)\right\|_{H_y^1}^2 \mathrm{~d} x \lesssim\left\|h(x, y)\right\|_{H_{x}^{k}H_{y}^{1}}^2.
    \end{aligned}
    $$
    Thus we have
    $$
    \begin{aligned}
    \|\overline{\partial_xp^{I,0}}\|_{H_x^3}
      &\leq C\|\partial_xp^{I,0}\|_{H_x^3H_y^1}
       \leq C\|\nabla p^{I,0}\|_{H^4},\\[0.2cm]
    \|\overline{\partial_xp^{I,0}}\|_{H_x^4}
      &\leq C\bigl(\|\nabla p^{I,0}\|_{H^4}
                      +\|\partial_x\nabla p^{I,0}\|_{H^4}\bigr).
    \end{aligned}
    $$
    The first bound is continuous in time, and the second is square-integrable. A similar treatment yields that
    $$
    \begin{aligned}
        -\partial_x\partial_t p^{I,0}(x,0,t) \in C([0,T];H_x^1) \cap L^2(0,T;H_x^2),\\[0.2cm]
        a\in C([0,T];H_{x}^{3}),\quad \partial_{x}a\in L^{2}(0,T;H_{x}^{3}).
    \end{aligned}
    $$
    Noting that the wall equation is
    \begin{equation}
    \label{eq:4.22}
        \begin{gathered}
            \partial_{t}a-\partial_{x}^{2}a+a\partial_{x}a=-\partial_{x}p^{I,0}(x,0,t),\quad a(x,0)=0.
        \end{gathered}
    \end{equation}
    Applying $\partial_x^j$ to \eqref{eq:4.22}, testing by $\partial_x^ja$, and summing over $0\leq j\leq5$, we obtain
    $$
    \begin{gathered}
         \frac{1}{2}\frac{d}{dt}\|a\|_{H_{x}^{5}}^{2}+\|\partial_{x}a\|_{H_{x}^{5}}^{2}=-\sum_{0\leq j\leq5}\left\langle\partial_{x}^{j}\left(a\partial_{x}a \right),\partial_{x}^{j}a\right\rangle-\sum_{0\leq j\leq 5} \left\langle \overline{\partial_{x}^{j+1}p^{I,0}},\partial_{x}^{j}a\right\rangle.
    \end{gathered}
    $$
    Noting that
    $$
    \begin{aligned}
        -\sum_{0\leq j\leq5}\left\langle \partial_{x}^{j}\left(a\partial_{x}a \right),\partial_{x}^{j}a\right\rangle \lesssim \left(1+\|a\|_{H_{x}^{2}}+\|\partial_{x}a\|_{H_{x}^{3}}^{2}+\|a\|_{H_{x}^{3}}^{2} \right) \|a\|_{H_{x}^{5}}^{2}, \\[0.2cm]
         -\sum_{0\leq j\leq5}\left\langle\overline{\partial_x^{j+1}p^{I,0}},\partial_x^ja\right\rangle
        \leq \frac14\|\partial_xa\|_{H_x^5}^2
        +C\|\overline{\partial_xp^{I,0}}\|_{H_x^4}^2+C\|a\|_2^2.
    \end{aligned}
    $$
    Combining these estimates and applying Gronwall's inequality gives
    $$
    a\in C([0,T];H_{x}^{5}),\quad \partial_{x}a \in L^{2}(0,T;H_{x}^{5}).
    $$
    A direct calculation yields
    $$
    \begin{aligned}
        \partial_{t}a & =\partial_{x}^{2}a-a\partial_{x}a-\overline{\partial_{x}p^{I,0}} \in C([0,T];H_{x}^{3}),\quad \partial_{x}\partial_{t}a \in L^{2}([0,T];H_{x}^{3}),\\[0.2cm]
        \partial_{t}^{2}a & =\partial_{x}^{2}\partial_{t}a-\partial_{t}\left(a\partial_{x}a \right)-\overline{\partial_{x}\partial_{t}p^{I,0}} \in C([0,T];H_{x}^{1}).
    \end{aligned}
    $$
    which is \eqref{eq:4.6}.
    
    It follows directly from the regularity of $\partial_{t}a$ that
    $$
    \alpha_0\in H_x^3.
    $$
    Using Minkowski’s inequality for integrals and Taylor’s formula:
    $$
    f(t)=\sum_{k=0}^{n}\frac{f^{(k)}(0)}{k!}t^k
       +\frac{1}{n!}\int_0^t f^{(n+1)}(s)(t-s)^n\,ds,
    $$
    we obtain
    $$
    \begin{aligned}
        \|a(t)\|_{H_{x}^{3}} &\lesssim \int_{0}^{t}\|\partial_{t}a\|_{H_{x}^{3}}ds \leq C t, \\[0.2cm]
        \|a(t)-t\alpha_0\|_{H_{x}^{1}} &\lesssim \int_{0}^{t} (t-s)\|\partial_{t}^{2}a\|_{H_{x}^{1}}ds \leq C t^{2}, \\[0.2cm]
        \|\partial_{t}a(t)-\alpha_0\|_{H_{x}^{1}} &\lesssim \int_{0}^{t} \|\partial_{t}^{2}a\|_{H_{x}^{1}}ds \leq C t.
    \end{aligned}
    $$
    A direct calculation yields
    $$
    \begin{aligned}
        \|\overline{\partial_yu_1^{I,0}}\|_{H_x^2} \lesssim \|u_{1}^{I,0}\|_{H^{4}} \in C([0,T]),\quad \|\overline{\partial_x\partial_yu_1^{I,0}}\|_{H_x^2} \lesssim \|\partial_{x}u_{1}^{I,0}\|_{H^{4}} \in L^{2}(0,T), \\[0.2cm]
        \|\overline{\partial_y^2u_1^{I,0}}\|_{H_x^1}\lesssim \|u_{1}^{I,0}\|_{H^{4}} \in C([0,T]),\quad \|\overline{\partial_x\partial_y^2u_1^{I,0}}\|_{H_x^1}\lesssim \|\partial_{x}u_{1}^{I,0}\|_{H^{4}} \in L^{2}(0,T). 
    \end{aligned}
    $$
    The proof is complete.
\end{proof}

\subsubsection{Regularity of the leading boundary-layer equation}
In this section, we present the full heat lifting and the self-similar profiles required for the corner analysis. The following finite-order estimates in the normal direction follow directly from the heat-kernel representation.
\begin{lemma}
    Let $g(x,0)=0$ and let $b$ solve 
    \begin{equation}
    \label{eq:5.1}
        \begin{gathered}
            b_t-b_{x x}-b_{z z}=0,\left.\quad b\right|_{z=0}=g(x, t),\left.\quad b\right|_{t=0}=0, \quad b \rightarrow 0 \quad(z \rightarrow \infty) .
        \end{gathered}
    \end{equation}
    Then
    \begin{equation}
    \label{eq:5.2}
        \begin{gathered}
            b(x, z, t)=\int_0^t K(z, t-s) e^{(t-s) \partial_x^2} g(x, s) \mathrm{d} s, \quad K(z, r)=\frac{z}{2 \sqrt{\pi} r^{3 / 2}} e^{-z^2 /(4 r)} .
        \end{gathered}
    \end{equation}
    If $g(x,t)=t^\lambda c(x)$, where $\lambda>0$, $m\geq0$ is an integer, and $c\in H_x^{m+2}$, set $L(x,z,t)=t^\lambda c(x)\Phi_\lambda(z/\sqrt t)$. Then, $0<t\leq T$, 
    \begin{equation}
    \label{eq:5.3}
        \sup_{z\geq0}\|b(\cdot,z,t)-L(\cdot,z,t)\|_{H_x^m}
          \leq C_{\lambda,m}t^{\lambda+1}\|c\|_{H_x^{m+2}},
    \end{equation}
    where
    \begin{equation}
    \label{eq:5.4}
        \begin{gathered}
            \Phi_\lambda(\eta)=\frac{\eta}{2 \sqrt{\pi}} \int_0^1 \sigma^\lambda(1-\sigma)^{-3 / 2} e^{-\eta^2 /[4(1-\sigma)]} \mathrm{d} \sigma,
        \end{gathered}
    \end{equation}
    \begin{equation}
    \label{eq:5.5}
        \begin{gathered}
            \Phi_\lambda^{\prime \prime}+\frac{\eta}{2} \Phi_\lambda^{\prime}-\lambda \Phi_\lambda=0, \quad \Phi_\lambda(0)=1, \quad \Phi_\lambda(\infty)=0,
        \end{gathered}
    \end{equation}
    \begin{equation}
    \label{eq:5.6}
        \begin{gathered}
            \int_0^{\infty} \Phi_\lambda(\eta) \mathrm{d} \eta=\frac{\Gamma(\lambda+1)}{\Gamma(\lambda+3 / 2)}.
        \end{gathered}
    \end{equation}
\end{lemma}
\begin{proof}
    The theory of the two-dimensional heat equation yields the Dirichlet boundary Poisson kernel
    $$
    K(z, r)=\frac{z}{2 \sqrt{\pi} r^{3 / 2}} e^{-z^2 /(4 r)},
    $$
    and hence \eqref{eq:5.2}. Putting $s=t \sigma, z=\sqrt{t} \eta$ gives
    $$
    \int_{0}^{t}K(z,t-s)s^{\lambda}ds=\frac{\eta t^{\lambda}}{2\sqrt{\pi}} \int_{0}^{1} \sigma^{\lambda}(1-\sigma)^{-\frac{3}{2}}e^{-\frac{\eta^{2}}{4(1-\sigma)}}d \sigma.
    $$
    Thus, we obtain \eqref{eq:5.4}.
    Noting that
    $$
    e^{(t-s) \partial_x^2}c(x)-c(x)=\int_{0}^{t-s}\frac{d}{d\rho}e^{\rho\partial_{x}^{2}}c(x)d\rho=(t-s)\int_{0}^{1}e^{\theta(t-s)\partial_{x}^{2}}\partial_{x}^{2}c(x)d\theta.
    $$
    Then, we have
    $$
    b(x,z,t)-L(x,z,t)=\int_{0}^{t}K(z,t-s)s^{\lambda}(t-s)\int_{0}^{1}e^{\theta(t-s)\partial_{x}^{2}}\partial_{x}^{2}c(x)d\theta ds
    $$
    Putting $s=t \sigma, z=\sqrt{t} \eta$ gives
    $$
    b(x,z,t)-L(x,z,t)=t^{\lambda+1}\frac{\eta}{2\sqrt{\pi}}\int_{0}^{1}\sigma^{\lambda}(1-\sigma)^{-\frac{1}{2}} e^{-\frac{\eta^{2}}{4(1-\sigma)}} \int_{0}^{1}e^{\theta t (1-\sigma)\partial_{x}^{2}} \partial_{x}^{2} c(x) d\theta d\sigma. 
    $$
    By the contractivity of the heat semigroup, we obtain
    $$
    \|e^{\theta t (1-\sigma) \partial_{x}^{2}} \partial_{x}^{2}c(x)\|_{H_{x}^{m}} \leq \|\partial_{x}^{2}c(x)\|_{H_{x}^{m}} \leq \|c\|_{H_{x}^{m+2}},\quad \forall t\geq0.
    $$
    Thus, Using Minkowski’s inequality for integrals, we have
    $$
    \|(b-L)(x,z,t)\|_{H_{x}^{m}} \leq t^{\lambda+1} \Psi_{\lambda}(\eta)\|c\|_{H_{x}^{m+2}},\quad \Psi_{\lambda}(\eta)= \frac{\eta}{2\sqrt{\pi}} \int_{0}^{1} \sigma^{\lambda}(1-\sigma)^{-\frac{1}{2}} e^{-\frac{\eta^{2}}{4(1-\sigma)}} d\sigma
    $$
    Noting that 
    $$
    | \Psi_{\lambda}(\eta) | \leq \frac{\eta}{2\sqrt{\pi}} \int_{0}^{1} (1-\sigma)^{-\frac{1}{2}}e^{-\frac{\eta^{2}}{4}}d\sigma \leq \frac{\eta e^{-\frac{\eta^{2}}{4}}}{\sqrt{\pi}}.
    $$
    Thus, we have
    $$
    \sup_{\eta \geq 0} |\Psi_{\lambda}(\eta)| < +\infty, \quad \sup_{z\geq 0} \|(b-L)(x,z,t)\|_{H_{x}^{m}} \leq C t^{\lambda+1}\|c\|_{H_{x}^{m+2}},
    $$
    which is \eqref{eq:5.3}. Since $L(x,z,t)=c(x)\int_{0}^{t}K(z,t-s)s^{\lambda}ds$, it satisfies
    $$
    \partial_{t}L(x,z,t)=\partial_{z}^{2}L(x,z,t).
    $$
    A direct calculation yields:
    $$
    \Phi_\lambda^{\prime \prime}+\frac{\eta}{2} \Phi_\lambda^{\prime}-\lambda \Phi_\lambda=0.
    $$
    With the successive changes of variables $q=1-\sigma$ and $\frac{\eta}{2\sqrt{q}}=s$, we have
    $$
    \Phi_{\lambda}(\eta)=\frac{2}{\sqrt{\pi}}\int_{\eta/2}^{\infty}(1-\frac{\eta^{2}}{4s^{2}})^{\lambda} e^{-s^{2}}ds.
    $$
    Thus, we obtain $\Phi_{\lambda}(0)=1$. Noting that
    $$
    0\leq \Phi_{\lambda}(\eta) \leq \frac{2}{\sqrt{\pi}} \int_{\eta/2}^{\infty} e^{-s^{2}}ds, 
    $$
    Thus, we obtain $\Phi_{\lambda}(\infty)=0$. A direct calculation yields:
    $$
    \begin{aligned}
        \int_{0}^{\infty}\Phi_{\lambda}(\eta)d\eta &=\frac{1}{2\sqrt{\pi}} \int_{0}^{\infty}\int_{0}^{1} \eta \sigma^{\lambda}(1-\sigma)^{-\frac{3}{2}}e^{-\frac{\eta^{2}}{4(1-\sigma)}} d\sigma d\eta \\[0.2cm]
        &=\frac{1}{2\sqrt{\pi}} \int_{0}^{1}\sigma^{\lambda}(1-\sigma)^{-\frac{3}{2}} \left(\int_{0}^{\infty}\eta e^{-\frac{\eta^{2}}{4(1-\sigma)}} d\eta \right) d\sigma \\[0.2cm]
        &=\frac{1}{\sqrt{\pi}}\int_{0}^{1}\sigma^{\lambda}(1-\sigma)^{-\frac{1}{2}}d\sigma=\frac{\Gamma(\lambda+1)}{\Gamma(\lambda+3 / 2)}.
    \end{aligned}
    $$
    The proof is complete.
\end{proof}

\begin{corollary}
    The solution of \eqref{eq:5.4} for $\lambda=1$ is
    $$
    \Phi_1(\eta)=\left(1+\frac{\eta^2}{2}\right) \operatorname{erfc}\left(\frac{\eta}{2}\right)-\frac{\eta}{\sqrt{\pi}} \mathrm{e}^{-\eta^2 / 4} .
    $$
    It satisfies
    $$
    \Phi_1^{\prime}(\eta)=\eta \operatorname{erfc}\left(\frac{\eta}{2}\right)-\frac{2}{\sqrt{\pi}} \mathrm{e}^{-\eta^2 / 4}, \quad \Phi_1^{\prime \prime}(\eta)=\operatorname{erfc}\left(\frac{\eta}{2}\right),
    $$
    and
    $$
    \int_0^{\infty} \Phi_1(\eta) \mathrm{d} \eta=\frac{4}{3 \sqrt{\pi}}.
    $$
\end{corollary}

\begin{proof}
    Upon making the change of variables $s=\frac{\eta}{2\sqrt{1-\sigma}}$ and integrating by parts, we obtain the expression for $\Phi_{1}(\eta)$. The stated properties of $\Phi_{1}(\eta)$ then follow directly by differentiation and \eqref{eq:5.6}.
\end{proof}

\begin{lemma}[The finite heat-lift bounds]
\label{lem:heat-lift-scale}
Let $s\geq0$ be an integer, $g\in C([0,T];H_x^s)$, $g(0)=0$, and $c=g_t-g_{xx}\in C([0,T];H_x^s)$, with the identity understood in distributions. For the heat lifting in \eqref{eq:5.1}, every fixed polynomial weight $\ell\geq0$ and $0\leq t\leq\min\{1,T\}$ satisfy
\begin{equation}
\label{eq:5.7}
    \begin{aligned}
    \|\langle z\rangle^\ell\partial_z^jb(t)\|_{H_x^sL_z^2}
       &\leq C_{s,\ell}\sup_{0\leq r\leq T}\|c(r)\|_{H_x^s}\,t^{5/4-j/2},
          \quad j=0,1,2,\\[0.2cm]
    \|\langle z\rangle^\ell z\partial_zb(t)\|_{H_x^sL_z^2}
       &\leq C_{s,\ell}\sup_{0\leq r\leq T}\|c(r)\|_{H_x^s}\,t^{5/4}.
    \end{aligned}
\end{equation}
If, in addition, $c\in C([0,T];H_x^{s+1})$, set $M=\sup_{0\leq r\leq T}\|c(r)\|_{H_x^{s+1}}$. Then
\begin{equation}
\label{eq:5.8}
    \begin{gathered}
    \sum_{j=0}^s\left\|\sup_{z\ge0}\left|\partial_x^j
    \int_0^z  \partial_{x}b(x,r,t)\,dr\right|\right\|_{L_x^2}
     \le C_s M t^{3/2},
    \end{gathered}
\end{equation}
\begin{equation}
\label{eq:5.9}
    \begin{gathered}
        \left(\int_0^t\|\langle z\rangle^\ell\nabla b\|_{H_x^sL_z^2}^2\,dr\right)^{1/2}
         +\sqrt t\left(\int_0^t\|\langle z\rangle^\ell\nabla \partial_{z}b\|_{H_x^sL_z^2}^2\,dr\right)^{1/2}
     \le C_{s,\ell}M t^{5/4}.
    \end{gathered}
\end{equation}
\end{lemma}

\begin{proof}
    Let $Q(z,h)=\operatorname{erfc}(\frac{z}{2\sqrt{h}})$. Indeed $\partial_{h}\left(e^{h\partial_{x}^{2}}g(x,t-h)\right)=-e^{h\partial_{x}^{2}}c(x,t-h)$ and $\partial_{h}Q(z,h)=K(z,h)$. Thus, we have
    \begin{equation}
        \begin{gathered}
            b(x,z,t)=\int_{0}^{t}e^{h\partial_{x}^{2}}g(x,t-h)K(z,h)dh=\int_{0}^{t}e^{h\partial_{x}^{2}}c(x,t-h)Q(z,h)dh.
        \end{gathered}
    \end{equation}
    Since $Q(z,h)=Q(z/\sqrt h,1)$, we have
    $\partial_z^jQ(z,h)=h^{-j/2}(\partial_\rho^jQ)(z/\sqrt h,1)$. Noting that
    $$
    \begin{aligned}
        \|\left<z\right>^{\ell}\partial_{z}^{j}Q(z,h) \|_{L_{z}^{2}}^{2} &= h^{\frac{1}{2} -j}\int_{0}^{\infty}(1+\rho^{2}h)^{\ell} |\partial_\rho^jQ(\rho,1)|^{2}d\rho \lesssim h^{\frac{1}{2} -j}(1+h^{\ell}) \\[0.2cm]
        \|\left<z \right>^{\ell} z \partial_{z}Q(z,h)\|_{L_{z}^{2}}^{2}&=h^{\frac{1}{2}} \int_{0}^{\infty}(1+\rho^{2}h)^{\ell}\rho^{2}|\partial_\rho Q(\rho,1)|^{2}d \rho \lesssim h^{\frac{1}{2}}(1+h^{\ell}).
    \end{aligned}
    $$
    Minkowski's inequality and the contractivity of the horizontal heat semigroup give
    $$
    \begin{aligned}
    \|\langle z\rangle^\ell\partial_z^jb(t)\|_{H_x^sL_z^2}
       &\leq\int_0^t\|c(t-h)\|_{H_x^s}
                    \|\langle z\rangle^\ell\partial_z^jQ(\cdot,h)\|_{L_z^2}\,dh \leq C_{s,\ell}\sup_{0\leq r\leq T}\|c(r)\|_{H_x^s}\,t^{5/4-j/2},\\[0.2cm]
    \|\langle z\rangle^\ell z\partial_zb(t)\|_{H_x^sL_z^2}
       &\leq\int_0^t\|c(t-h)\|_{H_x^s}
                    \|\langle z\rangle^\ell z\partial_zQ(\cdot,h)\|_{L_z^2}\,dh\leq C_{s,\ell}\sup_{0\leq r\leq T}\|c(r)\|_{H_x^s}\,t^{5/4}.
    \end{aligned}
    $$
    A similar treatment yields that
    $$
    \begin{aligned}
    \sum_{j=0}^s\left\|\sup_{z\ge0}\left|\partial_x^j
    \int_0^z \partial_{x}b(x,r,t)\,dr\right|\right\|_{L_x^2} &\leq \int_{0}^{\infty}\|b(x,z,t)\|_{H_{x}^{s+1}} dz \\[0.2cm]
    &\leq \int_{0}^{t}\|e^{h\partial_{x}^{2}}c(x,t-h)\|_{H_{x}^{s+1}}\|Q(z,h)\|_{L_{z}^{1}} dh \\[0.2cm]
    &\leq C_s M t^{3/2}.
    \end{aligned}
    $$
    Using \eqref{eq:5.7}, we have
    $$
    \|\langle z\rangle^\ell\nabla b\|_{H_x^sL_z^2} \leq C_{s,\ell} M t^{\frac{3}{4}},\quad \|\langle z\rangle^\ell\nabla \partial_{z}b\|_{H_x^sL_z^2} \leq C_{s,\ell} M t^{\frac{1}{4}}. 
    $$
    A direct calculation yields:
    $$
    \left(\int_0^t\|\langle z\rangle^\ell\nabla b\|_{H_x^sL_z^2}^2\,dr\right)^{1/2}
    +\sqrt t\left(\int_0^t\|\langle z\rangle^\ell\nabla \partial_{z}b\|_{H_x^sL_z^2}^2\,dr\right)^{1/2}
     \le C_{s,\ell}M t^{5/4}.
    $$
\end{proof}

\begin{remark}
    By \eqref{eq:4.7}, $a(x,t)=t\alpha_0(x)+O_{H_{x}^{1}}(t^{2})$. The boundary value of the leading layer is $-a$, so the first heat contribution is $-t\alpha_0(x)\Phi_{1}(\frac{z}{\sqrt{t}})$. The corresponding first term in $u_{2}^{b,1}$, obtained from incompressibility and decay at infinity, is $-t^{\frac{3}{2}} \partial_{x}\alpha_0(x)\int_{z / \sqrt{t}}^{\infty} \Phi_1(\sigma) \mathrm{d} \sigma $. At the wall, this leading term equals $-\frac{4}{3 \sqrt{\pi}} t^{3 / 2} \partial_x \alpha_0(x)$. The nonlinear proof that these model terms are the first terms of the actual layer is given in \cref{sec:corner-nonlinear-remainder}.
\end{remark}

\begin{proposition}[Existence and uniqueness for the leading boundary-layer equation]
\label{pro:boundary-layer equation}
Under the assumptions of Theorem \ref{thm:main}, and let $u^{I,0}$ be the solution given by Proposition \ref{prop:outer-flow}. Put
$$
a(x,t)=u_1^{I,0}(x,0,t),
\qquad
\pi(x,t)=\partial_x p^{I,0}(x,0,t).
$$
Then the leading boundary-layer problem \eqref{eq:3.3}  admits a unique solution $u_1^{b,0}$ on $[0,T]$ in the weighted regularity class below. If $g=\varphi a$ and $f=u_1^{b, 0}+g$, then, for every $\ell \geq 0$, 
\begin{equation}
\label{eq:5.11}
    \begin{gathered}
        \sup_{0\leq t \leq T}\sum_{j=0}^{5}\|\left<z \right>^{\ell+5-j}\partial_{x}^{j}f\|_{2}^{2}+\int_{0}^{T}\sum_{j=0}^{5}\|\left<z \right>^{\ell+5-j} \nabla \partial_{x}^{j}f\|_{2}^{2}ds \leq C_{\ell,T}.
    \end{gathered}
\end{equation}
The homogeneous profile belongs to $C([0,T];H_x^5L_{z,\ell}^2)$ and also satisfies
\begin{align}
 \langle z\rangle^\ell\partial_tf&\in L_T^2H_x^4L_z^2,\label{eq:5.25}\\[0.2cm]
 \langle z\rangle^\ell\partial_z^2f&\in L_T^2H_x^4L_z^2,\label{eq:5.26}\\[0.2cm]
 \langle z\rangle^\ell\partial_zf&\in C([0,T];H_x^4L_z^2),\qquad \partial_zf(0)=0.
 \label{eq:5.27}
\end{align}
Moreover, let $b_a$ denote the Dirichlet heat lifting of the wall datum $-a$,
$$
\left\{
\begin{aligned}
&\partial_t b_a-\partial_x^2 b_a-\partial_z^2 b_a=0,\\[0.2cm]
&b_a(x,0,t)=-a(x,t),\\[0.2cm]
&b_a|_{t=0}=0,\qquad b_a\to0\quad (z\to\infty),
\end{aligned}
\right.
$$
and write
\begin{equation}
\label{eq:5.12}
    \begin{gathered}
        u_1^{b,0}=b_a+r_a.
    \end{gathered}
\end{equation}
Then, for every fixed $\ell\ge0$ and $0\le t\le T$,
\begin{equation}
\label{eq:5.13}
    \begin{gathered}
        \|\langle z\rangle^\ell b_a(t)\|_{H_x^2L_z^2}+t^{1/2}
    \|\langle z\rangle^\ell\partial_z b_a(t)\|_{H_x^1L_z^2}
    \le C_{T,\ell}\,t^{5/4},
    \end{gathered}
\end{equation}
whereas the nonlinear remainder satisfies the strictly higher-order bound
\begin{equation}
\label{eq:5.14}
    \begin{gathered}
        \|\langle z\rangle^\ell r_a(t)\|_{H_x^2L_z^2}+t^{1/2}
        \|\langle z\rangle^\ell\partial_z r_a(t)\|_{H_x^1L_z^2}
        \le C_{T,\ell}\,t^{13/4}.
    \end{gathered}
\end{equation}
Finally, the first three tangential moments of the leading
boundary-layer profile satisfy
\begin{equation}
\label{eq:5.15}
    \begin{gathered}
        \left\|\partial_x^b\int_0^\infty u_1^{b,0}(x,z,t)\,dz
        \right\|_{L_x^2} \le C_T t^{3/2},\qquad b=0,1,2.
    \end{gathered}
\end{equation}
All constants depend only on the fixed time interval, the indicated polynomial weight, and the outer-flow bounds furnished by Proposition \ref{prop:outer-flow}.
\end{proposition}

\begin{proof}
    Set
    $$
    u=u_1^{b, 0}, \quad v=a+u, \quad w=-\int_0^z \partial_x v(x, s, t) \mathrm{d} s, \quad q=\partial_z v=\partial_z u
    $$
    The wall equation for the outer flow and $\partial_y u_2^{I, 0}|_{y=0}=-\partial_x a$ transform \eqref{eq:3.3} into
    \begin{equation}
    \label{eq:5.16}
        \begin{gathered}
            v_t-\Delta_{x, z} v+v \partial_x v+w \partial_z v=-\pi, \quad \partial_x v+\partial_z w=0,\left.\quad v\right|_{z=0}=0,\quad v|_{t=0}=0.
        \end{gathered}
    \end{equation}
    Differentiating in $z$ gives the exact no-stretching equation
    \begin{equation}
    \label{eq:5.17}
        \begin{gathered}
            q_t-\Delta_{x, z} q+v \partial_x q+w \partial_z q=0,\left.\quad \partial_z q\right|_{z=0}=\pi,\left.\quad q\right|_{t=0}=0 .
        \end{gathered}
    \end{equation}
    $f=u+\varphi a$ transform \eqref{eq:3.3} into
    \begin{equation}
    \label{eq:5.18}
        \begin{gathered}
            f_t-\Delta f+v \partial_x f+w \partial_z f+\left(\partial_x a\right) f=F_g,\quad f(x,0,t)=0,\quad f(x,z,0)=0,
        \end{gathered}
    \end{equation}
    where
    $$
    F_g=g_t-\Delta g+g \partial_x a+v \partial_x g+w \partial_z g .
    $$
    \textit{Local construction.} We first construct a solution near $t=0$. The estimates in this paragraph use a fixed polynomial weight; higher weights will be propagated by the a priori estimates below. For $0<\tau\leq1$, set
    $$
    \begin{aligned}
    \|h\|_\tau^2={}&\sup_{0\leq t\leq\tau}
       \bigl(E_{5,2}[h](t)+E_{4,2}[\partial_zh](t)\bigr)\\[0.2cm]
       &\quad+\int_0^\tau\bigl(D_{5,2}[h]+D_{4,2}[\partial_zh]
                       +E_{4,2}[\partial_th]\bigr)\,dt,\\[0.2cm]
    \|H\|_{\tau,*}^2={}&\int_0^\tau
       \bigl(\|\langle z\rangle^7H\|_2^2+E_{4,2}[H]\bigr)\,dt.
    \end{aligned}
    $$
    The first norm is taken on functions with zero initial and Dirichlet traces, continuous in the two spaces appearing in its supremum. The second norm provides one fewer tangential derivative for the top velocity energy. Define
    $$
    v^{(0)}=(1-\varphi)a,\qquad
    w^{(0)}=-\left(z-\int_0^z\varphi(r)\,dr\right)a_x,
    \qquad
    F_g^{(0)}=g_t-\Delta g+a_xg+v^{(0)}g_x+w^{(0)}g_z.
    $$
    Thus $v_x^{(0)}+w_z^{(0)}=0$ and $w^{(0)}|_{z=0}=0$. The known affine drift is kept on the left of the linear problem
    \begin{equation}
    \label{eq:leading-local-linear}
    \left\{
    \begin{aligned}
    h_t-\Delta h+v^{(0)}h_x+w^{(0)}h_z+a_xh&=H,\\
    h|_{z=0}=0,\qquad h|_{t=0}&=0.
    \end{aligned}
    \right.
    \end{equation}
    Put
    $$
    \delta_\tau=\sqrt\tau+\|a\|_{C([0,\tau];H_x^5)}
       +\|a\|_{L^2(0,\tau;H_x^6)}
       +\|a_t\|_{L^2(0,\tau;H_x^4)}.
    $$
    Proposition \ref{prop:outer-flow} and $a(0)=0$ imply $\delta_\tau\to0$ as $\tau\downarrow0$. For $\delta_\tau$ sufficiently small, the weighted linear estimate is
    \begin{equation}
    \label{eq:leading-local-linear-bound}
       \|h\|_\tau\leq C\|H\|_{\tau,*},
       \qquad \|F_g^{(0)}\|_{\tau,*}\leq C(\delta_\tau+\delta_\tau^2),
    \end{equation}
    with $C$ independent of $\tau$.

    Here are the estimates for \eqref{eq:leading-local-linear-bound}. Test the $j$th tangential equation by $\langle z\rangle^{2(7-j)}\partial_x^jh$, $0\leq j\leq5$. The undifferentiated transport cancels, apart from the derivative of the weight. For $j\geq1$, integrate the source once in $x$. A derivative of the affine coefficient contributes
    $$
       \partial_x^{k+1}a\left(z-\int_0^z\varphi(r)\,dr\right)
                         \partial_z\partial_x^{j-k}h,
       \qquad 1\leq k\leq j.
    $$
    Its factor $z$ is controlled by the weight at index $j-k$. When $j=k=5$, put $\partial_x^6a$ in $L_t^2L_x^2$ and $\langle z\rangle^3h_z$ in $L_t^\infty L_x^\infty L_z^2$; the latter follows from $E_{4,2}[h_z]$. All other coefficient derivatives are controlled by $\|a\|_{CH_x^5}$. Young's inequality gives
    $$
    \sup_{0\leq t\leq\tau}E_{5,2}[h](t)+\int_0^\tau D_{5,2}[h] \,dt
    \leq C\|H\|_{\tau,*}^2+C\delta_\tau\|h\|_\tau^2.
    $$
    Next write $h_t-\Delta h=H-v^{(0)}h_x-w^{(0)}h_z-a_xh$ and apply the weighted zero-data heat estimate at tangential indices $0\leq j\leq4$, with weight $\langle z\rangle^{6-j}$. The extra normal weight in $D_{5,2}[h]$ controls the affine drift. Derivatives of the weights are bounded by the same or lower weights. Consequently,
    $$
    \begin{aligned}
    &\sup_{0\leq t\leq\tau}E_{4,2}[h_z](t)
       +\int_0^\tau\bigl(D_{4,2}[h_z]+E_{4,2}[h_t]\bigr)\,dt\\
    &\quad\leq C\|H\|_{\tau,*}^2+C\int_0^\tau D_{5,2}[h]\,dt
                  +C\delta_\tau^2\|h\|_\tau^2.
    \end{aligned}
    $$
    Substitute the first estimate into the term
    $\int_0^\tau D_{5,2}[h]\,dt$ on the right-hand side of the second estimate. Adding the resulting bounds gives
    \[
     \|h\|_\tau^2\leq C\|H\|_{\tau,*}^2
          +C(\delta_\tau+\delta_\tau^2)\|h\|_\tau^2.
    \]
    Decreasing $\tau$ so that the last coefficient is at most $1/2$ proves
    the first estimate in \eqref{eq:leading-local-linear-bound}. The second follows from the compact support of $g$ and the norms in $\delta_\tau$. The linear solution is obtained first on finite strips with zero Dirichlet traces by Galerkin approximation, followed by passage to the half-plane. These bounds are independent of the strip and approximation parameters, and the weighted heat estimate supplies the stated initial continuity.

    For the nonlinear terms, the weighted product estimates give
    \begin{equation}
    \label{eq:leading-local-products}
    \left\|h k_x\right\|_{\tau,*}
     +\left\|\left(\int_0^z h_x\,dr\right)k_z\right\|_{\tau,*}
       \leq C\|h\|_\tau\|k\|_\tau.
    \end{equation}
    To check the primitive term, use
    $$
    \left\|\int_0^z\partial_x^{j+1}h\,dr\right\|_{L_x^2L_z^\infty}
       \leq C\|\langle z\rangle^2\partial_x^{j+1}h\|_2,
       \qquad 0\leq j\leq4.
    $$
    At $j=4$, place this factor in $L_t^\infty$ and the weighted $k_z$ factor in $L_t^2H_x^1L_z^2$. When fewer derivatives fall on the primitive, use the horizontal Sobolev inequality on that factor instead. The same placements, with the normal Sobolev inequality for the highest derivative of $h$, prove the first term in \eqref{eq:leading-local-products}. The zeroth-order weight $\langle z\rangle^7$ is carried by the differentiated factor $k_x$ or $k_z$ in $D_{5,2}[k]$. The compactly supported lifting also satisfies
    $$
    \left\|h g_x-\left(\int_0^z h_x\,dr\right)g_z\right\|_{\tau,*}
        \leq C\delta_\tau\|h\|_\tau.
    $$

    Starting from $f^{(0)}=0$, solve \eqref{eq:leading-local-linear} with right-hand side
    $$
    H^{(n)}=F_g^{(0)}+f^{(n)}g_x
       -\left(\int_0^z f_x^{(n)}\,dr\right)g_z
       -f^{(n)}f_x^{(n)}
       +\left(\int_0^z f_x^{(n)}\,dr\right)f_z^{(n)}
    $$
    to define $f^{(n+1)}$. Equations \eqref{eq:leading-local-linear-bound}--\eqref{eq:leading-local-products} imply
    $$
    \begin{aligned}
    \|f^{(n+1)}\|_\tau
      &\leq C(\delta_\tau+\delta_\tau^2)
               +C\delta_\tau\|f^{(n)}\|_\tau+C\|f^{(n)}\|_\tau^2,\\[0.2cm]
    \|f^{(n+1)}-f^{(n)}\|_\tau
      &\leq C\bigl(\delta_\tau+\|f^{(n)}\|_\tau+
                    \|f^{(n-1)}\|_\tau\bigr)
                       \|f^{(n)}-f^{(n-1)}\|_\tau.
    \end{aligned}
    $$
    Choose $R>0$ so that $2CR<1/4$, and then choose $\tau$ so that
    $C(\delta_\tau+\delta_\tau^2)\leq R/2$ and $C\delta_\tau<1/4$.
    The iterates stay in the ball $\|f\|_\tau\leq R$, and their differences contract by a factor smaller than $1/2$. Their limit solves \eqref{eq:5.18}, since
    $$
    v=v^{(0)}+f,\qquad w=w^{(0)}-\int_0^z f_x\,dr.
    $$
    This construction uses only the wall norms in Proposition \ref{prop:outer-flow}. At the initial corner, the solution has the energy traces specified above; no time-differentiated Dirichlet condition is imposed.

    \textit{A priori estimates.} We now obtain bounds on every finite time interval. The following calculations can first be made for smooth approximations of the nonlinear solution and then passed to the limit. In particular, the shear equation \eqref{eq:5.17} is used for the nonlinear solution, not for the linear Picard iterates.
    
    Multiplying \eqref{eq:5.18} by $f$ and then integrating by parts, we obtain 
    $$
    \frac{1}{2}\frac{d}{dt}\|f\|_{2}^{2} +\|\nabla f\|_{2}^{2} = -\left\langle(\partial_{x}a)f,f\right\rangle+\left\langle F_{g}, f\right\rangle.
    $$
    Noting that
    $$
    \left\|\int_{0}^{z}\partial_{x}u(x,s,t)ds \right\|_{L^2(\mathbb{R} \times(0,1))} \leq C\left\|\partial_x u\right\|_{L^2(\mathbb{R} \times(0,1))}.
    $$
    On the support of $\varphi^{\prime}$, we have
    $$
    \begin{aligned}
    \left|\left\langle \int_{0}^{z}\partial_{x}u(x,s,t)ds \cdot \varphi^{\prime} a, f\right\rangle\right| & \leq C\|a\|_{L_x^{\infty}}\left(\left\|\partial_x f\right\|_2+\left\|\partial_x g\right\|_2\right)\|f\|_{L^2(\mathbb{R} \times(0,1))} \\[0.2cm]
    & \leq \frac{1}{8}\left\|\partial_x f\right\|_2^2+C\|a\|_{H_x^1}^2\|f\|_2^2 +C\|a\|_{H_x^1}^2 .
    \end{aligned}
    $$
    Thus, we obtain
    $$
    \frac{\mathrm{d}}{\mathrm{~d} t}\|f\|_2^2+\|\nabla f\|_2^2 \leq C\left(1+\|a\|_{H_{x}^{2}}+\|a\|_{H_{x}^{2}}^{2} \right) \|f\|_{2}^{2}+\|\partial_{t}a\|_{2}^{2}+\|a\|_{H_{x}^{2}}^{2},
    $$
    which together with \eqref{eq:4.6} and Gronwall inequality, implies
    \begin{equation}
    \label{eq:5.19}
        \begin{gathered}
            \sup _{0 \leq t \leq T}\|f(t)\|_2^2+\int_0^T\|\nabla f\|_2^2 \mathrm{~d} t \leq C_T .
        \end{gathered}
    \end{equation}
    In particular, $\partial_x u=\partial_x f-\partial_x g \in L^2\left(0, T ; L^2\right)$. Multiplying \eqref{eq:5.17} by $\left\langle z \right\rangle^{2m}q$ and then integrating by parts, we obtain
    $$
    \begin{aligned}
        \frac{1}{2}\frac{d}{dt}\|\left\langle z \right\rangle^{m}q\|_{2}^{2} +\|\left\langle z\right\rangle^{m}\nabla q\|_{2}^{2}=&-\int_{\mathbb{R}}\pi\cdot q|_{z=0} dx -2m \iint \left\langle z\right\rangle^{2m-2}z \cdot\partial_{z}q \cdot q dxdz \\[0.2cm]
        &-\left\langle v\partial_{x}q, \langle z\rangle^{2m}q\right\rangle-\left\langle w\partial_{z}q,\langle z\rangle^{2m}q\right\rangle.
    \end{aligned}
    $$
    Using trace inequality, we have
    $$
    \int_{\mathbb{R}} \pi\cdot q|_{z=0}dx \leq \|q|_{z=0}\|_{L_{x}^{2}}\|\pi\|_{L_{x}^{2}} \leq \frac{1}{8}\|\partial_{z}q\|_{2}^{2}+C\|q\|_{2}^{2}+C\|\pi\|_{L_{x}^{2}}^{2}.
    $$
    Noting that
    $$
    \begin{aligned}
        \left\langle v\partial_{x}q, \langle z\rangle^{2m}q\right\rangle+\left\langle w\partial_{z}q,\langle z\rangle^{2m}q\right\rangle =&m \iint \partial_{x}a\left\langle z\right\rangle^{2m-2} z^{2} q^{2}dxdz \\[0.2cm]
        &+ m \iint\left\langle z\right\rangle^{2m-2}z q^{2} \left(\int_0^z \partial_x u(x, s, t) \mathrm{d} s\right) dxdz.
    \end{aligned}
    $$
    For the first term on the right-hand side, a direct estimate gives:
    $$
    m \iint \partial_{x}a\left\langle z\right\rangle^{2m-2} z^{2} q^{2}dxdz \leq C\|a\|_{H_{x}^{2}} \|\left\langle z\right\rangle^{m}q\|_{2}^{2}
    $$
    For the second term, the 2-D Ladyzhenskaya inequality yields
    $$
    \begin{aligned}
        m \iint\left\langle z\right\rangle^{2m-2}z q^{2} \left(\int_0^z \partial_x u(x, s, t) \mathrm{d} s\right) dxdz &\leq C \|\frac{\int_0^z \partial_x u(x, s, t) \mathrm{d} s}{z}\|_{2}\|\left\langle z\right\rangle^{m}q\|_{4}^{2} \\[0.2cm]
        &\leq \frac{1}{8}\|\left\langle z\right\rangle^{m}\nabla q\|_{2}^{2}+C\left(1+\left\|\partial_{x}u\right\|_{2}^{2}\right) \|\left\langle z\right\rangle^{m}q\|_{2}^{2}.
    \end{aligned}
    $$
    Combining these estimates, we have
    $$
    \frac{d}{dt} \left\|\left\langle z\right\rangle^{m}q\right\|_{2}^{2}+\|\left\langle z\right\rangle^{m}\nabla q\|_{2}^{2}\leq C\left(1+\|a\|_{H_{x}^{2}} +\|\partial_{x}u\|_{2}^{2} \right) \|\left\langle z\right\rangle^{m}q\|_{2}^{2}+C\|\pi\|_{L_{x}^{2}}^{2},
    $$
    which together with \eqref{eq:4.6}, \eqref{eq:5.19} and Gronwall inequality, implies
    \begin{equation}
    \label{eq:5.20}
        \begin{gathered}
            \sup_{0\leq t \leq T} \|\left\langle z\right\rangle^{m}q\|_{2}^{2}+\int_{0}^{T}\|\left\langle z\right\rangle^{m}\nabla q\|_{2}^{2}dt \leq C_{T,m}.
        \end{gathered}
    \end{equation}
    Using 2-D Ladyzhenskaya inequality, we have
    \begin{equation}
    \label{eq:5.21}
        \begin{gathered}
            \int_{0}^{T}\|\left\langle z\right\rangle^{m}q\|_{4}^{4} dt \leq C\int_{0}^{T} \|\left\langle z\right\rangle^{m}q\|_{2}^{2} \|\nabla \left(\left\langle z\right\rangle^{m}q\right)\|_{2}^{2} dt \leq C_{T,m}.
        \end{gathered}
    \end{equation}

    We next establish the higher-order weighted regularity of $f$. The case $N=0$ follows from \eqref{eq:5.20}, since $u=-\int_z^\infty q\,dr$ and
    $$
    \|\langle z\rangle^m u\|_2+\|\langle z\rangle^m\partial_xu\|_2
    \leq C_m\bigl(\|\langle z\rangle^{m+2}q\|_2
                     +\|\langle z\rangle^{m+2}\partial_xq\|_2\bigr).
    $$
    For $N=1$, we use
    $\partial_x\partial_zf=\partial_xq+\varphi'\partial_xa$;
    its squared $L^2$ norm is integrable by \eqref{eq:5.20} and
    Proposition \ref{prop:outer-flow}. For $N\geq2$, the corresponding
    lower derivatives are controlled by the preceding tangential energy.

    For $1\leq N\leq5$, use the abbreviated notation
    \[
    E_{N,\ell}=E_{N,\ell}[f],\qquad D_{N,\ell}=D_{N,\ell}[f].
    \]
    Acting $\partial_{x}^{j}$ on both sides of \eqref{eq:5.18}, then taking $L^{2}$ inner product with $\left\langle z\right\rangle^{2(\ell+N-j)}\partial_{x}^{j}f$, integrate by parts and summing about $0\leq j \leq N$, we obtain
    \begin{equation}
    \label{eq:5.22}
        \begin{gathered}
            \frac{1}{2}\frac{d}{dt}E_{N,\ell}+D_{N,\ell}=\sum_{i=1}^{5}I_{i},
        \end{gathered}
    \end{equation}
    where
    $$
    \begin{gathered}
        I_{1}=-\sum_{j=0}^{N}2(\ell+N-j)\iint \left\langle z\right\rangle^{2(\ell+N-j)-2}z \partial_{x}^{j}\partial_{z}f \cdot \partial_{x}^{j}f dxdz, \\[0.2cm]
        I_{2}=-\sum_{j=0}^{N} \left\langle \partial_{x}^{j}\left(v\partial_{x}f\right),\left\langle z\right\rangle^{2(\ell+N-j)}\partial_{x}^{j}f \right\rangle, \quad I_{3}=-\sum_{j=0}^{N} \left\langle \partial_{x}^{j}\left(w\partial_{z}f\right),\left\langle z\right\rangle^{2(\ell+N-j)}\partial_{x}^{j}f \right\rangle, \\[0.2cm]
        I_{4}=-\sum_{j=0}^{N} \left\langle \partial_{x}^{j}\left((\partial_{x}a)f \right),\left\langle z\right\rangle^{2(\ell+N-j)}\partial_{x}^{j}f \right\rangle,\quad I_{5}=\sum_{j=0}^{N} \left\langle \partial_{x}^{j}F_{g},\left\langle z\right\rangle^{2(\ell+N-j)}\partial_{x}^{j}f \right\rangle.
    \end{gathered}
    $$
    For $I_{1}$, we have
    $$
    I_{1}\leq \frac{1}{16}D_{N,\ell}+CE_{N,\ell}.
    $$
    Since $\partial_xv+\partial_zw=0$, the undifferentiated
    transport terms cancel apart from the derivative of the weight.
    With $I_{34}$ defined below, write
    $$
    I_2=I_{21}+I_{22}-I_{34},
    $$
    where
    $$
    \begin{aligned}
        I_{21}&=-\sum_{j=1}^{N}\sum_{k=0}^{j-1}\binom{j}{k}
        \left\langle (1-\varphi)\partial_x^{j-k}a\,
        \partial_x^{k+1}f,
        \langle z\rangle^{2(\ell+N-j)}\partial_x^jf\right\rangle,\\[0.2cm]
        I_{22}&=-\sum_{j=1}^{N}\sum_{k=0}^{j-1}\binom{j}{k}
        \left\langle \partial_x^{j-k}f\,\partial_x^{k+1}f,
        \langle z\rangle^{2(\ell+N-j)}\partial_x^jf\right\rangle.
    \end{aligned}
    $$
    For $I_{21}$, we have
    $$
    I_{21}\leq C\left(1+\|a\|_{H_{x}^{6}}^{2}\right) E_{N,\ell}.
    $$
    For $I_{22}$, Lemma \ref{lemma:2.2} gives
    $$
    \begin{aligned}
    |I_{22}|\lesssim
    \sum_{j=1}^{N}\sum_{k=0}^{j-1}
    &\|\partial_x^{j-k}f\|^{1/2}
     \|\partial_z\partial_x^{j-k}f\|^{1/2}
     \|\langle z\rangle^{\ell+N-j}\partial_x^{k+1}f\|^{1/2}\\
    &\times\|\langle z\rangle^{\ell+N-j}\partial_x^{k+2}f\|^{1/2}
     \|\langle z\rangle^{\ell+N-j}\partial_x^jf\|.
    \end{aligned}
    $$
    When $N=1$, Young's inequality and
    $\partial_x\partial_zf=\partial_xq+\varphi'\partial_xa$ yield
    $$
    \begin{aligned}
    |I_{22}|
    &\leq \frac1{16}D_{1,\ell}
      +C\|\partial_xf\|^{2/3}\|\partial_x\partial_zf\|^{2/3}E_{1,\ell}\\
    &\leq \frac1{16}D_{1,\ell}
      +C\left(1+D_{0,\ell+2}+\|\partial_xq\|_2^2
                      +\|a\|_{H_x^1}^2\right)E_{1,\ell}.
    \end{aligned}
    $$
    For $2\leq N\leq5$, the same product estimate and the lower-order
    energies give
    $$
    |I_{22}|\leq\frac1{16}D_{N,\ell}
       +C\left(1+E_{N-1,\ell+2}+D_{N-1,\ell+2}\right)
          \left(1+E_{N,\ell}\right).
    $$
    For $I_{3}$, we have
    $$
    I_{3}=I_{31}+I_{32}+I_{33}+I_{34}+I_{35},
    $$
    where
    $$
    \begin{gathered}
        I_{31}=\sum_{j=0}^{N}\sum_{k=0}^{j-1} \binom{j}{k}\left\langle z\partial_{x}^{j-k+1}a\cdot\partial_{z}\partial_{x}^{k}f,\left\langle z\right\rangle^{2(\ell+N-j)} \partial_{x}^{j}f \right\rangle, \\[0.2cm]
        I_{32}=\sum_{j=0}^{N}\sum_{k=0}^{j-1} \binom{j}{k}\left\langle \int_{0}^{z}\partial_{x}^{j-k+1}f(x,s,t)ds \cdot\partial_{z}\partial_{x}^{k}f,\left\langle z\right\rangle^{2(\ell+N-j)} \partial_{x}^{j}f \right\rangle, \\[0.2cm]
        I_{33}=-\sum_{j=0}^{N}\sum_{k=0}^{j-1} \binom{j}{k}\left\langle \int_{0}^{z}\varphi(s)ds\partial_{x}^{j-k+1}a \cdot\partial_{z}\partial_{x}^{k}f,\left\langle z\right\rangle^{2(\ell+N-j)} \partial_{x}^{j}f \right\rangle, \\[0.2cm]
        I_{34}=\sum_{j=0}^{N}\frac{1}{2} \iint \partial_{z}w \left\langle z\right\rangle^{2(\ell+N-j)}(\partial_{x}^{j}f)^{2}dxdz,\\[0.2cm]
        I_{35}=\sum_{j=0}^{N} (\ell+N-j) \iint w \left\langle z\right\rangle^{2(\ell+N-j)-2}z(\partial_{x}^{j}f)^{2}dxdz.
    \end{gathered}
    $$
    For $I_{31}$ and $I_{33}$, a direct calculation yields:
    $$
    I_{31}+I_{33}\leq CD_{N-1,\ell+2} E_{N,\ell} + \|a\|_{H_{x}^{6}}^{2}.
    $$
    For $I_{32}$, we distinguish the following cases. When $k=0$,
    the identity $\partial_zf=q+\varphi'a$ and the compact support of
    $\varphi'$ give
    $$
    \begin{aligned}
    &\sum_{j=1}^{N}\left|
    \left\langle\int_0^z\partial_x^{j+1}f(x,s,t)\,ds\,
       \partial_zf,\langle z\rangle^{2(\ell+N-j)}\partial_x^jf
    \right\rangle\right|\\[0.2cm]
    &\quad\leq C\sum_{j=1}^{N}
       \left\|\frac{\int_0^z\partial_x^{j+1}f\,ds}{z}\right\|_2
       \left(\|z\langle z\rangle^{\ell+N-j}q\|_4
                         +C_{\ell,N}\|a\|_{H_x^1}\right)
       \|\langle z\rangle^{\ell+N-j}\partial_x^jf\|_4\\[0.2cm]
    &\quad\leq\frac1{16}D_{N,\ell}
       +C\left(1+\|\langle z\rangle^{\ell+N+1}q\|_4^4
                         +\|a\|_{H_x^1}^4\right)E_{N,\ell}.
    \end{aligned}
    $$
    For $N\geq 3, 1\leq k\leq j-2$, Cauchy–Schwarz in $z$ and the one-dimensional Sobolev inequality give
    $$
    \begin{aligned}
        &\sum_{j=3}^{N}\sum_{k=1}^{j-2} \left\langle \int_{0}^{z}\partial_{x}^{j-k+1}f(x,s,t)ds \cdot\partial_{z}\partial_{x}^{k}f,\left\langle z\right\rangle^{2(\ell+N-j)} \partial_{x}^{j}f \right\rangle \\[0.2cm]
        &\quad\leq C\sum_{j=3}^{N}\sum_{k=1}^{j-2} \|\partial_{x}^{j-k+1}f\|_{2} \left\|\langle z\rangle^{\ell+N-j+1} \partial_z \partial_x^k f\right\|_{H_x^1 L_z^2}\left\|\langle z\rangle^{\ell+N-j} \partial_x^j f\right\|_2  \\[0.2cm]
        &\quad\leq C\left(1+D_{N-1, \ell+2}\right) E_{N, \ell}.
    \end{aligned}
    $$
    For $N\geq 3, k=j-1$, we have
    $$
    \begin{aligned}
        &\sum_{j=2}^{N}\left\langle \int_{0}^{z}\partial_{x}^{2}f(x,s,t)ds \cdot\partial_{z}\partial_{x}^{j-1}f,\left\langle z\right\rangle^{2(\ell+N-j)} \partial_{x}^{j}f \right\rangle \\[0.2cm]
        & \leq C \sum_{j=2}^N\left\|\partial_x^2 f\right\|_{L_x^{\infty} L_z^2}\left\|\langle z\rangle^{\ell+N-j+1} \partial_z \partial_x^{j-1} f\right\|_2\left\|\langle z\rangle^{\ell+N-j} \partial_x^j f\right\|_2 \\[0.2cm]
        &\leq  C D_{N-1,\ell+2} \left(1+E_{N,\ell}\right)+C E_{N-1,\ell+2}.
    \end{aligned}
    $$
    For $N=2, k=1$, Proceeding as above, we obtain
    $$
    \begin{aligned}
        &\iint \int_{0}^{z}\partial_{x}^{2}f(x,s,t)ds \cdot\partial_{z}\partial_{x}f \cdot \left\langle z\right\rangle^{2\ell} \partial_{x}^{2}f dxdz \\[0.2cm]
        &\leq C\left\|\partial_x^2 f\right\|_{L_x^{\infty} L_z^2}\left\|\langle z\rangle^{\ell+1} \partial_z \partial_x f\right\|_2\left\|\langle z\rangle^{\ell} \partial_x^2 f\right\|_2  \\[0.2cm]
        &\leq \frac{1}{16} D_{2, \ell}+C\left(1+D_{1, \ell+2}\right) E_{2, \ell} ..
    \end{aligned}
    $$
    Thus, we obtain
    $$
    \begin{aligned}
    |I_{32}|\leq\frac{3}{16}D_{N,\ell}
      +C\bigl(1+E_{N-1,\ell+2}+D_{N-1,\ell+2}
      +\|\langle z\rangle^{\ell+N+1}q\|_4^4
      +\|a\|_{H_x^1}^4\bigr)(1+E_{N,\ell}).
    \end{aligned}
    $$
    The terms $-I_{34}$ in $I_2$ and $I_{34}$ in $I_3$ cancel exactly, since $\partial_zw=-\partial_xv$.
    Using Hardy inequality and 2D Ladyzhenskaya inequality, we have
    $$
    \begin{aligned}
        I_{35}&\leq \|a\|_{H_{x}^{2}}E_{N,\ell}+C\sum_{j=0}^{N} \|\frac{\int_{0}^{z}\partial_{x}f(x,s,t)ds }{z}\|_{2} \|\left\langle z\right\rangle^{\ell+N-j} \partial_{x}^{j}f\|_{4}^{2} \\[0.2cm]
        &\leq \|a\|_{H_{x}^{2}}E_{N,\ell}+C\sum_{j=0}^{N} \|\partial_{x}f\|_{2} \|\left\langle z\right\rangle^{\ell+N-j} \partial_{x}^{j}f\|_{2} \|\nabla\left(\left\langle z\right\rangle^{\ell+N-j} \partial_{x}^{j}f\right)\|_{2} \\[0.2cm]
        &\leq \frac{1}{16}D_{N,\ell}+C\left(1+\|a\|_{H_{x}^{2}}+\|\partial_{x}f\|_{2}^{2}\right)E_{N,\ell}. 
    \end{aligned}
    $$
    For $I_{4}$, we have
    $$
    \begin{aligned}
        I_{4}&= \sum_{j=1}^{N}\left\langle \partial_x^{j-1}\left((\partial_xa)f\right),\langle z\rangle^{2(\ell+N-j)}\partial_x^{j+1}f\right\rangle-\left\langle(\partial_xa)f,\langle z\rangle^{2(\ell+N)}f\right\rangle \\[0.2cm]
        &\leq \frac{1}{16}D_{N,\ell}+C\|a\|_{H_{x}^{6}}^{2}E_{N,\ell}.
    \end{aligned}
    $$
    For $I_{5}$, we have
    $$
    \begin{aligned}
        I_{5}&=  -\sum_{j=1}^{N}\left\langle\partial_x^{j-1}F_g,\langle z\rangle^{2(\ell+N-j)}\partial_x^{j+1}f\right\rangle+\left\langle F_g,\langle z\rangle^{2(\ell+N)}f\right\rangle \\[0.2cm]
        &\leq \frac{1}{16}D_{N,\ell}+\|a\|_{H_{x}^{6}}^{2}E_{N,\ell} + \|\partial_{t}a\|_{H_{x}^{4}}^{2}+\|a\|_{H_{x}^{6}}^{2}+\|a\|_{H_{x}^{5}}^{4}.
    \end{aligned}
    $$
    Combining these estimates, for $1\leq N\leq5$ we obtain
    $$
    \begin{aligned}
    \frac{d}{dt}E_{N,\ell}+D_{N,\ell}
    \leq C\Bigl(&1+\|a\|_{H_x^6}^2
       +E_{N-1,\ell+2}+D_{N-1,\ell+2}
       +\|\langle z\rangle^{\ell+N+1}q\|_4^4\\[0.2cm]
       &+\|a\|_{H_x^2}+\|\partial_xf\|_2^2
       +\|a\|_{H_x^5}^4+\|\partial_ta\|_{H_x^4}^2
       +\|\partial_xq\|_2^2\Bigr)(1+E_{N,\ell}).
    \end{aligned}
    $$
    The last coefficient is integrable by \eqref{eq:5.20}; it supplies
    the $N=1$ case. Starting with the weighted $N=0$ estimate above,
    Gronwall's inequality and induction through $N=5$ therefore give
    $$
    \sup_{0\leq t \leq T}\sum_{j=0}^{5}\|\left<z \right>^{\ell+5-j}\partial_{x}^{j}f\|_{2}^{2}+\int_{0}^{T}\sum_{j=0}^{5}\|\left<z \right>^{\ell+5-j} \nabla \partial_{x}^{j}f\|_{2}^{2}ds \leq C_{\ell,T}.
    $$
    \textit{Normal regularity, continuation, and uniqueness.} Write \eqref{eq:5.18} as
    $$
    f_t-\Delta f=F_g-vf_x-wf_z-a_xf.
    $$
    The weighted tangential bounds just proved imply that the right-hand side belongs to $L^2(0,T;H_x^4L_{z,\ell}^2)$. For the highest primitive coefficient, use
    $$
    \left\|\int_0^z\partial_x^5u\,dr\right\|_{L_x^2L_z^\infty}
      \leq C\|\langle z\rangle^2u\|_{H_x^5L_z^2},
    \qquad
    \|\langle z\rangle^{\ell+1}f_z\|_{L_T^2L_x^\infty L_z^2}
      \leq C_{T,\ell},
    $$
    placing the first factor in $L_T^\infty$ and the second in $L_T^2$. The remaining terms follow from the same product estimates and $a_t\in L^2H_x^4$. Lemma \ref{lem:zero-data-weighted-heat} yields \eqref{eq:5.25}--\eqref{eq:5.27}; the tangential energy identity gives $f\in C([0,T];H_x^5L_{z,\ell}^2)$.

    Higher polynomial weights are obtained by applying the same estimates with normal cutoffs and then letting the cutoffs tend to infinity. The local solution is already controlled in the fixed-weight class used above; the additional weighted estimates are linear in the highest weighted energy after the unweighted lower-order factors are fixed. Repeating the finite tangential induction proves the bounds for each prescribed weight. In particular, the estimates with $\ell=3$ and the finitely many higher weights needed in their induction, together with the energy identities and the weighted heat estimate, give strong endpoint traces with finite $E_{5,3}[f]+E_{4,3}[f_z]$. Smooth approximation of the wall history, with $a(0)=0$ and uniformly bounded norms in \eqref{eq:4.6}, justifies these estimates for the wall history of Proposition \ref{prop:outer-flow}.
    
    These bounds also provide continuation. Let $t_*$ be a finite endpoint,
    let $k(s)=e^{s\Delta_D}f(t_*)$ be the Dirichlet heat evolution of its trace,
    and write $f(t_*+s)=k(s)+h(s)$, so that $h(0)=0$.
    The wall coefficients retain their original histories and are evaluated at $t_*+s$.
    The weighted heat estimate with the initial energy retained gives
    $$
    \begin{gathered}
    \sup_{0\leq s\leq\tau}E_{5,3}[k](s)
       +\int_0^\tau D_{5,3}[k](s)\,ds<\infty,\\[0.2cm]
    \int_0^\tau D_{5,3}[k](s)\,ds\longrightarrow0
       \qquad(\tau\downarrow0).
    \end{gathered}
    $$
    The normal-gradient estimate has the same property with initial energy
    $E_{4,3}[\partial_zf(t_*)]$ retained. The known right-hand side after subtracting $k$ is
    $$
    \begin{aligned}
    F_g^{(0)}-v^{(0)}k_x-w^{(0)}k_z-a_xk
       +kg_x-\left(\int_0^zk_x\,dr\right)g_z
       -kk_x+\left(\int_0^zk_x\,dr\right)k_z.
    \end{aligned}
    $$
    In particular, the extra weight controls its affine-drift term:
    $$
    \begin{aligned}
    \|\langle z\rangle^7w^{(0)}k_z\|_{L^2(0,\tau;L^2)}
       &\leq C\|a_x\|_{L^\infty((t_*,t_*+\tau)\times\mathbb R)}
          \|\langle z\rangle^8k_z\|_{L^2(0,\tau;L^2)} \longrightarrow0\qquad(\tau\downarrow0).
    \end{aligned}
    $$
    The other known terms are controlled by the same product estimates,
    with one factor in a bounded supremum norm and the differentiated factor
    in its time-integrable dissipation norm; the wall source uses the local
    $L^2$ norms of $a_t$ and $a$. Thus the source norm tends to zero on the new interval.
    Keep all terms linear in $h$ on the left-hand side. The resulting weighted
    energy estimate uses, at the top primitive term,
    $$
    \begin{aligned}
    \left|\left\langle
       \int_0^z\partial_x^6h\,dr,\,
       \partial_zk\,\langle z\rangle^4\partial_x^5h
       \right\rangle\right| \leq C\|\langle z\rangle^2\partial_x^6h\|_2
       \|\langle z\rangle^2\partial_zk\|_{L_x^\infty L_z^2}
       \|\langle z\rangle^2\partial_x^5h\|_2.
    \end{aligned}
    $$
    Young's inequality places the highest derivative in the dissipation and the bounded lower derivative in the Gronwall coefficient. The normal-gradient estimate is then obtained at the lower tangential indices as above. The remaining terms are quadratic in the zero-initial remainder and satisfy \eqref{eq:leading-local-products}; hence contraction on a smaller ball gives a local extension. This argument requires bounded endpoint norms, not their smallness. The a priori estimates therefore exclude a finite endpoint on $[0,T]$.

    For uniqueness, let $\widetilde u$ be another profile with the same wall history and put $d=u-\widetilde u$. Its equation is
    $$
    \begin{aligned}
    d_t-\Delta d+(a+u)d_x+
       \left(-za_x-\int_0^zu_x\,dr\right)d_z
       +(a_x+\widetilde u_x)d
       -\widetilde u_z\int_0^zd_x\,dr=0,
    \end{aligned}
    $$
    with zero initial and boundary values. The transport has zero divergence and zero normal wall flux. Hardy's and Ladyzhenskaya's inequalities give
    $$
    \begin{aligned}
    \left|\iint\widetilde u_z\left(\int_0^zd_x\,dr\right)d\,dx\,dz\right|
       &\leq2\|d_x\|_2\|z\widetilde u_z\|_4\|d\|_4\\
       &\leq\tfrac12\|\nabla d\|_2^2
               +C\|z\widetilde u_z\|_4^4\|d\|_2^2.
    \end{aligned}
    $$
    Consequently,
    $$
    \frac{d}{dt}\|d\|_2^2+\|\nabla d\|_2^2
       \leq C\bigl(\|a_x\|_\infty+\|\widetilde u_x\|_\infty
                       +\|z\widetilde u_z\|_4^4\bigr)\|d\|_2^2.
    $$
    The coefficient is integrable by \eqref{eq:5.11}, \eqref{eq:5.21}, and \eqref{eq:5.27}. Gronwall's inequality gives $d=0$.

    \textit{Short-time estimates.} Since $\partial_{t}a-\partial_{x}^{2}a \in C([0,T];H^{3})$, an application of Lemma \ref{lem:heat-lift-scale} yields \eqref{eq:5.13}.
    It follows from the equations satisfied by $u_{1}^{b,0}$ and $b_{a}$ that $r_{a}$ satisfies the following equation:
    \begin{equation}
    \label{eq:5.23}
    \left\{\begin{array}{l}
        \partial_{t}r_{a}-\Delta r_{a}+(a+u)\partial_{x}r_{a}-\left(z\partial_{x}a+\int_{0}^{z}\partial_{x}u(x,s,t)ds \right) \partial_{z}r_{a} \\[0.2cm]
        \qquad +\left(\partial_{x}a+\partial_{x}b_{a}\right)r_{a}-\partial_{z}b_{a}\int_{0}^{z}\partial_{x}r_{a}(x,s,t)ds=F(x,z,t), \\[0.2cm]
        \left.r_a\right|_{t=0}=0,\quad r_{a}|_{z=0}=0, \quad r_a \rightarrow 0 \quad(z \rightarrow \infty),
    \end{array}\right.
    \end{equation}
    where
    $$
    F(x,z,t):=-b_{a}\partial_{x}a-(a+b_{a})\partial_{x}b_{a}+\left(z\partial_{x}a+\int_{0}^{z}\partial_{x}b_{a}(x,s,t)ds \right)\partial_{z}b_{a}.
    $$
    A direct calculation yields:
    $$
    \begin{gathered}
        \|\left\langle z\right\rangle^{\ell}b_{a}\partial_{x}a \|_{H_{x}^{2} L_{z}^{2}}+\|\left\langle z\right\rangle^{\ell}a\partial_{x}b_{a} \|_{H_{x}^{2} L_{z}^{2}} \lesssim \|a\|_{H_{x}^{3}}\|\left\langle z\right\rangle^{\ell} b_{a}\|_{H_{x}^{3}L_{z}^{2}} \leq Ct^{\frac{9}{4}}, \\[0.2cm]
        \|\left\langle z\right\rangle^{\ell} b_{a}\partial_{x}b_{a}\|_{H_{x}^{2} L_{z}^{2}} \lesssim \|b_{a}\|_{H_{x}^{3} L_{z}^{\infty}} \|\left\langle z\right\rangle^{\ell} b_{a}\|_{H_{x}^{3} L_{z}^{2}} \lesssim  \|b_a\|_{H_x^3L_z^2}^{1/2} \|\partial_{z}b_{a}\|_{H_{x}^{3} L_{z}^{2}}^{\frac{1}{2}} \|\left\langle z\right\rangle^{\ell} b_{a}\|_{H_{x}^{3} L_{z}^{2}} \leq Ct^{\frac{9}{4}}, \\[0.2cm]
        \|\left\langle z\right\rangle^{\ell}z \partial_{x}a\partial_{z}b_{a} \|_{H_{x}^{2} L_{z}^{2}} \lesssim \|a\|_{H_{x}^{3}} \|\left\langle z\right\rangle^{\ell}z\partial_{z}b_{a} \|_{H_{x}^{3} L_{z}^{2}}\leq Ct^{\frac{9}{4}}, \\[0.2cm]
        \|\left\langle z\right\rangle^{\ell}\int_{0}^{z}\partial_{x}b_{a}(x,s,t)ds \cdot \partial_{z}b_{a} \|_{H_{x}^{2} L_{z}^{2}} \lesssim \|\int_{0}^{z}\partial_{x}b_{a}(x,s,t)ds \|_{H_{x}^{2} L_{z}^{\infty}} \|\left\langle z\right\rangle^{\ell}\partial_{z}b_{a} \|_{H_{x}^{2}L_{z}^{2}} \leq Ct^{\frac{9}{4}}.
    \end{gathered}
    $$
    Thus, we have
    $$
    \|\left\langle z\right\rangle^{\ell}F(x,z,t) \|_{H_{x}^{2} L_{z}^{2}} \leq Ct^{\frac{9}{4}}.
    $$
    Acting $\partial_{x}^{j}$ on both sides of \eqref{eq:5.23}, then taking $L^{2}$ inner product with $\left\langle z\right\rangle^{2(\ell+2-j)}\partial_{x}^{j}r_{a}$, integrate by parts and summing about $0\leq j \leq 2$, we obtain
    $$
    \begin{aligned}
        \frac{d}{dt}E_{2,\ell}[r_{a}]+D_{2,\ell}[r_{a}] &\leq C K_{\ell}(t) E_{2,\ell}[r_{a}] + \|\left\langle z\right\rangle^{\ell+2} F(x,z,t)\|_{H_{x}^{2}L_{z}^{2}}\left(E_{2,\ell}[r_{a}]\right)^{\frac{1}{2}},
    \end{aligned}
    $$
    where
    $$
    \begin{aligned}
        K_{\ell}(t)=1+\|a\|_{H_{x}^{3}}&+E_{2,\ell}[u]+D_{2,\ell}[u]+\|a\|_{H_{x}^{3}}^{2} +\|\left\langle z\right\rangle^{\ell+2}b_{a}\|_{H_{x}^{3}L_{z}^{2}} \\[0.2cm]
        &+\|\left\langle z\right\rangle^{\ell+2}\partial_{z}b_{a}\|_{H_{x}^{3}L_{z}^{2}}+\|\left\langle z\right\rangle^{\ell+2}z\partial_{z}b_{a}\|_{H_{x}^{3}L_{z}^{2}}.
    \end{aligned}
    $$
    Put $\sqrt{E_{2,\ell}(t)+\delta^{2}}=Y_{\delta}(t)$, we have 
    $$
    \frac{d}{dt}Y_{\delta}(t)\leq K_{\ell}(t)Y_{\delta}(t)+t^{\frac{9}{4}}
    $$
    Letting $\delta\to 0$ in the above inequality and then applying Gronwall inequality, we obtain
    \begin{equation}
    \label{eq:5.24}
        \begin{gathered}
            E_{2,\ell}[r_{a}](t)^{\frac{1}{2}} \leq Ct^{\frac{13}{4}},\quad \int_{0}^{t}D_{2,\ell}[r_{a}](s)ds \leq Ct^{\frac{13}{2}}.
        \end{gathered}
    \end{equation}
    Writing $\partial_tr_a-\Delta r_a=H(x,z,t)$, we estimate the primitive terms using an integrable normal weight. Weighted Cauchy--Schwarz and the one-dimensional Sobolev inequality give
    $$
    \left\|\int_0^zu_x\,dr\right\|_{L_x^\infty L_z^\infty}
    +\left\|\int_0^zu_{xx}\,dr\right\|_{L_x^2L_z^\infty}
       \leq C\|\langle z\rangle^2u\|_{H_x^2L_z^2}.
    $$
    Applying the product rule in $x$, we therefore obtain
    $$
    \begin{aligned}
    &\int_0^t\left\|\langle z\rangle^\ell
       \left(\int_0^zu_x\,dr\right)\partial_zr_a\right\|_{H_x^1L_z^2}^2\,ds\\[0.2cm]
    &\quad\leq C\sup_{0\leq s\leq t}\|\langle z\rangle^2u(s)\|_{H_x^2L_z^2}^2
          \int_0^t\|\langle z\rangle^\ell\partial_zr_a(s)\|_{H_x^1L_z^2}^2\,ds\\[0.2cm]
    &\quad\leq C_T\int_0^tD_{2,\ell}[r_a](s)\,ds
          \leq C_Tt^{13/2},
    \end{aligned}
    $$
    and
    $$
    \begin{aligned}
    &\int_0^t\left\|\langle z\rangle^\ell\partial_zb_a
       \int_0^z\partial_xr_a(x,r,s)\,dr\right\|_{H_x^1L_z^2}^2\,ds\\[0.2cm]
    &\quad\leq C\int_0^t\|\langle z\rangle^\ell\partial_zb_a\|_{H_x^2L_z^2}^2
          \|\langle z\rangle^2r_a\|_{H_x^2L_z^2}^2\,ds
          \leq Ct^9.
    \end{aligned}
    $$
    Consequently, a direct estimate yields
    $$
    \int_{0}^{t}\|\left\langle z\right\rangle^{\ell}H\|_{H_{x}^{1}L_{z}^{2}}^{2} ds \leq Ct^{\frac{11}{2}}.
    $$
    Multiplying $\eqref{eq:5.23}_{1}$ by $\left\langle z\right\rangle^{2\ell} \partial_{t}r_{a}$, and then
    integrating by parts, we obtain 
    $$
    \frac{d}{dt}\|\left\langle z\right\rangle^{\ell}\nabla r_{a}\|_{2}^{2} \leq C\|\left\langle z\right\rangle^{\ell}\nabla r_{a}\|_{2}^{2}+\|\left\langle z\right\rangle^{\ell} H\|_{2}^{2}.
    $$
    Choosing \(\tau\) such that $\|\left\langle z\right\rangle^{\ell} \nabla r_{a}(\tau)\|_{2}^{2} \leq \frac{2}{t} \int_{\frac{t}{2}}^{t}\|\left\langle z\right\rangle^{\ell} \nabla r_{a}\|_{2}^{2}ds $, and integrating the above inequality over \([\tau,t]\), we obtain
    $$
    \begin{aligned}
        \|\left\langle z\right\rangle^{\ell}\nabla r_{a}(t)\|_{2}^{2} &\leq C \|\left\langle z\right\rangle^{\ell}\nabla r_{a}(\tau)\|_{2}^{2} + C\int_{0}^{t}\|\left\langle z \right\rangle^{\ell}H\|_{2}^{2}ds \\[0.2cm]
        &\lesssim \frac{1}{t}\int_{0}^{t}D_{2,\ell}[r_{a}](s)ds+C\int_{0}^{t}\|\left\langle z \right\rangle^{\ell}H\|_{2}^{2}ds \leq Ct^{\frac{11}{2}}.
    \end{aligned}
    $$
    Multiplying $\partial_{x}\eqref{eq:5.23}_{1}$ by $\left\langle z\right\rangle^{2\ell} \partial_{t}\partial_{x}r_{a}$, and then
    integrating by parts, a similar argument yields that
    $$
    \|\left\langle z\right\rangle^{\ell}\nabla \partial_{x} r_{a}(t)\|_{2}^{2} \leq Ct^{\frac{11}{2}},
    $$
    Combining with \eqref{eq:5.24}, we get
    $$
    \begin{gathered}
        \|\langle z\rangle^\ell r_a(t)\|_{H_x^2L_z^2}+t^{1/2}
        \|\langle z\rangle^\ell\partial_z r_a(t)\|_{H_x^1L_z^2}
        \le C_{T,\ell}\,t^{13/4}.
    \end{gathered}
    $$
    The heat-history representation and the normal $L^1$ norm of its kernel give
    \[
     \int_0^\infty\|b_a(\cdot,z,t)\|_{H_x^2}\,dz\leq C_Tt^{3/2}.
    \]
    Weighted Cauchy--Schwarz and \eqref{eq:5.14} then yield
    \[
    \begin{aligned}
     \left\|\int_0^\infty u_1^{b,0}(\cdot,z,t)\,dz\right\|_{H_x^2}
     &\leq \int_0^\infty\|b_a(\cdot,z,t)\|_{H_x^2}\,dz
          +\int_0^\infty\|r_a(\cdot,z,t)\|_{H_x^2}\,dz\\
     &\leq C_Tt^{3/2}
          +C\|\langle z\rangle^2r_a(t)\|_{H_x^2L_z^2}
     \leq C_Tt^{3/2}.
    \end{aligned}
    \]
    This proves \eqref{eq:5.15} and completes the proof.
\end{proof}

\subsubsection{Regularity of Higher-Order Inner and Outer Profiles}
\begin{lemma}
\label{pro:6.1}
    Under the assumptions of Theorem \ref{thm:main}, there exists a unique solution $u^{I,1}$ of the problem \eqref{eq:3.5} on $[0,T]$, satisfying 
    $$
    u^{I,1} \in C([0,T];H^{3}),\quad \partial_{x}u^{I,1}\in L^{2}(0,T;H^{3}),\quad \partial_{t}u^{I,1} \in L^{2}(0,T;H^{2}),
    $$
    and
    $$
    u_{1}^{I,1}|_{y=0} \in C([0,T];H_{x}^{3}),\quad \partial_{x}u_{1}^{I,1}|_{y=0} \in L^{2}(0,T;H_{x}^{3}),\quad  \partial_{t}u_{1}^{I,1}|_{y=0} \in L^{2}(0,T;H_{x}^{2}).
    $$
\end{lemma}
\begin{proof}
	Denoting $\hat{u}^{I,1}=u^{I,1}+\left(-\varphi^{\prime}(y)\int_{0}^{+\infty}u_{1}^{b,0}(x,s,t)ds,\varphi(y)\int_{0}^{+\infty}\partial_{x}u_{1}^{b,0}(x,s,t)ds\right)^{\top}$ with $\varphi$ defined in $\eqref{eq:1.16}$, and using \eqref{eq:3.5}, we have
	\begin{equation}
    \label{eq:6.1}
		\left\{\begin{array}{l}
			\partial_t \hat{u}^{I,1}+\hat{u}^{I,1} \cdot \nabla u^{I,0}+u^{I,0} \cdot \nabla \hat{u}^{I,1}+\nabla p^{I, 1}-\partial_{x}^{2} \hat{u}^{I,1}+I_{1}=0, \\[0.2cm]
			\partial_x \hat{u}_{1}^{I,1} + \partial_y \hat{u}_{2}^{I,1}=0, \\[0.2cm]
			\hat{u}_{2}^{I,1}(x,0,t)=0, \quad \hat{u}^{I,1}(x, y, 0)=0 .
		\end{array}\right.
	\end{equation}
	where
	$$
	r(x,y,t)=\left(-\varphi^{\prime}(y)\int_{0}^{+\infty}u_{1}^{b,0}(x,s,t)ds,\varphi(y)\int_{0}^{\infty}\partial_{x}u_{1}^{b,0}(x,s,t)ds\right)^{\top},
	$$
	and
	$$
	-I_{1}=\partial_{t}r+r \cdot \nabla u^{I,0}+u^{I,0}\cdot \nabla r-\partial_{x}^{2}r.
	$$
	Using Proposition \ref{prop:outer-flow} and Proposition \ref{pro:boundary-layer equation}, a direct calculation yields that
	$$
	\begin{gathered}
		 u^{I,0}\in C([0,T];H^{4}),\quad \partial_{x}u^{I,0} \in L^{2}(0,T;H^{4}),\quad I_{1} \in L^{2}(0,T;H^{3}).
	\end{gathered}
	$$
    Indeed, weighted Cauchy--Schwarz gives
    $$
    \int_0^\infty u_1^{b,0}\,dz\in C([0,T];H_x^5)\cap L^2(0,T;H_x^6),
    \qquad
    \int_0^\infty\partial_tu_1^{b,0}\,dz\in L^2(0,T;H_x^4).
    $$
    These bounds control $r_t-r_{xx}$ in $L^2H^3$ and the two transport terms in $L^2H^3$. Lemma \ref{lemma:2.14} with $s=3$ therefore applies. Since $\varphi'(0)=0$, the tangential trace of $r$ is zero, so the same wall bounds hold for $u^{I,1}$.
\end{proof}
\begin{remark}
	The regularity of $p^{b,2}$ follows from the identity obtained by combining \eqref{eq:3.3} and \eqref{eq:3.7}:
    $$
    \begin{aligned}
    \mathcal P_2={}&4(\partial_xa)u_2^{b,1}
     +2(\partial_x^2a)\int_z^\infty u_1^{b,0}\,dr
     +\partial_x^2\int_z^\infty (u_1^{b,0})^2\,dr\\[0.2cm]
     &+2z(\partial_x^2a)u_1^{b,0}
     -2\bigl(\partial_xu_2^{b,1}+\overline{\partial_xu_2^{I,1}}\bigr)u_1^{b,0}.
    \end{aligned}
    $$
    To verify this, set $K=\int_z^\infty u_1^{b,0}\,dr$ and
    $w=-z\partial_xa+u_2^{b,1}+\overline{u_2^{I,1}}$. Integrating \eqref{eq:3.3} in $z$ gives
    $$
    (\partial_t-\partial_x^2-\partial_z^2)K
    =-a u_2^{b,1}-\partial_x\int_z^\infty (u_1^{b,0})^2\,dr
     -2(\partial_xa)K+w u_1^{b,0}.
    $$
    Differentiate once in $x$ and substitute into \eqref{eq:3.7}; the time derivatives cancel and yield the displayed identity. Weighted Cauchy--Schwarz and the tangential product estimates now give
    $$
    \mathcal P_2,\ p^{b,2}\in C([0,T];H_x^3L_{z,\ell}^2)
    $$
    for every fixed $\ell\geq0$. For the square in the tail integral, one places the two factors in $L_z^2$ before integrating, so no additional normal derivative is required.
\end{remark}
Next, for $u_{1}^{b,1}$, we have the following Proposition.
\begin{lemma}
\label{pro:6.3}
    Under the assumptions of Theorem \ref{thm:main}, there exists a unique solution $u_{1}^{b,1}$ of the problem \eqref{eq:3.6} on $[0,T]$, satisfying 
    $$
    u_{1}^{b,1} \in C([0,T];H_{x}^{3}L_{z,\ell}^{2}),\quad \nabla u_{1}^{b,1}\in L^{2}(0,T;H_{x}^{3}L_{z,\ell}^{2}),
    $$
    and
    $$
    \partial_{z}u_{1}^{b,1} \in L^{\infty}(0,T;H_{x}^{2}L_{z,\ell}^{2}),\quad \partial_{t}u_{1}^{b,1},\partial_{z}^{2}u_{1}^{b,1}\in L^{2}(0,T;H_{x}^{2}L_{z,\ell}^{2}).
    $$
\end{lemma}
\begin{proof}
	Denoting $\hat{u}_{1}^{b,1}(x,z,t)=u_{1}^{b,1}(x,z,t)+\varphi(z)u_{1}^{I,1}(x,0,t)$ with $\varphi$ defined in \eqref{eq:1.16}, and using \eqref{eq:3.6}, we have
	\begin{equation}
		\label{eq:6.2}
		\left\{\begin{array}{l}
			\partial_t \hat{u}_{1}^{b,1} - \partial_{x}^{2} \hat{u}_{1}^{b,1} - \partial_{z}^{2} \hat{u}_{1}^{b,1}+\left(a+u_{1}^{b,0}\right) \partial_x \hat{u}_{1}^{b,1} \\[0.2cm]
			\quad+ \left( -z\partial_{x}a-\int_{0}^{z}\partial_{x}u_{1}^{b,0}(x,s,t)ds\right) \partial_z \hat{u}_{1}^{b,1}+\left( \partial_x u_{1}^{b,0}+\partial_{x}a\right) \hat{u}_{1}^{b,1}\\[0.2cm]
			\quad-\left(\int_{0}^{z}\partial_{x}\hat{u}_{1}^{b,1}(x,s,t)ds\right) \partial_z u_{1}^{b,0}=B_{1}(x,z,t) \\[0.2cm]
			\hat{u}_{1}^{b,1}(x, z,0)=0, \\[0.2cm]
			\displaystyle \lim _{z \rightarrow \infty} \hat{u}_{1}^{b,1}(x,z,t)=0, \quad \hat{u}_{1}^{b,1}(x,0,t)=0.
		\end{array}\right.
	\end{equation}
	where
	$$
	h(x,z,t)=\varphi(z) u_{1}^{I,1}(x,0,t),
	$$
	and
	$$
	\begin{aligned}
		B_{1}(x,z,t)=&\partial_{t}h-\partial_{x}^{2}h-\partial_{z}^{2}h+(a+u_{1}^{b,0})\partial_{x}h \\[0.2cm]
		&+\left(-z\partial_{x}a-\int_{0}^{z}\partial_{x}u_{1}^{b,0}(x,s,t)ds\right)\partial_{z}h+(\partial_{x}u_{1}^{b,0}+\partial_{x}a)h \\[0.2cm]
		&-(\overline{u_{1}^{I,1}}+z\overline{\partial_{y}u_{1}^{I,0}})\partial_{x}u_{1}^{b,0}-z\overline{\partial_{y}\partial_{x}u_{1}^{I,0}}u_{1}^{b,0}-u_{1}^{b,0}\overline{\partial_{x}u_{1}^{I,1}}-u_{2}^{b,1}\overline{\partial_{y}u_{1}^{I,0}} \\[0.2cm]
		&-(\int_{0}^{z}\partial_{x}h(x,z^{\prime},t)dz^{\prime}+z\overline{\partial_{y}u_{2}^{I,1}}+\frac{z^{2}}{2}\overline{\partial_{y}^{2}u_{2}^{I,0}})\partial_{z}u_{1}^{b,0}.
	\end{aligned}
	$$
	Using Proposition \ref{prop:outer-flow}, Proposition \ref{pro:boundary-layer equation} and Lemma \ref{pro:6.1}, a direct calculation yields that
	$$
	\begin{gathered}
		\partial_{x}u_{1}^{b,0},\partial_{z}u_{1}^{b,0} \in L^{2}(0,T;H_{x}^{5}L_{z,\ell}^{2}), \\[0.2cm]
        \partial_{x}a \in L^{2}(0,T;H_{x}^{5}),\quad B_{1}(x,z,t) \in L^{2}(0,T;H_{x}^{2}L_{z,\ell+5}^{2}). 
	\end{gathered}
	$$
	Lemma \ref{lemma:2.13} with $s=3$ therefore gives a unique solution $\hat u_1^{b,1}$ with the stated regularity. Subtracting $h$ yields the desired solution $u_1^{b,1}$ and completes the proof.
\end{proof}

\begin{lemma}
\label{pro:6.4}
    Under the assumptions of Theorem \ref{thm:main}, there exists a unique solution $u^{I,2}$ of the problem \eqref{eq:3.8} on $[0,T]$, satisfying 
    $$
    u^{I,2} \in C([0,T];H^{2}),\quad  \partial_{x}u^{I,2}\in L^{2}(0,T;H^{2}), \qquad \partial_{t}u^{I,2},\nabla p^{I,2}\in L^{2}(0,T;H^{1}),
    $$
    and
    $$
    u_{1}^{I,2}|_{y=0} \in C([0,T];H_{x}^{2}),\quad \partial_{x}u_{1}^{I,2}|_{y=0} \in L^{2}(0,T;H_{x}^{2}),\quad  \partial_{t}u_{1}^{I,2}|_{y=0} \in L^{2}(0,T;H_{x}^{1}).
    $$
    Furthermore, there exists a unique solution $u_{1}^{b,2}$ of problem \eqref{eq:3.9} on $[0,T]$, satisfying
    $$
    u_{1}^{b,2} \in C([0,T];H_{x}^{2}L_{z,\ell}^{2}),\quad \nabla u_{1}^{b,2}\in L^{2}(0,T;H_{x}^{2}L_{z,\ell}^{2}),
    $$
    and
    $$
    \partial_{z}u_{1}^{b,2} \in L^{\infty}(0,T;H_{x}^{1}L_{z,\ell}^{2}),\quad \partial_{t}u_{1}^{b,2},\partial_{z}^{2}u_{1}^{b,2}\in L^{2}(0,T;H_{x}^{1}L_{z,\ell}^{2}).
    $$
\end{lemma}
\begin{proof}
	For the second outer profile, put
    $$
    M_1(x,t)=\int_0^\infty u_1^{b,1}(x,z,t)\,dz.
    $$
    Lemma \ref{pro:6.3} and weighted Cauchy--Schwarz imply
    $$
    M_1\in C([0,T];H_x^3)\cap L^2(0,T;H_x^4),\qquad
    \partial_tM_1\in L^2(0,T;H_x^2),\qquad M_1(0)=0.
    $$
    We use a harmonic lifting of the normal boundary value. Let
    $$
    \widehat K(\xi,y,t)=e^{-|\xi|y}\widehat M_1(\xi,t),\qquad
    r=(\partial_yK,-\partial_xK),\qquad
    \widehat\Phi(\xi,y,t)=\frac{\mathrm i\xi}{|\xi|}e^{-|\xi|y}\widehat M_1(\xi,t).
    $$
    The value at $\xi=0$ is immaterial. Direct calculation gives
    $$
    r=\nabla\Phi,\qquad \operatorname{div}r=\operatorname{curl}r=0,
    \qquad r_2|_{y=0}=-\partial_xM_1,
    $$
    and integration in $y$ gives
    $$
    \begin{gathered}
    \|r\|_{H^2}\leq C\|M_1\|_{H_x^3},\qquad
    \|\nabla r\|_{H^2}\leq C\|M_1\|_{H_x^4},\\[0.2cm]
    \|r_t\|_{H^1}\leq C\|\partial_tM_1\|_{H_x^2},\qquad
    \|r_{xx}\|_{H^1}\leq C\|M_1\|_{H_x^4}.
    \end{gathered}
    $$
    Only the gradient of $\Phi$ is used; the potential need not be in $L^2$.
    Writing $\hat u^{I,2}=u^{I,2}-r$ and $\hat p^{I,2}=p^{I,2}+\Phi_t-\Phi_{xx}$, we obtain
    $$
    \begin{gathered}
    \partial_t\hat u^{I,2}-\partial_x^2\hat u^{I,2}
     +u^{I,0}\cdot\nabla\hat u^{I,2}
     +\hat u^{I,2}\cdot\nabla u^{I,0}+\nabla\hat p^{I,2}
    =g^{I,2}-u^{I,0}\cdot\nabla r-r\cdot\nabla u^{I,0},\\[0.2cm]
    \operatorname{div}\hat u^{I,2}=0,\qquad \hat u_2^{I,2}|_{y=0}=0,
    \qquad \hat u^{I,2}(0)=0.
    \end{gathered}
    $$
    The right-hand side belongs to $L^2H^2$: $g^{I,2}=\partial_y^2u^{I,0}-u^{I,1}\cdot\nabla u^{I,1}$, and all the other terms are controlled by the displayed lifting bounds. Lemma \ref{lemma:2.14} with $s=2$ proves the bulk estimates. Recovering the original pressure uses
    $$
    \nabla p^{I,2}=\nabla\hat p^{I,2}-r_t+r_{xx}\in L^2H^1.
    $$
    At the wall, the Fourier multiplier for $r_1$ is $-|\xi|$, so its trace is in $CH_x^2\cap L^2H_x^3$, with time derivative in $L^2H_x^1$. Adding this trace to that of $\hat u^{I,2}$ proves all the stated wall estimates. In particular, no third tangential derivative of $\partial_tM_1$ is needed.

    For $u_1^{b,2}$, set $h=\varphi(z)u_1^{I,2}(x,0,t)$ and $\hat u_1^{b,2}=u_1^{b,2}+h$. The equation for $\hat u_1^{b,2}$ is \eqref{eq:2.18} with the coefficients determined by $a,u_1^{b,0}$. Its right-hand side is $g^{b,2}$ plus the terms obtained by applying that linear operator to $h$. It belongs to $L^2H_x^1L_{z,\ell+4}^2$. Indeed, the outer factors needed in $g^{b,2}$ are
    $$
    \begin{gathered}
    \overline{\partial_yu_1^{I,0}}\in CH_x^2\cap L^2H_x^3,\qquad
    \overline{\partial_y^2u_1^{I,0}}\in CH_x^1\cap L^2H_x^2,\\
    \overline{\partial_yu_1^{I,1}}\in CH_x^1\cap L^2H_x^2,\qquad
    \overline{u_1^{I,2}}\in CH_x^2\cap L^2H_x^3.
    \end{gathered}
    $$
    Normal derivatives of $u_2^{I,j}$ are replaced by tangential derivatives of $u_1^{I,j}$ using incompressibility. The powers of $z$ are absorbed by the polynomial weights. The terms involving $u_1^{b,1}$ use Lemma \ref{pro:6.3}, and the pressure term uses the preceding remark. Finally, $h_t,h_{xx}$ have exactly the required $L^2H_x^1$ regularity. Lemma \ref{lemma:2.13} with $s=2$ therefore proves the inner estimates and uniqueness.
\end{proof}

\subsection{Construction of the approximate solution}
\label{sec:finite-approximation}
Based on the analysis in Section \ref{subsec:Asymptotic analysis}, we define an approximate solution for the system \eqref{eq:1.1} as follows:
\begin{equation}
\label{eq:3.70}
    \left\{\begin{array}{l}
	u^{a}(x, y, t)=u^{I}(x, y, t)+u^{b}\left(x, \frac{y}{\varepsilon}, t\right)+\varepsilon^{3} S(x, y, t), \\[0.2cm]
	p^{a}(x, y, t)=p^{I}(x, y, t)+p^{b}\left(x,\frac{y}{\varepsilon},t\right),
\end{array}\right.
\end{equation}
where
$$
\begin{gathered}
	u^{I}=u^{I, 0}+\varepsilon u^{I, 1}+\varepsilon^{2} u^{I, 2}, \quad p^{I}=p^{I, 0}+\varepsilon p^{I, 1}+\varepsilon^{2} p^{I, 2},  \\[0.2cm]
	u^{b}=u^{b,0}+\varepsilon u^{b, 1}+\varepsilon^{2} u^{b, 2}+\varepsilon^{3}\left(0, u_2^{b, 3}\right)^{\top}, \quad p^{b}=\varepsilon^{2}p^{b,2}.
\end{gathered}
$$
and
\begin{equation}
\begin{gathered}
\label{eq:S}
    S(x, y, t)=\binom{\varphi^{\prime}(y) \int_{0}^{+\infty}u_{1}^{b,2}(x,s,t)ds}{-\varphi(y)\int_{0}^{+\infty}\partial_{x}u_{1}^{b,2}(x,s,t)ds}.
\end{gathered}
\end{equation}
with $\varphi$ the cut-off function defined by \eqref{eq:1.16}. Let$(u,p)$ be the solution of problem \eqref{eq:1.1}-\eqref{eq:1.2}. Then, we define the error terms by
\begin{equation}
	\label{eq:3.38}
	E^\varepsilon=u^\varepsilon-u^a,\qquad \pi^\varepsilon=p^\varepsilon-p^a.
\end{equation}
Substituting $(u^{a},p^{a})$ into \eqref{eq:1.1}-\eqref{eq:1.2}, with the help of the equations of inner and outer layer profiles in Section \ref{subsec:Asymptotic analysis}, we find that the error functions satisfy the following problem:
\begin{equation}
	\label{eq:3.39}
	\left\{\begin{array}{l}
		\partial_tE^\varepsilon+(E^\varepsilon\cdot\nabla)u^a+(u^a\cdot\nabla)E^\varepsilon+(E^\varepsilon\cdot\nabla)E^\varepsilon
        +\nabla \pi^\varepsilon-\partial_x^2E^\varepsilon-\varepsilon^2\partial_y^2E^\varepsilon=-\mathcal{R}^\varepsilon \\[0.2cm]
		\operatorname{div} E^\varepsilon=0, \\[0.2cm]
		E^\varepsilon(x, y, 0)=0, \quad E^\varepsilon(x, 0, t)=0,
	\end{array}\right.
\end{equation}
where
$$
\begin{gathered}
	\mathcal{R}^\varepsilon_1=\partial_tu^a_1+(u^a\cdot\nabla)u^a_1+\partial_xp^a
          -\partial_x^2u^a_1-\varepsilon^2\partial_y^2u^a_1,\\[0.2cm]
	\mathcal{R}^\varepsilon_2=\partial_tu^a_2+(u^a\cdot\nabla)u^a_2+\partial_yp^a
          -\partial_x^2u^a_2-\varepsilon^2\partial_y^2u^a_2.
\end{gathered}
$$
Moreover, using the equations of inner and outer layer profiles, we can split $\mathcal{R}_{1}^{\varepsilon}$ and $\mathcal{R}_{2}^{\varepsilon}$ as follows.
$$
\mathcal{R}_1^{\varepsilon}=-\sum_{k=1}^{19} F_k, \quad \mathcal{R}_2^{\varepsilon}=-\sum_{k=1}^{20} G_k .
$$
See $(\ref{eq:B1})$ and $(\ref{eq:B2})$ for the detailed expressions.

\section{Justification of the vanishing vertical viscosity limit}
\label{sr:finite-criterion}

In this section, we give the proof of Theorems \ref{thm:main}. All the constants $C$ used in the proof may depend on $T$ and the bounds obtained in Section \ref{sec:profile-regularity}, but are independent of $\varepsilon$. Without loss of generality, we assume that $\varepsilon \in (0,1]$ in the rest of the paper.

\subsection{Estimates of the source terms}
\label{subsec:source terms}

To begin with, we give some basic estimates of the approximate solutions which will be frequently used later.
\begin{lemma}
\label{lemma:4.1}
    Under the assumptions of Theorem \ref{thm:main}, there exists a constant $C_T$ independent of $\varepsilon$ such that
	$$
    \|u_1^a\|_{L_T^\infty L_{xy}^\infty}
    +\|S\|_{L_T^\infty H_{xy}^1}+\|S\|_{L_T^2H_{xy}^2}\leq C_T,
    $$
	and
	$$
    \begin{aligned}
    \int_0^T\Bigl(&\|(\nabla u^I,\nabla S,\partial_xu^b,
     \varepsilon^{-1}\partial_zu_2^b,\langle z\rangle\partial_zu_1^b)\|_\infty^2
     +\|\partial_xu^a\|_\infty^2
     +\|\partial_x^2u^a\|_{L_x^2L_y^\infty}^2\\[0.2cm]
     &+\|\nabla\partial_y(u_1^I+\varepsilon^3S_1)\|_2^2
     +\|\langle z\rangle\partial_x\partial_zu_1^b\|_{H_x^1L_z^2}^2
     +\|\langle z\rangle^2\partial_z^2u_1^b\|_{H_x^1L_z^2}^2\Bigr)\,dt
     \leq C_T.
    \end{aligned}
    $$
\end{lemma}
\begin{proof}
	The proof can be completed by directly using the Proposition \ref{prop:outer-flow},\ref{pro:boundary-layer equation} and Lemmas \ref{pro:6.1}-\ref{pro:6.4}, H\"older and Sobolev inequalities with the expressions of the approximate solutions. We omit the details for brevity.
\end{proof}
The following estimates can be organized in the following form, which also records the cancellations needed for the spatial error energies. The residual $\mathcal{R}^{\varepsilon}$ enters the two velocity estimates and the zeroth-order shear estimate, while $\partial_{x}\mathcal{R}^{\varepsilon}_{1}$ enters the mixed shear estimate. The residual pressure is controlled by $ \partial_{x}\mathcal{R}^{\varepsilon}_{1}+\partial_{y}\mathcal{R}^{\varepsilon}_{2}$ and the wall Fourier integral of $\mathcal{R}^{\varepsilon}_{2}|_{y=0}$. We give the regrouping explicitly; the original terms are listed in Appendix \ref{app:residual}.
\begin{lemma}
\label{prop:residual}
    Under the assumptions of Theorem \ref{thm:main}, there exists a constant $C$ independent of $\varepsilon$ such that
    \begin{equation}
    \label{residual-estimates}
    \begin{gathered}
         \|\mathcal{R}^{\varepsilon}\|_{L_t^2L_{xy}^2}
        +\|\partial_x\mathcal{R}_1^{\varepsilon}\|_{L_t^2L_{xy}^2}
        \leq C_T\varepsilon^{5/2}, \\[0.2cm] 
        \|\partial_y\mathcal{R}_2^{\varepsilon}\|_{L_t^2L_{xy}^2}
        \leq C_T\varepsilon^{3/2}, \\[0.2cm]
         \int_0^T\int_{\mathbb{R}}|\xi|
           |\widehat{\mathcal{R}^{\varepsilon}_2|_{y=0}}(\xi,t)|^2\,d\xi\,dt
             \le C_T\varepsilon^4.
    \end{gathered}
    \end{equation}
\end{lemma}
\begin{proof}
	Thanks to $(\ref{eq:B1})$, we can estimate $\mathcal{R}_{1}$ term by term. For $F_{1}$ using Lemmas \ref{pro:6.4} and \ref{lemma:4.1}, we have
	$$
	\begin{aligned}
		\left\|F_{1}\right\|_{L_{t}^{2} L_{x y}^2} & \leq \varepsilon^{2}\left\|\partial_t u_{1}^{b, 2}\right\|_{L_t^{2} L_{x y}^2}+\varepsilon^{3}\left\|\partial_t S_{1}\right\|_{L_t^{2} L_{x y}^2} \\[0.2cm]
		& \leq \varepsilon^{\frac{5}{2}}\left\|\partial_t u_{1}^{b, 2}\right\|_{L_t^{2} L_{xz}^2}+\varepsilon^{3}\left\|\partial_t S_{1}\right\|_{L_t^{2} L_{x y}^2} \\[0.2cm]
		& \leq C \varepsilon^{\frac{5}{2}}.
	\end{aligned}
	$$
	For $F_{2}$, we use the Taylor's formula,
	$$
	\begin{aligned}
		\left\|F_{2}\right\|_{L_t^{2} L_{x y}^2} &\leq\left\|\frac{u_{1}^{I, 0}(x, y, t)-u_{1}^{I, 0}(x, 0, t)-\partial_{y}u_{1}^{I,0}(x,0,t)y}{\frac{y^{2}}{2}} \frac{1}{2}\varepsilon^{2} z^{2} \partial_{x} u_{1}^{b, 0}\right\|_{L_t^{2} L_{x y}^2} \\[0.2cm]
		&\quad+\left\|\frac{u_{1}^{I, 0}(x, y, t)-u_{1}^{I, 0}(x, 0, t)}{y} \varepsilon z \varepsilon \partial_{x} u_{1}^{b, 1}\right\|_{L_t^{2} L_{x y}^2} \\[0.2cm]
		& \leq C \varepsilon^{\frac{5}{2}}\left\|u_1^{I, 0}\right\|_{L_T^{\infty} H_{x y}^{4}}\left\|\langle z\rangle^{2} u_{1}^{b, 0}\right\|_{L_t^{2} H_{x}^{1}L_{z}^{2} }+C \varepsilon^{\frac{5}{2}}\left\|u_{1}^{I, 0}\right\|_{L_T^{\infty} H_{x y}^{3}}\left\|\langle z\rangle u_{1}^{b, 1}\right\|_{L_T^{\infty} H_{x}^{1}L_{z}^{2} } \\[0.2cm]
		& \leq C \varepsilon^{\frac{5}{2}}.
	\end{aligned}
	$$
	Using Proposition \ref{prop:outer-flow}, Lemmas \ref{pro:6.1}, \ref{pro:6.4}, \ref{lemma:4.1}, and Sobolev embedding theorem, we have
	$$
	\begin{aligned}
		\left\|F_{3}\right\|_{L_t^{2} L_{x y}^2} & \leq  \varepsilon^{\frac{5}{2}}\left\|u_{1}^{I, 0}\right\|_{L_T^{\infty} L_{x y}^{\infty}}\left\|\partial_{x}u_{1}^{b, 2}\right\|_{L_t^{2} L_{xz}^{2}}+\varepsilon^{3}\left\|u_{1}^{I, 0}\right\|_{L_T^{\infty} L_{x y}^{\infty}}\left\|\partial_{x}S_{1}\right\|_{L_t^{2} L_{xy}^{2}}\\[0.2cm]
        &\quad+\varepsilon^{3} \left\| u_{1}^{I,1}\right\|_{L_T^{\infty} L_{x y}^{\infty}} \left\| \partial_{x}u_{1}^{I,2}\right\|_{L_{t}^{2} L_{xy}^{2}} \\[0.2cm]
		&\leq C \varepsilon^{\frac{5}{2}} \left\| u_{1}^{I,0}\right\|_{L_{T}^{\infty} H_{xy}^{2}} \left\| u_{1}^{b,2}\right\|_{L_{t}^{2} H_{x}^{1}L_{z}^{2}}+C \varepsilon^{3} \left\| u_{1}^{I,0}\right\|_{L_{T}^{\infty} H_{xy}^{2}} \left\| S_{1}\right\|_{L_{t}^{2} H_{xy}^{1}}\\[0.2cm]
        &\quad+C \varepsilon^{3}\left\| u_{1}^{I,1}\right\|_{L_{T}^{\infty} H_{xy}^{2}} \left\| u_{1}^{I,2}\right\|_{L_{t}^{2} H_{xy}^{1}} \\[0.2cm]
		&\leq C \varepsilon^{\frac{5}{2}} .
	\end{aligned}
	$$
	For $F_{4}$, similar to the estimate of $F_{2}$, we have
	$$
	\begin{aligned}
		\left\| F_{4}\right\|_{L_{t}^{2} L_{xy}^{2}} \leq C \varepsilon^{\frac{5}{2}} \left\| u_{1}^{I,1}\right\|_{L_{T}^{\infty} H_{xy}^{3}} \left\| \langle z \rangle u_{1}^{b,0}\right\|_{L_{t}^{2} H_{x}^{1}L_{z}^{2}}  \leq C \varepsilon^{\frac{5}{2}}.
	\end{aligned}
	$$
	A similar treatment yields that
	$$
	\begin{aligned}
		\left\| F_{5}\right\|_{L_{t}^{2} L_{xy}^{2}} & \leq \varepsilon^{\frac{5}{2}} \left\| u_{1}^{I,1}\right\|_{L_{T}^{\infty} H_{xy}^{2}} \left(\left\| u_{1}^{b,1}\right\|_{L_{t}^{2} H_{x}^{1}L_{z}^{2}} +\varepsilon \left\| u_{1}^{b,2}\right\|_{L_{t}^{2} H_{x}^{1}L_{z}^{2}}+\varepsilon^{\frac{3}{2}} \left\| S_{1}\right\|_{L_{t}^{2} H_{xy}^{1}}\right) \\[0.2cm]
		& \leq C \varepsilon^{\frac{5}{2}}
	\end{aligned}
	$$
	Using Proposition \ref{pro:boundary-layer equation} and Lemmas \ref{pro:6.1}, \ref{pro:6.3}, \ref{pro:6.4}, \ref{lemma:4.1}, we have
	$$
	\begin{aligned}
		\left\|F_{6}\right\|_{L_t^{2} L_{x y}^2} &\leq  \varepsilon^{\frac{5}{2}}\left\|u_{1}^{I, 2}\right\|_{L_T^{\infty} H_{x y}^{2}}\left(\varepsilon^{\frac{1}{2}}\left\|u_{1}^{I, 1}\right\|_{L_t^{2} H_{xy}^{1}}+\varepsilon^{\frac{3}{2}}\left\|u_{1}^{I, 2}\right\|_{L_t^{2} H_{xy}^{1}}+\left\| u_{1}^{b,0}\right\|_{L_{t}^{2} H_{x}^{1}L_{z}^{2}} \right. \\[0.2cm]
		& \left.\quad+\varepsilon\left\|u_{1}^{b, 1}\right\|_{L_t^{2} H_{x}^{1}L_{z}^{2}}+\varepsilon^{2}\left\| u_{1}^{b,2}\right\|_{L_{t}^{2} H_{x}^{1}L_{z}^{2} }+\left\| S_{1}\right\|_{L_{t}^{2} H_{xy}^{1}} \right) \\[0.2cm]
		&\leq  C \varepsilon^{\frac{5}{2}} .
	\end{aligned}
	$$
	For $F_{7}$, similar to the estimate of $F_{2}$, we have
	$$
	\begin{aligned}
		\left\|F_{7} \right\|_{L_{t}^{2} L_{xy}^{2}} &\leq  C \varepsilon^{\frac{5}{2}} \left\|u_{1}^{I,0} \right\|_{L_{t}^{2} H_{x}^{2}H_{y}^{3}} \left\|\langle z \rangle^{2}u_{1}^{b,0} \right\|_{L_{T}^{\infty} L_{xz}^{2}} + C \varepsilon^{\frac{5}{2}}\left\|u_{1}^{I,1} \right\|_{L_{t}^{2} H_{x}^{2}H_{y}^{2}} \left\|\langle z \rangle u_{1}^{b,0} \right\|_{L_{T}^{\infty} L_{xz}^{2}} \\[0.2cm]
		& \leq C \varepsilon^{\frac{5}{2}}.
	\end{aligned}
	$$
	A similar treatment yields that
	$$
	\begin{aligned}
		\left\|F_{8} \right\|_{L_{t}^{2} L_{xy}^{2}} &\leq  C \varepsilon^{\frac{5}{2}} \left\|u_{1}^{I,2} \right\|_{L_{t}^{2} H_{x}^{2}H_{y}^{1}} \left\|u_{1}^{b,0} \right\|_{L_{T}^{\infty} L_{xz}^{2}} +C \varepsilon^{\frac{5}{2}} \left\|u_{1}^{b,0} \right\|_{L_{T}^{\infty} H_{x}^{1}H_{z}^{1}} \left\|u_{1}^{b,2} \right\|_{L_{t}^{2} H_{x}^{1}L_{z}^{2}} \\[0.2cm]
		&\quad+ C \varepsilon^{3} \left\|u_{1}^{b,0} \right\|_{L_{T}^{\infty} H_{x}^{1}H_{z}^{1}} \left\|S_{1} \right\|_{L_{t}^{2} H_{xy}^{1}} \\[0.2cm]
		& \leq C \varepsilon^{\frac{5}{2}}.
	\end{aligned}
	$$
	For $F_{9}$, similar to the estimate of $F_{2}$, we have
	$$
	\begin{aligned}
		\left\|F_{9} \right\|_{L_{t}^{2} L_{xy}^{2}} \leq  C \varepsilon^{\frac{5}{2}} \left\|u_{1}^{I,0} \right\|_{L_{T}^{\infty} H_{xy}^{4}} \left\|\langle z 
		\rangle u_{1}^{b,1} \right\|_{L_{t}^{2} L_{xz}^{2}} \leq C \varepsilon^{\frac{5}{2}}.
	\end{aligned}
	$$
	A similar treatment yields that
	$$
	\begin{aligned}
		\left\|F_{10} \right\|_{L_{t}^{2} L_{xy}^{2}} &\leq  C \varepsilon^{\frac{5}{2}} \left\|u_{1}^{I,1} \right\|_{L_{T}^{\infty} H_{xy}^{3}} \left\|u_{1}^{b,1} \right\|_{L_{t}^{2} L_{xz}^{2}} +C \varepsilon^{\frac{7}{2}} \left\|u_{1}^{I,2} \right\|_{L_{t}^{2} H_{x}^{2}H_{y}^{1}} \left\|u_{1}^{b,1} \right\|_{L_{T}^{\infty} L_{xz}^{2}} \\[0.2cm]
		&\quad+ C \varepsilon^{\frac{5}{2}} \left\|u_{1}^{b,1} \right\|_{L_{T}^{\infty} H_{x}^{1}H_{z}^{1}} \left(\left\|u_{1}^{b,1} \right\|_{L_{t}^{2} H_{x}^{1}L_{z}^{2} }+\varepsilon \left\|u_{1}^{b,2} \right\|_{L_{t}^{2} H_{x}^{1}L_{z}^{2}} +\varepsilon^{\frac{3}{2}} \left\|S_{1} \right\|_{L_{t}^{2} H_{xy}^{1} }\right) \\[0.2cm]
		&\leq C \varepsilon^{\frac{5}{2}}.
	\end{aligned}
	$$
	Using the definition of $u_{1}^{a}$, it is easy to get
	$$
	\left\|\partial_x u_1^a\right\|_{L_t^{2} L_{x y}^{\infty}}  \leq C .
	$$
	Hence, for $F_{11}$, we have
	$$
	\begin{aligned}
		\left\|F_{11}\right\|_{L_t^{2} L_{x y}^2} &\leq C \varepsilon^{\frac{5}{2}}\left\|\partial_x u_1^a\right\|_{L_t^{2} L_{x y}^{\infty}}\left(\left\|u_1^{b, 2}\right\|_{L_T^{\infty} L_{xz}^2 }+\varepsilon^{\frac{1}{2}}\left\|S_1\right\|_{L_T^{\infty} L_{x y}^2}\right) \\[0.2cm]
		& \leq C \varepsilon^{\frac{5}{2}}
	\end{aligned}
	$$
	For $F_{12}$, similar to the estimate of $F_{2}$, we have
	$$
	\begin{aligned}
		\left\|F_{12} \right\|_{L_{t}^{2} L_{xy}^{2}} &\leq  C \varepsilon^{\frac{5}{2}} \left\|u_{2}^{I,0} \right\|_{L_{t}^{2} H_{x}^{1}H_{y}^{4}} \left\|\langle z \rangle^{3} u_{1}^{b,0} \right\|_{L_{T}^{\infty} L_{x}^{2}H_{z}^{1}} + C \varepsilon^{\frac{5}{2}} \left\|u_{2}^{I,0} \right\|_{L_{T}^{\infty} H_{xy}^{4}} \left\|\langle z \rangle^{2} u_{1}^{b,1} \right\|_{L_{t}^{2} L_{x}^{2}H_{z}^{1}} \\[0.2cm]
		& \leq C \varepsilon^{\frac{5}{2}}.
	\end{aligned}
	$$
	For $F_{13}$, similar to the estimate of $F_{2}$, we have
	$$
	\begin{aligned}
		\left\|F_{13} \right\|_{L_{t}^{2} L_{xy}^{2}} &\leq  C \varepsilon^{\frac{5}{2}} \left\|u_{2}^{I,0} \right\|_{L_{T}^{\infty} H_{xy}^{3}} \left\|\langle z \rangle u_{1}^{b,2} \right\|_{L_{t}^{2} L_{x}^{2}H_{z}^{1}} + C \varepsilon^{3} \left\|u_{2}^{I,0} \right\|_{L_{T}^{\infty} H_{xy}^{2}} \left\|S_{1} \right\|_{L_{t}^{2} H_{xy}^{1}} \\[0.2cm]
		&\quad+ C \varepsilon^{3} \left\|u_{2}^{I,1} \right\|_{L_{T}^{\infty} H_{xy}^{2}} \left\|u_{1}^{I,2} \right\|_{L_{t}^{2} H_{xy}^{1}} \\[0.2cm]
		& \leq C \varepsilon^{\frac{5}{2}}.
	\end{aligned}
	$$
	For $F_{14}$, similar to the estimate of $F_{2}$, we have
	$$
	\begin{aligned}
		\left\|F_{14} \right\|_{L_{t}^{2} L_{xy}^{2}} &\leq  C \varepsilon^{\frac{5}{2}} \left\|u_{2}^{I,1} \right\|_{L_{t}^{2} H_{x}^{1}H_{y}^{3}} \left\|\langle z \rangle^{2} u_{1}^{b,0} \right\|_{L_{T}^{\infty} L_{x}^{2}H_{z}^{1}} + C \varepsilon^{\frac{5}{2}} \left\|u_{2}^{I,1} \right\|_{L_{T}^{\infty} H_{xy}^{3}} \left\|\langle z \rangle u_{1}^{b,1} \right\|_{L_{t}^{2} L_{x}^{2}H_{z}^{1}} \\[0.2cm]
		& \leq C \varepsilon^{\frac{5}{2}}.
	\end{aligned}
	$$
	A similar treatment yields that
	$$
	\begin{aligned}
		\left\|F_{15} \right\|_{L_{t}^{2} L_{xy}^{2}} &\leq  C \varepsilon^{\frac{5}{2}} \left\|u_{2}^{I,1} \right\|_{L_{T}^{\infty} H_{xy}^{2}} \left\| u_{1}^{b,2} \right\|_{L_{t}^{2} L_{x}^{2}H_{z}^{1}} + C \varepsilon^{4} \left\|u_{2}^{I,1} \right\|_{L_{T}^{\infty} H_{xy}^{2}} \left\|S_{1} \right\|_{L_{t}^{2} H_{xy}^{1}} \\[0.2cm]
		&\quad+ C \varepsilon^{3} \left\|u_{1}^{I,1} \right\|_{L_{T}^{\infty} H_{xy}^{3}} \left\|u_{2}^{I,2} \right\|_{L_{t}^{2} L_{xy}^{2}} +C \varepsilon^{4} \left\|u_{2}^{I,2} \right\|_{L_{T}^{\infty} H_{xy}^{2}} \left\|u_{1}^{I,2} \right\|_{L_{t}^{2} H_{xy}^{1}} \\[0.2cm]
		& \leq C \varepsilon^{\frac{5}{2}}.
	\end{aligned}
	$$
	For $F_{16}$, similar to the estimate of $F_{2}$, we have
	$$
	\begin{aligned}
		\left\|F_{16} \right\|_{L_{t}^{2} L_{xy}^{2}} &\leq  C \varepsilon^{\frac{5}{2}} \left\|u_{2}^{I,2} \right\|_{L_{t}^{2} H_{x}^{1}H_{y}^{2}} \left\|\langle z \rangle u_{1}^{b,0} \right\|_{L_{T}^{\infty} L_{x}^{2}H_{z}^{1}} + C \varepsilon^{\frac{5}{2}} \left\|u_{2}^{I,2} \right\|_{L_{T}^{\infty} H_{xy}^{2}} \left\|u_{1}^{b,1} \right\|_{L_{t}^{2} L_{x}^{2}H_{z}^{1}} \\[0.2cm]
		&\quad+ C \varepsilon^{\frac{7}{2}} \left\|u_{2}^{I,2} \right\|_{L_{T}^{\infty} H_{xy}^{2}} \left\|u_{1}^{b,2} \right\|_{L_{t}^{2} L_{x}^{2}H_{z}^{1}} +C \varepsilon^{5} \left\|u_{2}^{I,2} \right\|_{L_{T}^{\infty} H_{xy}^{2}} \left\|S_{1} \right\|_{L_{t}^{2} H_{xy}^{1}} \\[0.2cm]
		& \leq C \varepsilon^{\frac{5}{2}}.
	\end{aligned}
	$$
	For $F_{17}$, similar to the estimate of $F_{2}$, we have
	$$
	\begin{aligned}
		\left\|F_{17} \right\|_{L_{t}^{2} L_{xy}^{2}} &\leq  C \varepsilon^{\frac{5}{2}} \left\|u_{1}^{I,0} \right\|_{L_{T}^{\infty} H_{xy}^{4}} \left\|\langle z \rangle u_{2}^{b,1} \right\|_{L_{t}^{2} L_{xz}^{2}} + C \varepsilon^{\frac{5}{2}} \left\|u_{1}^{I,1} \right\|_{L_{T}^{\infty} H_{xy}^{3}} \left\|u_{2}^{b,1} \right\|_{L_{t}^{2} L_{xz}^{2}} \\[0.2cm]
		&\quad+ \varepsilon^{\frac{5}{2}} \left\|u_{2}^{b,1} \right\|_{L_{T}^{\infty} H_{x}^{1}H_{z}^{1}}\left( \varepsilon^{\frac{1}{2}}\left\|u_{1}^{I,2} \right\|_{L_{t}^{2} H_{xy}^{1}}+\left\|u_{1}^{b,2} \right\|_{L_{t}^{2} L_{x}^{2}H_{z}^{1}}+\varepsilon^{\frac{3}{2}}\left\|S_{1} \right\|_{L_{t}^{2} H_{xy}^{1}}\right)  \\[0.2cm]
		&\leq C \varepsilon^{\frac{5}{2}}.
	\end{aligned}
	$$
	Using the definition of $u_{1}^{a}$, it is easy to get
	$$
	\left\|\partial_y u_1^a \right\|_{L_t^{2} L_{x y}^{\infty}}\leq \frac{C}{\varepsilon},\quad \left\|\partial_y u_1^a\right\|_{L_T^{\infty} L_{x y}^{2}} \leq \frac{C}{\sqrt{\varepsilon}}, \quad\left\|\partial_y\left(u_1^a-u_1^{b, 0}\right)\right\|_{L_T^{\infty} L_{x y}^{2}} \leq C .
	$$
	Hence, for $F_{18}$, we have
	$$
	\begin{aligned}
		\left\|F_{18}\right\|_{L_t^{2} L_{x y}^2} &\leq  \varepsilon^{\frac{5}{2}} \left\|\partial_{y}\left(u_{1}^{I,0}+\varepsilon u_{1}^{I,1}+\varepsilon^{2}u_{1}^{I,2}+\varepsilon^{3}S_{1}\right)\right\|_{L_{t}^{2} L_{xy}^{\infty}} \left\|u_{2}^{b,2}\right\|_{L_{T}^{\infty} L_{xz}^{2}} \\[0.2cm]
		&\quad+\varepsilon^{\frac{5}{2}}\left\|u_{2}^{b,2}\right\|_{L_{t}^{2} H_{x}^{1}H_{z}^{1}}\left(\left\|u_{1}^{b,1}\right\|_{L_{T}^{\infty} L_{x}^{2}H_{z}^{1}} +\varepsilon \left\|u_{1}^{b,2}\right\|_{L_{T}^{\infty} L_{x}^{2}H_{z}^{1}} \right) \\[0.2cm]
        &\quad+\varepsilon^{\frac{7}{2}}\left\|\partial_{y}u_{1}^{a} \right\|_{L_t^{2} L_{xy}^{\infty}}\left\|u_{2}^{b,3}\right\|_{L_T^{\infty} L_{xz}^{2}} \\[0.2cm]
		& \quad+C \varepsilon^{\frac{5}{2}}\left\|S_2\right\|_{L_t^{2} H_{x}^{1}H_{y}^{1}}\left(\varepsilon^{\frac{1}{2}}\left\|\partial_y\left(u_1^a-u_1^{b, 0}\right)\right\|_{L_T^{\infty} L_{x y}^{2}} + \left\|u_1^{b, 0}\right\|_{L_T^{\infty} L_{x}^{2}H_{z}^1} \right) \\[0.2cm]
        &\quad+\varepsilon^{\frac{5}{2}}\left\|p^{b, 2}\right\|_{L_T^{\infty} H_{x}^{1}L_{z}^{2}} \\[0.2cm]
		& \leq C \varepsilon^{\frac{5}{2}} .
	\end{aligned}
	$$
	Finally, we can estimate $F_{19}$ as follows.
	$$
	\begin{aligned}
		\left\|F_{19} \right\|_{L_{t}^{2} L_{xy}^{2}} &\leq  \varepsilon^{\frac{5}{2}} \left\|u_{1}^{b,2} \right\|_{L_{t}^{2} H_{x}^{2}L_{z}^{2}} + \varepsilon^{3} \left\|S_{1} \right\|_{L_{t}^{2} H_{xy}^{2}} +\varepsilon^{3} \left\|u_{1}^{I,1} \right\|_{L_{t}^{2} H_{xy}^{2}} \\[0.2cm]
		&\quad+\varepsilon^{4} \left\|u_{1}^{I,2} \right\|_{L_{t}^{2} H_{xy}^{2}} +\varepsilon^{\frac{5}{2}} \left\|u_{1}^{b,2} \right\|_{L_{t}^{2} L_{x}^{2}H_{z}^{2}} +\varepsilon^{5} \left\|S_{1} \right\|_{L_{t}^{2} H_{xy}^{2}} \\[0.2cm]
		& \leq C \varepsilon^{\frac{5}{2}}.
	\end{aligned}
	$$
	Combining the estimates of $F_{1}, \cdots, F_{19}$, we have
	$$
	\left\|\mathcal{R}^{\varepsilon}_{1} \right\|_{L_{t}^{2} L_{xy}^{2}} \leq \displaystyle \sum_{i=1}^{19} \left\|F_{i} \right\|_{L_{t}^{2} L_{xy}^{2}} \leq C \varepsilon^{\frac{5}{2}}.
	$$
    We next estimate $\partial_x\mathcal R_1^\varepsilon$ without differentiating the full normal residual. The time-derivative group $F_1$ uses $\partial_x\partial_tu_1^{b,2}\in L_t^2L_{z,\ell}^2L_x^2$ and $\partial_x\partial_tS_1=\varphi'\int_0^\infty\partial_x\partial_tu_1^{b,2}\,dz$. For the Taylor groups $F_2,F_4,F_7,F_9,F_{12},F_{14},F_{16},F_{17}$, use
    $$
    \begin{aligned}
    &\left\|\left(f-\sum_{j=0}^{m-1}\frac{y^j}{j!}\overline{\partial_y^jf}\right)
               Q(x,y/\varepsilon)\right\|_2 \leq C_m\varepsilon^{m+1/2}
       \|\partial_y^mf\|_{L_x^2L_y^\infty}
       \|z^mQ\|_{L_x^\infty L_z^2}.
    \end{aligned}
    $$
    This is Taylor's integral formula followed by $y=\varepsilon z$. If $\partial_x$ falls on an outer Taylor remainder, use $f=\partial_xu_i^{I,j}$ in this estimate. In the highest case, $F_{12}$, incompressibility gives
    $$
    \partial_x\partial_y^3u_2^{I,0}=-\partial_x^2\partial_y^2u_1^{I,0};
    $$
    its $L_t^2L_x^2L_y^\infty$ norm is bounded by $u^{I,0}\in CH^4$ and $\partial_xu^{I,0}\in L_t^2H^4$. The prefactor $\varepsilon^{-1}$ leaves the order $\varepsilon^{5/2}$.

    The product groups $F_3,F_5,F_6,F_8,F_{10},F_{11},F_{13},F_{15},F_{18}$ use the same mixed Sobolev estimates. At the terminal normal profile, put
    $$
    \partial_xu_2^{b,3}=\int_z^\infty\partial_x^2u_1^{b,2}\,dr
    $$
    in $L_T^\infty L_x^2L_z^\infty$, and put the leading shear factor in $L_T^\infty L_x^\infty L_z^2$. All other highest tangential derivatives are placed in their $L_t^2$ dissipation norms, while the other factor is bounded in time. The diffusion group $F_{19}$ uses
    $$
    \partial_x^3u_1^{b,2},\ \partial_x\partial_z^2u_1^{b,2}\in L_t^2L_{z,\ell}^2L_x^2,
    \qquad
    \partial_x\partial_y^2u^{I,2}\in L_t^2L_{xy}^2.
    $$
    These estimates cover all nineteen groups and give
    $$
    \sum_{i=1}^{19}\|\partial_xF_i\|_{L_t^2L_{xy}^2}
    \leq C_T\varepsilon^{5/2},\qquad
    \|\partial_x\mathcal R_1^\varepsilon\|_{L_t^2L_{xy}^2}
    \leq C_T\varepsilon^{5/2}.
    $$
    
    For the source term $\mathcal{R}_{2}$, we have the following estimates. For $G_{1}$, similar to the estimate of $F_{1}$, we have
	$$
	\begin{aligned}
		\left\|G_{1}\right\|_{L_t^{2} L_{x y}^2} \leq \varepsilon^{\frac{5}{2}}\left\|\partial_t u_{2}^{b, 2}\right\|_{L_t^{2} L_{xz}^{2} }+\varepsilon^{\frac{7}{2}}\left\|\partial_t u_{2}^{b,3}\right\|_{L_t^{2} L_{xz}^2 } +\varepsilon^{3} \left\|\partial_{t} S_{2} \right\|_{L_{t}^{2} L_{xy}^{2}}  \leq C \varepsilon^{\frac{5}{2}}.
	\end{aligned}
	$$
	For $G_{2}$, similar to the estimate of $F_{2}$, we have
	$$
	\begin{aligned}
		\left\|G_{2}\right\|_{L_t^{2} L_{x y}^2} \leq &  C \varepsilon^{\frac{5}{2}}\left\|u_1^{I, 0}\right\|_{L_T^{\infty} H_{x y}^{3}}\left\|\langle z\rangle u_{2}^{b, 1}\right\|_{L_t^{2} H_{x}^{1}L_{z}^{2}} \leq C \varepsilon^{\frac{5}{2}}.
	\end{aligned}
	$$
	Based on the regularity of boundary profiles, we can give
	$$
	\begin{aligned}
		\left\|G_{3}\right\|_{L_t^{2} L_{x y}^2} &\leq  C \varepsilon^{\frac{5}{2}} \left\|u_{1}^{I,0} \right\|_{L_{T}^{\infty} H_{xy}^{2}} \left( \left\| u_{2}^{b,2}\right\|_{L_{t}^{2} H_{x}^{1}L_{z}^{2}} + \varepsilon \left\|u_{2}^{b,3} \right\|_{L_{t}^{2} H_{x}^{1}L_{z}^{2}} +\varepsilon^{\frac{1}{2}} \left\|S_{2} \right\|_{L_{t}^{2} H_{xy}^{1}} \right) \\[0.2cm]
		&\quad+C \varepsilon^{3} \left\|u_{1}^{I,1} \right\|_{L_{T}^{\infty} H_{xy}^{2}} \left\|u_{2}^{I,2} \right\|_{L_{t}^{2} H_{xy}^{1}} \\[0.2cm]
		& \leq C \varepsilon^{\frac{5}{2}}.
	\end{aligned}
	$$
	A similar argument yields that
	$$
	\begin{aligned}
		\left\|G_{4}\right\|_{L_t^{2} L_{x y}^2} & \leq C \varepsilon^{\frac{5}{2}} \left\|u_{1}^{I,1} \right\|_{L_{T}^{\infty} H_{xy}^{2}} \left( \left\| u_{2}^{b,1}\right\|_{L_{t}^{2} H_{x}^{1}L_{z}^{2}} + \varepsilon \left\|u_{2}^{b,2} \right\|_{L_{t}^{2} H_{x}^{1}L_{z}^{2}} \right. \\[0.2cm]
		&\left.\quad+ \varepsilon^{2} \left\|u_{2}^{b,3} \right\|_{L_{t}^{2} H_{x}^{1}L_{z}^{2}} +\varepsilon^{\frac{3}{2}} \left\|S_{2} \right\|_{L_{t}^{2} H_{xy}^{1}} \right) \\[0.2cm]
		& \leq C \varepsilon^{\frac{5}{2}}.
	\end{aligned}
	$$
	A similar argument yields that
	$$
	\begin{aligned}
		\left\|G_{5}\right\|_{L_t^{2} L_{x y}^2} &\leq C \left\|u_{1}^{I,2} \right\|_{L_{T}^{\infty} H_{xy}^{2}} \left( \varepsilon^{3}\left\| u_{2}^{I,1}\right\|_{L_{t}^{2} H_{xy}^{1}} + \varepsilon^{4} \left\|u_{2}^{I,2} \right\|_{L_{t}^{2} H_{xy}^{1}} +\varepsilon^{\frac{7}{2}} \left\|u_{2}^{b,1} \right\|_{L_{t}^{2} H_{x}^{1}L_{z}^{2}} \right. \\[0.2cm]
		&\left.\quad+ \varepsilon^{\frac{9}{2}} \left\|u_{2}^{b,2} \right\|_{L_{t}^{2} H_{x}^{1}L_{z}^{2}} +\varepsilon^{\frac{11}{2}} \left\|u_{2}^{b,3} \right\|_{L_{t}^{2} H_{x}^{1}L_{z}^{2}}+\varepsilon^{5} \left\|S_{2} \right\|_{L_{t}^{2} H_{xy}^{1}} \right) \\[0.2cm]
		& \leq C \varepsilon^{3}.
	\end{aligned}
	$$
	For $G_{6}$, similar to the estimate of $F_{2}$, we have 
	$$
	\begin{aligned}
		\left\|G_{6}\right\|_{L_t^{2} L_{x y}^2} &\leq   C \varepsilon^{\frac{5}{2}}\left\|u_{2}^{I, 0}\right\|_{L_t^{2} H_{x}^{2}H_{y}^{3}}\left\|\langle z\rangle^{2} u_{1}^{b, 0}\right\|_{L_T^{\infty} L_{xz}^{2}}+C \varepsilon^{\frac{5}{2}} \left\|u_{2}^{I, 1}\right\|_{L_t^{2} H_{x}^{2}H_{y}^{2}}\left\|\langle z\rangle u_{1}^{b, 0}\right\|_{L_T^{\infty} L_{xz}^{2}} \\[0.2cm]
		& \leq C \varepsilon^{\frac{5}{2}}.
	\end{aligned}
	$$
	A similar argument yields that
	$$
	\begin{aligned}
		\left\|G_{7}\right\|_{L_t^{2} L_{x y}^2} &\leq C \varepsilon^{\frac{5}{2}}\left\|u_{2}^{I,2} \right\|_{L_{t}^{2} H_{x}^{2}H_{y}^{1}} \left\|u_{1}^{b,0} \right\|_{L_{T}^{\infty} L_{xz}^{2}}+C\varepsilon^{\frac{5}{2}} \left\|u_{1}^{b,0} \right\|_{L_{T}^{\infty} H_{x}^{1}H_{z}^{1}} \left\|u_{2}^{b,2} \right\|_{L_{t}^{2} H_{x}^{1}L_{z}^{2}}  \\[0.2cm]
		&\quad+C\varepsilon^{\frac{7}{2}} \left\|u_{1}^{b,0} \right\|_{L_{T}^{\infty}H_{x}^{1}H_{z}^{1}} \left\|u_{2}^{b,3} \right\|_{L_{t}^{2} H_{x}^{1}L_{z}^{2}} +C\varepsilon^{3} \left\|u_{1}^{b,0} \right\|_{L_{T}^{\infty} H_{x}^{1}H_{z}^{1}} \left\|S_{2} \right\|_{L_{t}^{2} H_{xy}^{1} }  \\[0.2cm]
		& \leq C \varepsilon^{\frac{5}{2}}.
	\end{aligned}
	$$
	For $G_{8}$, similar to the estimate of $F_{2}$, we have
	$$
	\begin{aligned}
		\left\|G_{8}\right\|_{L_t^{2} L_{x y}^2} \leq &  C \varepsilon^{\frac{5}{2}}\left\|u_{2}^{I, 0}\right\|_{L_T^{\infty} H_{x y}^{4}}\left\|\langle z\rangle u_{1}^{b,1}\right\|_{L_t^{2} L_{xz}^{2}}  \leq C \varepsilon^{\frac{5}{2}}.
	\end{aligned}
	$$
	A similar argument yields that
	$$
	\begin{aligned}
		\left\|G_{9}\right\|_{L_t^{2} L_{x y}^2} &\leq C \varepsilon^{\frac{5}{2}}\left\|u_{2}^{I,1} \right\|_{L_{T}^{\infty} H_{xy}^{3}} \left\|u_{1}^{b,1} \right\|_{L_{t}^{2} L_{xz}^{2}}+C\varepsilon^{3} \left\|u_{1}^{b,1} \right\|_{L_{T}^{\infty} H_{x}^{1}H_{z}^{1}} \left\|u_{2}^{I,2} \right\|_{L_{t}^{2} H_{xy}^{1}}  \\[0.2cm]
		&\quad+C\varepsilon^{\frac{5}{2}} \left\|u_{1}^{b,1} \right\|_{L_{t}^{2} H_{x}^{1}H_{z}^{1}} \left\|u_{2}^{b,1} \right\|_{L_{T}^{\infty} H_{x}^{1}L_{z}^{2}} +C\varepsilon^{\frac{7}{2}} \left\|u_{1}^{b,1} \right\|_{L_{t}^{2} H_{x}^{1}H_{z}^{1}} \left\|u_{2}^{b,2} \right\|_{L_{T}^{\infty} H_{x}^{1}L_{z}^{2} }  \\[0.2cm]
		&\quad+C\varepsilon^{\frac{9}{2}} \left\|u_{1}^{b,1} \right\|_{L_{T}^{\infty} H_{x}^{1}H_{z}^{1}} \left\|u_{2}^{b,3} \right\|_{L_{t}^{2} H_{x}^{1}L_{z}^{2}} +C\varepsilon^{4} \left\|u_{1}^{b,1} \right\|_{L_{T}^{\infty} H_{x}^{1}H_{z}^{1}} \left\|S_{2} \right\|_{L_{t}^{2} H_{xy}^{1} }  \\[0.2cm]
		& \leq C \varepsilon^{\frac{5}{2}}.
	\end{aligned}
	$$
	Using the definition of $u_{2}^{a}$, it is easy to get
	$$
	\left\|\partial_{x} \left(u_{2}^{a}-\varepsilon^{3}u_{2}^{b,3}\right)  \right\|_{L_{t}^{2} L_{xy}^{\infty}} \leq C.
	$$
	Hence, for $G_{10}$, we have
	$$
	\begin{aligned}
		\left\|G_{10} \right\|_{L_{t}^{2} L_{xy}^{2}} &\leq  \varepsilon^{\frac{5}{2}} \left\|\partial_{x} \left(u_{2}^{a}-\varepsilon^{3}u_{2}^{b,3}\right) \right\|_{L_{t}^{2} L_{xy}^{\infty}} \left(\left\|u_{1}^{b,2} \right\|_{L_{T}^{\infty} L_{xz}^{2}} + \varepsilon^{\frac{1}{2}} \left\|S_{1} \right\|_{L_{T}^{\infty} L_{xy}^{2}} \right) \\[0.2cm]
		&\quad+\varepsilon^{\frac{11}{2}}\left(\left\|u_{1}^{b,2}\right\|_{L_{t}^{2} H_{x}^{1}H_{z}^{1}}+\varepsilon \left\|S_{1}\right\|_{L_{t}^{2} H_{x}^{1}H_{y}^{1}} \right) \left\|u_{2}^{b,3}\right\|_{L_{T}^{\infty} H_{x}^{1}L_{z}^{2}} \\[0.2cm]
		& \leq C \varepsilon^{\frac{5}{2}}.
	\end{aligned}
	$$
	For $G_{11}$, similar to the estimate of $F_{2}$, we have
	$$
	\begin{aligned}
		\left\|G_{11}\right\|_{L_t^{2} L_{x y}^2} &\leq   C \varepsilon^{\frac{5}{2}}\left\|u_{2}^{I, 0}\right\|_{L_T^{\infty} H_{x y}^{4}}\left\|\langle z\rangle^{2} u_{2}^{b, 1}\right\|_{L_t^{2} L_{x}^{2}H_{z}^{1}}+C \varepsilon^{\frac{5}{2}} \left\|u_{2}^{I, 0}\right\|_{L_T^{\infty} H_{x y}^{3}}\left\|\langle z\rangle u_{2}^{b, 2}\right\|_{L_t^{2} L_{x}^{2}H_{z}^{1}}  \leq C \varepsilon^{\frac{5}{2}}.
	\end{aligned}
	$$
	A similar argument yields that
	$$
	\begin{aligned}
		\left\|G_{12}\right\|_{L_t^{2} L_{x y}^2} &\leq C \varepsilon^{\frac{5}{2}}\left\|u_{2}^{I,0} \right\|_{L_{T}^{\infty} H_{xy}^{2}} \left\|u_{2}^{b,3} \right\|_{L_{t}^{2} L_{x}^{2}H_{z}^{1}}+C\varepsilon^{3} \left\|u_{2}^{I,0} \right\|_{L_{T}^{\infty} H_{xy}^{2}} \left\|S_{2} \right\|_{L_{t}^{2} H_{xy}^{1}}  \\[0.2cm]
		&\quad+C\varepsilon^{3} \left\|u_{2}^{I,1} \right\|_{L_{T}^{\infty} H_{xy}^{2}} \left\|u_{2}^{I,2} \right\|_{L_{t}^{2} H_{xy}^{1}} \\[0.2cm]
		& \leq C \varepsilon^{\frac{5}{2}}.
	\end{aligned}
	$$
	For $G_{13}$, similar to the estimate of $F_{2}$, we have
	$$
	\begin{aligned}
		\left\|G_{13}\right\|_{L_t^{2} L_{x y}^2} &\leq   C \varepsilon^{\frac{5}{2}}\left\|u_{2}^{I, 1}\right\|_{L_T^{\infty} H_{x y}^{3}}\left\|\langle z\rangle u_{2}^{b, 1}\right\|_{L_t^{2} L_{x}^{2}H_{z}^{1}}+C \varepsilon^{\frac{5}{2}} \left\|u_{2}^{I, 1}\right\|_{L_T^{\infty} H_{x y}^{2}}\left\| u_{2}^{b, 2}\right\|_{L_t^{2} L_{x}^{2}H_{z}^{1}} \\[0.2cm]
		&\quad+C \varepsilon^{\frac{7}{2}} \left\|u_{2}^{I,1} \right\|_{L_{T}^{\infty} H_{xy}^{2}} \left\|u_{2}^{b,3} \right\|_{L_{t}^{2} L_{x}^{2}H_{z}^{1}}+C\varepsilon^{4} \left\|u_{2}^{I,1} \right\|_{L_{T}^{\infty} H_{xy}^{2}} \left\|S_{2} \right\|_{L_{t}^{2} H_{xy}^{1}} \\[0.2cm]
		& \leq C \varepsilon^{\frac{5}{2}}.
	\end{aligned}
	$$
	A similar argument yields that
	$$
	\begin{aligned}
		\left\|(G_{14},G_{15})\right\|_{L_t^{2}  L_{x y}^2} &\leq  C \varepsilon^{\frac{5}{2}}\left\|u_{2}^{I,2} \right\|_{L_{t}^{2} H_{xy}^{2}} \left( \varepsilon^{\frac{1}{2}}\left\|u_{2}^{I,1} \right\|_{L_{T}^{\infty} H_{xy}^{1}}+\varepsilon^{\frac{3}{2}}\left\|u_{2}^{I,2} \right\|_{L_{T}^{\infty}  H_{xy}^{1}}  \right. \\[0.2cm]
		&\left.\quad +\left\|u_{2}^{b,1} \right\|_{L_{T}^{\infty} L_{x}^{2}H_{z}^{1}} + \varepsilon \left\|u_{2}^{b,2} \right\|_{L_{T}^{\infty} L_{x}^{2}H_{z}^{1}} +\varepsilon^{2} \left\|u_{2}^{b,3} \right\|_{L_{T}^{\infty} L_{x}^{2}H_{z}^{1}}+\left\|S_{2} \right\|_{L_{T}^{\infty}H_{xy}^{1}} \right) \\[0.2cm]
		& \leq C \varepsilon^{\frac{5}{2}}.
	\end{aligned}
	$$
	For $G_{16}$, similar to the estimate of $F_{2}$, we have
	$$
	\begin{aligned}
		\left\|G_{16}\right\|_{L_t^{2}  L_{x y}^2} \leq & C \varepsilon^{\frac{5}{2}}\left\|u_{2}^{I,0} \right\|_{L_{T}^{\infty} H_{xy}^{4}} \left\|\langle z \rangle u_{2}^{b,1}\right\|_{L_{t}^{2} L_{xz}^{2}} \leq C \varepsilon^{\frac{5}{2}}.
	\end{aligned}
	$$
	A similar argument yields that
	$$
	\begin{aligned}
		\left\|G_{17}\right\|_{L_t^{2}  L_{x y}^2} &\leq  C \varepsilon^{\frac{5}{2}}\left\|u_{2}^{I,1} \right\|_{L_{t}^{2} H_{xy}^{3}} \left\|u_{2}^{b,1}\right\|_{L_{T}^{\infty} L_{xz}^{2}} +C \varepsilon^{\frac{5}{2}} \left\|u_{2}^{b,1} \right\|_{L_{T}^{\infty}H_{x}^{1}H_{z}^{1}} \left( \varepsilon^{\frac{1}{2}}\left\|u_{2}^{I,2} \right\|_{L_{t}^{2} H_{xy}^{1}} \right. \\[0.2cm]
		&\left.\quad+\left\|u_{2}^{b,2} \right\|_{L_{t}^{2} L_{x}^{2}H_{z}^{1}} +\varepsilon \left\|u_{2}^{b,3} \right\|_{L_{t}^{2} L_{x}^{2}H_{z}^{1}} +\varepsilon^{\frac{3}{2}} \left\|S_{2} \right\|_{L_{t}^{2} H_{xy}^{1}} \right) \\[0.2cm]
		& \leq C \varepsilon^{\frac{5}{2}}.
	\end{aligned}
	$$
	Using the definition of $u_{2}^{a}$, it is easy to get
	$$
	\left\|\partial_{y}\left(u_{2}^{a}-u_{2}^{I,0}\right) \right\|_{L_{t}^{2} L_{xy}^{\infty}} \leq  C,\quad \left\|\partial_{y}\left(u_{2}^{a}-u_{2}^{I,0}\right) \right\|_{L_{T}^{\infty} L_{xy}^{2}} \leq \sqrt{\varepsilon} C.
	$$
	Hence, for $G_{18}$, we have
	$$
	\begin{aligned}
		\left\|G_{18} \right\|_{L_{t}^{2} L_{xy}^{2}}&\leq C\varepsilon^{\frac{5}{2}} \left\|u_{2}^{I,0}\right\|_{L_{T}^{\infty} H_{xy}^{3}} \left(\left\|u_{2}^{b,2}\right\|_{L_{t}^{2} L_{xz}^{2}}+\varepsilon \left\|u_{2}^{b,3}\right\|_{L_{t}^{2} L_{xz}^{2}}+\varepsilon^{\frac{1}{2}}\left\|S_{2}\right\|_{L_{t}^{2} L_{xy}^{2}} \right) \\[0.2cm]
		&\quad +C \varepsilon^{2} \left(\left\|u_{2}^{b,2}\right\|_{L_{t}^{2} H_{x}^{1}H_{z}^{1}}+\varepsilon \left\|S_{2}\right\|_{L_{t}^{2} H_{xy}^{2}}\right) \left\|\partial_{y}\left(u_{2}^{a}-u_{2}^{I,0}\right) \right\|_{L_{T}^{\infty} L_{xy}^{2}}\\[0.2cm]
		&\quad+\varepsilon^{\frac{7}{2}} \left\|\partial_{y}\left(u_{2}^{a}-u_{2}^{I,0}\right) \right\|_{L_{t}^{2} L_{xy}^{\infty}} \left\|u_{2}^{b,3}\right\|_{L_{T}^{\infty} L_{xz}^{2}} \\[0.2cm]
		& \leq C \varepsilon^{\frac{5}{2}}.
	\end{aligned}
	$$
	A similar argument yields that
	$$
	\begin{aligned}
		\left\|G_{19}\right\|_{L_t^{2}  L_{x y}^2} &\leq  \varepsilon^{\frac{5}{2}} \left\|u_{2}^{b,2} \right\|_{L_{t}^{2} H_{x}^{2}L_{z}^{2}}+\varepsilon^{\frac{7}{2}} \left\|u_{2}^{b,3} \right\|_{L_{t}^{2} H_{x}^{2}L_{z}^{2}} +\varepsilon^{3} \left\|S_{2} \right\|_{L_{t}^{2} H_{xy}^{2} } \leq C \varepsilon^{\frac{5}{2}}.
	\end{aligned}
	$$
	A similar argument yields that
	$$
	\begin{aligned}
		\left\|G_{20}\right\|_{L_t^{2}  L_{x y}^2} &\leq  \varepsilon^{3} \left\|u_{2}^{I,1} \right\|_{L_{t}^{2} H_{xy}^{2}}+\varepsilon^{4} \left\|u_{2}^{I,2} \right\|_{L_{t}^{2} H_{xy}^{2}} +\varepsilon^{\frac{5}{2}} \left\|u_{1}^{b,1} \right\|_{L_{t}^{2} H_{x}^{1}H_{z}^{1}} \\[0.2cm]
		&\quad+\varepsilon^{\frac{7}{2}} \left\|u_{1}^{b,2} \right\|_{L_{t}^{2} H_{x}^{1}H_{z}^{1}} +\varepsilon^{5} \left\|S_{2} \right\|_{L_{t}^{2}  H_{xy}^{2}} \\[0.2cm]
		& \leq C \varepsilon^{\frac{5}{2}}.
	\end{aligned}
	$$
	Combining $G_{i}$, we have
	$$
	\|\mathcal{R}^{\varepsilon}_{2}\|_{L_t^{2} L_{x y}^2}\leq\sum_{i=1}^{20}\left\|G_i\right\|_{L_t^{2} L_{x y}^2} \leq C \varepsilon^{\frac{5}{2}}.
	$$
    For $\partial_y\mathcal R_2^\varepsilon$, differentiate the displayed groups $G_i$ once in the physical normal variable. The Taylor groups $G_2,G_6,G_8,G_{11},G_{13},G_{16}$ use the preceding Taylor estimate, with either one less power of $y$ or one additional normal derivative of the inner factor. In each case the resulting order is at least $\varepsilon^{3/2}$. The product groups $G_3,G_4,G_5,G_7,G_9,G_{10},G_{12},G_{14},G_{15},G_{17},G_{18}$ use the same mixed-norm placement and incompressibility. The lowest-order inner contribution after differentiation is $\varepsilon$ times a weighted profile, whose physical $L^2$ norm is $O(\varepsilon^{3/2})$.

    For the terminal time and diffusion groups $G_1,G_{19},G_{20}$, the identities needed are
    $$
    \begin{gathered}
    \partial_z\partial_tu_2^{b,3}=-\partial_x\partial_tu_1^{b,2},\qquad
    \partial_z^3u_2^{b,3}=-\partial_x\partial_z^2u_1^{b,2},\\[0.2cm]
    \partial_x^2\partial_zu_2^{b,3}=-\partial_x^3u_1^{b,2},\qquad
    \partial_y^3u_2^{I,2}=-\partial_x\partial_y^2u_1^{I,2}.
    \end{gathered}
    $$
    All right-hand sides are in the stated $L_t^2L^2$ spaces. The derivatives of $S$ are controlled from \eqref{eq:S}. Hence
    $$
    \sum_{i=1}^{20}\|\partial_yG_i\|_{L_t^2L_{xy}^2}
    \leq C_T\varepsilon^{3/2},\qquad
    \|\partial_y\mathcal R_2^\varepsilon\|_{L_t^2L_{xy}^2}
    \leq C_T\varepsilon^{3/2}.
    $$
    No estimate for $\partial_x\mathcal R_2^\varepsilon$ is used.

    Finally, For any scalar $f$, horizontal Fourier transformation and the fundamental theorem in $y$ give the integer bulk estimate:
    \begin{equation}
    \label{eq:Fourier-fundamental}
        \int|\xi||\widehat{f|_{y=0}}|^2\,d\xi \le C\|f_x\|_2\|f_y\|_2.
    \end{equation}
    using \eqref{eq:3.39}, \eqref{eq:3.3}-\eqref{eq:3.5} and \eqref{eq:3.7}, we have
    $$
    \mathcal{R}^{\varepsilon}_{2}|_{y=0}= \partial_{y}p^{a}|_{y=0}+\varepsilon^{2}\partial_{x}\partial_{y}u_{1}^{a}|_{y=0}, \quad \left(\partial_{y}p^{I,1}+\partial_{z}p^{b,2}+\partial_{x}\partial_{z}u_{1}^{b,0}\right)|_{y=0}=0.
    $$
    Thus, we get
    \begin{equation}
        \label{r4:wall-residual}
        \begin{gathered}
        \mathcal{R}^{\varepsilon}_{2}|_{y=0}=
        \varepsilon^{2}\left(\partial_{y}p^{I,2}|_{y=0}
        +\partial_{x}\partial_{y}u_{1}^{I,0}|_{y=0}
        +\partial_{x}\partial_{z}u_{1}^{b,1}|_{z=0}\right)\\[0.2cm]
        +\varepsilon^{3}\left(\partial_{x}\partial_{y}u_{1}^{I,1}|_{y=0}+\partial_{x}\partial_{z}u_{1}^{b,2}|_{z=0} \right) \\[0.2cm]
        +\varepsilon^{4}\partial_{x}\partial_{y}u_{1}^{I,2}|_{y=0}+\varepsilon^{5}\partial_{x}\partial_{y}S_{1}|_{y=0}.
    \end{gathered}
    \end{equation}
    For the last inner wall trace, \eqref{eq:Fourier-fundamental} gives
    $$
    \int_{\mathbb R}|\xi|
       |\widehat{\partial_x\partial_zu_1^{b,2}|_{z=0}}|^2\,d\xi
    \leq2\|\partial_x^2\partial_zu_1^{b,2}\|_2
             \|\partial_x\partial_z^2u_1^{b,2}\|_2.
    $$
    Both factors are in $L_t^2$. The outer trace $\overline{\partial_x\partial_yu_1^{I,2}}$ is controlled by $\partial_xu^{I,2}\in L_t^2H^2$, and $\overline{\partial_yp^{I,2}}$ by $\nabla p^{I,2}\in L_t^2H^1$. The lower-order traces have at least the same regularity, and the final lifting trace follows directly from \eqref{eq:S}. Applying these bounds to \eqref{r4:wall-residual}, we obtain
    $$
    \int_0^T\int_{\mathbb{R}}|\xi|
           |\widehat{\mathcal{R}^{\varepsilon}_2|_{y=0}}(\xi,t)|^2\,d\xi\,dt
             \le C_T\varepsilon^4.
    $$
    The proof is complete.
\end{proof}

\subsection{Estimates for the error terms}\label{subsec:error terms}
To avoid imposing strong compatibility conditions involving mixed normal derivatives at the initial time, we estimate $E^{\varepsilon},\partial_{x}E^{\varepsilon},\partial_{y}E_{1}^{\varepsilon}$and $\partial_{x}\partial_{y}E_{1}^{\varepsilon}$ separately in $L_{T}^{\infty}L_{xy}^{2}$. To begin with, we derive the following $L_{T}^{\infty} L_{xy}^{2}$ estimate for $(E_{1}^{\varepsilon},E_{2}^{\varepsilon})$.
\begin{lemma}
	\label{lemma:4.3}
	Under the assumptions of Theorem \ref{thm:main}, for any $0<\varepsilon\leq1$, there exists a constant $C$ independent of $\varepsilon$, such that
	$$
	\sup_{0\leq t\leq T}\left\|E^{\varepsilon} \right\|^{2}+\int_{0}^{T} \left\|\partial_{x}E^{\varepsilon}\right\|^{2} dt+\varepsilon^{2} \int_{0}^{T}\left\|\partial_{y}E^{\varepsilon}\right\|^{2} dt \leq C_{T} \varepsilon^{5}.
	$$
\end{lemma}
\begin{proof}
	Multiplying $(\ref{eq:3.39})$ by $E^{\varepsilon}$, and then integrating by parts, we obtain
	\begin{equation}
		\label{eq:E^{v}}
		\begin{aligned}
			\frac{1}{2} &\frac{d}{dt}\left\|(E_{1}^{\varepsilon},E_{2}^{\varepsilon}) \right\|^{2} +\left\|(\partial_{x}E_{1}^{\varepsilon},\partial_{x}E_{2}^{\varepsilon})\right\|^{2}+\varepsilon^{2}\left\|(\partial_{y}E_{1}^{\varepsilon},\partial_{y}E_{2}^{\varepsilon})\right\|^{2} \\[0.2cm]
			&=-\left\langle E^{\varepsilon} \cdot \nabla u_{1}^{a}, E_{1}^{\varepsilon} \right\rangle -\left\langle E^{\varepsilon} \cdot \nabla u_{2}^{a}, E_{2}^{\varepsilon} \right\rangle -\left\langle \mathcal{R}^{\varepsilon}_{1},E_{1}^{\varepsilon} \right\rangle-\left\langle \mathcal{R}^{\varepsilon}_{2},E_{2}^{\varepsilon} \right\rangle.
		\end{aligned}
	\end{equation}
	Using the definition of $u_{1}^{a}$, we handle the first term on the right hand side of \eqref{eq:E^{v}} as follows,
	$$
	\begin{aligned}
		-\left\langle E^{\varepsilon}\cdot\nabla u_{1}^{a},E_{1}^{\varepsilon} \right\rangle&=-\left\langle E^{\varepsilon}\cdot \nabla (u_{1}^{I}+\varepsilon^{3}S_{1})+E_{1}^{\varepsilon} \partial_{x}u_{1}^{b}, E_{1}^{\varepsilon} \right\rangle - \left\langle E_{2}^{\varepsilon}\partial_{y}u_{1}^{b},E_{1}^{\varepsilon} \right\rangle \\[0.2cm]
		&=-\left\langle E^{\varepsilon}\cdot \nabla (u_{1}^{I}+\varepsilon^{3}S_{1})+E_{1}^{\varepsilon} \partial_{x}u_{1}^{b}, E_{1}^{\varepsilon} \right\rangle - \left\langle \frac{1}{\varepsilon} E_{2}^{\varepsilon}\partial_{z}u_{1}^{b},E_{1}^{\varepsilon} \right\rangle \\[0.2cm]
		&=-\left\langle E^{\varepsilon}\cdot \nabla (u_{1}^{I}+\varepsilon^{3}S_{1})+E_{1}^{\varepsilon} \partial_{x}u_{1}^{b}, E_{1}^{\varepsilon} \right\rangle - \left\langle \frac{1}{y} E_{2}^{\varepsilon}z\partial_{z}u_{1}^{b},E_{1}^{\varepsilon} \right\rangle.
	\end{aligned}
	$$
    For the first term on the right-hand side, we directly obtain
    $$
    \begin{aligned}
    \left\langle E^{\varepsilon}\cdot \nabla (u_{1}^{I}+\varepsilon^{3}S_{1})+E_{1}^{\varepsilon} \partial_{x}u_{1}^{b}, E_{1}^{\varepsilon} \right\rangle \leq C\left\|(E_{1}^{\varepsilon},E_{2}^{\varepsilon}) \right\|^{2}+\left\|\left(\nabla u_{1}^{I},\nabla S_{1},\partial_{x}u_{1}^{b}\right) \right\|_{L^{\infty}}^{2} \left\|E_{1}^{\varepsilon} \right\|^{2}
    \end{aligned}
    $$
    For the second term, by Hardy's inequality, we have
    $$
    \begin{aligned}
    \left\langle \frac{1}{y} E_{2}^{\varepsilon}z\partial_{z}u_{1}^{b},E_{1}^{\varepsilon} \right\rangle &\leq  \left\|\partial_{y}E_{2}^{\varepsilon} \right\| \left\|\langle z \rangle  \partial_{z}u_{1}^{b}\right\|_{L^{\infty}}\left\|E_{1}^{\varepsilon} \right\| \\[0.2cm]
    &= \left\|\partial_{x}E_{1}^{\varepsilon} \right\| \left\|\langle z \rangle  \partial_{z}u_{1}^{b}\right\|_{L^{\infty}}\left\|E_{1}^{\varepsilon} \right\| \\[0.2cm]
    &\leq \frac{1}{8} \left\|\partial_{x}E_{1}^{\varepsilon} \right\|^{2} +C\left\|\langle z \rangle \partial_{z}u_{1}^{b} \right\|_{L^{\infty}}^{2} \left\|E_{1}^{\varepsilon} \right\|^{2}.
    \end{aligned}
    $$
    Thus, we obtain
    $$
    \begin{aligned}
    -\left\langle E^{\varepsilon}\cdot\nabla u_{1}^{a},E_{1}^{\varepsilon} \right\rangle \leq  \frac{1}{8} \left\|\partial_{x}E_{1}^{\varepsilon} \right\|^{2}+C\left\|(E_{1}^{\varepsilon},E_{2}^{\varepsilon}) \right\|^{2} +C\left\|\left(\nabla u_{1}^{I},\nabla S_{1},\partial_{x}u_{1}^{b},\langle z \rangle \partial_{z}u_{1}^{b}\right) \right\|_{L^{\infty}}^{2} \left\|E_{1}^{\varepsilon} \right\|^{2}.
    \end{aligned}
    $$
	Noting that $u_{2}^{b,0}=0$, we have
	$$
	-\left\langle E^{\varepsilon}\cdot \nabla u_{2}^{a},E_{2}^{\varepsilon} \right\rangle \leq C\left\|(E_{1}^{\varepsilon},E_{2}^{\varepsilon}) \right\|^{2}+C\left\|\left(\nabla u_{2}^{I},\nabla S_{2},\partial_{x}u_{2}^{b}, \varepsilon^{-1}\partial_{z}u_{2}^{b}\right) \right\|_{L^{\infty}}^{2} \left\|E_{2}^{\varepsilon} \right\|^{2}
	$$
	Combining these estimates, we have
    \begin{equation}
    \label{eq:E-est}
        \begin{aligned}
		\frac{d}{dt}&\left\|(E_{1}^{\varepsilon},E_{2}^{\varepsilon}) \right\|^{2}+\left\|(\partial_{x}E_{1}^{\varepsilon},\partial_{x}E_{2}^{\varepsilon}) \right\|^{2}+\varepsilon^{2}\left\| (\partial_{y}E_{1}^{\varepsilon},\partial_{y}E_{2}^{\varepsilon})\right\|^{2} \\[0.2cm]
		&\leq \left\|\mathcal{R}^{\varepsilon}\right\|^{2} +C \left(\left\|\nabla u^{I},\nabla S,\partial_{x}u^{b},  \varepsilon^{-1}\partial_{z}u_{2}^{b}, \langle z\rangle \partial_{z}u_{1}^{b} \right\|_{L^{\infty}}^{2}+1\right) \left\|(E_{1}^{\varepsilon},E_{2}^{\varepsilon}) \right\|^{2}.
	\end{aligned}
    \end{equation}
	which, together with Gronwall inequality, Lemmas \ref{lemma:4.1}-\ref{prop:residual}, implies
	$$
	\sup_{0\leq t\leq T}\left\|E^{\varepsilon} \right\|^{2}+\int_{0}^{T} \left\|\partial_{x}E^{\varepsilon}\right\|^{2} dt+\varepsilon^{2} \int_{0}^{T}\left\|\partial_{y}E^{\varepsilon}\right\|^{2} dt \leq C_{T} \varepsilon^{5}.
	$$
	Keeping the forcing in the form $\|\mathcal R^\varepsilon\|_2\|E^\varepsilon\|_2$ also gives
    \begin{equation}
    \label{eq:error-initial-L2}
    \|E^\varepsilon(t)\|_2
       \leq C_T\int_0^t\|\mathcal R^\varepsilon(s)\|_2\,ds
       \leq C_T\varepsilon^{5/2}t^{1/2}.
    \end{equation}
    Indeed, the coefficient in \eqref{eq:E-est} has a uniformly bounded time integral by Lemma \ref{lemma:4.1}, so this follows from the norm-level Gronwall inequality and $E^\varepsilon(0)=0$. This completes the proof.
\end{proof}

Next, we have the following $L_{T}^{\infty} L_{xy}^{2}$ estimate of $(\partial_{x}E_{1}^{\varepsilon},\partial_{x}E_{2}^{\varepsilon})$.
\begin{lemma}
	\label{lemma:4.4}
	Under the assumptions of Theorem \ref{thm:main}, for any $0<\varepsilon\leq1$, there exists a constant $C$ independent of $\varepsilon$, such that
	$$
	\sup_{0\leq t\leq T}\left\|\partial_{x}E^{\varepsilon} \right\|^{2}+\int_{0}^{T}\left\|\partial_{x}^{2}E^{\varepsilon} \right\|^{2}dt+\varepsilon^{2}\int_{0}^{T}\left\|\partial_{y}\partial_{x}E^{\varepsilon} \right\|^{2}dt \leq C_{T}\varepsilon^{3}.
	$$
\end{lemma}
\begin{proof}
	Multiplying $\partial_{x}(\ref{eq:3.39})$ by $\partial_{x}E^{\varepsilon}$, and then integrating by parts, we obtain
	\begin{equation}
		\label{eq:xE^{v}}
		\begin{aligned}
			\frac{1}{2}\frac{d}{dt}&\left\|(\partial_{x}E_{1}^{\varepsilon},\partial_{x}E_{2}^{\varepsilon}) \right\|^{2}+\left\|(\partial_{x}^{2}E_{1}^{\varepsilon},\partial_{x}^{2}E_{2}^{\varepsilon}) \right\|^{2}+\varepsilon^{2}\left\|(\partial_{y}\partial_{x}E_{1}^{\varepsilon},\partial_{y}\partial_{x}E_{2}^{\varepsilon}) \right\|^{2}=\sum_{i=1}^{6}I_{i},
		\end{aligned}
	\end{equation}
    where
    $$
    \begin{gathered}
        I_{1}=-\left\langle (\partial_{x}E^{\varepsilon}) \cdot \nabla u^{a},\partial_{x}E^{\varepsilon} \right\rangle, \quad I_{2}=-\left\langle E^{\varepsilon}\cdot \nabla (\partial_{x}u^{a}),\partial_{x}E^{\varepsilon}\right\rangle, \\[0.2cm]
    I_{3}=\left\langle E^{\varepsilon} \otimes \partial_{x}u^{a},\nabla \partial_{x}E^{\varepsilon} \right\rangle, \quad I_{4}=-\left\langle \partial_{x}E^{\varepsilon} \cdot\nabla E^{\varepsilon}, \partial_{x}E^{\varepsilon} \right\rangle, \\[0.2cm]
     I_5=\left\langle\mathcal R_1^\varepsilon,\partial_x^2E_1^\varepsilon\right\rangle,
    \qquad I_6=\left\langle\mathcal R_2^\varepsilon,\partial_x^2E_2^\varepsilon\right\rangle.
    \end{gathered}
    $$
	For $I_{1}$ and $I_{2}$, integrating by parts, we have
	$$
	\begin{aligned}
		I_{1}+I_{2}=\left\langle E^{\varepsilon} \cdot \nabla u^{a},\partial_{x}^{2}E^{\varepsilon} \right\rangle
		&\leq  \frac{1}{16}\left\|\partial_x^2E^\varepsilon\right\|^2+C\left\|(\nabla u^{I},\nabla S,\partial_{x}u^{b}, \varepsilon^{-1}\partial_{z}u_{2}^{b}) \right\|_{L^{\infty}}^{2} \left\|E^{\varepsilon} \right\|^{2} \\[0.2cm]
        &\quad +\|\left\langle z\right\rangle \partial_{z}u_{1}^{b}\|_{L^{\infty}}^{2}\|\partial_{x}E_{1}^{\varepsilon}\|^{2}
	\end{aligned}
	$$
	For $I_{3}$, we have
	$$
	\begin{aligned}
		I_{3} &\leq \frac{1}{16} \left\|(\partial_{x}^{2}E_{1}^{\varepsilon},\partial_{x}^{2}E_{2}^{\varepsilon}) \right\|^{2} +C \left\|\partial_{x}u_{1}^{a} \right\|_{L^{\infty}}^{2}\left\|(E_{1}^{\varepsilon},E_{2}^{\varepsilon}) \right\|^{2} \\[0.2cm]
		&\quad+\frac{\varepsilon^{2}}{16} \left\|(\partial_{y}\partial_{x}E_{1}^{\varepsilon},\partial_{y}\partial_{x}E_{2}^{\varepsilon}) \right\|^{2}+\frac{C}{\varepsilon^{2}}\left\|\partial_{x}u_{2}^{a} \right\|_{L^{\infty}}^{2}\left\|(E_{1}^{\varepsilon},E_{2}^{\varepsilon}) \right\|^{2}.
	\end{aligned}
	$$
	For $I_{4}$ using the Lemma \ref{lemma:2.2}, we have
	$$
	\begin{aligned}
		I_{4} &\leq \frac{3}{16} \left\|(\partial_{x}^{2}E_{1}^{\varepsilon},\partial_{x}^{2}E_{2}^{\varepsilon}) \right\|^{2}+\frac{\varepsilon^{2}}{16}\left\|(\partial_{y}\partial_{x}E_{1}^{\varepsilon},\partial_{y}\partial_{x}E_{2}^{\varepsilon}) \right\|^{2} \\[0.2cm]
		&\quad+C(1+\frac{1}{\varepsilon^{4}}) \left\|(E_{1}^{\varepsilon},E_{2}^{\varepsilon}) \right\|^{2} \left\|(\partial_{y}E_{1}^{\varepsilon},\partial_{y}E_{2}^{\varepsilon}) \right\|^{2}\left\|(\partial_{x}E_{1}^{\varepsilon},\partial_{x}E_{2}^{\varepsilon}) \right\|^{2}.
	\end{aligned}
	$$
	For the last two terms, we have
	$$
	 I_5+I_6\leq\frac1{16}\|\partial_x^2E^\varepsilon\|_2^2+C\|\mathcal R^\varepsilon\|_2^2.
	$$
	Substituting the estimates of $I_{1}, \cdots, I_{6}$ into $(\ref{eq:xE^{v}})$, we have
	$$
	\begin{aligned}
		&\frac{d}{dt}\left\|(\partial_{x}E_{1}^{\varepsilon},\partial_{x}E_{2}^{\varepsilon}) \right\|^{2}+\left\|(\partial_{x}^{2}E_{1}^{\varepsilon},\partial_{x}^{2}E_{2}^{\varepsilon}) \right\|^{2}+\varepsilon^{2}\left\|(\partial_{y}\partial_{x}E_{1}^{\varepsilon},\partial_{y}\partial_{x}E_{2}^{\varepsilon}) \right\|^{2} \\[0.2cm]
		&\leq C\left( \|\langle z \rangle \partial_{z}u_{1}^{b} \|_{L^{\infty}}^{2} +(1+\frac{1}{\varepsilon^{4}})\left\|(E_{1}^{\varepsilon},E_{2}^{\varepsilon}) \right\|^{2}\left\|(\partial_{y}E_{1}^{\varepsilon},\partial_{y}E_{2}^{\varepsilon}) \right\|^{2}+1 \right) \left\|(\partial_{x}E_{1}^{\varepsilon},\partial_{x}E_{2}^{\varepsilon}) \right\|^{2} \\[0.2cm]
		&+C \left(\left\|(\nabla u^{I},\nabla S,\partial_{x}u^{b}, \varepsilon^{-1}\partial_{z}u_{2}^{b}) \right\|_{L^{\infty}}^{2}+\left\|\partial_{x}u_{1}^{a} \right\|_{L^{\infty}}^{2}+\frac{1}{\varepsilon^{2}}\left\|\partial_{x}u_{2}^{a} \right\|_{L^{\infty}}^{2}\right)\left\|(E_{1}^{\varepsilon},E_{2}^{\varepsilon}) \right\|^{2}+ C\left\|\mathcal R^\varepsilon\right\|^2.
	\end{aligned}
	$$
	which together with Lemmas \ref{lemma:4.1}-\ref{lemma:4.3} and Gronwall inequality, implies
	$$
	\sup_{0\leq t \leq T}\left\|\partial_{x}E^{\varepsilon} \right\|^{2}+\int_{0}^{T} \left\|\partial_{x}^{2}E^{\varepsilon} \right\|^{2}dt +\varepsilon^{2}\int_{0}^{T}\left\|\partial_{y}\partial_{x}E^{\varepsilon} \right\|^{2}dt \leq C_{2}T e^{C_{2}T} \varepsilon^{3}.
	$$
	The proof is complete.
\end{proof}

Next, we give the $L_{T}^{\infty} L_{xy}^{2}$ estimate of $\partial_{y}E_{1}^{\varepsilon}$.
\begin{lemma}
\label{lemma:4.5}
    Under the assumptions of Theorem \ref{thm:main}, for any $0<\varepsilon\leq1$, there exists a constant $C$ independent of $\varepsilon$, such that
	$$
	\sup_{0\leq t\leq T}\left\|\partial_{y}E_{1}^{\varepsilon} \right\|^{2}+\int_{0}^{T}\left\|\partial_{x}\partial_{y}E_{1}^{\varepsilon} \right\|^{2}dt+\varepsilon^{2}\int_{0}^{T}\left\|\partial_{y}^{2} E_{1}^{\varepsilon} \right\|^{2}dt \leq C_{T}\varepsilon^{3}.
	$$
\end{lemma}
\begin{proof}
For $0<t_0<T$, we first make the following calculations on $[t_0,T]$.
Multiplying the first equation in \eqref{eq:3.39} by
$-\partial_y^2E_1^\varepsilon$,
 and then integrating by parts,
we obtain
\begin{equation}
\label{eq:yE1}
\begin{aligned}
&\frac12\frac{d}{dt}\|\partial_yE_1^\varepsilon\|^2
 +\|\partial_x\partial_yE_1^\varepsilon\|^2
 +\varepsilon^2\|\partial_y^2E_1^\varepsilon\|^2
 =\sum_{i=1}^{3}M_i,
\end{aligned}
\end{equation}
where
$$
\begin{aligned}
M_1
&=\left\langle
 (u^a+E^\varepsilon)\cdot\nabla(u_1^a+E_1^\varepsilon)
 -u^a\cdot\nabla u_1^a,
 \partial_y^2E_1^\varepsilon
 \right\rangle,\\[0.2cm]
M_2
&=\left\langle
 \partial_x\pi^\varepsilon,
 \partial_y^2E_1^\varepsilon
 \right\rangle,\qquad
M_3
=\left\langle
 \mathcal R_1^\varepsilon,
 \partial_y^2E_1^\varepsilon
 \right\rangle.
\end{aligned}
$$

For $M_1$, using incompressibility and integrating by parts, we have
$$
\begin{aligned}
M_1
&=-\left\langle
 E^\varepsilon\cdot\nabla\partial_yu_1^a,
 \partial_yE_1^\varepsilon
 \right\rangle\\[0.2cm]
&=\left\langle
 (\partial_yE^\varepsilon)\cdot
 \nabla(u_1^I+\varepsilon^3S_1),
 \partial_yE_1^\varepsilon
 \right\rangle +
 \left\langle
 E^\varepsilon\cdot\nabla(u_1^I+\varepsilon^3S_1),
 \partial_y^2E_1^\varepsilon
 \right\rangle\\[0.2cm]
&\quad-
 \frac1\varepsilon
 \left\langle
 E_1^\varepsilon\partial_x\partial_zu_1^b,
 \partial_yE_1^\varepsilon
 \right\rangle
 -\frac1{\varepsilon^2}
 \left\langle
 E_2^\varepsilon\partial_z^2u_1^b,
 \partial_yE_1^\varepsilon
 \right\rangle.
\end{aligned}
$$
Since $\partial_yE_2^\varepsilon=-\partial_xE_1^\varepsilon$,
the first two terms satisfy
$$
\begin{aligned}
&\left|
 \left\langle
 (\partial_yE^\varepsilon)\cdot
 \nabla(u_1^I+\varepsilon^3S_1),
 \partial_yE_1^\varepsilon
 \right\rangle
 \right| +
 \left|
 \left\langle
 E^\varepsilon\cdot\nabla(u_1^I+\varepsilon^3S_1),
 \partial_y^2E_1^\varepsilon
 \right\rangle
 \right|\\[0.2cm]
&\le
 \frac{\varepsilon^2}{32}
 \|\partial_y^2E_1^\varepsilon\|^2
 +C\left(
 1+\|\nabla(u_1^I+\varepsilon^3S_1)\|_{L^\infty}^2
 \right)\|\partial_yE_1^\varepsilon\|^2\\[0.2cm]
&\quad+
 C\|\partial_xE_1^\varepsilon\|^2
 +\frac C{\varepsilon^2}
 \|\nabla(u_1^I+\varepsilon^3S_1)\|_{L^\infty}^2
 \|E^\varepsilon\|^2.
\end{aligned}
$$
For the boundary-layer terms, the boundary conditions give
$$
\begin{aligned}
E_1^\varepsilon(x,y,t)
&=\int_0^y\partial_yE_1^\varepsilon(x,r,t)\,dr,\\[0.2cm]
E_2^\varepsilon(x,y,t)
&=-\int_0^y\partial_xE_1^\varepsilon(x,r,t)\,dr\\[0.2cm]
&=-\int_0^y(y-r)
 \partial_x\partial_yE_1^\varepsilon(x,r,t)\,dr.
\end{aligned}
$$
Consequently,
$$
\begin{aligned}
|E_1^\varepsilon(x,y,t)|
&\le y^{1/2}
 \|\partial_yE_1^\varepsilon(x,\cdot,t)\|_{L_y^2},\\[0.2cm]
|E_2^\varepsilon(x,y,t)|
&\le \frac{y^{3/2}}{\sqrt3}
 \|\partial_x\partial_yE_1^\varepsilon(x,\cdot,t)\|_{L_y^2}.
\end{aligned}
$$
Using these inequalities, the change of variables
$y=\varepsilon z$, and the one-dimensional Sobolev
inequality in $x$, we obtain
$$
\begin{aligned}
\frac1\varepsilon
 \left|
 \left\langle
 E_1^\varepsilon\partial_x\partial_zu_1^b,
 \partial_yE_1^\varepsilon
 \right\rangle
 \right| &\le
 \|z^{1/2}\partial_x\partial_zu_1^b\|_{L_x^\infty L_z^2}
 \|\partial_yE_1^\varepsilon\|^2\\[0.2cm]
&\le
 C\|\langle z\rangle\partial_x\partial_zu_1^b\|_{H_x^1L_z^2}
 \|\partial_yE_1^\varepsilon\|^2,
\end{aligned}
$$
and
$$
\begin{aligned}
\frac1{\varepsilon^2}
 \left|
 \left\langle
 E_2^\varepsilon\partial_z^2u_1^b,
 \partial_yE_1^\varepsilon
 \right\rangle
 \right|&\le
 C\|z^{3/2}\partial_z^2u_1^b\|_{L_x^\infty L_z^2}
 \|\partial_x\partial_yE_1^\varepsilon\|
 \|\partial_yE_1^\varepsilon\|\\[0.2cm]
&\le
 \frac1{32}\|\partial_x\partial_yE_1^\varepsilon\|^2
 +C\|\langle z\rangle^2\partial_z^2u_1^b\|_{H_x^1L_z^2}^2
 \|\partial_yE_1^\varepsilon\|^2.
\end{aligned}
$$
Combining these estimates yields
\begin{equation}
\label{eq:yE1-M1-repaired}
\begin{aligned}
M_1
&\le
 \frac1{32}\|\partial_x\partial_yE_1^\varepsilon\|^2
 +\frac{\varepsilon^2}{32}
 \|\partial_y^2E_1^\varepsilon\|^2 
 +C\left(
 1+\|\nabla(u_1^I+\varepsilon^3S_1)\|_{L^\infty}^2
 +\|\langle z\rangle\partial_x\partial_zu_1^b\|_{H_x^1L_z^2}^2 \right. \\[0.2cm]
 &\quad \left. +\|\langle z\rangle^2\partial_z^2u_1^b\|_{H_x^1L_z^2}^2
 \right)\|\partial_yE_1^\varepsilon\|^2 +
 C\|\partial_xE_1^\varepsilon\|^2
 +\frac C{\varepsilon^2}
 \|\nabla(u_1^I+\varepsilon^3S_1)\|_{L^\infty}^2
 \|E^\varepsilon\|^2.
\end{aligned}
\end{equation}
For $M_2$, we decompose
$\pi^\varepsilon=\pi_H+\pi_R+\pi_s$, where
$$
\left\{
\begin{aligned}
-\Delta\pi_H
&=2\operatorname{div}\left(
 u_1^a\partial_xE^\varepsilon
 +E_1^\varepsilon\partial_xu^a
 +E_1^\varepsilon\partial_xE^\varepsilon
 \right),\\
\partial_y\pi_H|_{y=0}&=0,
\end{aligned}
\right.
$$
$$
\left\{
\begin{aligned}
-\Delta\pi_R&=\operatorname{div}\mathcal R^\varepsilon,\\
\partial_y\pi_R|_{y=0}&=-\mathcal R_2^\varepsilon|_{y=0},
\end{aligned}
\right.
$$
and
$$
\left\{
\begin{aligned}
\Delta\pi_s&=0,\\
\partial_y\pi_s|_{y=0}
&=-\varepsilon^2\partial_x\partial_yE_1^\varepsilon|_{y=0}.
\end{aligned}
\right.
$$
The energy estimates for the finite-energy Neumann problems give
$$
\begin{aligned}
\|\nabla\pi_H\|^2
&\le
4\left\|
 u_1^a\partial_xE^\varepsilon
 +E_1^\varepsilon\partial_xu^a
 +E_1^\varepsilon\partial_xE^\varepsilon
\right\|^2\\[0.2cm]
&\le
 C\|u_1^a\|_{L^\infty}^2\|\partial_xE^\varepsilon\|^2
 +C\|\partial_xu^a\|_{L^\infty}^2\|E^\varepsilon\|^2
 +C\|E_1^\varepsilon\partial_xE^\varepsilon\|^2,
\end{aligned}
$$
and
$$
\|\nabla\pi_R\|^2\le\|\mathcal R^\varepsilon\|^2.
$$
For the nonlinear product, the one-dimensional Sobolev
inequalities yield
$$
\begin{aligned}
\|E_1^\varepsilon\partial_xE^\varepsilon\|^2
&\le
 \|E_1^\varepsilon\|_{L_x^2L_y^\infty}^2
 \|\partial_xE^\varepsilon\|_{L_x^\infty L_y^2}^2\\[0.2cm]
&\le
4\|E_1^\varepsilon\|
 \|\partial_yE_1^\varepsilon\|
 \|\partial_xE^\varepsilon\|
 \|\partial_x^2E^\varepsilon\|.
\end{aligned}
$$
By Lemmas \ref{lemma:4.3} and \ref{lemma:4.4},
$$
\sup_{0<t<T}\|E^\varepsilon(t)\|\le C_T\varepsilon^{5/2},
\qquad
\sup_{0<t<T}\|\partial_xE^\varepsilon(t)\|
\le C_T\varepsilon^{3/2}.
$$
Therefore,
\begin{equation}
\label{eq:yE1-pressure-product}
\begin{aligned}
\frac1{\varepsilon^2}
\|E_1^\varepsilon\partial_xE^\varepsilon\|^2
&\le
 C_T\varepsilon^2
 \|\partial_yE_1^\varepsilon\|
 \|\partial_x^2E^\varepsilon\|\\[0.2cm]
&\le
 C_T\varepsilon^2
 \left(
 \|\partial_yE_1^\varepsilon\|^2
 +\|\partial_x^2E^\varepsilon\|^2
 \right).
\end{aligned}
\end{equation}
For the harmonic pressure, the unitary Fourier transform in $x$
gives
$$
\widehat{\pi_s}(\xi,y,t)
=
\varepsilon^2\frac{\mathrm i\xi}{|\xi|}
\widehat{\partial_yE_1^\varepsilon}(\xi,0,t)e^{-|\xi|y}.
$$
Hence,
$$
\begin{aligned}
\|\partial_x\pi_s\|^2
=
\frac{\varepsilon^4}{2}
\int_{\mathbb R}|\xi|
\left|
\widehat{\partial_yE_1^\varepsilon}(\xi,0,t)
\right|^2\,d\xi  \leq\varepsilon^4
\|\partial_x\partial_yE_1^\varepsilon\|
\|\partial_y^2E_1^\varepsilon\|.
\end{aligned}
$$
It follows from Young's inequality and $0<\varepsilon\le1$ that
$$
\begin{aligned}
\left|
\left\langle
\partial_x\pi_s,\partial_y^2E_1^\varepsilon
\right\rangle
\right|
\le
\varepsilon^2
\|\partial_x\partial_yE_1^\varepsilon\|^{1/2}
\|\partial_y^2E_1^\varepsilon\|^{3/2}\le
\frac14\|\partial_x\partial_yE_1^\varepsilon\|^2
+\frac{3\varepsilon^2}{4}
\|\partial_y^2E_1^\varepsilon\|^2.
\end{aligned}
$$
Combining these estimates, we obtain
\begin{equation}
\label{eq:yE1-M2-repaired}
\begin{aligned}
M_2
&\le
\frac14\|\partial_x\partial_yE_1^\varepsilon\|^2
+\frac{25\varepsilon^2}{32}
\|\partial_y^2E_1^\varepsilon\|^2 +
C_T\varepsilon^2
\left(
\|\partial_yE_1^\varepsilon\|^2
+\|\partial_x^2E^\varepsilon\|^2
\right) \\[0.2cm]
&\quad+
\frac C{\varepsilon^2}
\left(
\|u_1^a\|_{L^\infty}^2\|\partial_xE^\varepsilon\|^2
+\|\partial_xu^a\|_{L^\infty}^2\|E^\varepsilon\|^2
+\|\mathcal R^\varepsilon\|^2
\right).
\end{aligned}
\end{equation}
For $M_3$, Young's inequality gives
$$
M_3
\le
\frac{\varepsilon^2}{32}\|\partial_y^2E_1^\varepsilon\|^2
+\frac C{\varepsilon^2}\|\mathcal R_1^\varepsilon\|^2.
$$
Substituting the estimates of $M_1,M_2,M_3$
into \eqref{eq:yE1}, we obtain
\begin{equation}
\label{eq:yE1-closed-repaired}
\begin{aligned}
&\frac{d}{dt}\|\partial_yE_1^\varepsilon\|^2
+\frac14\|\partial_x\partial_yE_1^\varepsilon\|^2
+\frac{\varepsilon^2}{4}\|\partial_y^2E_1^\varepsilon\|^2\\[0.2cm]
&\quad\le
C_T\Bigl(
1+\|\nabla(u_1^I+\varepsilon^3S_1)\|_{L^\infty}^2
+\|\langle z\rangle\partial_x\partial_zu_1^b\|_{H_x^1L_z^2}^2
+\|\langle z\rangle^2\partial_z^2u_1^b\|_{H_x^1L_z^2}^2
\Bigr)\|\partial_yE_1^\varepsilon\|^2\\[0.2cm]
&\qquad+
C\|\partial_xE_1^\varepsilon\|^2
+C_T\varepsilon^2\|\partial_x^2E^\varepsilon\|^2+
\frac C{\varepsilon^2}
\left(
\|\nabla(u_1^I+\varepsilon^3S_1)\|_{L^\infty}^2
+\|\partial_xu^a\|_{L^\infty}^2
\right)\|E^\varepsilon\|^2\\[0.2cm]
&\qquad\quad+
\frac C{\varepsilon^2}
\|u_1^a\|_{L^\infty}^2\|\partial_xE^\varepsilon\|^2
+\frac C{\varepsilon^2}\|\mathcal R^\varepsilon\|^2.
\end{aligned}
\end{equation}
For the term containing $\varepsilon^{-2}\|\partial_xE^\varepsilon\|_2^2$,
Lemma \ref{lemma:4.3} gives the required dissipation bound:
\[
 \varepsilon^{-2}\int_0^T
 \|u_1^a\|_\infty^2\|\partial_xE^\varepsilon\|_2^2\,dt
 \leq C_T\varepsilon^{-2}\int_0^T
             \|\partial_xE^\varepsilon\|_2^2\,dt
 \leq C_T\varepsilon^3.
\]
The remaining source terms are controlled by Lemmas \ref{lemma:4.1}--\ref{lemma:4.4}.
Using the profile regularity in Propositions \ref{prop:outer-flow},
\ref{pro:boundary-layer equation}, and Lemmas \ref{pro:6.1}--\ref{pro:6.4},
Gronwall's inequality applied to \eqref{eq:yE1-closed-repaired} on $[t_0,T]$ yields
\begin{equation}
\label{eq:yE1-positive-time-repaired}
\begin{aligned}
\sup_{t_0\le t\le T}\|\partial_yE_1^\varepsilon(t)\|^2
&+\int_{t_0}^T
 \|\partial_x\partial_yE_1^\varepsilon(t)\|^2\,dt\\[0.2cm]
&+
\varepsilon^2\int_{t_0}^T
\|\partial_y^2E_1^\varepsilon(t)\|^2\,dt
\le
C_T\left(
\|\partial_yE_1^\varepsilon(t_0)\|^2+\varepsilon^3
\right),
\end{aligned}
\end{equation}
where $C_T$ is independent of $\varepsilon$ and $t_0$.

To pass to the initial time, use the norm-level estimate \eqref{eq:error-initial-L2} in the velocity energy identity \eqref{eq:E^{v}}. This gives
$$
\begin{aligned}
\varepsilon^2\int_0^t
\|\partial_yE^\varepsilon(s)\|^2\,ds
\le
C_T\left(
\int_0^t\|\mathcal R^\varepsilon(s)\|\,ds
\right)^2 \le
C_Tt\int_0^t\|\mathcal R^\varepsilon(s)\|^2\,ds.
\end{aligned}
$$
For each fixed $\varepsilon>0$, it follows that
$$
\frac1t\int_0^t
\|\partial_yE_1^\varepsilon(s)\|^2\,ds
\le
\frac{C_T}{\varepsilon^2}
\int_0^t\|\mathcal R^\varepsilon(s)\|^2\,ds
\longrightarrow0
\qquad(t\downarrow0).
$$
Hence there exists a sequence $t_n\downarrow0$, chosen
among times for which the preceding energy estimates
hold, such that
$$
\|\partial_yE_1^\varepsilon(t_n)\|\longrightarrow0.
$$
Taking $t_0=t_n$ in
\eqref{eq:yE1-positive-time-repaired} and letting
$n\to\infty$, we obtain
$$
\begin{aligned}
\sup_{0\le t\le T}\|\partial_yE_1^\varepsilon(t)\|^2
+\int_0^T
\|\partial_x\partial_yE_1^\varepsilon(t)\|^2\,dt +
\varepsilon^2\int_0^T
\|\partial_y^2E_1^\varepsilon(t)\|^2\,dt
\le C_T\varepsilon^3.
\end{aligned}
$$
The proof is complete.
\end{proof}
Finally, we give the $L_{T}^{\infty}L_{xy}^{2}$ estimate of $\partial_{x}\partial_{y}E_{1}^{\varepsilon}$.
\begin{lemma}
\label{lemma:4.6}
    Under the assumptions of Theorem \ref{thm:main}, for any $0<\varepsilon\leq1$, there exists a constant $C$ independent of $\varepsilon$, such that
	$$
	\sup_{0\leq t\leq T}\left\|\partial_{x}\partial_{y}E_{1}^{\varepsilon} \right\|^{2}+\int_{0}^{T}\left\|\partial_{x}^{2}\partial_{y}E_{1}^{\varepsilon} \right\|^{2}dt+\varepsilon^{2}\int_{0}^{T}\left\|\partial_{x}\partial_{y}^{2} E_{1}^{\varepsilon} \right\|^{2}dt \leq C_{T}\varepsilon.
	$$
\end{lemma}
\begin{proof}
Multiplying the first equation in \eqref{eq:3.39} by
$\partial_x^2\partial_y^2E_1^\varepsilon$, and then
integrating by parts, we obtain
\begin{equation}
\label{eq:xyE1}
\begin{aligned}
&\frac12\frac{d}{dt}
 \|\partial_x\partial_yE_1^\varepsilon\|^2
 +\|\partial_x^2\partial_yE_1^\varepsilon\|^2
 +\varepsilon^2
  \|\partial_x\partial_y^2E_1^\varepsilon\|^2
 =\sum_{i=1}^{5}N_i,
\end{aligned}
\end{equation}
where
$$
\begin{gathered}
    N_1=-\left\langle
 (\partial_xE^\varepsilon)\cdot\nabla\partial_yu_1^a,
 \partial_x\partial_yE_1^\varepsilon
 \right\rangle,\quad
N_2=-\left\langle
 E^\varepsilon\cdot\nabla\partial_x\partial_yu_1^a,
 \partial_x\partial_yE_1^\varepsilon
 \right\rangle,\\[0.2cm]
N_3=\left\langle
 \partial_x^2\pi^\varepsilon,
 \partial_x\partial_y^2E_1^\varepsilon
 \right\rangle,\quad 
N_4=\left\langle
 \partial_x\mathcal R_1^\varepsilon,
 \partial_x\partial_y^2E_1^\varepsilon
 \right\rangle,\\[0.2cm]
N_5=-\left\langle
 \partial_x(u^a+E^\varepsilon)
       \cdot\nabla\partial_yE_1^\varepsilon,
 \partial_x\partial_yE_1^\varepsilon
 \right\rangle.
\end{gathered}
$$
For $N_1$ and $N_2$, integrating by parts in $x$, we have
$$
\begin{aligned}
N_1+N_2
&=-\left\langle
 \partial_x\left(
 E^\varepsilon\cdot\nabla\partial_yu_1^a
 \right),
 \partial_x\partial_yE_1^\varepsilon
 \right\rangle\\[0.2cm]
&=\left\langle
 E^\varepsilon\cdot\nabla\partial_yu_1^a,
 \partial_x^2\partial_yE_1^\varepsilon
 \right\rangle\\[0.2cm]
&\le
 \frac1{32}\|\partial_x^2\partial_yE_1^\varepsilon\|^2
 +C\|E^\varepsilon\cdot\nabla\partial_yu_1^a\|^2.
\end{aligned}
$$
By the anisotropic Sobolev inequalities and
$\partial_y^2E_2^\varepsilon
=-\partial_x\partial_yE_1^\varepsilon$, we obtain
$$
\begin{aligned}
\|E^\varepsilon\|_{L^\infty}^2 &\leq C \|E_{1}^{\varepsilon}\|^{2}+\|\partial_{x}E_{1}^{\varepsilon}\|^{2}+\|\partial_{y}E_{1}^{\varepsilon}\|^{2}+\|\partial_{x}\partial_{y}E_{1}^{\varepsilon}\|^{2} \\[0.2cm]
 &\qquad +\|E_{2}^{\varepsilon}\|^{2}+\|\partial_{x}E^{\varepsilon}\|^{2}+\|\partial_{y}^{2}E_{2}^{\varepsilon}\|^{2} \\[0.2cm]
&\le C\bigl(
 \|E^\varepsilon\|^2
 +\|\partial_xE^\varepsilon\|^2 +\|\partial_yE_1^\varepsilon\|^2
 +\|\partial_x\partial_yE_1^\varepsilon\|^2
 \bigr).
\end{aligned}
$$
Moreover, the one-dimensional Sobolev inequality gives
$$
\|E_1^\varepsilon\|_{L_x^2L_y^\infty}^2
\le
2\|E_1^\varepsilon\|
 \|\partial_yE_1^\varepsilon\|.
$$
Using the boundary condition
$E_2^\varepsilon|_{y=0}=0$ and incompressibility, we also have
$$
\begin{aligned}
|E_2^\varepsilon(x,y,t)|^2
=\left|
 \int_0^y\partial_xE_1^\varepsilon(x,r,t)\,dr
 \right|^2 \le
 y\int_0^\infty
 |\partial_xE_1^\varepsilon(x,r,t)|^2\,dr.
\end{aligned}
$$
Consequently, after the change of variables $y=\varepsilon z$,
$$
\begin{aligned}
\frac1{\varepsilon^2}
 \left\|
 E_1^\varepsilon(x,y,t)
 \partial_x\partial_zu_1^b
       \left(x,\frac{y}{\varepsilon},t\right)
 \right\|^2 &\le
 \frac C\varepsilon
 \|E_1^\varepsilon\|_{L_x^2L_y^\infty}^2
 \|\partial_x\partial_zu_1^b\|_{H_x^1L_z^2}^2\\[0.2cm]
&\le
 \frac C\varepsilon
 \|E_1^\varepsilon\|
 \|\partial_yE_1^\varepsilon\|
 \|\partial_x\partial_zu_1^b\|_{H_x^1L_z^2}^2,
\end{aligned}
$$
and
$$
\begin{aligned}
\frac1{\varepsilon^4}
 \left\|
 E_2^\varepsilon(x,y,t)
 \partial_z^2u_1^b
       \left(x,\frac{y}{\varepsilon},t\right)
 \right\|^2 \le
 \frac C{\varepsilon^2}
 \|\partial_xE_1^\varepsilon\|^2
 \|\langle z\rangle\partial_z^2u_1^b\|_{H_x^1L_z^2}^2.
\end{aligned}
$$
Thus, we obtain
\begin{equation}
\label{eq:xyE1-N12}
\begin{aligned}
N_1+N_2
&\le
 \frac1{32}\|\partial_x^2\partial_yE_1^\varepsilon\|^2 +
 C\|\nabla\partial_y(u_1^I+\varepsilon^3S_1)\|^2
 \bigl(
 \|E^\varepsilon\|^2
 +\|\partial_xE^\varepsilon\|^2
 +\|\partial_yE_1^\varepsilon\|^2
 +\|\partial_x\partial_yE_1^\varepsilon\|^2
 \bigr)\\[0.2cm]
&\quad+
 \frac C\varepsilon
 \|E_1^\varepsilon\|
 \|\partial_yE_1^\varepsilon\|
 \|\partial_x\partial_zu_1^b\|_{H_x^1L_z^2}^2 +
 \frac C{\varepsilon^2}
 \|\partial_xE_1^\varepsilon\|^2
 \|\langle z\rangle\partial_z^2u_1^b\|_{H_x^1L_z^2}^2.
\end{aligned}
\end{equation}

For $N_3$, use the decomposition $\pi^\varepsilon=\pi_H+\pi_R+\pi_s$ from the preceding proof. The residual pressure is estimated without differentiating $\mathcal R_2^\varepsilon$ in $x$. Indeed,
$$
\|\partial_x^2\pi_R\|^2\leq C\left(
\|\operatorname{div}\mathcal R^\varepsilon\|^2+
\int_{\mathbb R}|\xi|\left|\widehat{\mathcal R_2^\varepsilon|_{y=0}}(\xi,t)\right|^2\,d\xi\right).
$$
To see this, split the Neumann solution into a zero-Neumann part and a harmonic part. Even extension in $y$ and the multiplier $\xi^2/(\xi^2+\zeta^2)$ control the first part by $\|\operatorname{div}\mathcal R^\varepsilon\|$. For the harmonic part the Fourier formula is
$$
\widehat{\pi} (\xi,y)=\frac{\widehat{\mathcal R_2^\varepsilon|_{y=0}}(\xi)}{|\xi|}e^{-|\xi|y},\qquad
\|\partial_x^2\pi\|^2=\frac12\int_{\mathbb R}|\xi|\left|\widehat{\mathcal R_2^\varepsilon|_{y=0}}(\xi)\right|^2\,d\xi.
$$
The calculation is justified first away from zero frequency and then at the derivative level. For the Stokes pressure, the same Fourier trace calculation as above gives
$$
\|\partial_x^2\pi_s\|^2\leq\varepsilon^4
\|\partial_x^2\partial_yE_1^\varepsilon\|
\|\partial_x\partial_y^2E_1^\varepsilon\|.
$$
Combining these bounds with the estimate of $\pi_H$, we obtain
\begin{equation}
\label{eq:xyE1-N3}
\begin{aligned}
N_3
&\le
\frac14
\|\partial_x^2\partial_yE_1^\varepsilon\|^2
+\frac{25\varepsilon^2}{32}
\|\partial_x\partial_y^2E_1^\varepsilon\|^2 +
\frac C{\varepsilon^2}
 \left(\|\operatorname{div}\mathcal R^\varepsilon\|^2+
\int_{\mathbb R}|\xi|\left|\widehat{\mathcal R_2^\varepsilon|_{y=0}}\right|^2\,d\xi\right)\\[0.2cm]
&\quad+
\frac C{\varepsilon^2}
\left\|
\partial_x\left(
u_1^a\partial_xE^\varepsilon
+E_1^\varepsilon\partial_xu^a
+E_1^\varepsilon\partial_xE^\varepsilon
\right)
\right\|^2.
\end{aligned}
\end{equation}
For $N_4$, Young's inequality yields
$$
N_4
\le
\frac{\varepsilon^2}{32}
\|\partial_x\partial_y^2E_1^\varepsilon\|^2
+\frac C{\varepsilon^2}
\|\partial_x\mathcal R_1^\varepsilon\|^2.
$$
Finally, for $N_5$, we have
$$
\begin{aligned}
N_5
&=
-\int_\Omega
(\partial_xu_1^a)
|\partial_x\partial_yE_1^\varepsilon|^2\,dx\,dy-
\int_\Omega
(\partial_xu_2^a)
(\partial_y^2E_1^\varepsilon)
(\partial_x\partial_yE_1^\varepsilon)\,dx\,dy\\[0.2cm]
&\quad-
\int_\Omega
(\partial_xE_1^\varepsilon)
|\partial_x\partial_yE_1^\varepsilon|^2\,dx\,dy-
\int_\Omega
(\partial_xE_2^\varepsilon)
(\partial_y^2E_1^\varepsilon)
(\partial_x\partial_yE_1^\varepsilon)\,dx\,dy.
\end{aligned}
$$
The first two terms satisfy
$$
\begin{aligned}
&\left|
\int_\Omega
(\partial_xu_1^a)
|\partial_x\partial_yE_1^\varepsilon|^2\,dx\,dy
\right|+
\left|
\int_\Omega
(\partial_xu_2^a)
(\partial_y^2E_1^\varepsilon)
(\partial_x\partial_yE_1^\varepsilon)\,dx\,dy
\right|\\[0.2cm]
&\quad\le
C\left(
\|\partial_xu_1^a\|_{L^\infty}
+\|\partial_xu_2^a\|_{L^\infty}^2
\right)
\|\partial_x\partial_yE_1^\varepsilon\|^2
+C\|\partial_y^2E_1^\varepsilon\|^2.
\end{aligned}
$$
For the third term, using Lemma \ref{lemma:2.2}, we obtain
$$
\begin{aligned}
&\left|
\int_\Omega
(\partial_xE_1^\varepsilon)
|\partial_x\partial_yE_1^\varepsilon|^2\,dx\,dy
\right|\\[0.2cm]
&\quad\le
C\|\partial_xE_1^\varepsilon\|^{1/2}
 \|\partial_x^2E_1^\varepsilon\|^{1/2}
 \|\partial_x\partial_yE_1^\varepsilon\|^{3/2}
 \|\partial_x\partial_y^2E_1^\varepsilon\|^{1/2}\\[0.2cm]
&\quad\le
\frac{\varepsilon^2}{32}
\|\partial_x\partial_y^2E_1^\varepsilon\|^2 +
C\varepsilon^{-2/3}
\|\partial_xE^\varepsilon\|^{2/3}
\|\partial_x^2E^\varepsilon\|^{2/3}
\|\partial_x\partial_yE_1^\varepsilon\|^2.
\end{aligned}
$$
For the last term, using
$\partial_y\partial_xE_2^\varepsilon
=-\partial_x^2E_1^\varepsilon$
and Lemma \ref{lemma:2.2}, we have
$$
\begin{aligned}
&\left|
\int_\Omega
(\partial_xE_2^\varepsilon)
(\partial_y^2E_1^\varepsilon)
(\partial_x\partial_yE_1^\varepsilon)\,dx\,dy
\right|\\[0.2cm]
&\quad\le
C\|\partial_y^2E_1^\varepsilon\|
 \|\partial_xE_2^\varepsilon\|^{1/2}
 \|\partial_x^2E_1^\varepsilon\|^{1/2}
 \|\partial_x\partial_yE_1^\varepsilon\|^{1/2}
 \|\partial_x^2\partial_yE_1^\varepsilon\|^{1/2}\\[0.2cm]
&\quad\le
C\|\partial_y^2E_1^\varepsilon\|^2
+C\|\partial_xE^\varepsilon\|
 \|\partial_x^2E^\varepsilon\|
 \|\partial_x\partial_yE_1^\varepsilon\|
 \|\partial_x^2\partial_yE_1^\varepsilon\|\\[0.2cm]
&\quad\le
\frac1{32}
\|\partial_x^2\partial_yE_1^\varepsilon\|^2
+C\|\partial_y^2E_1^\varepsilon\|^2+
C\|\partial_xE^\varepsilon\|^2
 \|\partial_x^2E^\varepsilon\|^2
 \|\partial_x\partial_yE_1^\varepsilon\|^2.
\end{aligned}
$$
Therefore,
\begin{equation}
\label{eq:xyE1-N5}
\begin{aligned}
N_5
&\le
\frac1{32}
\|\partial_x^2\partial_yE_1^\varepsilon\|^2
+\frac{\varepsilon^2}{32}
\|\partial_x\partial_y^2E_1^\varepsilon\|^2+
C\left(
\|\partial_xu_1^a\|_{L^\infty}
+\|\partial_xu_2^a\|_{L^\infty}^2 \right. \\[0.2cm]
&\quad \left. +\varepsilon^{-2/3}
\|\partial_xE^\varepsilon\|^{2/3}
\|\partial_x^2E^\varepsilon\|^{2/3}+\|\partial_xE^\varepsilon\|^2 \|\partial_x^2E^\varepsilon\|^2
\right)
\|\partial_x\partial_yE_1^\varepsilon\|^2 +C\|\partial_y^2E_1^\varepsilon\|^2 .
\end{aligned}
\end{equation}
Substituting the estimates of $N_1,\ldots,N_5$
into \eqref{eq:xyE1}, we obtain
$$
\begin{aligned}
&\frac{d}{dt}
\|\partial_x\partial_yE_1^\varepsilon\|^2
+\frac14
\|\partial_x^2\partial_yE_1^\varepsilon\|^2
+\frac{\varepsilon^2}{4}
\|\partial_x\partial_y^2E_1^\varepsilon\|^2\\[0.2cm]
&\le
C\left(
 \|\nabla\partial_y(u_1^I+\varepsilon^3S_1)\|^2
 +\|\partial_xu_1^a\|_{L^\infty}
 +\|\partial_xu_2^a\|_{L^\infty}^2+\varepsilon^{-2/3}
  \|\partial_xE^\varepsilon\|^{2/3}
  \|\partial_x^2E^\varepsilon\|^{2/3}\right. \\[0.2cm]
&\left.\quad
 +\|\partial_xE^\varepsilon\|^2
  \|\partial_x^2E^\varepsilon\|^2
 \right)
 \|\partial_x\partial_yE_1^\varepsilon\|^2 +
C\|\nabla\partial_y(u_1^I+\varepsilon^3S_1)\|^2
 \left(
 \|E^\varepsilon\|^2
 +\|\partial_xE^\varepsilon\|^2
 +\|\partial_yE_1^\varepsilon\|^2
 \right)\\[0.2cm]
&\qquad+
\frac C\varepsilon
\|E_1^\varepsilon\|
\|\partial_yE_1^\varepsilon\|
\|\partial_x\partial_zu_1^b\|_{H_x^1L_z^2}^2+
\frac C{\varepsilon^2}
\|\partial_xE_1^\varepsilon\|^2
\|\langle z\rangle\partial_z^2u_1^b\|_{H_x^1L_z^2}^2
+C\|\partial_y^2E_1^\varepsilon\|^2\\[0.2cm]
&\qquad+
\frac C{\varepsilon^2}
\left\|
\partial_x\left(
u_1^a\partial_xE^\varepsilon
+E_1^\varepsilon\partial_xu^a
+E_1^\varepsilon\partial_xE^\varepsilon
\right)
\right\|^2\\[0.2cm]
&\qquad+\frac C{\varepsilon^2}
 \left(\|\partial_x\mathcal R_1^\varepsilon\|^2+
\|\operatorname{div}\mathcal R^\varepsilon\|^2+
\int_{\mathbb R}|\xi|\left|\widehat{\mathcal R_2^\varepsilon|_{y=0}}\right|^2\,d\xi\right).
\end{aligned}
$$
It remains to estimate the product containing the unknown error before applying Gronwall's inequality. For this calculation, put
$$
X(t)=\|\partial_x\partial_yE_1^\varepsilon(t)\|^2,
\qquad H(t)=\|\partial_x^2E^\varepsilon(t)\|^2.
$$
The linear part of the product is
$$
u_1^a\partial_x^2E^\varepsilon
+(\partial_xu_1^a)\partial_xE^\varepsilon
+(\partial_xE_1^\varepsilon)\partial_xu^a
+E_1^\varepsilon\partial_x^2u^a.
$$
For the last of these terms, the one-dimensional Sobolev inequality gives
$$
\|E_1^\varepsilon\partial_x^2u^a\|^2
\leq C\|E_1^\varepsilon\|\|\partial_xE_1^\varepsilon\|
\|\partial_x^2u^a\|_{L_x^2L_y^\infty}^2.
$$
By Lemmas \ref{lemma:4.1}, \ref{lemma:4.3}, and \ref{lemma:4.4}, the time integral of the squared linear part, divided by $\varepsilon^2$, is at most $C_T\varepsilon$. For the two nonlinear terms, we have
$$
\begin{aligned}
\|(\partial_xE_1^\varepsilon)\partial_xE^\varepsilon\|^2
&\leq C\|\partial_xE^\varepsilon\|^2
\|\partial_x^2E^\varepsilon\|
\|\partial_x\partial_yE^\varepsilon\|,\\[0.2cm]
\|E_1^\varepsilon\|_{L^\infty}^2
&\leq C\bigl(\|E_1^\varepsilon\|\|\partial_xE_1^\varepsilon\|
\|\partial_yE_1^\varepsilon\|\|\partial_x\partial_yE_1^\varepsilon\|\bigr)^{1/2}
\leq C_T\varepsilon^{11/4}X(t)^{1/4}.
\end{aligned}
$$
The first estimate and the lower energy bounds imply
$$
\varepsilon^{-2}\int_0^T
\|(\partial_xE_1^\varepsilon)\partial_xE^\varepsilon\|^2\,dt
\leq C_T\varepsilon^3.
$$
The remaining nonlinear term satisfies, without any prior bound on $X$,
$$
\varepsilon^{-2}\|E_1^\varepsilon\partial_x^2E^\varepsilon\|^2
\leq C_T\varepsilon^{3/4}H(t)X(t)^{1/4}
\leq C_T\varepsilon^{3/4}H(t)(1+X(t)).
$$
Since $\int_0^T H(t)\,dt\leq C_T\varepsilon^3$, the additional coefficient of $X$ is integrable uniformly in $\varepsilon$. Moreover, Lemma \ref{prop:residual} gives
$$
\varepsilon^{-2}\int_0^T\left(
\|\partial_x\mathcal R_1^\varepsilon\|^2+
\|\operatorname{div}\mathcal R^\varepsilon\|^2+
\int_{\mathbb R}|\xi|\left|\widehat{\mathcal R_2^\varepsilon|_{y=0}}\right|^2\,d\xi
\right)dt\leq C_T\varepsilon.
$$
All other coefficients of $X$ in the preceding differential inequality have bounded time integrals by Lemmas \ref{lemma:4.1} and \ref{lemma:4.3}--\ref{lemma:4.5}; the remaining source integrals are at most $C_T\varepsilon$. In particular, the terms with $\varepsilon^{-2/3}$ and $\|\partial_xE^\varepsilon\|^2\|\partial_x^2E^\varepsilon\|^2$ have integrals bounded by $C_T\varepsilon^{4/3}$ and $C_T\varepsilon^6$, respectively. Gronwall's inequality on $[t_0,T]$ therefore gives

\begin{equation}
\label{eq:xyE1-positive-time}
\begin{aligned}
\sup_{t_0\le t\le T}
\|\partial_x\partial_yE_1^\varepsilon(t)\|^2
&+\int_{t_0}^T
\|\partial_x^2\partial_yE_1^\varepsilon(t)\|^2\,dt\\[0.2cm]
&+
\varepsilon^2\int_{t_0}^T
\|\partial_x\partial_y^2E_1^\varepsilon(t)\|^2\,dt \le
C_T\left(
\|\partial_x\partial_yE_1^\varepsilon(t_0)\|^2
+\varepsilon
\right),
\end{aligned}
\end{equation}
where $C_T$ is independent of $\varepsilon$ and $t_0$.

We now justify the initial limit for each fixed $\varepsilon>0$. The zero-data heat estimate used in \eqref{eq:5.27} and in Lemma \ref{lemma:2.13} gives initial continuity of the normal derivatives of the homogeneous profiles $f$, $\hat u_1^{b,1}$, and $\hat u_1^{b,2}$. Their wall liftings vanish initially. Consequently,
$$
\|\partial_x\partial_zu_1^{b,j}(t)\|_{L_{xz}^2}\longrightarrow0,
\qquad j=0,1,2,\qquad t\downarrow0.
$$
At the lowest regularity, $j=2$, this uses only
$\partial_x\partial_tu_1^{b,2}$, $\partial_x^3u_1^{b,2}$,
$\partial_x^2\partial_zu_1^{b,2}$, and $\partial_x\partial_z^2u_1^{b,2}$ in $L_t^2L_{xz}^2$, all supplied by Lemma \ref{pro:6.4}. Since
$$
\left\|\partial_x\partial_y\bigl[u_1^{b,j}(x,y/\varepsilon,t)\bigr]\right\|_{L_{xy}^2}
=\varepsilon^{-1/2}\|\partial_x\partial_zu_1^{b,j}(t)\|_{L_{xz}^2},
$$
the outer $H^2$ continuity and
$\partial_x\partial_yS_1=\varphi''(y)\int_0^\infty\partial_xu_1^{b,2}\,dz$ imply
$$
\partial_x\partial_yu_1^a(t)\longrightarrow\partial_x\partial_y\tilde u_1
\quad\hbox{in }L^2,\qquad \varepsilon>0\hbox{ fixed}.
$$
By the fixed-viscosity $H^2$ continuity in Proposition \ref{prop:exact-solution}, it follows that
$$
\|\partial_x\partial_yE_1^\varepsilon(t)\|\longrightarrow0\qquad(t\downarrow0).
$$
No uniform modulus in this last limit is needed. Letting $t_0\downarrow0$ in \eqref{eq:xyE1-positive-time}, whose constants are independent of $t_0$ and $\varepsilon$, we obtain

$$
\begin{aligned}
&\sup_{0\le t\le T}
\|\partial_x\partial_yE_1^\varepsilon(t)\|^2
+\int_0^T
\|\partial_x^2\partial_yE_1^\varepsilon(t)\|^2\,dt +
\varepsilon^2\int_0^T
\|\partial_x\partial_y^2E_1^\varepsilon(t)\|^2\,dt
\le C_T\varepsilon.
\end{aligned}
$$
The proof is complete.
\end{proof}

\subsection{Convergence rate}\label{subsec:Convergence rate}
\begin{lemma}
	\label{lemma:4.7}
	Under the assumptions of Theorem \ref{thm:main}, for any $0<\varepsilon\leq1$, there exists a constant $C$ independent of $\varepsilon$, such that
	$$
	\left\|E_{1}^{\varepsilon} \right\|_{L_{T}^{\infty}L_{xy}^{\infty}}\leq C\varepsilon^{\frac{3}{2}} \quad \text{and} \quad \left\|E_{2}^{\varepsilon} \right\|_{L_{T}^{\infty}L_{xy}^{\infty}}\leq C\varepsilon^{\frac{3}{2}}.
	$$
\end{lemma}
\begin{proof}
	Using the Sobolev inequality, we have
	$$
	\left\|E_{1}^{\varepsilon} \right\|_{L_{T}^{\infty}L_{xy}^{\infty}}\leq C \left\|E_{1}^{\varepsilon} \right\|_{L_{T}^{\infty}L_{x}^{\infty}L_{y}^{2}}^{\frac{1}{2}} \left\|\partial_{y}E_{1}^{\varepsilon} \right\|_{L_{T}^{\infty}L_{x}^{\infty}L_{y}^{2}}^{\frac{1}{2}}.
	$$
	where
	$$
	\left\|E_{1}^{\varepsilon} \right\|_{L_{T}^{\infty}L_{x}^{\infty}L_{y}^{2}} \leq C \left\|E_{1}^{\varepsilon} \right\|_{L_{T}^{\infty}L_{xy}^{2}}^{\frac{1}{2}} \left\|\partial_{x}E_{1}^{\varepsilon} \right\|_{L_{T}^{\infty}L_{xy}^{2}}^{\frac{1}{2}} \leq C\varepsilon^{(\frac{5}{4}+\frac{3}{4})}=C\varepsilon^{2},
	$$
	and
	$$
	\begin{aligned}
		\left\|\partial_{y}E_{1}^{\varepsilon} \right\|_{L_{T}^{\infty}L_{x}^{\infty}L_{y}^{2}} &\leq C \left\|\partial_{y}E_{1}^{\varepsilon} \right\|_{L_{T}^{\infty}L_{xy}^{2}}^{\frac{1}{2}} \left\|\partial_{x}\partial_{y}E_{1}^{\varepsilon} \right\|_{L_{T}^{\infty}L_{xy}^{2}}^{\frac{1}{2}} \\[0.2cm]
		&\leq C \varepsilon^{\frac{3}{4}+\frac{1}{4}}=C\varepsilon.
	\end{aligned}
	$$
	Hence, we get
	$$
	\left\|E_{1}^{\varepsilon} \right\|_{L_{T}^{\infty}L_{xy}^{\infty}}\leq C\left\|E_{1}^{\varepsilon} \right\|_{L_{T}^{\infty}L_{x}^{\infty}L_{y}^{2}}^{\frac{1}{2}} \left\|\partial_{y}E_{1}^{\varepsilon} \right\|_{L_{T}^{\infty}L_{x}^{\infty}L_{y}^{2}}^{\frac{1}{2}} \leq C\varepsilon^{\frac{3}{2}}.
	$$
	Similarly, we have
	$$
    \begin{aligned}
        \|E_{2}^{\varepsilon}\|_{L_{T}^{\infty}L_{xy}^{\infty}}&\leq C\|E_{2}^{\varepsilon}\|_{L_{T}^{\infty} L_{xy}^{2}}^{\frac{1}{4}} \|\partial_{x}E_{2}^{\varepsilon}\|_{L_{T}^{\infty} L_{xy}^{2}}^{\frac{1}{2}} \|\partial_{y}^{2}E_{2}^{\varepsilon}\|_{L_{T}^{\infty} L_{xy}^{2}}^{\frac{1}{4}} \\[0.2cm]
        &\leq C\|E_{2}^{\varepsilon}\|_{L_{T}^{\infty} L_{xy}^{2}}^{\frac{1}{4}} \|\partial_{x}E_{2}^{\varepsilon}\|_{L_{T}^{\infty} L_{xy}^{2}}^{\frac{1}{2}} \|\partial_{x}\partial_{y}E_{1}^{\varepsilon}\|_{L_{T}^{\infty} L_{xy}^{2}}^{\frac{1}{4}} \\[0.2cm]
        &\leq C\varepsilon^{\frac{3}{2}}.
    \end{aligned}
	$$
\end{proof}
\subsection{Proof of Theorem \ref{thm:main}}
\label{r4:error}
Using  Proposition \ref{prop:outer-flow},\ref{pro:boundary-layer equation}, Lemmas \ref{pro:6.1}-\ref{pro:6.4}, \ref{lemma:4.7}, Sobolev inequality and the definition of $E^{\varepsilon}$, we have
\begin{equation}
\label{r4:uniform-error}
	\begin{aligned}
		&\left\|u_{1}(x,y,t)-u_{1}^{I,0}(x,y,t)-u_{1}^{b,0}(x,\frac{y}{\varepsilon},t) \right\|_{L_{T}^{\infty} L_{xy}^{\infty}} \\[0.2cm]
		&\quad\lesssim \varepsilon \left\|u_{1}^{I,1} \right\|_{L_{T}^{\infty} L_{xy}^{\infty}} +\varepsilon^{2} \left\|u_{1}^{I,2} \right\|_{L_{T}^{\infty} L_{xy}^{\infty}}+\varepsilon \left\|u_{1}^{b,1} \right\|_{L_{T}^{\infty} L_{xy}^{\infty}} \\[0.2cm]
		&\qquad +\varepsilon^{2} \left\|u_{1}^{b,2} \right\|_{L_{T}^{\infty} L_{xy}^{\infty}}+\varepsilon^{3} \left\|S_{1} \right\|_{L_{T}^{\infty} L_{xy}^{\infty}}+\left\|E_{1}^{\varepsilon}(x,y,t) \right\|_{L_{T}^{\infty}L_{xy}^{\infty}} \leq C\varepsilon,
	\end{aligned}
\end{equation}
and
\begin{equation}
	\begin{aligned}
		&\left\|u_{2}(x,y,t)-u_{2}^{I,0}(x,y,t) \right\|_{L_{T}^{\infty}L_{xy}^{\infty}} \\[0.2cm]
		&\quad\lesssim \varepsilon \left\|u_{2}^{I,1} \right\|_{L_{T}^{\infty}L_{xy}^{\infty}} +\varepsilon^{2} \left\|u_{2}^{I,2} \right\|_{L_{T}^{\infty}L_{xy}^{\infty}}+\varepsilon \left\|u_{2}^{b,1} \right\|_{L_{T}^{\infty}L_{xy}^{\infty}}+\varepsilon^{2} \left\|u_{2}^{b,2} \right\|_{L_{T}^{\infty}L_{xy}^{\infty}} \\[0.2cm]
		&\qquad+\varepsilon^{3} \left\|u_{2}^{b,3} \right\|_{L_{T}^{\infty}L_{xy}^{\infty}}+\varepsilon^{3} \left\|S_{2} \right\|_{L_{T}^{\infty}L_{xy}^{\infty}}+\left\|E_{2}^{\varepsilon}(x,y,t) \right\|_{L_{T}^{\infty}L_{xy}^{\infty}} \leq C\varepsilon,
	\end{aligned}
\end{equation}
Together with Lemma \ref{lemma:4.7}, this completes the proof of Theorem \ref{thm:main}.

\begin{corollary}
\label{cor:initial-error-modulus}
Under the assumptions of Theorem \ref{thm:main}, fix \(T>0\).
For \(0<\varepsilon\leq1\),
the actual error \(E^\varepsilon=u^\varepsilon-u^a\) satisfies
\begin{equation}
 \|E^\varepsilon(t)\|_{L^\infty}
 \leq C_T\varepsilon^{3/2}t^{1/8},
 \qquad 0\leq t\leq T.
 \label{eq:initial-error-infty}
\end{equation}
More generally, for \(2\leq p\leq\infty\), with the convention
\(1/\infty=0\), we have
\begin{equation}
 \|E^\varepsilon(t)\|_{L^p(\Omega)}
 \leq C_{T,p}\varepsilon^{3/2+2/p}t^{1/8+3/(4p)},
 \qquad 0\leq t\leq T.
 \label{eq:initial-error-p}
\end{equation}
In particular, the full velocity error converges strongly to zero as
\(t\to0^+\). No claim is made here that the temporal exponent \(1/8\)
is optimal.
\end{corollary}

\begin{proof}
We retain the initial \(L^2\)-bound already transferred to the actual
error and combine it with the corresponding componentwise embeddings.
In particular,
\[
 \|E^\varepsilon(t)\|_2
 \leq C_T\varepsilon^{5/2}t^{1/2},
\]
while the four-level energy estimates yield
\[
 \|E_x^\varepsilon(t)\|_2\leq C_T\varepsilon^{3/2},
 \qquad
 \|E_{1y}^\varepsilon(t)\|_2\leq C_T\varepsilon^{3/2},
 \qquad
 \|E_{1xy}^\varepsilon(t)\|_2\leq C_T\varepsilon^{1/2}.
\]
Hence, applying the four-factor anisotropic embedding to the tangential
component gives
\[
 \|E_1^\varepsilon(t)\|_{L^\infty}
 \leq C_T
 \bigl(
   \varepsilon^{5/2}t^{1/2}
   \varepsilon^{3/2}
   \varepsilon^{3/2}
   \varepsilon^{1/2}
 \bigr)^{1/4}
 =C_T\varepsilon^{3/2}t^{1/8}.
\]
For the normal component, using
\[
 E_{2yy}^\varepsilon=-E_{1xy}^\varepsilon,
 \qquad
 E_2^\varepsilon|_{y=0}=E_{2y}^\varepsilon|_{y=0}=0,
\]
we obtain
\[
 \|E_2^\varepsilon(t)\|_{L^\infty}
 \leq
 C\|E_2^\varepsilon(t)\|_2^{1/4}
  \|E_{2x}^\varepsilon(t)\|_2^{1/2}
  \|E_{2yy}^\varepsilon(t)\|_2^{1/4}
 \leq C_T\varepsilon^{3/2}t^{1/8}.
\]
This proves \eqref{eq:initial-error-infty}.
The above spatial inequalities hold initially for almost every time.
By choosing the corresponding time-continuous representatives and using
lower semicontinuity, the estimates extend to every \(t>0\).
Finally, interpolation between \(L^2\) and \(L^\infty\) gives
\[
 \|E^\varepsilon(t)\|_{L^p}
 \leq
 \|E^\varepsilon(t)\|_2^{2/p}
 \|E^\varepsilon(t)\|_{L^\infty}^{1-2/p},
\]
and therefore
\[
 \|E^\varepsilon(t)\|_{L^p(\Omega)}
 \leq
 C_{T,p}\varepsilon^{3/2+2/p}
 t^{1/8+3/(4p)},
\]
which proves \cref{eq:initial-error-p}.
The right-hand side tends to zero as $t\to0^+$, consistently with $E^\varepsilon(0)=0$. This interpolation uses the error estimates already proved, including their initial traces, and requires no additional compatibility condition.
\end{proof}

\section{Tail Estimates, Corner Effects, and Optimality of the Convergence Rate}
\label{sec:thickness-sharpness}

The three results in this section concern three distinct limiting regimes:
\[
\begin{array}{lll}
\text{Fixed time interval:}
& \varepsilon\downarrow0,
& t\in[0,T],\quad T>0\ \text{fixed},\\[0.2cm]
\text{Corner self-similar scale:}
& t\downarrow0,
& z/\sqrt{t}\ \text{fixed},\\[0.2cm]
\text{Sharp convergence rate:}
& \varepsilon\downarrow0,
& t=t_*>0\ \text{fixed}.
\end{array}
\]
Distinguishing these limiting regimes explicitly prevents confusion among the geometric thickness of the boundary layer, the magnitude of the error, and the temporal scale associated with the corner structure.

\subsection{Proof of Theorem \ref{thm:thickness}}
\begin{proof}
Theorem \ref{thm:main}, together with the rescaling \(z=y/\varepsilon\), immediately yields \eqref{eq:interval-tail}. It therefore remains only to establish a uniform-in-time modulus of continuity in the normal variable and the corresponding decay of the boundary-layer profile.
The lower-order shear and horizontal estimates imply that, for every \(h\geq0\),
\[
 \sup_{x,t}|u^{b,0}_1(x,z+h,t)-u^{b,0}_1(x,z,t)|
 \le h^{1/2}\sup_t\|u^{b,0}_{1z}(t)\|_{L_x^\infty L_z^2}
 \le C_T h^{1/2}.
\]
Hence,
\[
 0\le F_T(\rho)-F_T(\rho+h)
 \le C_T h^{1/2}.
\]
Moreover, the weighted estimates give
\begin{equation}
 \sup_{0\le t\le T}\|\langle z\rangle u^{b,0}_1(t)\|_\infty<\infty,
 \qquad
 F_T(\rho)\le \frac{C_T}{\langle\rho\rangle}.
 \label{eq:B-tail-weight}
\end{equation}
Thus, the tail function on the time interval is continuous and tends to zero as \(\rho\to\infty\). In particular,
\begin{equation}
 \|u^\varepsilon-u^{I,0}\|_{L^\infty([0,T]\times\Omega_{\delta_\varepsilon})}
 \le C_T\left(\varepsilon+\frac1{\langle\delta_\varepsilon/\varepsilon\rangle}\right).
 \label{eq:thickness-upper}
\end{equation}
Applying the same argument at a fixed time yields \cref{eq:tail-limit}. If \(u^{b,0}_1\) is not identically zero on the time interval under consideration, then \(F_T(0)>0\). By continuity, this rules out the possibility that the defect tends to zero when the layer thickness is \(o(\varepsilon)\).
\end{proof}
\begin{corollary}
\label{cor:uncorrected-limits}
Under the assumptions of Theorem \ref{thm:main}, fix \(T>0\). Then
\begin{equation}
 \lim_{\varepsilon\downarrow0}
 \left\lVert u^\varepsilon-u^{I,0}\right\rVert_{{L^\infty([0,T]\times\Omega)}}
 =
 \left\lVert u^{b,0}_1\right\rVert_{{L^\infty([0,T]\times\Omega_{\mathrm b})}}.
\label{eq:uncorrected-Linf}
\end{equation}
In particular, for every fixed \(t_0\in[0,T]\),
\[
 \lim_{\varepsilon\downarrow0}
 \left\lVert u^\varepsilon(t_0)-u^{I,0}(t_0)\right\rVert_{L^\infty(\Omega)}
 =
 \left\lVert u^{b,0}_1(t_0)\right\rVert_{L^\infty(\Omega_{\mathrm b})}.
\]
\end{corollary}

\subsection{Proof of Theorem \ref{thm:corner}}
\label{sec:corner-nonlinear-remainder}
We first collect the lower-order short-time estimates established in \eqref{eq:5.13}--\eqref{eq:5.14}, and supplement them with bounds for the normal factors and the underlying functions. These estimates will serve as common inputs for the two-order corner expansion developed below.
\begin{lemma}
\label{lem:nonlinear-corner-scale}
For every integer \(\ell\geq0\), there exists
\(0<t_c\leq\min\{1,T\}\) such that, for all
\(0\leq t\leq t_c\),
\[
\begin{gathered}
    \left\|\langle z\rangle^{\ell} u_1^{b, 0}(t)\right\|_{H_x^2 L_z^2}
    \leq C_{\ell} t^{5/4}, \qquad
    \left\|\langle z\rangle^{\ell} \partial_z u_1^{b, 0}(t)\right\|_{H_x^1 L_z^2}
    \leq C_{\ell} t^{3/4},\\[0.2cm]
    \left\|\langle z\rangle^{\ell} z \partial_z u_1^{b, 0}(t)\right\|_{H_x^1 L_z^2}
    \leq C_{\ell} t^{5/4},\\[0.2cm]
    \sum_{j=0}^1
    \left\|
    \sup_{z\geq0}
    \left|
    \partial_x^j
    \int_0^z
    \partial_x u_1^{b,0}(x,r,t)\,dr
    \right|
    \right\|_{L_x^2}
    \leq C_{\ell} t^{3/2}.
\end{gathered}
\]
\end{lemma}
\begin{proof}
The heat-lifting estimate \eqref{eq:5.7} gives
\[
 \|\langle z\rangle^\ell b_a(t)\|_{H_x^2L_z^2}
 \leq C_\ell t^{5/4},
 \quad
 \|\langle z\rangle^\ell\partial_z b_a(t)\|_{H_x^1L_z^2}
 \leq C_\ell t^{3/4},\quad
 \|\langle z\rangle^\ell z\partial_z b_a(t)\|_{H_x^1L_z^2}
 \leq C_\ell t^{5/4}.
\]
By \eqref{eq:5.14}, the nonlinear remainder satisfies
\begin{equation}
 \|\langle z\rangle^\ell r_a(t)\|_{H_x^2L_z^2}
 \leq C_\ell t^{13/4},
 \qquad
 \|\langle z\rangle^\ell \partial_z r_a(t)\|_{H_x^1L_z^2}
 \leq C_\ell t^{11/4}.
 \label{eq:nonlinear-remainder-corner-bounds}
\end{equation}
Moreover,
\[
 \sum_{j=0}^1
 \left\|
 \sup_{z\geq0}
 \left|
 \int_0^z\partial_x^{j+1}r_a\,ds
 \right|
 \right\|_{L_x^2}
 \leq
 C_m\|\langle z\rangle^m r_a\|_{H_x^2L_z^2}
 \leq C_m t^{13/4}.
\]
Since
\[
 u^{b,0}_1=b_a+r_a,
\]
the stated estimates follow.
\end{proof}

Using the above lemma, we now prove Theorem \ref{thm:corner}.
\begin{proof}
Set$c=a_t-a_{xx}$. The boundary equation implies that
\[
c\in C H_x^3,\qquad
\partial_t c\in C H_x^1,
\qquad
c(x,0)=\alpha_0.
\]
The heat-lifting term corresponding to the boundary value \(-a\) is given by
\begin{equation}
 b_a(x,z,t)
 =-\int_0^t e^{h\partial_x^2}c(x,t-h)
       \operatorname{erfc}\!\left(\frac{z}{2\sqrt{h}}\right)\,dh .
 \label{r5:corner-heat-history}
\end{equation}
For \(\theta\in(0,1)\), define $Q_\theta(\eta)=\operatorname{erfc}\!\left(\frac{\eta}{2\sqrt{\theta}}\right)$. Making the change of variables \(h=t\theta\), we obtain
\[
 t^{-1}b_a(x,\sqrt{t}\,\eta,t)
 =-\int_0^1
 Q_\theta(\eta)
 e^{t\theta\partial_x^2}
 c\bigl(x,t(1-\theta)\bigr)\,d\theta.
\]
By the contractivity of the horizontal heat semigroup and the standard
heat-semigroup difference estimate,
\begin{align*}
 \left\|
 e^{t\theta\partial_x^2}c\bigl(x,t(1-\theta)\bigr)
 -\alpha_0
 \right\|_{H_x^1}
 &\leq
 \left\|
 c\bigl(x,t(1-\theta)\bigr)-c(x,0)
 \right\|_{H_x^1}
 +
 \left\|
 (e^{t\theta\partial_x^2}-1)\alpha_0
 \right\|_{H_x^1}
 \\[0.2cm]
 &\leq
 Ct(1-\theta)
 +Ct\theta\|\alpha_0\|_{H_x^3}
 \leq Ct.
\end{align*}
A direct calculation gives
\[
 \|Q_\theta\|_{L_\eta^2}
 =C\theta^{1/4},
 \qquad
 \|\partial_\eta Q_\theta\|_{L_\eta^2}
 =C\theta^{-1/4}.
\]
Both powers are integrable over \((0,1)\). Hence, by Minkowski's inequality,
\[
 \left\|
 t^{-1}b_a(x,\sqrt{t}\,\eta,t)
 +\alpha_0\Phi_1(\eta)
 \right\|_{H_\eta^1(H_x^1)}
 \leq Ct,
 \qquad
 \Phi_1(\eta)=\int_0^1Q_\theta(\eta)\,d\theta.
\]
The same estimate remains valid after inserting any fixed polynomial
weight in \(\eta\). The explicit formula for \(\Phi_1\) is given in
\eqref{eq:Phi1-explicit}.
The lower-order construction yields
\eqref{eq:nonlinear-remainder-corner-bounds}. Therefore,
\[
\begin{gathered}
 \left\|
 t^{-1}r_a(x,\sqrt{t}\,\eta,t)
 \right\|_{L_\eta^2(H_x^1)}
 =
 t^{-5/4}\|r_a(t)\|_{H_x^1L_z^2}
 \leq Ct^2,
 \\[0.2cm]
 \left\|
 \partial_\eta
 \bigl[t^{-1}r_a(x,\sqrt{t}\,\eta,t)\bigr]
 \right\|_{L_\eta^2(H_x^1)}
 =
 t^{-3/4}\|(r_a)_z(t)\|_{H_x^1L_z^2}
 \leq Ct^2.
\end{gathered}
\]
Combining these two estimates proves \cref{eq:corner-limit}. The
\(L^\infty\)-version \cref{eq:corner-limit-Linf} then follows from the
Sobolev embedding theorem. Finally, the resulting remainder also satisfies, in the physical boundary-layer variables,
\begin{equation}
 \|\langle z\rangle^\ell(u^{b,0}_1+t\alpha_0\Phi_1\!\left(\eta \right))(t)\|_{H_x^1L_z^2}
 +\sqrt{t}\,
 \|\langle z\rangle^\ell\partial_{z}(u^{b,0}_1+t\alpha_0\Phi_1\!\left(\eta \right))(t)\|_{H_x^1L_z^2}
 \leq C_\ell t^{9/4}.
 \label{eq:Rc-terminal-H1}
\end{equation}
This completes the proof.
\end{proof}

We next prove Corollary \ref{cor:second-corner}.
\begin{proof}
Using the identity
\[
 e^{h\partial_x^2}\alpha_0-\alpha_0-h\partial_x^2\alpha_0
 =\int_0^h
 (e^{s\partial_x^2}-\mathrm{Id})\partial_x^2\alpha_0\,ds,
\]
we obtain, uniformly for \(0\leq\theta\leq1\),
\begin{equation}
 e^{t\theta\partial_x^2}c(x,t(1-\theta))
 =
 \alpha_0
 +t\Bigl(
 (1-\theta)\partial_t c(x,0)
 +\theta\partial_x^2\alpha_0
 \Bigr)
 +o_{H_x^1}(t).
 \label{r6:heat-second-taylor}
\end{equation}
Applying Minkowski's integral inequality, we may pass the expansion
\eqref{r6:heat-second-taylor} to the space \(H_\eta^1(H_x^1)\). The two
coefficient profiles are given by
\[
 \int_0^1(1-\theta)Q_\theta\,d\theta
 =\frac12\Phi_2,
 \qquad
 \int_0^1\theta Q_\theta\,d\theta
 =\Phi_1-\frac12\Phi_2.
\]
Substituting these identities into the representation of \(b_a\) yields
the expansion \eqref{r6:second-corner-expansion} for the heat-lifting term. For the nonlinear remainder constructed above, we have
\begin{align*}
 \|t^{-2}r_a(x,\sqrt{t}\eta,t)\|_{L_\eta^2(H_x^1)}
 &\leq C t^{-9/4}t^{13/4}
 =Ct,\\[0.2cm]
 \|\partial_\eta[t^{-2}r_a(x,\sqrt{t}\eta,t)]\|_{L_\eta^2(H_x^1)}
 &\leq C t^{-7/4}t^{11/4}
 =Ct.
\end{align*}
If $\Psi_2=0$, its trace at $\eta=0$ gives $\beta_0=0$. Since $\Phi_1-\Phi_2$ is not identically zero, the formula for $\Psi_2$ then gives $\alpha_{0,xx}=0$. An affine function in $L^2(\mathbb R)$ must vanish, contradicting $\alpha_0\not\equiv0$. Hence $\Psi_2\neq0$ in that case. Finally, $H_\eta^1(H_x^1)\hookrightarrow L^\infty_{x,\eta}$ and continuity of the norm give both limits in \eqref{r6:sharp-corner-remainder}. This completes the proof.
\end{proof}

\begin{remark}
The second-order profile also admits the explicit representation
\[
 \Phi_2(\eta)
 =
 \left(1+\eta^2+\frac{\eta^4}{12}\right)\operatorname{erfc}(\eta/2)
 -\frac{1}{\sqrt{\pi}}
 \left(\frac{5\eta}{3}+\frac{\eta^3}{6}\right)
 e^{-\eta^2/4}.
\]
By \cref{eq:3.3}, the corresponding leading-order term of the normal
profile is given by
\[
 -t^{3/2}\partial_x\alpha_0(x)
 \int_{z/\sqrt{t}}^\infty \Phi_1(\sigma)\,\,\mathrm{d}\sigma.
\]
In particular, its boundary value satisfies
\begin{equation}
 u^{b,1}_2(x,0,t)
 =
 -\frac{4}{3\sqrt{\pi}}t^{3/2}\partial_x\alpha_0
 +O_{L_x^2}(t^{5/2}),
 \qquad t\downarrow0.
\label{eq:N1-wall-asymptotic}
\end{equation}
\end{remark}

\subsection{Size of the Corner Layer in Physical Variables}
In the physical variables, define
\[
 K_{\mathrm c}^\varepsilon(x,y,t)
 =-t\alpha_0(x)\Phi_1\!\left(\frac{y}{\varepsilon\sqrt t}\right).
\]
For \(2\leq p<\infty\), the change of variables
\(y=\varepsilon\sqrt t\,\eta\) yields the exact identities
\begin{align}
 \left\lVert K_{\mathrm c}^\varepsilon(t)\right\rVert_{L^p_{x,y}}
 &=\varepsilon^{1/p}t^{1+1/(2p)}
   \left\lVert \alpha_0\Phi_1\right\rVert_{L^p_{x,\eta}},
\label{eq:corner-physical-Lp}\\
 \left\lVert \partial_yK_{\mathrm c}^\varepsilon(t)\right\rVert_{L^p_{x,y}}
 &=\varepsilon^{-1+1/p}t^{1/2+1/(2p)}
   \left\lVert \alpha_0\Phi_1'\right\rVert_{L^p_{x,\eta}}.
\label{eq:corner-physical-gradient}
\end{align}
Moreover, for every \(c\geq0\),
\begin{equation}
 \left\lVert K_{\mathrm c}^\varepsilon(t)\right\rVert_{{L^\infty(\mathbb{R}\times[c\varepsilon\sqrt t,\infty))}}
 =
 t\left\lVert \alpha_0\Phi_1\right\rVert_{{L^\infty(\mathbb{R}\times[c,\infty))}}.
\label{eq:corner-physical-tail}
\end{equation}
Together with \cref{eq:corner-limit-Linf}, these identities show that,
provided \(\alpha_0\neq0\), the scale \(O(\varepsilon\sqrt t)\) is the sharp
matching thickness of the leading-order corner profile, rather than
merely an upper bound for its thickness. The next result identifies the
same profile in the simultaneous limiting regime for the exact tangential
defect.

For \(0<t\leq t_c\), let
\[
 L=-t\alpha_0\Phi_1(z/\sqrt t).
\]
Then, for \(0\leq k\leq3\) and \(j\geq0\), a direct change of variables gives
\[
 \|\partial_x^k\partial_z^jL\|_{L^2_{x,z}}
 =
 t^{5/4-j/2}
 \|\partial_x^k\alpha_0\|_{L_x^2}
 \|\Phi_1^{(j)}\|_{L_\eta^2}.
\]
After returning to the physical variables, this norm acquires the
additional factor \(\varepsilon^{1/2-j}\).
For \(j=0,1,2,3,4\), the corresponding powers of \(t\) are, respectively, $\frac54,\frac34,\frac14,-\frac14,-\frac34$. Consequently, for a nontrivial leading-order profile, the third normal
derivative is square-integrable in time, whereas the fourth normal
derivative is, in general, not square-integrable near \(t=0\).
This calculation concerns only the explicit profile \(L\); derivatives
of the nonlinear profile remain controlled by the finite-order estimates established above.

\begin{corollary}
\label{cor:exact-corner-coupled}
Fix initial data satisfying the assumptions of Theorem \ref{thm:main}, and set
\[
 d^\varepsilon(x,y,t)=u_1^\varepsilon(x,y,t)-u_1^{I,0}(x,y,t).
\]
For $\eta\geq0$ and $0<t\leq t_c$, the estimate
\begin{equation}
\label{eq:exact-corner-coupled-bound}
\|t^{-1}d^\varepsilon(x,\varepsilon\sqrt t\,\eta,t)+\alpha_0\Phi_1\|_{L^\infty_{x,\eta}}
\leq
C_T t
+\varepsilon t^{-3/4}\omega(t)
+C_T\varepsilon^{3/2}t^{-1},
\qquad 0<t\leq t_c,
\end{equation}
holds, where
\[
\omega(t)\to0
\qquad\text{as }t\downarrow0
\]
is the modulus associated with the fixed initial data in
\cref{r6:first-correction-onset}.

For each fixed choice of initial data, the exact corner limit follows
whenever
\begin{equation}
 t_\varepsilon>0,\qquad
 t_\varepsilon\to0,\qquad
 \liminf_{\varepsilon\downarrow0}
 \frac{t_\varepsilon}{\varepsilon^{4/3}}>0.
\label{eq:corner-coupling}
\end{equation}
In particular, one may take
\[
t_\varepsilon=\varepsilon^\gamma,
\qquad 0<\gamma\leq\frac43.
\]
These conditions provide sufficient coupling regimes. At the endpoint
\(\gamma=4/3\), the argument uses the modulus of continuity associated
with the fixed initial data; no uniformity of this modulus over the
entire \(H^4\)-ball of initial data is asserted. Here, the two parameters \(\varepsilon\) and \(t_\varepsilon\) tend to
zero simultaneously subject to the ratio condition
\eqref{eq:corner-coupling}.
\end{corollary}

\begin{proof}
Using the zero initial data of the correction terms together with the
previously established estimates for their time derivatives, we obtain
$$
\begin{aligned}
    \|u^{I,1}(t)\|_\infty \le C\sqrt t\,\|\partial_{t}u^{I,1}\|_{L^2(0,t;H^2)}, \quad \|u^{b,1}_1(t)\|_\infty^2 \le C\|u^{b,1}_1(t)\|_{H_x^1L_z^2}\|\partial_{z}u_{1}^{b,1}(t)\|_{H_x^1L_z^2} \\[0.2cm]
    \|u^{b,1}_1(t)\|_\infty \le C_T t^{1/4}\left(\int_0^t\|\partial_{t}u_{1}^{b,1}(s)\|_{H_x^1L_z^2}^2\,ds\right)^{1/4}, \quad \|u^{b,1}_2(t)\|_\infty\le C\|\langle z\rangle^2u^{b,0}_1(t)\|_{H_x^2L_z^2} \le C_Tt^{5/4}.
\end{aligned}
$$
Consequently,
\begin{equation}
 \|u^{I,1}(t)\|_\infty
 +\|u^{b,1}_1(t)\|_\infty
 +\|u^{b,1}_2(t)\|_\infty
 =
 t^{1/4}\omega(t),
 \qquad
 \omega(t)\to0
 \quad\text{as }t\downarrow0.
 \label{r6:first-correction-onset}
\end{equation}
Using the uniform bounds for the higher-order profiles together with
\cref{r4:uniform-error}, we further obtain
\[
 \|u^\varepsilon-u^{I,0}-((u^{b,0}_1)^\varepsilon,0)\|_\infty
 \leq
 \varepsilon t^{1/4}\omega(t)+C_T\varepsilon^{3/2}.
\]
If $t_\varepsilon\geq c\varepsilon^{4/3} $ for some \(c>0\), then
\[
 \varepsilon t_\varepsilon^{-3/4}\omega(t_\varepsilon)
 \leq
 c^{-3/4}\omega(t_\varepsilon)
 \longrightarrow0,
\]
and
\[
 \frac{\varepsilon^{3/2}}{t_\varepsilon}
 \leq
 c^{-1}\varepsilon^{1/6}
 \longrightarrow0.
\]
Combining these estimates with \(t_\varepsilon\to0\) proves both the
endpoint case and the general coupled-limit assertion.
\end{proof}

\begin{corollary}
\label{cor:r7-exact-second-corner}
Fix initial data satisfying the assumptions of Theorem \ref{thm:main}, and use
$d^\varepsilon$ as defined in Corollary~\ref{cor:exact-corner-coupled}.
For \(0<t\leq t_c\), the quantitative remainder estimate
\[
 \left\|
 t^{-2}\bigl[d^\varepsilon(x,\varepsilon\sqrt t\,\eta,t)
 +t\alpha_0\Phi_1\bigr]+\Psi_2
 \right\|_{L^\infty_{x,\eta}}
 \leq
 \rho(t)
 +\varepsilon t^{-7/4}\omega(t)
 +C_T\varepsilon^{3/2}t^{-2},
\]
holds, where
\[
 \rho(t)\to0,
 \qquad
 \omega(t)\to0
 \qquad\text{as }t\downarrow0.
\]
Here, \(\rho\) arises from the second-order expansion of the corner profile,
whereas \(\omega\) is associated with the first-order correction for the fixed initial data.

If
\begin{equation}
 t_\varepsilon>0,\qquad
 t_\varepsilon\to0,\qquad
 \liminf_{\varepsilon\downarrow0}
 \frac{t_\varepsilon}{\varepsilon^{4/7}}>0,
 \label{r7:second-corner-coupling}
\end{equation}
then
\begin{equation}
 \left\|
 t_\varepsilon^{-2}
 \bigl[
 d^\varepsilon(x,\varepsilon\sqrt{t_\varepsilon}\,\eta,t_\varepsilon)
 +t_\varepsilon\alpha_0\Phi_1
 \bigr]
 +\Psi_2
 \right\|_{L^\infty_{x,\eta}}
 \longrightarrow0
 \qquad\text{as }\varepsilon\downarrow0.
 \label{r7:exact-second-corner}
\end{equation}
In particular, one may take $ t_\varepsilon=\varepsilon^\gamma,\quad 0<\gamma\leq\frac47 $.
\end{corollary}

\begin{proof}
By \eqref{r6:second-corner-expansion}, the profile remainder, after division by \(t^2\), admits an \(L^\infty\)-modulus \(\rho(t)\to0\) as \(t\downarrow0\). Retaining the modulus \(\omega(t)\to0\) from \cref{r6:first-correction-onset} in the full error expansion, we obtain
\[
 \left\|
 t^{-2}\bigl[d^\varepsilon(x,\varepsilon\sqrt t\,\eta,t)+t\alpha_0\Phi_1\bigr]
 +\Psi_2
 \right\|_\infty
 \leq
 \rho(t)
 +\varepsilon t^{-7/4}\omega(t)
 +C_T\varepsilon^{3/2}t^{-2}.
\]
If \(t=t_\varepsilon\geq c\varepsilon^{4/7}\), then the right-hand side is bounded by
\[
 \rho(t)
 +c^{-7/4}\omega(t)
 +C_Tc^{-2}\varepsilon^{5/14}
 \longrightarrow0.
\]
Thus, the coefficient \(\Psi_2\) is transferred to the exact solution without introducing any additional profile terms or imposing stronger assumptions on the initial data.
\end{proof}

\begin{corollary}
\label{cor:r8-corner-tail}
Fix initial data satisfying the assumptions of Theorem \ref{thm:main}, and let
\(t_\varepsilon\) satisfy \cref{eq:corner-coupling}. Set
\[
 r_\varepsilon=\frac{\delta_\varepsilon}{\varepsilon\sqrt{t_\varepsilon}},
 \qquad
 \delta_\varepsilon\geq0.
\]
Then, for any such sequence of widths \(\{\delta_\varepsilon\}\),
\begin{equation}
 \left|
 \frac{
 \|d^\varepsilon(t_\varepsilon)\|_{L^\infty(\mathbb{R}\times[\delta_\varepsilon,\infty))}
 }{t_\varepsilon}
 -
 \|\alpha_0\|_\infty\Phi_1(r_\varepsilon)
 \right|
 \longrightarrow0
 \qquad\text{as }\varepsilon\downarrow0.
 \label{r8:normalized-corner-tail}
\end{equation}
Consequently, provided that \(\alpha_0\not\equiv0\), the normalized exterior
defect converges to zero if and only if \(r_\varepsilon\to\infty\).
If \(r_\varepsilon\to\ell<\infty\), then
\[
 \frac{
 \|d^\varepsilon(t_\varepsilon)\|_{L^\infty(\mathbb{R}\times[\delta_\varepsilon,\infty))}
 }{t_\varepsilon}
 \longrightarrow
 \|\alpha_0\|_\infty\Phi_1(\ell)>0.
\]
If \(\delta_\varepsilon\) is to represent a vanishing physical thickness, one must
in addition require \(\delta_\varepsilon\to0\). We emphasize that the quantity
considered here is the defect normalized by the vanishing amplitude
\(t_\varepsilon\).
\end{corollary}

\begin{proof}
Restrict the uniform approximation in Corollary~\ref{cor:exact-corner-coupled} to the parameter-dependent region \(\eta\geq r_\varepsilon\). Then
\[
 \left|
 t_\varepsilon^{-1}
 \|d^\varepsilon(t_\varepsilon)\|_{L^\infty(\mathbb{R}\times[\delta_\varepsilon,\infty))}
 -
 \|\alpha_0\Phi_1\|_{L^\infty(\mathbb{R}\times[r_\varepsilon,\infty))}
 \right|
 \leq
 \|t_\varepsilon^{-1}d^\varepsilon(x,\varepsilon\sqrt{t_\varepsilon}\,\eta,t_\varepsilon)+\alpha_0\Phi_1\|_{L^\infty_{x,\eta}}
 \longrightarrow0.
\]
Moreover,
\[
 \Phi_1(\eta)
 =
 \int_0^1
 \operatorname{erfc}\!\left(\frac{\eta}{2\sqrt{\theta}}\right)\,d\theta.
\]
Since \(\Phi_1\) is positive and decreasing on \([0,\infty)\), it follows that
\[
 \|\alpha_0\Phi_1\|_{L^\infty(\mathbb{R}\times[r_\varepsilon,\infty))}
 =
 \|\alpha_0\|_\infty\Phi_1(r_\varepsilon).
\]
This proves \eqref{r8:normalized-corner-tail} and, in particular, the
sufficiency of \(r_\varepsilon\to\infty\).

Conversely, if \(r_\varepsilon\) does not tend to infinity, then there exists a
subsequence, still denoted by \(r_\varepsilon\), such that \(r_\varepsilon\leq R\) for
some finite \(R>0\). Along this subsequence,
\[
 \|\alpha_0\|_\infty\Phi_1(r_\varepsilon)
 \geq
 \|\alpha_0\|_\infty\Phi_1(R)>0,
\]
provided that \(\alpha_0\not\equiv0\). Hence the normalized exterior defect
cannot converge to zero, which proves the necessity. Finally, if \(r_\varepsilon\to\ell<\infty\), the stated limit follows directly
from the continuity of \(\Phi_1\).
\end{proof}

\begin{remark}
For \(0<\varepsilon\leq1\) and \(0<t\leq\min\{t_c,1\}\),
\eqref{eq:initial-error-infty} and \eqref{r6:first-correction-onset} yield
\[
 \|u^\varepsilon-u^{I,0}-((u^{b,0}_1)^\varepsilon,0)\|_\infty
 \leq
 \varepsilon t^{1/4}\omega(t)
 +C_T\varepsilon^2
 +C_T\varepsilon^{3/2}t^{1/8}.
\]
Here, the term \(C_T\varepsilon^2\) accounts for the second- and higher-order
profiles that are not absorbed into the exact error, without assigning
to these profiles any additional time-decay rates that have not been
established. Consequently, the first- and second-order normalized error
estimates can be refined to
$$
\begin{gathered}
    \|t^{-1}d^\varepsilon(x,\varepsilon\sqrt t\,\eta,t)+\alpha_0\Phi_1\|_{L^\infty_{x,\eta}}\leq C_Tt +\varepsilon t^{-3/4}\omega(t) +C_T\varepsilon^2t^{-1} +C_T\varepsilon^{3/2}t^{-7/8},\\[0.2cm]
    \|t^{-2}[d^\varepsilon(x,\varepsilon\sqrt t\,\eta,t)+t\alpha_0\Phi_1]+\Psi_2\|_\infty \leq \rho(t) +\varepsilon t^{-7/4}\omega(t) +C_T\varepsilon^2t^{-2} +C_T\varepsilon^{3/2}t^{-15/8}.
\end{gathered}
$$
When \(t\geq c\varepsilon^{4/3}\), the last two terms in the first estimate are
of order \(O(\varepsilon^{2/3})\) and \(O(\varepsilon^{1/3})\), respectively. Likewise,
when \(t\geq c\varepsilon^{4/7}\), the corresponding terms in the second estimate
are of order \(O(\varepsilon^{6/7})\) and \(O(\varepsilon^{3/7})\), respectively.
The contribution involving \(\omega(t)\) from the first-order correction
remains present. Hence, these improved remainder bounds do not enlarge
the sufficient coupling regimes given in
\cref{eq:corner-coupling,r7:second-corner-coupling}.
\end{remark}

\subsection{Explicit Initial Data  }

\begin{proposition}[A Family of Gaussian Initial Data]
\label{prop:Gaussian-data}
Fix \(\kappa>0\), choose
\(0\neq\beta\in C_c^\infty((0,\infty))\), and define
\begin{equation}
 \psi(x,y)=\mathrm{e}^{-\kappa x^2}\beta(y),
 \qquad
 \tilde u=(\partial_y\psi,-\partial_x\psi).
\label{eq:Gaussian-data}
\end{equation}
Then \(\tilde u\) is smooth and divergence-free, and satisfies the no-slip
boundary condition. Moreover, the corresponding coefficient
\[
 \alpha_0=-\partial_xp^{I,0}(x,0,0)
\]
is not identically zero.
\end{proposition}

\begin{proof}
Let $\gamma_\kappa(x)=\mathrm{e}^{-\kappa x^2}$. Since the support of \(\beta\) is separated from the boundary by a positive
distance, all boundary traces of \(\tilde u\) vanish. The source term in the
pressure equation \eqref{eq:4.21} is
\begin{equation}
 S_0
 =
 2\bigl(
 (\gamma_\kappa')^2(\beta')^2
 -\gamma_\kappa\gamma_\kappa''\beta\beta''
 \bigr).
\label{eq:S0-Gaussian}
\end{equation}
Suppose, for contradiction, that \(\alpha_0\equiv0\). Let
\(\lambda=|\xi|>0\). The Neumann Green function for the operator
\(-\partial_y^2+\lambda^2\) gives
\[
 \widehat{p^{I,0}(\cdot,\cdot,0)}(\xi,0)
 =
 \frac{1}{\lambda}
 \int_0^\infty
 \mathrm{e}^{-\lambda y}\widehat S_0(\xi,y)\,\,\mathrm{d} y.
\]
Since
\[
 -i\xi\,
 \widehat{p^{I,0}(\cdot,\cdot,0)}(\xi,0)
 =
 \widehat{\alpha_0}(\xi)
 =
 0,
\]
we obtain
\begin{equation}
 \int_0^\infty
 \mathrm{e}^{-\lambda y}\widehat S_0(\xi,y)\,\,\mathrm{d} y
 =
 0,
 \qquad \xi\neq0.
\label{eq:pressure-Laplace-zero}
\end{equation}
A direct computation using the Fourier transform of the Gaussian yields
\begin{equation}
 \widehat{(\gamma_\kappa')^2}
 =
 \left(
 \kappa-\frac{\xi^2}{4}
 \right)
 \widehat{\gamma_\kappa^2},
 \qquad
 \widehat{\gamma_\kappa\gamma_\kappa''}
 =
 -\left(
 \kappa+\frac{\xi^2}{4}
 \right)
 \widehat{\gamma_\kappa^2}.
\label{eq:Gaussian-Fourier}
\end{equation}
Define
\[
 J_0(\lambda)
 =
 \int_0^\infty\mathrm{e}^{-\lambda y}\beta^2\,\,\mathrm{d} y,
 \qquad
 J_1(\lambda)
 =
 \int_0^\infty\mathrm{e}^{-\lambda y}(\beta')^2\,\,\mathrm{d} y,
\]
and
\[
 J_2(\lambda)
 =
 \int_0^\infty\mathrm{e}^{-\lambda y}\beta\beta''\,\,\mathrm{d} y.
\]
Substituting \eqref{eq:S0-Gaussian},\eqref{eq:Gaussian-Fourier} into
\eqref{eq:pressure-Laplace-zero} and dividing by the strictly positive
quantity \(\widehat{\gamma_\kappa^2}\), we obtain
\begin{equation}
 \left(
 \kappa-\frac{\lambda^2}{4}
 \right)J_1
 +
 \left(
 \kappa+\frac{\lambda^2}{4}
 \right)J_2
 =
 0.
\label{eq:I-relation1}
\end{equation}
Integrating by parts twice, with all boundary terms vanishing, gives
\[
 J_2
 =
 -J_1+\frac{\lambda^2}{2}J_0.
\]
Hence, \cref{eq:I-relation1} reduces to
\begin{equation}
 J_1(\lambda)
 =
 \left(
 \kappa+\frac{\lambda^2}{4}
 \right)
 J_0(\lambda),
 \qquad \lambda>0.
\label{eq:I-relation2}
\end{equation}
Since both \(\beta^2\) and its first derivative vanish at \(y=0\), another
two integrations by parts yield
\[
 \lambda^2
 \int_0^\infty
 \mathrm{e}^{-\lambda y}\beta(y)^2\,\,\mathrm{d} y
 =
 \int_0^\infty
 \mathrm{e}^{-\lambda y}(\beta(y)^2)''\,\,\mathrm{d} y.
\]
By the uniqueness of the one-sided Laplace transform,
\cref{eq:I-relation2} implies
\[
 (\beta')^2
 =
 \kappa\beta^2
 +\frac14(\beta^2)''.
\]
Equivalently,
\begin{equation}
 (\beta')^2-\beta\beta''
 =
 2\kappa\beta^2.
\label{eq:B-contradiction-ode}
\end{equation}
On any connected component of the set \(\{\beta\neq0\}\), dividing
\cref{eq:B-contradiction-ode} by \(\beta^2\) and using
\[
 \left(\frac{\beta'}{\beta}\right)'
 =
 \frac{\beta\beta''-(\beta')^2}{\beta^2},
\]
we obtain
\[
 \frac{\beta'(y)}{\beta(y)}
 =
 -2\kappa y+c,
 \qquad
 \beta(y)
 =
 C\mathrm{e}^{-\kappa y^2+cy}.
\]
A nontrivial function of this form cannot vanish at a finite endpoint of
the connected component, contradicting the compact support of \(\beta\).
Therefore,
\[
 \alpha_0\not\equiv0.
\]
This completes the proof.
\end{proof}

\begin{remark}
A fully explicit choice is
\[
 \beta_*(y)=
 \begin{cases}
 \exp\!\left[-((y-1)(2-y))^{-1}\right], & 1<y<2,\\
 0, & \text{otherwise}.
 \end{cases}
\]
For every initial datum constructed in Proposition \ref{prop:Gaussian-data}, uniform
convergence without the boundary-layer correction fails. Indeed, the
Taylor expansion at the boundary, together with the embedding
\(H_x^1(\mathbb{R})\hookrightarrow L_x^\infty(\mathbb{R})\), gives
\[
 a(t)=t\alpha_0+O_{L_x^\infty}(t^2).
\]
Since \(\alpha_0\not\equiv0\), one may choose \(t_{\rm u}>0\) sufficiently
small such that \(a(t_{\rm u})\not\equiv0\). As the exact solution satisfies
the no-slip boundary condition,
\[
 0<\left\lVert a(t_{\rm u})\right\rVert_{L_x^\infty}
 \leq
 \left\lVert 
 u^\varepsilon(t_{\rm u})-u^{I,0}(t_{\rm u})\right\rVert_{L^\infty(\Omega)}
 \qquad\text{for every }\varepsilon>0.
\]
More precisely, Corollary \ref{cor:uncorrected-limits} identifies the positive limit as $\|u^{b,0}_1(t_{\rm u})\|_{L^\infty(\Omega_{\mathrm b})}$,whose boundary trace is \(-a(t_{\rm u})\).
\end{remark}

\subsection{Full Moments and Optimality of the Corrected Convergence Rate}
Set $M_0=\int_0^\infty u^{b,0}_1\,dz$ and $c=a_t-a_{xx}$.
Then
\[
 M_0(t)
 =-\frac{2}{\sqrt{\pi}}
 \int_0^t
 \sqrt{h}\,e^{h\partial_x^2}c(t-h)\,dh
 +\int_0^\infty r_a(z,t)\,dz,
\]
and
\[
 \left\|
 M_0(t)+\frac{4}{3\sqrt{\pi}}t^{3/2}\alpha_0
 \right\|_{H_x^1}
 \leq Ct^{5/2}.
\]
Using \eqref{r6:heat-second-taylor}, we further obtain
\[
 M_0(t)
 =
 -\frac{4}{3\sqrt{\pi}}t^{3/2}\alpha_0
 -\frac{t^{5/2}}{15\sqrt{\pi}}
 \bigl(8\beta_0+4\partial_x^2\alpha_0\bigr)
 +o_{H_x^1}(t^{5/2}).
\]
The divergence-free condition together with the matching relation gives
$ u^{I,1}_2|_{y=0}=-\partial_x M_0 $.
Hence,
\begin{equation}
 \left\|
 u^{I,1}_2(\cdot,0,t)
 -\frac{4}{3\sqrt{\pi}}t^{3/2}\partial_x\alpha_0
 \right\|_{L_x^2}
 \leq Ct^{5/2}.
 \label{eq:first-interior-nonzero}
\end{equation}
For any initial data satisfying \(\alpha_0\not\equiv0\), the fact that
\(\alpha_0\in L_x^2(\mathbb{R})\) implies \(\partial_x\alpha_0\not\equiv0\). Therefore,
for every sufficiently small fixed \(t>0\), $ u^{I,1}(t)\neq0$. The Gaussian family constructed above provides explicit initial data
satisfying this nondegeneracy condition. The following elementary
two-scale limit then identifies the exact pointwise coefficient.

\begin{lemma}
\label{lem:two-scale-sup}
Let \(F\) and \(G\) denote an outer field and an inner field, respectively.
Assume that \(F\in C_0(\overline\Omega;\mathbb{R}^d)\) is uniformly continuous, and that
\(G\in C(\mathbb{R}\times\mathbb{R}_+;\mathbb{R}^d)\) is bounded and satisfies
\[
 \sup_x |G(x,z)|\longrightarrow0
 \qquad\text{as } z\to\infty.
\]
Then
\begin{equation}
 \lim_{\varepsilon\downarrow0}
 \left\lVert F(x,y)+G(x,y/\varepsilon)\right\rVert_{L^\infty(\Omega)}
 =
 \max\left\{
 \left\lVert F\right\rVert_{L^\infty(\Omega)},\
 \left\lVert F(x,0)+G(x,z)\right\rVert_{L^\infty_{x,z}}
 \right\}.
\label{eq:two-scale-sup}
\end{equation}
\end{lemma}

\begin{proof}
Fix \(\eta>0\). Choose \(R>0\) such that
\[
 \sup_x |G(x,z)|\leq \eta,
 \quad z\geq R,
\]
and then choose \(d>0\) such that
\[
 |F(x,y)-F(x,0)|\leq \eta,
 \quad 0\leq y\leq d,
\]
uniformly in \(x\). For \(\varepsilon R\leq d\), split the domain into
\(y\leq \varepsilon R\) and \(y\geq \varepsilon R\). On the first region,
\[
 |F(x,y)+G(x,y/\varepsilon)|
 \leq
 |F(x,0)+G(x,y/\varepsilon)|+\eta,
\]
whereas on the second,
\[
 |F(x,y)+G(x,y/\varepsilon)|
 \leq \|F\|_{L^\infty(\Omega)}+\eta.
\]
Hence,
\[
 \limsup_{\varepsilon\downarrow0}
 \|F(x,y)+G(x,y/\varepsilon)\|_{L^\infty(\Omega)}
 \leq
 \max\left\{
 \|F\|_{L^\infty(\Omega)},
 \|F(x,0)+G(x,z)\|_{L^\infty_{x,z}}
 \right\}.
\]

For the reverse inequality, choose approximate maximizers. To recover
\(\|F\|_{L^\infty(\Omega)}\), evaluate at an approximate maximizer
\((x_0,y_0)\) of \(|F|\). If \(y_0>0\), then
\(G(x_0,y_0/\varepsilon)\to0\). If \(y_0=0\), evaluate instead at
\((x_0,\sqrt{\varepsilon})\), for which
\[
 F(x_0,\sqrt{\varepsilon})\to F(x_0,0),
 \quad
 G(x_0,\varepsilon^{-1/2})\to0.
\]
To recover the second term, take an approximate maximizer
\((x_1,z_1)\) of \(|F(x,0)+G(x,z)|\) and evaluate at
\((x_1,\varepsilon z_1)\). Letting first \(\varepsilon\downarrow0\) and then
\(\eta\downarrow0\) proves \eqref{eq:two-scale-sup}.
\end{proof}

\begin{corollary}
\label{cor:corrected-rate-sharp}
Under the assumptions of Theorem \ref{thm:main}, suppose in addition that
\(\alpha_0\not\equiv0\). Then, for every sufficiently small fixed time
\(t_*>0\), there exists a constant \(C_*(t_*)>0\) such that
\begin{align}
&\lim_{\varepsilon\downarrow0}
 \varepsilon^{-1}
 \left\lVert 
 u^\varepsilon(t_*)-u^{I,0}(t_*)
 -(u^{b,0}_1(x,y/\varepsilon,t_*),0)\right\rVert_{L^\infty(\Omega)}
 \notag\\
&\quad=:C_*
 =\max\Bigl\{
 \left\lVert u^{I,1}(t_*)\right\rVert_{L^\infty(\Omega)},
 \notag\\
&\qquad\qquad
 \left\lVert 
 u^{I,1}(x,0,t_*)
 +(u^{b,1}_1(x,z,t_*),u^{b,1}_2(x,z,t_*))\right\rVert_{L^\infty_{x,z}}
 \Bigr\}
 >0.
\label{eq:corrected-sharp-Linf}
\end{align}
Equivalently, the corrected defect satisfies
\begin{equation}
 \left\lVert 
 u^\varepsilon(t_*)-u^{I,0}(t_*)
 -(u^{b,0}_1(x,y/\varepsilon,t_*),0)\right\rVert_{L^\infty(\Omega)}
 =
 C_*\varepsilon+o(\varepsilon).
\label{eq:corrected-sharp-asymptotic}
\end{equation}
Consequently, for the above nondegenerate initial data and every
sufficiently small fixed positive time \(t_*\), the corrected
\(L^\infty\)-convergence rate \(O(\varepsilon)\) is optimal. The nondegenerate
lower bound requires \(t_*>0\) to be fixed; at the initial time, the
defect vanishes.
\end{corollary}

\begin{proof}
At the fixed positive time \(t_*\), the exact error estimate gives
\[
 \varepsilon^{-1}\left\lVert E^\varepsilon(t_*)\right\rVert_{L^\infty(\Omega)}
 \leq C_T\varepsilon^{1/2}\longrightarrow0.
\]
After subtracting the leading-order boundary-layer profile from the composite expansion and dividing by \(\varepsilon\), all second- and higher-order profile terms vanish in \(L^\infty\) as \(\varepsilon\downarrow0\). The remaining two fields are
\[
 u^{I,1}(x,y,t_*)
 +\bigl(u^{b,1}_1(x,y/\varepsilon,t_*),u^{b,1}_2(x,y/\varepsilon,t_*)\bigr).
\]
These fields satisfy the assumptions of Lemma \ref{lem:two-scale-sup}. Therefore, applying that lemma yields \cref{eq:corrected-sharp-Linf}. Since
\[
 u^{I,1}(t_*)\neq0,
\]
the first term in the maximum is strictly positive, and hence \(C_*>0\). Multiplying the limiting relation by \(\varepsilon\) then gives \cref{eq:corrected-sharp-asymptotic}.
\end{proof}

\subsection{Strong Convergence for Finite Exponents and the First-Order Interior Correction}
Throughout this subsection, write
\[
 E_{\rm lead}^\varepsilon
 =u^\varepsilon-u^{I,0}-((u^{b,0}_1)^\varepsilon,0).
\]

Fix $T>0$. The uniform $L^p$ estimate for $E^\varepsilon$, $2\leq p\leq\infty$,
follows from \eqref{eq:initial-error-p}. Using the normal rescaling identity
\[
 \|F(x,y/\varepsilon,t)\|_{L^p(\Omega)}
 =\varepsilon^{1/p}\|F(x,z,t)\|_{L^p(\Omega_{\mathrm b})},
\]
together with the exact decomposition
\begin{align*}
 E_{\rm lead}^\varepsilon-\varepsilon(u^{I,1}+(u^{b,1})^\varepsilon)
 &=
 \varepsilon^2u^{I,2}
 +\varepsilon^2(u^{b,2}_1,u^{b,2}_2)^\varepsilon
 +\varepsilon^3S
 +\varepsilon^3(0,(u^{b,3}_2)^\varepsilon)
 +E^\varepsilon,
\end{align*}
we obtain
\begin{equation}
\begin{aligned}
 \sup_{0\leq t\leq T}
 \|E_{\rm lead}^\varepsilon(t)
   -\varepsilon(u^{I,1}(t)+(u^{b,1})^\varepsilon(t))\|_{L^p(\Omega)}
 \leq
 C_{T,p}\bigl(\varepsilon^2+\varepsilon^{3/2+2/p}\bigr),
 \qquad 2\leq p\leq\infty.
\end{aligned}
\label{eq:shared-first-correction}
\end{equation}
Here \(0<\varepsilon\leq1\), with the convention \(1/\infty=0\).
The higher-order profile terms carry the factors
\[
 \varepsilon^2,\qquad
 \varepsilon^{2+1/p},\qquad
 \varepsilon^3,\qquad
 \varepsilon^{3+1/p},
\]
respectively, and hence are all bounded by \(C\varepsilon^2\).

\begin{corollary}
\label{cor:finite-lp}
Under the assumptions of Theorem \ref{thm:main} and with the notation
introduced above in this subsection, fix \(T>0\). For
\(2\leq p\leq\infty\), with the convention \(1/\infty=0\), we have
\begin{align}
 \sup_{0\leq t\leq T}\|E^\varepsilon(t)\|_{L^p(\Omega)}
 &\leq C_{T,p}\varepsilon^{3/2+2/p},
 \label{eq:finite-lp-full}\\
 \sup_{0\leq t\leq T}\|E_{\rm lead}^\varepsilon(t)\|_{L^p(\Omega)}
 &\leq C_{T,p}\varepsilon.
 \label{eq:finite-lp-leading}
\end{align}
Moreover, for every finite \(2\leq p<\infty\),
\begin{align}
 &\sup_{0\leq t\leq T}
 \left|
 \varepsilon^{-1/p}\|u^\varepsilon(t)-u^{I,0}(t)\|_{L^p(\Omega)}
 -\|u^{b,0}_1(t)\|_{L^p(\Omega_{\mathrm b})}
 \right|
 \leq C_{T,p}\varepsilon^{1-1/p},
 \label{eq:finite-lp-profile}\\
 &\sup_{0\leq t\leq T}
 \|\varepsilon^{-1}E_{\rm lead}^\varepsilon(t)-u^{I,1}(t)\|_{L^p(\Omega)}
 \leq C_{T,p}\varepsilon^{1/p}.
 \label{eq:finite-lp-first-interior}
\end{align}
In particular, the uncorrected convergence rate in \(L^p\), for finite
\(p\), is \(O(\varepsilon^{1/p})\). If a fixed time \(t_*>0\) satisfies
\[
 \|u^{b,0}_1(t_*)\|_{L^p(\Omega_{\mathrm b})}>0,
\]
then the leading-order coefficient is precisely this profile norm, and
hence the rate \(O(\varepsilon^{1/p})\) is sharp at that time.
\end{corollary}

\begin{proof}
The estimate \eqref{eq:finite-lp-full} follows from \eqref{eq:initial-error-p} by taking the supremum over $0\leq t\leq T$. By
\cref{eq:shared-first-correction}, the uniform \(L^p\)-bounds for
\(u^{I,1}\) and \(u^{b,1}\), and the normal rescaling identity, we obtain \eqref{eq:finite-lp-leading}.
Dividing \cref{eq:shared-first-correction} by \(\varepsilon\), for every finite
\(p\geq2\) we have
\[
 \sup_{0\leq t\leq T}
 \|\varepsilon^{-1}E_{\rm lead}^\varepsilon(t)-u^{I,1}(t)\|_{L^p(\Omega)}
 \leq
 C_{T,p}
 \bigl(
 \varepsilon^{1/p}
 +\varepsilon
 +\varepsilon^{1/2+2/p}
 \bigr)
 \leq C_{T,p}\varepsilon^{1/p}.
\]
This proves \cref{eq:finite-lp-first-interior}. Finally, using the decomposition
\[
 u^\varepsilon-u^{I,0}
 =
 ((u^{b,0}_1)^\varepsilon,0)+E_{\rm lead}^\varepsilon,
\]
together with the reverse triangle inequality and the normal rescaling
identity, we obtain \cref{eq:finite-lp-profile}.
\end{proof}
\begin{remark}
The convergence to the interior field in
\eqref{eq:finite-lp-first-interior} is restricted to finite \(p\).
At the \(L^\infty\)-endpoint, the norm of the first-order boundary-layer
correction \((u^{b,1}_1,u^{b,1}_2)^\varepsilon\) is not reduced by the normal rescaling.
Accordingly, the \(L^\infty\)-limit of
\(\varepsilon^{-1}E_{\rm lead}^\varepsilon\), without subtracting \(u^{I,1}\), is described
by the two-scale formula in \cref{eq:corrected-sharp-Linf}. After the
interior field \(u^{I,1}\) is subtracted, the coefficient associated with the
remaining thin-layer norm is given by the endpoint relation stated in
\cref{eq:thin-coefficient-infinity-endpoint}.
\end{remark}

\begin{corollary}
\label{cor:finite-p-coefficients}
Under the assumptions and notation of Corollary \ref{cor:finite-lp}, for every
fixed \(2\leq p<\infty\),
\begin{equation}
 \sup_{0\leq t\leq T}
 \left|
 \varepsilon^{-1}\|E_{\rm lead}^\varepsilon(t)\|_{L^p(\Omega)}
 -\|u^{I,1}(t)\|_{L^p(\Omega)}
 \right|
 \leq C_{T,p}\varepsilon^{1/p}.
 \label{eq:finite-p-leading-coefficient}
\end{equation}
Moreover, the boundary-layer coefficient associated with the first interior
limit satisfies
\begin{align}
 &\sup_{0\leq t\leq T}
 \left|
 \varepsilon^{-1/p}
 \|\varepsilon^{-1}E_{\rm lead}^\varepsilon(t)-u^{I,1}(t)\|_{L^p(\Omega)}
 -\|u^{b,1}(t)\|_{L^p(\Omega_{\mathrm b})}
 \right|
 \notag\\
 &\qquad\leq
 C_{T,p}\left(
 \varepsilon^{1-1/p}+\varepsilon^{1/2+1/p}
 \right).
 \label{eq:finite-p-thin-coefficient}
\end{align}
If \(\alpha_0\not\equiv0\), then, for every sufficiently small fixed
\(t_*>0\), both leading-order coefficients are strictly positive, and
\begin{align*}
 \|E_{\rm lead}^\varepsilon(t_*)\|_{L^p(\Omega)}
 &=
 \varepsilon\|u^{I,1}(t_*)\|_{L^p(\Omega)}
 +O_{T,p}(\varepsilon^{1+1/p}),\\
 \|\varepsilon^{-1}E_{\rm lead}^\varepsilon(t_*)-u^{I,1}(t_*)\|_{L^p(\Omega)}
 &=
 \varepsilon^{1/p}\|u^{b,1}(t_*)\|_{L^p(\Omega_{\mathrm b})}
 +o(\varepsilon^{1/p}).
\end{align*}
Consequently, at such nondegenerate times, the \(O(\varepsilon)\) convergence
rate after subtraction of the leading-order boundary-layer correction and
the \(O(\varepsilon^{1/p})\) rate in \cref{eq:finite-lp-first-interior} are both
sharp.
\end{corollary}

\begin{proof}
The first estimate follows directly from
\cref{eq:finite-lp-first-interior} and the reverse triangle inequality.
For the second estimate, divide \cref{eq:shared-first-correction} by
\(\varepsilon^{1+1/p}\), and use
\[
 \|(u^{b,1})^\varepsilon(t)\|_{L^p(\Omega)}
 =
 \varepsilon^{1/p}\|u^{b,1}(t)\|_{L^p(\Omega_{\mathrm b})}
\]
together with the reverse triangle inequality. This yields
\cref{eq:finite-p-thin-coefficient}.

By \cref{eq:first-interior-nonzero}, the boundary trace
\(u^{I,1}_{2}|_{y=0}\) is nonzero for every sufficiently small fixed
positive time. Indeed, if \(\alpha_0\in L_x^2(\mathbb{R})\) is not identically zero,
then \(\partial_x\alpha_0\not\equiv0\), while the leading-order term of this trace is $ \frac{4}{3\sqrt{\pi}}t^{3/2}\partial_x\alpha_0$.
It follows that $ u^{I,1}(t_*)\neq0$.
Moreover, since $ u^{b,1}_2|_{z=0}=-u^{I,1}_{2}|_{y=0}$,
we also have $ u^{b,1}(t_*)\neq0$.
Both fields possess the previously established finite-\(L^p\) regularity
and the corresponding traces; therefore, their \(L^p\)-norms are strictly
positive. Multiplying the two coefficient estimates by the corresponding
powers of \(\varepsilon\) yields the stated asymptotic formulas.
\end{proof}

\begin{remark}
\label{rem:linfty-corrector-objects}
For all \(0<\varepsilon\leq1\), after subtracting the first-order boundary-layer
correction, we have the uniform strong approximation
\begin{equation}
 \sup_{0\leq t\leq T}
 \|\varepsilon^{-1}E_{\rm lead}^\varepsilon(t)-u^{I,1}(t)-(u^{b,1})^\varepsilon(t)\|_{L^\infty(\Omega)}
 \leq C_T\varepsilon^{1/2}.
 \label{eq:thin-corrector-infinity-remainder}
\end{equation}
Consequently, the boundary-layer coefficient also satisfies the endpoint
formula
\begin{equation}
 \sup_{0\leq t\leq T}
 \left|
 \|\varepsilon^{-1}E_{\rm lead}^\varepsilon(t)-u^{I,1}(t)\|_{L^\infty(\Omega)}
 -\|u^{b,1}(t)\|_{L^\infty(\Omega_{\mathrm b})}
 \right|
 \leq C_T\varepsilon^{1/2}.
 \label{eq:thin-coefficient-infinity-endpoint}
\end{equation}
This is the norm-coefficient endpoint counterpart of
\cref{eq:finite-p-thin-coefficient}, corresponding formally to
\(1/\infty=0\).

Indeed, taking \(p=\infty\) in \cref{eq:shared-first-correction} and
dividing by \(\varepsilon\), the right-hand side becomes
\[
 C_T(\varepsilon+\varepsilon^{1/2})\leq C_T\varepsilon^{1/2},
\]
which proves \cref{eq:thin-corrector-infinity-remainder}. Since $\|(u^{b,1})^\varepsilon(t)\|_{L^\infty(\Omega)}
 =
 \|u^{b,1}(t)\|_{L^\infty(\Omega_{\mathrm b})}$,
the reverse triangle inequality then yields
\cref{eq:thin-coefficient-infinity-endpoint}.
If a fixed time \(t_*>0\) satisfies
\(u^{b,1}(t_*)\neq0\), then
\[
 \|\varepsilon^{-1}E_{\rm lead}^\varepsilon(t_*)-u^{I,1}(t_*)\|_{L^\infty(\Omega)}
 \longrightarrow
 \|u^{b,1}(t_*)\|_{L^\infty(\Omega_{\mathrm b})}
 >0.
\]
Thus, the strong convergence
\(\varepsilon^{-1}E_{\rm lead}^\varepsilon-u^{I,1}\to0\), which holds for finite \(p\), cannot in general be
extended to the endpoint \(p=\infty\). Without subtracting the interior
field, the quantity
\[
 \varepsilon^{-1}E_{\rm lead}^\varepsilon
\]
retains the superposition of the outer and inner fields. At each fixed
positive time, its limiting \(L^\infty\)-norm is therefore still described
by the two-scale maximum formula in \cref{eq:corrected-sharp-Linf}.
\end{remark}

\appendix
\crefalias{section}{appendix}
\section{Derivation of inner and outer profiles}
\label{sec:appendixA}
In this section, we will give a formal derivation of the inner and outer profiles with the corresponding initial and boundary conditions(see Chapter 4 of \cite{Holmes2013} or Appendix A of \cite{WangWen2024} for more detailed illustrations).

\textit{Step 1. The initial and boundary conditions.}

Substituting \eqref{eq:3.1} into initial and boundary conditions \eqref{eq:1.2}, we find the initial and boundary conditions should satisfy
\begin{equation}
	u^{I,0}|_{t=0} = \tilde{u}(x,y),\quad u^{I,j}|_{t=0} =0,j \geq 1, \quad  u^{b,i}|_{t=0}=0,i \geq 0.
	\label{eq:A1}
\end{equation}
and
\begin{equation}
	u^{I,j}(x,0,t) + u^{b,j}(x,0,t) = 0, j \geq 0.
	\label{eq:A2}
\end{equation}

\textit{Step 2. Equations of leading order profiles.}

Plugging \eqref{eq:3.1} into \eqref{eq:1.1}, we have
\begin{equation}
	\begin{aligned}
		&\partial_{t} \sum_{j=0}^{+\infty} \varepsilon^{j}(u_{1}^{I,j}+u_{1}^{b,j})+\left(\sum_{j=0}^{+\infty}\varepsilon^{j}(u^{I,j}+u^{b,j}) \cdot \nabla\right)\sum_{k=0}^{+\infty}\varepsilon^{k}(u_{1}^{I,k}+u_{1}^{b,k}) \\[0.2cm]
		&\qquad+\partial_{x} \sum_{j=0}^{+\infty} \varepsilon^{j}(p^{I, j}+p^{b, j})-\partial_{x}^{2}\sum_{j=0}^{+\infty}\varepsilon^{j}(u_{1}^{I,j}+u_{1}^{b,j})-\varepsilon^{2}\partial_{y}^{2}\sum_{j=0}^{+\infty}\varepsilon^{j}(u_{1}^{I,j}+u_{1}^{b,j})=0.
	\end{aligned}
	\label{eq:A3}
\end{equation}
\begin{equation}
	\begin{aligned}
		&\partial_{t} \sum_{j=0}^{+\infty} \varepsilon^{j}(u_{2}^{I,j}+u_{2}^{b,j})+\left(\sum_{j=0}^{+\infty}\varepsilon^{j}(u^{I,j}+u^{b,j}) \cdot \nabla\right)\sum_{k=0}^{+\infty}\varepsilon^{k}(u_{2}^{I,k}+u_{2}^{b,k}) \\[0.2cm]
		&\qquad+\partial_{y} \sum_{j=0}^{+\infty} \varepsilon^{j}(p^{I, j}+p^{b, j})-\partial_{x}^{2}\sum_{j=0}^{+\infty}\varepsilon^{j}(u_{2}^{I,j}+u_{2}^{b,j})-\varepsilon^{2}\partial_{y}^{2}\sum_{j=0}^{+\infty}\varepsilon^{j}(u_{2}^{I,j}+u_{2}^{b,j})=0.
	\end{aligned}
	\label{eq:A4}
\end{equation}
and
\begin{equation}
	\operatorname{div} \sum_{j=0}^{+\infty} \varepsilon^{j}\left(u^{I, j}+u^{b, j}\right)=0 .
	\label{eq:A5}
\end{equation}
Formally, let $z \rightarrow+\infty$, we get
\begin{equation}
	\left\{\begin{array}{l}
		\partial_t u_{1}^{I,j}+\displaystyle \sum_{\ell=0}^j u^{I, \ell} \cdot \nabla u_{1}^{I, j-\ell}  +\partial_{x} p^{I,j} -\partial_{x}^{2}u_{1}^{I,j}=\partial_{y}^{2}u_{1}^{I,j-2} \\[0.2cm]
		\partial_t u_{2}^{I,j}+\displaystyle \sum_{\ell=0}^j u^{I, \ell} \cdot \nabla u_{2}^{I, j-\ell}  +\partial_{y} p^{I,j} -\partial_{x}^{2}u_{2}^{I,j}=\partial_{y}^{2}u_{2}^{I,j-2} \\[0.2cm]
		\operatorname{div} u^{I, j}=0.
	\end{array}\right.
	\label{eq:A6}
\end{equation}
for $j \geq 0$, where $u_{1}^{I,-1}=u_{1}^{I,-2}=u_{2}^{I,-1}=u_{2}^{I,-2}=0$.
\begin{lemma}
	\label{lemma:A1}
	The zeroth order outer profiles $(u^{I,0},p^{I,0})$ satisfies the limit problem \eqref{eq:1.3}-\eqref{eq:1.4}. The zeroth order inner profiles $u_{2}^{b,0}$, $p^{b,0}$ vanish identically,i.e.,
	\begin{equation}
		u_{2}^{b,0}=0,\quad p^{b,0}=0.
		\label{eq:A7}
	\end{equation}
	The zeroth order inner profile $u_{1}^{b,0}$ satisfies problem \eqref{eq:3.3}.
\end{lemma}
\begin{proof}
    Near the boundary, subtracting $(\ref{eq:A6})_1$ from \eqref{eq:A3},
    we identify the required coefficients using $y=\varepsilon z$ and
    the finite Taylor formula
    \[
    \begin{aligned}
    f(x,\varepsilon z,t)
    ={}&\sum_{\ell=0}^{m-1}\frac{(\varepsilon z)^\ell}{\ell!}
             \partial_y^\ell f(x,0,t)\\
     &+\frac1{(m-1)!}\int_0^{\varepsilon z}
        (\varepsilon z-r)^{m-1}\partial_y^mf(x,r,t)\,dr.
    \end{aligned}
    \]
    The following coefficient series is formal; only the finite
    remainders needed in Section~\ref{subsec:source terms} are used in
    the estimates. For the outer profiles $(u^{I,j},p^{I,j})$, we obtain
	$$
	\sum_{j=-1}^{+\infty} \varepsilon^{j} \mathcal{F}^{j}(x, z, t)=0,
	$$
	where
	$$
	\mathcal{F}^{-1} = (u_{2}^{I,0}(x,0,t)+u_{2}^{b,0}) \partial_{z} u_{1}^{b,0},
	$$
	$$
	\begin{aligned}
		\mathcal{F}^{0} = \partial_{t} u_{1}^{b,0} &+ (u_{1}^{I,0}(x,0,t)+u_{1}^{b,0}) \partial_{x} u_{1}^{b,0}+(u_{2}^{I,1}(x,0,t)+u_{2}^{b,1}+z \partial_{y}u_{2}^{I,0}(x,0,t))\partial_{z}u_{1}^{b,0} \\[0.2cm]
		& +u_{1}^{b,0} \partial_{x}u_{1}^{I,0}(x,0,t) + u_{2}^{b,0}\partial_{y}u_{1}^{I,0}(x,0,t) + \partial_{x}p^{b,0}-\partial_{x}^{2}u_{1}^{b,0} \\[0.2cm]
		& +(u_{2}^{I,0}(x,0,t)+u_{2}^{b,0})\partial_{z}u_{1}^{b,1} -\partial_{z}^{2}u_{1}^{b,0},
	\end{aligned}
	$$
	$$
	\begin{aligned}
		\mathcal{F}^{1}=&z\partial_{y}u_{1}^{I,0}(x,0,t)\partial_{x}u_{1}^{b,0}+zu_{1}^{b,0}\partial_{y}\partial_{x}u_{1}^{I,0}(x,0,t)+\frac{z^{2}}{2}\partial_{y}^{2}u_{2}^{I,0}(x,0,t)\partial_{z}u_{1}^{b,0} \\[0.2cm]
		&+zu_{2}^{b,0}\partial_{y}^{2}u_{1}^{I,0}(x,0,t) +\partial_{t}u_{1}^{b,1}+u_{1}^{I,1}(x,0,t)\partial_{x}u_{1}^{b,0}+u_{1}^{b,1}\partial_{x}u_{1}^{I,0}(x,0,t)+u_{1}^{b,1}\partial_{x}u_{1}^{b,0} \\[0.2cm]
		&+u_{1}^{I,0}(x,0,t)\partial_{x}u_{1}^{b,1}+u_{1}^{b,0}\partial_{x}u_{1}^{I,1}(x,0,t)+u_{1}^{b,0}\partial_{x}u_{1}^{b,1}+z\partial_{y}u_{2}^{I,1}(x,0,t)\partial_{z}u_{1}^{b,0} \\[0.2cm]
		&+u_{2}^{b,1}\partial_{y}u_{1}^{I,0}(x,0,t)+z\partial_{y}u_{2}^{I,0}(x,0,t)\partial_{z}u_{1}^{b,1}+u_{2}^{b,0}\partial_{y}u_{1}^{I,1}(x,0,t)+\partial_{x}p^{b,1}-\partial_{x}^{2}u_{1}^{b,1} \\[0.2cm]
		&+u_{2}^{I,2}(x,0,t)\partial_{z}u_{1}^{b,0}+u_{2}^{b,2}\partial_{z}u_{1}^{b,0}+u_{2}^{I,1}(x,0,t)\partial_{z}u_{1}^{b,1}+u_{2}^{b,1}\partial_{z}u_{1}^{b,1}+u_{2}^{I,0}(x,0,t)\partial_{z}u_{1}^{b,2} \\[0.2cm]
		&+u_{2}^{b,0}\partial_{z}u_{1}^{b,2}-\partial_{z}^{2}u_{1}^{b,1},
	\end{aligned}
	$$
	$$
	\begin{aligned}
		\mathcal{F}^{2}=&\frac{z^{2}}{2}\partial_{y}^{2}u_{1}^{I,0}(x,0,t)\partial_{x}u_{1}^{b,0}+\frac{z^{2}}{2}u_{1}^{b,0}\partial_{y}^{2}\partial_{x}u_{1}^{I,0}(x,0,t)+\frac{z^{3}}{6}\partial_{y}^{3}u_{2}^{I,0}(x,0,t)\partial_{z}u_{1}^{b,0}\\[0.2cm]
		&+\frac{z^{2}}{2}u_{2}^{b,0}\partial_{y}^{3}u_{1}^{I,0}(x,0,t)+z\partial_{y}u_{1}^{I,1}(x,0,t)\partial_{x}u_{1}^{b,0}+zu_{1}^{b,1}\partial_{y}\partial_{x}u_{1}^{I,0}(x,0,t) \\[0.2cm]
		&+z\partial_{y}u_{1}^{I,0}(x,0,t)\partial_{x}u_{1}^{b,1}+zu_{1}^{b,0}\partial_{y}\partial_{x}u_{1}^{I,1}(x,0,t)+\frac{z^{2}}{2}\partial_{y}^{2}u_{2}^{I,1}(x,0,t)\partial_{z}u_{1}^{b,0}\\[0.2cm]
		&+zu_{2}^{b,1}\partial_{y}^{2}u_{1}^{I,0}(x,0,t)+\frac{z^{2}}{2}\partial_{y}^{2}u_{2}^{I,0}(x,0,t)\partial_{z}u_{1}^{b,1}+zu_{2}^{b,0}\partial_{y}^{2}u_{1}^{I,1}(x,0,t)+\partial_{t}u_{1}^{b,2} \\[0.2cm]
		&+u_{1}^{I,2}(x,0,t)\partial_{x}u_{1}^{b,0}+u_{1}^{b,2}\partial_{x}u_{1}^{I,0}(x,0,t)+u_{1}^{b,2}\partial_{x}u_{1}^{b,0}+u_{1}^{I,1}(x,0,t)\partial_{x}u_{1}^{b,1}\\[0.2cm]
		&+u_{1}^{b,1}\partial_{x}u_{1}^{I,1}(x,0,t)+u_{1}^{b,1}\partial_{x}u_{1}^{b,1}+u_{1}^{I,0}(x,0,t)\partial_{x}u_{1}^{b,2}+u_{1}^{b,0}\partial_{x}u_{1}^{I,2}(x,0,t)\\[0.2cm]
		&+u_{1}^{b,0}\partial_{x}u_{1}^{b,2}+z\partial_{y}u_{2}^{I,2}(x,0,t)\partial_{z}u_{1}^{b,0}+u_{2}^{b,2}\partial_{y}u_{1}^{I,0}(x,0,t)+z\partial_{y}u_{2}^{I,1}(x,0,t)\partial_{z}u_{1}^{b,1}\\[0.2cm]
		&+u_{2}^{b,1}\partial_{y}u_{1}^{I,1}(x,0,t)+z\partial_{y}u_{2}^{I,0}(x,0,t)\partial_{z}u_{1}^{b,2}+u_{2}^{b,0}\partial_{y}u_{1}^{I,2}(x,0,t)+\partial_{x}p^{b,2}-\partial_{x}^{2}u_{1}^{b,2}\\[0.2cm]
		&+u_{2}^{I,3}(x,0,t)\partial_{z}u_{1}^{b,0}+u_{2}^{b,3}\partial_{z}u_{1}^{b,0}+u_{2}^{I,2}(x,0,t)\partial_{z}u_{1}^{b,1}+u_{2}^{b,2}\partial_{z}u_{1}^{b,1}+u_{2}^{I,1}(x,0,t)\partial_{z}u_{1}^{b,2}\\[0.2cm]
		&+u_{2}^{b,1}\partial_{z}u_{1}^{b,2}+u_{2}^{I,0}(x,0,t)\partial_{z}u_{1}^{b,3}+u_{2}^{b,0}\partial_{z}u_{1}^{b,3}-\partial_{z}^{2}u_{1}^{b,2},\\[0.2cm]
		&......
	\end{aligned}
	$$
	Similar to the derivation of $\mathcal{F}^{j}$, from \eqref{eq:A4} and $(\ref{eq:A6})_{2}$, we obtain
	$$
	\sum_{j=-1}^{+\infty} \varepsilon^{j} \mathcal{G}^j(x, z, t)=0,
	$$
	where
	$$
	\mathcal{G}^{-1}=(u_{2}^{I,0}(x,0,t)+u_{2}^{b,0}) \partial_{z}u_{2}^{b,0}+\partial_{z}p^{b,0},
	$$
	$$
	\begin{aligned}
		\mathcal{G}^{0}=&\partial_{t}u_{2}^{b,0}+(u_{1}^{I,0}(x,0,t)+u_{1}^{b,0})\partial_{x}u_{2}^{b,0}+(u_{2}^{I,1}(x,0,t)+u_{2}^{b,1}+z\partial_{y}u_{2}^{I,0}(x,0,t))\partial_{z}u_{2}^{b,0} \\[0.2cm]
		&+u_{1}^{b,0}\partial_{x}u_{2}^{I,0}(x,0,t)+u_{2}^{b,0}\partial_{y}u_{2}^{I,0}(x,0,t)-\partial_{x}^{2}u_{2}^{b,0}+(u_{2}^{I,0}(x,0,t)+u_{2}^{b,0})\partial_{z}u_{2}^{b,1} \\[0.2cm]
		&+\partial_{z}p^{b,1}-\partial_{z}^{2}u_{2}^{b,0},
	\end{aligned}
	$$
	$$
	\begin{aligned}
		\mathcal{G}^{1}=&z\partial_{y}u_{1}^{I,0}(x,0,t)\partial_{x}u_{2}^{b,0}+zu_{1}^{b,0}\partial_{y}\partial_{x}u_{2}^{I,0}(x,0,t)+\frac{z^{2}}{2}\partial_{y}^{2}u_{2}^{I,0}(x,0,t)\partial_{z}u_{2}^{b,0} \\[0.2cm]
		&+zu_{2}^{b,0}\partial_{y}^{2}u_{2}^{I,0}(x,0,t)+\partial_{t}u_{2}^{b,1}+u_{1}^{I,1}(x,0,t)\partial_{x}u_{2}^{b,0}+u_{1}^{b,1}\partial_{x}u_{2}^{I,0}(x,0,t)+u_{1}^{b,1}\partial_{x}u_{2}^{b,0}\\[0.2cm]
		&+u_{1}^{I,0}(x,0,t)\partial_{x}u_{2}^{b,1}+u_{1}^{b,0}\partial_{x}u_{2}^{I,1}(x,0,t)+u_{1}^{b,0}\partial_{x}u_{2}^{b,1}+z\partial_{y}u_{2}^{I,1}(x,0,t)\partial_{z}u_{2}^{b,0}\\[0.2cm]
		&+u_{2}^{b,1}\partial_{y}u_{2}^{I,0}(x,0,t)+z\partial_{y}u_{2}^{I,0}(x,0,t)\partial_{z}u_{2}^{b,1}+u_{2}^{b,0}\partial_{y}u_{2}^{I,1}(x,0,t)-\partial_{x}^{2}u_{2}^{b,1}\\[0.2cm]
		&+u_{2}^{I,2}(x,0,t)\partial_{z}u_{2}^{b,0}+u_{2}^{b,2}\partial_{z}u_{2}^{b,0}+u_{2}^{I,1}(x,0,t)\partial_{z}u_{2}^{b,1}+u_{2}^{b,1}\partial_{z}u_{2}^{b,1}+u_{2}^{I,0}(x,0,t)\partial_{z}u_{2}^{b,2}\\[0.2cm]
		&+u_{2}^{b,0}\partial_{z}u_{2}^{b,2}+\partial_{z}p^{b,2}-\partial_{z}^{2}u_{2}^{b,1},
	\end{aligned}
	$$
	$$
	\begin{aligned}
		\mathcal{G}^{2}=&\frac{z^{2}}{2}\partial_{y}^{2}u_{1}^{I,0}(x,0,t)\partial_{x}u_{2}^{b,0}+\frac{z^{2}}{2}u_{1}^{b,0}\partial_{y}^{2}\partial_{x}u_{2}^{I,0}(x,0,t)+\frac{z^{3}}{6}\partial_{y}^{3}u_{2}^{I,0}(x,0,t)\partial_{z}u_{2}^{b,0} \\[0.2cm]
		&+\frac{z^{2}}{2}u_{2}^{b,0}\partial_{y}^{3}u_{2}^{I,0}(x,0,t)+z\partial_{y}u_{1}^{I,1}(x,0,t)\partial_{x}u_{2}^{b,0}+zu_{1}^{b,1}\partial_{y}\partial_{x}u_{2}^{I,0}(x,0,t) \\[0.2cm]
		&+z\partial_{y}u_{1}^{I,0}(x,0,t)\partial_{x}u_{2}^{b,1}+zu_{1}^{b,0}\partial_{y}\partial_{x}u_{2}^{I,1}(x,0,t)+\frac{z^{2}}{2}\partial_{y}^{2}u_{2}^{I,1}(x,0,t)\partial_{z}u_{2}^{b,0} \\[0.2cm]
		&+zu_{2}^{b,1}\partial_{y}^{2}u_{2}^{I,0}(x,0,t)+\frac{z^{2}}{2}\partial_{y}^{2}u_{2}^{I,0}(x,0,t)\partial_{z}u_{2}^{b,1}+zu_{2}^{b,0}\partial_{y}^{2}u_{2}^{I,1}(x,0,t)+\partial_{t}u_{2}^{b,2} \\[0.2cm]
		&+u_{1}^{I,2}(x,0,t)\partial_{x}u_{2}^{b,0}+u_{1}^{b,2}\partial_{x}u_{2}^{I,0}(x,0,t)+u_{1}^{b,2}\partial_{x}u_{2}^{b,0}+u_{1}^{I,1}(x,0,t)\partial_{x}u_{2}^{b,1} \\[0.2cm]
		&+u_{1}^{b,1}\partial_{x}u_{2}^{I,1}(x,0,t)+u_{1}^{b,1}\partial_{x}u_{2}^{b,1}+u_{1}^{I,0}(x,0,t)\partial_{x}u_{2}^{b,2}+u_{1}^{b,0}\partial_{x}u_{2}^{I,2}(x,0,t) \\[0.2cm]
		&+u_{1}^{b,0}\partial_{x}u_{2}^{b,2}+z\partial_{y}u_{2}^{I,2}(x,0,t)\partial_{z}u_{2}^{b,0}+u_{2}^{b,2}\partial_{y}u_{2}^{I,0}(x,0,t)+z\partial_{y}u_{2}^{I,1}(x,0,t)\partial_{z}u_{2}^{b,1} \\[0.2cm]
		&+u_{2}^{b,1}\partial_{y}u_{2}^{I,1}(x,0,t)+z\partial_{y}u_{2}^{I,0}(x,0,t)\partial_{z}u_{2}^{b,2}+u_{2}^{b,0}\partial_{y}u_{2}^{I,2}(x,0,t)-\partial_{x}^{2}u_{2}^{b,2} \\[0.2cm]
		&+u_{2}^{I,3}(x,0,t)\partial_{z}u_{2}^{b,0}+u_{2}^{b,3}\partial_{z}u_{2}^{b,0}+u_{2}^{I,2}(x,0,t)\partial_{z}u_{2}^{b,1}+u_{2}^{b,2}\partial_{z}u_{2}^{b,1}+u_{2}^{I,1}(x,0,t)\partial_{z}u_{2}^{b,2} \\[0.2cm]
		&+u_{2}^{b,1}\partial_{z}u_{2}^{b,2}+u_{2}^{I,0}(x,0,t)\partial_{z}u_{2}^{b,3}+u_{2}^{b,0}\partial_{z}u_{2}^{b,3}+\partial_{z}p^{b,3}-\partial_{z}^{2}u_{2}^{b,2}, \\[0.2cm]
		&......
	\end{aligned}
	$$
	Moreover, from \eqref{eq:A5} and $(\ref{eq:A6})_{3}$, we have
	\begin{equation}
		\partial_{x}u_{1}^{b,j}+\partial_{z}u_{2}^{b,j+1}=0,\quad \forall j \geq0.
		\label{eq:A8}
	\end{equation}
	In particular, $\partial_{z}u_{2}^{b,0}=0$, combining with $\displaystyle \lim _{z \rightarrow +\infty} u_{2}^{b,0}(x,z,t)=0$ yield
	\begin{equation}
		u_{2}^{b,0}=0.
		\label{eq:A9}
	\end{equation}
	which together with the boundary condition \eqref{eq:A2} implies
	\begin{equation}
		u_{2}^{I,0}(x,0,t)=0.
		\label{eq:A10}
	\end{equation}
	Letting $j=0$ in \eqref{eq:A6}, combining with \eqref{eq:A1} and \eqref{eq:A10}, we find that $(u^{I,0},p^{I,0})$ satisfies the limit problem \eqref{eq:1.3}-\eqref{eq:1.4}. Next, from $\mathcal{G}^{-1}=0$ and \eqref{eq:A9}, we have
	$$
	\partial_{z}p^{b,0}=0,
	$$
	which implies
	\begin{equation}
		p^{b,0}=0.
		\label{eq:A11}
	\end{equation}
	From$\mathcal{F}^{0}=0$, \eqref{eq:A8}-\eqref{eq:A11},the initial condition \eqref{eq:A1} and the boundary condition \eqref{eq:A2}, we find that $u_{1}^{b,0}$ satisfies the problem \eqref{eq:3.3}.The proof is complete.
\end{proof}

\textit{Step 3. Equations of first order profiles.} From \eqref{eq:A8}, we have the following Corollary.
\begin{corollary}
	\label{corollary:A2}
	The first order boundary layer profile $u_{2}^{b,1}$ satisfies
	\begin{equation}
		u_{2}^{b,1}=\int_{z}^{+\infty} \partial_{x}u_{1}^{b,0}(x,s,t)ds.
		\label{eq:A12}
	\end{equation}
	As a consequence(using \eqref{eq:A2}), we have the following boundary conditions:
	\begin{equation}
		\label{eq:A13}
		u_{2}^{I,1}(x,0,t)=-\int_{0}^{+\infty}\partial_{x}u_{1}^{b,0}(x,s,t)ds.
	\end{equation}
\end{corollary}
\begin{lemma}
	\label{lemma:A3}
	The first order outer profiles $(u^{I,1},p^{I,1})$ satisfies problem \eqref{eq:3.5}. The first order inner profile $p^{b,1}=0$. The first order inner profile $u_{1}^{b,1}$ satisfies problem \eqref{eq:3.6}.
\end{lemma}
\begin{proof}
	Taking $j=1$ in \eqref{eq:A6}, we have $(u^{I,1},p^{I,1})$ satisfies
	$$
	\left\{\begin{array}{l}
		\partial_t u^{I, 1}+u^{I, 1} \cdot \nabla u^{I,0}+u^{I,0} \cdot \nabla u^{I,1}+\nabla p^{I, 1}-\partial_{x}^{2} u^{I, 1}=0, \\[0.2cm]
		\partial_x u_1^{I, 1}+\partial_y u_2^{I, 1}=0.
	\end{array}\right.
	$$
	which together with \eqref{eq:A1}, \eqref{eq:A2}, and \eqref{eq:A13} implies that $(u^{I,1},p^{I,1})$ satisfies the linear problem \eqref{eq:3.5}. Next, from $\mathcal{G}^{0}=0$,\eqref{eq:A9},and \eqref{eq:A10} we have
	$$
	\partial_{z}p^{b,1}=0,
	$$
	which implies
	\begin{equation}
		\label{eq:A14}
		p^{b,1}=0
	\end{equation}
	Moreover, from $\mathcal{F}^{1}=0$, \eqref{eq:A8}-\eqref{eq:A10} and \eqref{eq:A14}, we have
	\begin{equation}
		\begin{aligned}
			\partial_t u_{1}^{b,1} &- \partial_{x}^{2} u_{1}^{b,1} - \partial_{z}^{2} u_{1}^{b,1} +\left(\overline{u_{1}^{I,0}} + u_{1}^{b,0}\right) \partial_x u_{1}^{b,1}\\[0.2cm]
			&+ \left( \overline{u_{2}^{I,1}} + u_{2}^{b,1} + z \overline{\partial_y u_{2}^{I,0}}\right) \partial_z u_{1}^{b,1}+\left( \partial_x u_{1}^{b,0}+\overline{\partial_x u_{1}^{I,0}}\right) u_{1}^{b,1} \\[0.2cm]
			&+\left(z \overline{\partial_{y} \partial_{x} u_{1}^{I,0}} u_{1}^{b,0} +u_{1}^{b,0} \overline{\partial_x u_{1}^{I,1}} +u_{2}^{b,1} \overline{\partial_y u_{1}^{I,0}} \right) \\[0.2cm]
			&+\left( \overline{u_{1}^{I,1}} + z \overline{\partial_y u_{1}^{I,0}}\right) \partial_x u_{1}^{b,0}+\left(z \overline{\partial_y u_{2}^{I,1}} + \frac{z^2}{2} \overline{\partial_{y}^{2} u_{2}^{I,0}}\right) \partial_z u_{1}^{b,0}\\[0.2cm]
			&+\left(\int_{z}^{+\infty}\partial_{x}u_{1}^{b,1}(x,s,t)ds - \int_{0}^{+\infty}\partial_{x}u_{1}^{b,1}(x,s,t)ds\right)\partial_z u_{1}^{b,0}=0.
		\end{aligned}
	\end{equation}
	which together \eqref{eq:A1} and \eqref{eq:A2} implies that $u_{1}^{b,1}$ satisfies problem \eqref{eq:3.6}. The proof is complete.
\end{proof}

\textit{Step 4. Equations of second order profiles.} From \eqref{eq:A8}, we have the following Corollary.
\begin{corollary}
	\label{corollary:A4}
	The second boundary layer profile $u_{2}^{b,2}$ satisfies
	\begin{equation}
		\label{eq:A16}
		u_{2}^{b,2} = \int_{z}^{+\infty} \partial_{x}u_{1}^{b,1}(x,s,t)ds.
	\end{equation}
	As a consequence(using \ref{eq:A2}), we have the following boundary conditions:
	\begin{equation}
		\label{eq:A17}
		u_{2}^{I,2}(x,0,t)=-\int_{0}^{+\infty}\partial_{x}u_{1}^{b,1}(x,s,t)ds.
	\end{equation}
\end{corollary}
\begin{lemma}
	\label{lemma:A5}
	The second order outer profiles $(u^{I,2},p^{I,2})$ satisfies problem \eqref{eq:3.8}. The second order inner profiles $p^{b,2}$ and $u_{1}^{b,2}$ satisfy problem \eqref{eq:3.7},\eqref{eq:3.9}, respectively.
\end{lemma}
\begin{proof}
	Taking $j=2$ in \eqref{eq:A6}, we find $(u^{I,2},p^{I,2})$ satisfies
	$$
	\left\{\begin{array}{l}
		\partial_t u^{I, 2}+u^{I, 2} \cdot \nabla u^{I,0}+u^{I,0} \cdot \nabla u^{I,2}+\nabla p^{I, 2}-\partial_{x}^{2} u^{I, 2}=\partial_{y}^{2}u^{I,0}-u^{I,1} \cdot \nabla u^{I,1}, \\[0.2cm]
		\partial_x u_1^{I, 2}+\partial_y u_2^{I, 2}=0.
	\end{array}\right.
	$$
	which together with \eqref{eq:A1} and \eqref{eq:A17} implies that $(u^{I,2},p^{I,2})$ satisfies the linear problem \eqref{eq:3.8}. Next from $\mathcal{G}^{1}=0$, \eqref{eq:A9}-\eqref{eq:A10} and $\displaystyle \lim _{z \rightarrow +\infty} p^{b,2}(x,z,t)=0$, we have 
	\begin{equation}
		\label{eq:A18}
		p^{b,2}=-\int_{z}^{+\infty}\mathcal{P}_{2}(x,s,t)ds.
	\end{equation}
	where
	$$
	\begin{aligned}
		\mathcal{P}_{2}(x,z,t)&=-\partial_{t}u_{2}^{b,1}+\partial_{x}^{2}u_{2}^{b,1}+\partial_{z}^{2}u_{2}^{b,1}-\overline{u_{1}^{I,0}}\partial_{x}u_{2}^{b,1} \\[0.2cm]
		&-(\partial_{x}u_{2}^{b,1}+\overline{\partial_{x}u_{2}^{I,1}})u_{1}^{b,0}-\overline{\partial_{y}u_{2}^{I,0}}u_{2}^{b,1}-(u_{2}^{b,1}+\overline{u_{2}^{I,1}})\partial_{z}u_{2}^{b,1} \\[0.2cm]
		&-z\overline{\partial_{y}\partial_{x}u_{2}^{I,0}}u_{1}^{b,0}-z\overline{\partial_{y}u_{2}^{I,0}}\partial_{z}u_{2}^{b,1}.
	\end{aligned}
	$$
	Moreover, from $\mathcal{F}^{2}=0$,\eqref{eq:A8}-\eqref{eq:A10}, we have
	\begin{equation}
		\begin{aligned}
			\partial_t u_{1}^{b,2} &- \partial_{x}^{2} u_{1}^{b,2} - \partial_{z}^{2} u_{1}^{b,2} +\left(\overline{u_{1}^{I,0}} + u_{1}^{b,0}\right) \partial_x u_{1}^{b,2}\\[0.2cm]
			&+ \left( \overline{u_{2}^{I,1}} + u_{2}^{b,1} + z \overline{\partial_y u_{2}^{I,0}}\right) \partial_z u_{1}^{b,2}+\left( \partial_x u_{1}^{b,0}+\overline{\partial_x u_{1}^{I,0}}\right) u_{1}^{b,2} \\[0.2cm]
			&+\left(\overline{u_{1}^{I,1}}+u_{1}^{b,1}+z\overline{\partial_{y}u_{1}^{I,0}}\right) \partial_{x} u_{1}^{b,1} +\left(z \overline{\partial_{y}u_{2}^{I,1}} +\frac{z^2}{2}\overline{\partial_{y}^{2}u_{2}^{I,0}}\right)\partial_{z}u_{1}^{b,1} \\[0.2cm] &+\left(\int_{z}^{+\infty}\partial_{x}u_{1}^{b,1}(x,s,t)ds-\int_{0}^{+\infty}\partial_{x}u_{1}^{b,1}(x,s,t)ds \right) \partial_{z}u_{1}^{b,1} \\[0.2cm]
			&+ \left(\overline{u_{1}^{I,2}}+z \overline{\partial_{y}u_{1}^{I,1}} + \frac{z^2}{2} \overline{\partial_{y}^{2}u_{1}^{I,0}}\right) \partial_{x} u_{1}^{b,0} + \left( z \overline{\partial_{y} u_{2}^{I,2}} + \frac{z^2}{2}\overline{\partial_{y}^{2}u_{2}^{I,1}} + \frac{z^3}{6} \overline{\partial_{y}^{3}u_{2}^{I,0}}\right) \partial_{z} u_{1}^{b,0} \\[0.2cm]
			&+\left(\int_{z}^{+\infty}\partial_{x}u_{1}^{b,2}(x,s,t)ds - \int_{0}^{+\infty}\partial_{x}u_{1}^{b,2}(x,s,t)ds\right) \partial_{z} u_{1}^{b,0} \\[0.2cm]
			&+\frac{z^2}{2}\overline{\partial_{y}^{2}\partial_{x}u_{1}^{I,0}} u_{1}^{b,0} + z \overline{\partial_{y}\partial_{x}u_{1}^{I,1}} u_{1}^{b,0} + \overline{\partial_{x} u_{1}^{I,2}} u_{1}^{b,0}+ z \overline{\partial_{y}^{2}u_{1}^{I,0}} u_{2}^{b,1}+ \overline{\partial_{y}u_{1}^{I,1}} u_{2}^{b,1} \\[0.2cm]
			&+z \overline{\partial_{y} \partial_{x} u_{1}^{I,0}} u_{1}^{b,1}+ \overline{\partial_{x}u_{1}^{I,1}} u_{1}^{b,1}+\overline{\partial_{y} u_{1}^{I,0}} u_{2}^{b,2} + \partial_{x} p^{b,2}=0.
		\end{aligned}
	\end{equation}
	Combining \eqref{eq:A1}-\eqref{eq:A2},we find that $u_{1}^{b,2}$ satisfies problem \eqref{eq:3.9}. The proof is complete.
\end{proof}

\textit{Step 5. Some higher order profiles.}
\begin{lemma}
	\label{lemma:A6}
    The terminal normal velocity profile $u_2^{b,3}$ satisfies \eqref{eq:3.10}.
\end{lemma}
\begin{proof}
    Integrating \eqref{eq:A8} for $j=2$ and using decay at infinity gives
    \begin{equation}
        \label{eq:A20}
        u_2^{b,3}(x,z,t)=\int_z^\infty
              \partial_xu_1^{b,2}(x,s,t)\,ds.
    \end{equation}
\end{proof}

\section{Expressions of some source terms.}
\label{app:residual}
In this section, we present the complete expressions of some source terms in \eqref{eq:3.39}, i.e.,
\begin{equation}
	\label{eq:B1}
	\begin{aligned}
		\mathcal{R}_{1}^{\varepsilon}=-\displaystyle \sum_{i=1}^{19} F_{i},
	\end{aligned}
\end{equation}
where 
\begin{align*}
-F_{1}&=\varepsilon^{2}\partial_{t}u_{1}^{b,2}+\varepsilon^{3}\partial_{t}S_{1}, \\[0.1cm]
-F_{2}&=(u_{1}^{I,0}-\overline{u_{1}^{I,0}}-y\overline{\partial_{y}u_{1}^{I,0}})\partial_{x}u_{1}^{b,0}+\varepsilon(u_{1}^{I,0}-\overline{u_{1}^{I,0}})\partial_{x}u_{1}^{b,1}, \\[0.1cm]
-F_{3}&=\varepsilon^{2}u_{1}^{I,0}\partial_{x}u_{1}^{b,2}+\varepsilon^{3}u_{1}^{I,0}\partial_{x}S_{1}+\varepsilon^{3}u_{1}^{I,1}\partial_{x}u_{1}^{I,2}, \\[0.1cm]
-F_{4}&=\varepsilon(u_{1}^{I,1}-\overline{u_{1}^{I,1}})\partial_{x}u_{1}^{b,0}, \\[0.1cm]
-F_{5}&=\varepsilon^{2}u_{1}^{I,1}\partial_{x}u_{1}^{b,1}+\varepsilon^{3}u_{1}^{I,1}\partial_{x}u_{1}^{b,2}+\varepsilon^{4}u_{1}^{I,1}\partial_{x}S_{1}, \\[0.1cm]
-F_{6}&=\varepsilon^{2}u_{1}^{I,2}\partial_{x}(\varepsilon u_{1}^{I,1}+\varepsilon^{2}u_{1}^{I,2}+u_{1}^{b,0}+\varepsilon u_{1}^{b,1}+\varepsilon^{2}u_{1}^{b,2}+\varepsilon^{3}S_{1}), \\[0.1cm]
-F_{7}&=u_{1}^{b,0}(\partial_{x}u_{1}^{I,0}-\overline{\partial_{x}u_{1}^{I,0}}-y\overline{\partial_{y}\partial_{x}u_{1}^{I,0}})+\varepsilon u_{1}^{b,0}(\partial_{x}u_{1}^{I,1}-\overline{\partial_{x}u_{1}^{I,1}}), \\[0.1cm]
-F_{8}&=u_{1}^{b,0}\partial_{x}(\varepsilon^{2}u_{1}^{I,2}+\varepsilon^{2}u_{1}^{b,2}+\varepsilon^{3}S_{1}), \\[0.1cm]
-F_{9}&=\varepsilon u_{1}^{b,1}(\partial_{x}u_{1}^{I,0}-\overline{\partial_{x}u_{1}^{I,0}}), \\[0.1cm]
-F_{10}&=\varepsilon u_{1}^{b,1}\partial_{x}(\varepsilon u_{1}^{I,1}+\varepsilon^{2}u_{1}^{I,2}+\varepsilon u_{1}^{b,1}+\varepsilon^{2}u_{1}^{b,2}+\varepsilon^{3}S_{1}), \\[0.1cm]
-F_{11}&=(\varepsilon^{2}u_{1}^{b,2}+\varepsilon^{3}S_{1})\partial_{x}u_{1}^{a}, \\[0.1cm]
-F_{12}&=\frac{1}{\varepsilon}(u_{2}^{I,0}-y\overline{\partial_{y}u_{2}^{I,0}}-\frac{y^{2}}{2}\overline{\partial_{y}^{2}u_{2}^{I,0}})\partial_{z}u_{1}^{b,0} +(u_{2}^{I,0}-y\overline{\partial_{y}u_{2}^{I,0}})\partial_{z}u_{1}^{b,1}, \\[0.1cm]
-F_{13}&=(u_{2}^{I,0}-\overline{u_{2}^{I,0}})\partial_{y}(\varepsilon^{2}u_{1}^{b,2}+\varepsilon^{3}S_{1})+\varepsilon^{3}u_{2}^{I,1}\partial_{y}u_{1}^{I,2}, \\[0.1cm]
-F_{14}&=(u_{2}^{I,1}-\overline{u_{2}^{I,1}}-y\overline{\partial_{y}u_{2}^{I,1}})\partial_{z}u_{1}^{b,0}+\varepsilon (u_{2}^{I,1}-\overline{u_{2}^{I,1}})\partial_{z}u_{1}^{b,1}, \\[0.1cm]
-F_{15}&=\varepsilon u_{2}^{I,1} \partial_{y}(\varepsilon^{2}u_{1}^{b,2}+\varepsilon^{3}S_{1}) +\varepsilon^{2}u_{2}^{I,2}\partial_{y}(\varepsilon u_{1}^{I,1}+\varepsilon^{2}u_{1}^{I,2}), \\[0.1cm]
-F_{16}&=\varepsilon (u_{2}^{I,2}-\overline{u_{2}^{I,2}}) \partial_{z} u_{1}^{b,0}+\varepsilon^{2}u_{2}^{I,2} \partial_{y}(\varepsilon u_{1}^{b,1}+\varepsilon^{2}u_{1}^{b,2}+\varepsilon^{3}S_{1}), \\[0.1cm]
-F_{17}&=\varepsilon u_{2}^{b,1}(\partial_{y}u_{1}^{I,0}-\overline{\partial_{y}u_{1}^{I,0}}) +\varepsilon u_{2}^{b,1}\partial_{y}(\varepsilon u_{1}^{I,1}+\varepsilon^{2}u_{1}^{I,2}+\varepsilon^{2}u_{1}^{b,2}+\varepsilon^{3}S_{1}), \\[0.1cm]
-F_{18}&=\varepsilon^{2}u_{2}^{b,2}\partial_{y}(u_{1}^{a}-u_{1}^{b,0})+(\varepsilon^{3}u_{2}^{b,3}+\varepsilon^{3}S_{2})\partial_{y}u_{1}^{a}+\varepsilon^{2}\partial_{x}p^{b,2}, \\[0.1cm]
-F_{19}&=-\partial_{x}^{2}(\varepsilon^{2}u_{1}^{b,2}+\varepsilon^{3}S_{1}) -\varepsilon^{2}\partial_{y}^{2}(\varepsilon u_{1}^{I,1}+\varepsilon^{2}u_{1}^{I,2}+\varepsilon^{2}u_{1}^{b,2}+\varepsilon^{3}S_{1}).
\end{align*} 
Finally, for the components of $\mathcal{R}_{2}^{\varepsilon}$, we have
\begin{equation}
	\label{eq:B2}
	\begin{aligned}
		\mathcal{R}_{2}^{\varepsilon}=-\sum_{i=1}^{20}G_{i},
	\end{aligned}
\end{equation}
where 
\begin{align*}
-G_{1}&=\varepsilon^{2}\partial_{t}u_{2}^{b,2}+\varepsilon^{3}\partial_{t}u_{2}^{b,3}+\varepsilon^{3}\partial_{t}S_{2}, \\[0.1cm]
-G_{2}&=\varepsilon(u_{1}^{I,0}-\overline{u_{1}^{I,0}})\partial_{x}u_{2}^{b,1}, \\[0.1cm]
-G_{3}&=u_{1}^{I,0}\partial_{x}(\varepsilon^{2}u_{2}^{b,2}+\varepsilon^{3}u_{2}^{b,3}+\varepsilon^{3}S_{2}) +\varepsilon^{3}u_{1}^{I,1}\partial_{x}u_{2}^{I,2}, \\[0.1cm]
-G_{4}&=\varepsilon u_{1}^{I,1}\partial_{x}(\varepsilon u_{2}^{b,1}+\varepsilon^{2}u_{2}^{b,2}+\varepsilon^{3}u_{2}^{b,3}+\varepsilon^{3}S_{2}), \\[0.1cm]
-G_{5}&=\varepsilon^{2}u_{1}^{I,2}\partial_{x}(\varepsilon u_{2}^{I,1}+\varepsilon^{2}u_{2}^{I,2}+\varepsilon u_{2}^{b,1}+\varepsilon^{2}u_{2}^{b,2}+\varepsilon^{3}u_{2}^{b,3}+\varepsilon^{3}S_{2}), \\[0.1cm]
-G_{6}&=(\partial_{x}u_{2}^{I,0}-\overline{\partial_{x}u_{2}^{I,0}}-y\overline{\partial_{y}\partial_{x}u_{2}^{I,0}})u_{1}^{b,0}+\varepsilon(\partial_{x}u_{2}^{I,1}-\overline{\partial_{x}u_{2}^{I,1}})u_{1}^{b,0}, \\[0.1cm]
-G_{7}&=u_{1}^{b,0}\partial_{x}(\varepsilon^{2}u_{2}^{I,2}+\varepsilon^{2}u_{2}^{b,2}+\varepsilon^{3}u_{2}^{b,3}+\varepsilon^{3}S_{2}), \\[0.1cm]
-G_{8}&=\varepsilon(\partial_{x}u_{2}^{I,0}-\overline{\partial_{x}u_{2}^{I,0}})u_{1}^{b,1}, \\[0.1cm]
-G_{9}&=\varepsilon u_{1}^{b,1}\partial_{x}(\varepsilon u_{2}^{I,1}+\varepsilon^{2}u_{2}^{I,2}+\varepsilon u_{2}^{b,1}+\varepsilon^{2}u_{2}^{b,2}+\varepsilon^{3}u_{2}^{b,3}+\varepsilon^{3}S_{2}), \\[0.1cm]
-G_{10}&=(\varepsilon^{2}u_{1}^{b,2}+\varepsilon^{3}S_{1})\partial_{x}(u_{2}^{a}-u_{2}^{b,0}), \\[0.1cm]
-G_{11}&=(u_{2}^{I,0}-\overline{u_{2}^{I,0}}-y\overline{\partial_{y}u_{2}^{I,0}})\partial_{z}u_{2}^{b,1}+\varepsilon(u_{2}^{I,0}-\overline{u_{2}^{I,0}})\partial_{z}u_{2}^{b,2}, \\[0.1cm]
-G_{12}&=u_{2}^{I,0}(\varepsilon^{2}\partial_{z}u_{2}^{b,3}+\varepsilon^{3}\partial_{y}S_{2})+\varepsilon^{3}u_{2}^{I,1}\partial_{y}u_{2}^{I,2}, \\[0.1cm]
-G_{13}&=\varepsilon(u_{2}^{I,1}-\overline{u_{2}^{I,1}})\partial_{z}u_{2}^{b,1}+\varepsilon u_{2}^{I,1}(\varepsilon \partial_{z}u_{2}^{b,2}+\varepsilon^{2}\partial_{z}u_{2}^{b,3}+\varepsilon^{3}\partial_{y}S_{2}), \\[0.1cm]
-G_{14}&=\varepsilon^{2}u_{2}^{I,2}\partial_{y}(\varepsilon u_{2}^{I,1}+\varepsilon^{2}u_{2}^{I,2}), \\[0.1cm]
-G_{15}&=\varepsilon^{2}u_{2}^{I,2}(\partial_{z}u_{2}^{b,1}+\varepsilon\partial_{z}u_{2}^{b,2}+\varepsilon^{2}\partial_{z}u_{2}^{b,3}+\varepsilon^{3}\partial_{y}S_{2}), \\[0.1cm]
-G_{16}&=\varepsilon u_{2}^{b,1}(\partial_{y}u_{2}^{I,0}-\overline{\partial_{y}u_{2}^{I,0}}), \\[0.1cm]
-G_{17}&=\varepsilon u_{2}^{b,1}\partial_{y}(\varepsilon u_{2}^{I,1}+\varepsilon^{2}u_{2}^{I,2})+\varepsilon u_{2}^{b,1}(\varepsilon \partial_{z}u_{2}^{b,2}+\varepsilon^{2}\partial_{z}u_{2}^{b,3}+\varepsilon^{3}\partial_{y}S_{2}), \\[0.1cm]
-G_{18}&=(\varepsilon^{2}u_{2}^{b,2}+\varepsilon^{3}u_{2}^{b,3}+\varepsilon^{3}S_{2})\partial_{y}(u_{2}^{a}-u_{2}^{b,0}), \\[0.1cm]
-G_{19}&=-\partial_{x}^{2}(\varepsilon^{2}u_{2}^{b,2}+\varepsilon^{3}u_{2}^{b,3}+\varepsilon^{3}S_{2}), \\[0.1cm]
-G_{20}&=-\varepsilon^{2}\partial_{y}^{2}(\varepsilon u_{2}^{I,1}+\varepsilon^{2}u_{2}^{I,2})-\varepsilon^{2}(\partial_{z}^{2}u_{2}^{b,2}+\varepsilon\partial_{z}^{2}u_{2}^{b,3}+\varepsilon^{3}\partial_{y}^{2}S_{2}). \\[0.1cm]
\end{align*} 

\paragraph{Research provenance and AI-assisted contributions.} 
An earlier version of the pointwise convergence result in Theorem \ref{thm:main} formed the principal result of Siwei Chen's undergraduate thesis. The original manuscript required $H^{20}$ initial regularity; the present $H^4$ formulation incorporates the subsequent developments described below.

The authors used ChatGPT (OpenAI), with \mbox{GPT-5.5} and \mbox{GPT-5.6}, in developing the fixed-time tail limit \eqref{eq:tail-limit} in Theorem \ref{thm:thickness} and the corner-profile results in Theorem \ref{thm:corner}. \mbox{GPT-6 Astra} provided the proof ideas and draft arguments that enabled the reduction of the initial regularity assumption from $H^{20}$ to $H^4$. \mbox{DeepSeek-V4} was used to organize and rewrite proofs generated with the GPT models.

The authors manually reviewed the AI-generated arguments for the results identified above and reorganized and rewrote them in their own words. The final organization and presentation of all proofs and text were determined by the authors. The authors take full responsibility for the mathematical correctness, originality, and integrity of the entire manuscript, including every argument developed with AI assistance.

\paragraph{Acknowledgments.}
 Y.~H. Wang was supported in part by the National Natural Science Foundation of China (Grant No.~12401274) and the Natural Science Foundation of Hunan Province (Grant No.~2024JJ6302).

\paragraph{Data availability.}
No new data were created or analyzed in this study.

\paragraph{Conflict of interest.}
The authors declare that they have no conflict of interest.

\end{document}